\pdfoutput=1
\documentclass{amsart} % reqno places equation numbers on the right
\newif\ifpersonal
\newif\ifpersonalsection
\usepackage{amsmath,amsthm,amssymb,mathrsfs,mathtools,bm,centernot} % math related
\usepackage{microtype,lmodern} % latex technical issues
\usepackage{enumitem,comment,braket,xspace,tikz-cd,adjustbox} % utilities
\usepackage[utf8]{inputenc} % input encoding
\usepackage[T1]{fontenc} % font encoding 4518376
\usepackage[centering,vscale=0.7,hscale=0.7]{geometry}
\usepackage[hidelinks,linktoc=all]{hyperref}
\usepackage[capitalize]{cleveref}
\allowdisplaybreaks

\numberwithin{equation}{section}
\theoremstyle{plain}
\newtheorem{theorem}[equation]{Theorem}
\newtheorem{lemma}[equation]{Lemma}
\newtheorem*{lemma*}{Lemma}

\newtheorem{claim}[equation]{Claim}
\newtheorem*{claim*}{Claim}

\newtheorem{proposition}[equation]{Proposition}
\newtheorem*{proposition*}{Proposition}
\newtheorem{corollary}[equation]{Corollary}
\theoremstyle{definition}
\newtheorem{definition}[equation]{Definition}
\newtheorem{definition-theorem}[equation]{Definition-Theorem}
\newtheorem{definition-lemma}[equation]{Definition-Lemma}
\newtheorem{construction}[equation]{Construction}
\newtheorem{assumption}[equation]{Assumption}
\newtheorem{notation}[equation]{Notation}
\newtheorem{example}[equation]{Example}
\newtheorem{remark}[equation]{Remark}

\ifpersonal
\newcommand{\todo}[1]{\textcolor{red}{(Todo: #1)}}
\newcommand{\personal}[1]{\textcolor[rgb]{0,0,1}{(Personal: #1)}}
\newcommand{\discussion}[1]{\textcolor{violet}{(Discussion: #1)}}
\else
\newcommand{\todo}[1]{\ignorespaces}
\newcommand{\personal}[1]{\ignorespaces}
\newcommand{\discussion}[1]{\ignorespaces}
\fi

\providecommand{\abs}[1]{\lvert#1\rvert}
\providecommand{\norm}[1]{\lVert#1\rVert}

\newcommand{\bbA}{\mathbb A}
\newcommand{\bbB}{\mathbb B}
\newcommand{\bbC}{\mathbb C}
\newcommand{\bbD}{\mathbb D}

\newcommand{\bbG}{\mathbb G}

\newcommand{\bbN}{\mathbb N}
\newcommand{\bbP}{\mathbb P}
\newcommand{\bbQ}{\mathbb Q}
\newcommand{\bbR}{\mathbb R}
\newcommand{\bbT}{\mathbb T}

\newcommand{\bbZ}{\mathbb Z}

\newcommand{\fC}{\mathfrak C}
\newcommand{\fD}{\mathfrak D}

\newcommand{\fF}{\mathfrak F}
\newcommand{\fG}{\mathfrak G}
\newcommand{\fH}{\mathfrak H}

\newcommand{\fS}{\mathfrak S}
\newcommand{\fT}{\mathfrak T}
\newcommand{\fU}{\mathfrak U}

\newcommand{\fX}{\mathfrak X}
\newcommand{\fY}{\mathfrak Y}
\newcommand{\fZ}{\mathfrak Z}

\newcommand{\cA}{\mathcal A}

\newcommand{\cC}{\mathcal C}

\newcommand{\cE}{\mathcal E}

\newcommand{\cH}{\mathcal H}

\newcommand{\cK}{\mathcal K}

\newcommand{\cM}{\mathcal M}
\newcommand{\cN}{\mathcal N}
\newcommand{\cO}{\mathcal O}

\newcommand{\cS}{\mathcal S}
\newcommand{\cT}{\mathcal T}
\newcommand{\cU}{\mathcal U}
\newcommand{\cV}{\mathcal V}
\newcommand{\cW}{\mathcal W}
\newcommand{\cX}{\mathcal X}

\newcommand{\cZ}{\mathcal Z}

\newcommand{\bP}{\mathbf P}

\newcommand{\sC}{\mathscr C}

\makeatletter
\let\save@mathaccent\mathaccent
\newcommand*\if@single[3]{%
	\setbox0\hbox{${\mathaccent"0362{#1}}^H$}%
	\setbox2\hbox{${\mathaccent"0362{\kern0pt#1}}^H$}%
	\ifdim\ht0=\ht2 #3\else #2\fi
}
\newcommand*\rel@kern[1]{\kern#1\dimexpr\macc@kerna}
\newcommand*\widebar[1]{\@ifnextchar^{{\wide@bar{#1}{0}}}{\wide@bar{#1}{1}}}
\newcommand*\wide@bar[2]{\if@single{#1}{\wide@bar@{#1}{#2}{1}}{\wide@bar@{#1}{#2}{2}}}
\newcommand*\wide@bar@[3]{%
	\begingroup
	\def\mathaccent##1##2{%
		\let\mathaccent\save@mathaccent
		\if#32 \let\macc@nucleus\first@char \fi
		\setbox\z@\hbox{$\macc@style{\macc@nucleus}_{}$}%
		\setbox\tw@\hbox{$\macc@style{\macc@nucleus}{}_{}$}%
		\dimen@\wd\tw@
		\advance\dimen@-\wd\z@
		\divide\dimen@ 3
		\@tempdima\wd\tw@
		\advance\@tempdima-\scriptspace
		\divide\@tempdima 10
		\advance\dimen@-\@tempdima
		\ifdim\dimen@>\z@ \dimen@0pt\fi
		\rel@kern{0.6}\kern-\dimen@
		\if#31
		\overline{\rel@kern{-0.6}\kern\dimen@\macc@nucleus\rel@kern{0.4}\kern\dimen@}%
		\advance\dimen@0.4\dimexpr\macc@kerna
		\let\final@kern#2%
		\ifdim\dimen@<\z@ \let\final@kern1\fi
		\if\final@kern1 \kern-\dimen@\fi
		\else
		\overline{\rel@kern{-0.6}\kern\dimen@#1}%
		\fi
	}%
	\macc@depth\@ne
	\let\math@bgroup\@empty \let\math@egroup\macc@set@skewchar
	\mathsurround\z@ \frozen@everymath{\mathgroup\macc@group\relax}%
	\macc@set@skewchar\relax
	\let\mathaccentV\macc@nested@a
	\if#31
	\macc@nested@a\relax111{#1}%
	\else
	\def\gobble@till@marker##1\endmarker{}%
	\futurelet\first@char\gobble@till@marker#1\endmarker
	\ifcat\noexpand\first@char A\else
	\def\first@char{}%
	\fi
	\macc@nested@a\relax111{\first@char}%
	\fi
	\endgroup
}
\makeatother

\newcommand{\oB}{\widebar B}
\newcommand{\oC}{\widebar C}

\newcommand{\oF}{\widebar F}

\newcommand{\oI}{\widebar I}
\newcommand{\oJ}{\widebar J}

\newcommand{\oM}{\widebar M}
\newcommand{\oP}{\widebar P}

\newcommand{\oS}{\widebar S}

\newcommand{\oU}{\widebar U}

\newcommand{\oX}{\widebar X}
\newcommand{\oY}{\widebar Y}
\newcommand{\oZ}{\widebar Z}

\newcommand{\oGamma}{\widebar\Gamma}
\newcommand{\oLambda}{\widebar\Lambda}

\newcommand{\oPsi}{\widebar\Psi}
\newcommand{\oSigma}{\widebar\Sigma}

\newcommand{\odelta}{\widebar\delta}

\newcommand{\otheta}{\widebar\theta}

\newcommand{\hC}{\widehat C}
\newcommand{\hD}{\widehat D}

\newcommand{\hF}{\widehat F}

\newcommand{\hL}{\widehat L}

\newcommand{\hR}{\widehat R}
\newcommand{\hS}{\widehat S}

\newcommand{\hY}{\widehat Y}

\newcommand{\hGamma}{\widehat\Gamma}

\newcommand{\tA}{\widetilde A}

\newcommand{\tD}{\widetilde D}
\newcommand{\tE}{\widetilde E}

\newcommand{\tH}{\widetilde H}

\newcommand{\tQ}{\widetilde Q}
\newcommand{\tR}{\widetilde R}
\newcommand{\tS}{\widetilde S}

\newcommand{\tV}{\widetilde V}
\newcommand{\tW}{\widetilde W}

\newcommand{\tY}{\widetilde Y}

\newcommand{\tbeta}{\widetilde\beta}
\newcommand{\tgamma}{\widetilde\gamma}
\newcommand{\tdelta}{\widetilde\delta}

\newcommand{\tchi}{\widetilde\chi}

\newcommand{\hooklongrightarrow}{\lhook\joinrel\longrightarrow}
\newcommand{\hooklongleftarrow}{\longleftarrow\joinrel\rhook}
\newcommand{\longto}{\longrightarrow}

\newcommand{\hh}{\widehat h}
\newcommand{\hv}{\widehat v}

\newcommand{\hgamma}{\widehat\gamma}

\newcommand{\tf}{\widetilde f}

\newcommand{\tk}{\widetilde k}

\newcommand{\tw}{\widetilde w}

\newcommand{\ob}{\widebar b}

\newcommand{\of}{\widebar f}

\newcommand{\oh}{\widebar h}
\newcommand{\ol}{\bar l}

\newcommand{\op}{\widebar p}

\newcommand{\os}{\widebar s}

\newcommand{\ov}{\widebar v}
\newcommand{\ow}{\widebar w}

\newcommand{\obbD}{\widebar{\bbD}}

\newcommand{\obbB}{\widebar{\bbB}}

\newcommand{\ocM}{\widebar{\cM}}

\newcommand{\obP}{\widebar{\bP}}
\newcommand{\ocX}{\widebar{\cX}}

\newcommand{\tcK}{\widetilde{\cK}}

\newcommand{\tcS}{\widetilde{\cS}}

\newcommand{\tcV}{\widetilde\cV}

\newcommand{\ff}{\mathfrak f}
\newcommand{\fm}{\mathfrak m}

\newcommand{\fw}{\mathfrak w}
\newcommand {\fd} {\mathfrak{d}}

\newcommand{\rt}{\mathrm{t}}
\DeclareFontFamily{U}{BOONDOX-calo}{\skewchar\font=45 }
\DeclareFontShape{U}{BOONDOX-calo}{m}{n}{<-> s*[1.05] BOONDOX-r-calo}{}
\DeclareFontShape{U}{BOONDOX-calo}{b}{n}{<-> s*[1.05] BOONDOX-b-calo}{}
\DeclareMathAlphabet{\mathcalboondox}{U}{BOONDOX-calo}{m}{n}
\newcommand{\cf}{\mathcalboondox f}

\newcommand{\Cut}{\mathrm{Cut}}
\newcommand{\DC}{\mathrm{DC}}
\newcommand{\Glue}{\mathrm{Glue}}
\newcommand{\GHKK}{\mathrm{GHKK}}
\newcommand{\ISk}{\mathrm{ISk}}
\newcommand{\NT}{\mathrm{NT}}
\newcommand{\Pairs}{\mathrm{Pairs}}
\newcommand{\SH}{\mathrm{SH}}
\newcommand{\SP}{\mathrm{SP}}
\newcommand{\RSP}{\mathrm{RSP}}
\newcommand{\TC}{\mathrm{TC}}

\newcommand{\an}{\mathrm{an}}

\newcommand{\bir}{\mathrm{bir}}
\newcommand{\dom}{\mathrm{dom}}
\newcommand{\ess}{\mathrm{ess}}

\newcommand{\ev}{\mathrm{ev}}

\newcommand{\gp}{\mathrm{gp}}
\newcommand{\id}{\mathrm{id}}
\newcommand{\idt}{\mathrm{idt}}

\newcommand{\init}{\mathrm{init}}

\newcommand{\irr}{\mathrm{irr}}

\newcommand{\marked}{\mathrm{marked}}

\newcommand{\npr}{\mathrm{npr}}
\newcommand{\pr}{\mathrm{pr}}

\newcommand{\sd}{\mathrm{sd}}

\newcommand{\sing}{\mathrm{sing}}
\newcommand{\st}{\mathrm{st}}
\newcommand{\sm}{\mathrm{sm}}
\newcommand{\tail}{\mathrm{tail}}
\newcommand{\tr}{\mathrm{tr}}

\newcommand{\trop}{\mathrm{trop}}

\newcommand{\prin}{\mathrm{prin}}
\newcommand{\ver}{\mathrm{ver}}
\newcommand{\hor}{\mathrm{hor}}

\newcommand{\Gm}{\bbG_\mathrm{m}}

\newcommand{\Ih}{I^\mathrm{h}}
\newcommand{\Iv}{I^\mathrm{v}}
\newcommand{\Zaffine}{$\bbZ$-affine\xspace}
\newcommand{\inv}{^{-1}}
\newcommand{\kanal}{$k$-analytic\xspace}
\newcommand{\kc}{k^\circ}
\newcommand{\kcc}{k^{\circ\circ}}

\newcommand{\llp}{(\!(}

\newcommand{\rrp}{)\!)}

\DeclareMathOperator{\Bl}{Bl}

\DeclareMathOperator{\CH}{CH}

\DeclareMathOperator{\Coeff}{Coeff}

\DeclareMathOperator{\Cont}{Cont}

\DeclareMathOperator{\Frac}{Frac}
\DeclareMathOperator{\Hom}{Hom}

\DeclareMathOperator{\NE}{NE}
\DeclareMathOperator{\Nef}{Nef}

\DeclareMathOperator{\Proj}{Proj}
\DeclareMathOperator{\Sk}{Sk}
\DeclareMathOperator{\oSk}{\overline{Sk}}

\DeclareMathOperator{\Spec}{Spec}
\DeclareMathOperator{\Spf}{Spf}
\DeclareMathOperator{\Sp}{Sp}
\DeclareMathOperator{\Star}{Star}
\DeclareMathOperator{\Trop}{Trop}
\DeclareMathOperator{\TV}{TV}
\DeclareMathOperator{\Wall}{Wall}
\DeclareMathOperator{\can}{can}

\DeclareMathOperator{\degtwig}{degtwig}

\DeclareMathOperator{\length}{length}
\DeclareMathOperator*{\ord}{ord}
\DeclareMathOperator{\rank}{rank}
\DeclareMathOperator{\supp}{supp}
\DeclareMathOperator{\Trace}{Trace}

\DeclareMathOperator{\val}{val}

\renewenvironment{abstract}{%
	\quotation
	\small
	\textbf{\textit{\abstractname.}} % with a normal space
}{\endquotation}

\begin{document}
	
\title[The Frobenius structure theorem]
{The Frobenius structure theorem for affine log Calabi-Yau varieties containing a torus}
\author{Sean Keel}
\address{Sean Keel, Department of Mathematics, 1 University Station C1200, Austin, TX 78712-0257, USA}
\email{keel@math.utexas.edu}
\author{Tony Yue YU}
\address{Tony Yue YU, Department of Mathematics M/C 253-37, California Institute of Technology, 1200 E California Blvd, Pasadena, CA 91125, USA}
\email{yuyuetony@gmail.com}
\date{January 10, 2022}
%\date{\today}

\maketitle

\begin{abstract}
Let $U$ be an affine log Calabi-Yau variety containing an open algebraic torus.
We show that the naive counts of rational curves in $U$ uniquely determine a commutative associative algebra equipped with a compatible multilinear form.
This proves a variant of the Frobenius structure conjecture by Gross-Hacking-Keel in mirror symmetry, and the spectrum of this algebra is supposed to give the hypothetical mirror family.
Although the statement of our theorem involves only elementary algebraic geometry, our proof employs Berkovich non-archimedean analytic methods.
We construct the structure constants of the algebra via counting non-archimedean analytic disks in the analytification of $U$.
%Our counts are naive counts in the sense that they are simply cardinalities of finite sets, without need of virtual fundamental classes.
We establish various properties of the counting, notably deformation invariance, symmetry, gluing formula and convexity.
In the special case when $U$ is a Fock-Goncharov skew-symmetric X-cluster variety, we prove that our algebra generalizes, and gives a direct geometric construction of, the mirror algebra of Gross-Hacking-Keel-Kontsevich.
The comparison is proved via a canonical scattering diagram constructed from counts of infinitesimal non-archimedean analytic cylinders, without using the Kontsevich-Soibelman algorithm.
Several combinatorial conjectures of GHKK, as well as the positivity in the Laurent phenomenon, follow readily from the geometric description.
%We build a canonical scattering diagram directly by counting infinitesimal non-archimedean cylinders, without using the Kontsevich-Soibelman algorithm.
%Our geometric approach implies directly the positivity of the coefficients of the scattering functions, and hence the positivity in the Laurent phenomenon for cluster algebras.
%Moreover, the broken line convexity conjecture of GHKK also follows from the geometric description.
%We introduce a new notion of skeletal curves in non-archimedean geometry, which plays a fundamental role throughout our theory.
\end{abstract}

%\textcolor{red}{Preliminary version, beware of typos.}

\setcounter{tocdepth}{1} % hide subsections in tableofcontents
\tableofcontents

\section{Introduction and statements of main results} \label{sec:introduction}

We begin with a brief summary of our main results, followed by precise definitions and statements.

Let $U$ be a smooth affine log Calabi-Yau variety containing an open algebraic torus, and $U\subset Y$ any snc compactification.
We show that the naive counts of rational curves in $U$ uniquely determine a commutative associative algebra $A$ equipped with a compatible multilinear form.
This is directly inspired by a very similar conjecture of Gross-Hacking-Keel in mirror symmetry, known as the Frobenius structure conjecture, see \cref{rem:after_main_theorem}.
The mirror algebra $A$ is a free $\bbZ[\NE(Y)]$-module, with basis $\Sk(U,\bbZ)$, the integer points in the essential skeleton of $U$.
We construct the algebra structure on $A$ by giving non-negative integer structure constants as naive counts of non-archimedean analytic disks, see \cref{def:structure_constant}.
The counts are naive in the sense that they are simply cardinalities of finite sets, without use of virtual fundamental classes.
We prove moreover that the spectrum of this algebra is the total space of a family whose generic fiber is an affine log Calabi-Yau variety of the same dimension as $U$, with at worst log canonical singularities.

We build a canonical scattering diagram directly by counting infinitesimal non-archimedean cylinders, without using the Kontsevich-Soibelman algorithm, see \cref{sec:intro:scattering}.
In the special case where $U$ is a Fock-Goncharov skew symmetric $X$-cluster variety, we prove that our (geometrically defined) scattering diagram agrees with the (combinatorially constructed) scattering diagram of \cite{Gross_Canonical_bases}, see \cref{thm:scattering_diagram_comparison}.
As a consequence, we show that our algebra generalizes, and in particular gives a new and direct geometric construction of, the mirror algebra of Gross-Hacking-Keel-Kontsevich \cite{Gross_Canonical_bases}, (which under additional assumptions is isomorphic to the associated Fomin-Zelevinski upper cluster algebra), see \cref{thm:X-cluster_intro}.
We refer to \cref{sec:intro:scattering} for the implications of this comparison theorem, including the broken-line convexity conjecture, the positivity of the coefficients of scattering functions (thus the positivity in the Laurent phenomenon for cluster algebras), and the independence of the mirror algebra on the cluster structure.

The \emph{heart} of our work is a simple definition of counts associated to trees in the essential skeleton $\Sk(U)$ of $U$, building on the ideas of \cite{Yu_Enumeration_of_holomorphic_cylinders_I}, as naive counts of non-archimedean analytic curves with the given intrinsic skeleton, independent of any choice of non-archimedean SYZ fibration, see \cref{sec:intro:naive_counts}.
We establish various properties of such counts, notably deformation invariance (\cref{thm:deformation_invariance_truncated}), symmetry (\cref{thm:forgetting_interior_marked_points}), gluing formula (\cref{thm:gluing_concatenate}) and convexity (\cref{thm:Ftrop_structure_constants}).

\bigskip
Now we give precise statements of our results. 
\medskip

Let $k$ be any field of characteristic zero.
Let $U$ be a connected affine smooth  log Calabi-Yau $k$-variety, with volume form $\omega$, and $k(U)$ the field of rational functions on $U$.
Let
\begin{multline*}
\Sk(U,\bbZ)\coloneqq\{0\}\sqcup\\
\big\{m\nu\ \big|\ m\in\bbN_{>0},\ \nu\text{ is a divisorial valuation on }k(U)\text{ where }\omega\text{ has a pole}\big\}.
\end{multline*}

Fix a projective snc compactification $U\subset Y$ with complement $D\coloneqq Y\setminus U$.

\begin{definition}[see \cref{fig:Frobenius_form}]\label{def:counting_Frobenius}
Given $n\ge 2$, $\bP\coloneqq(P_1,\dots,P_n)$ with $P_j \in \Sk(U,\bbZ)$, and a curve class $\beta\in\NE(Y)\subset N_1(Y,\bbZ)$, we define a number $\eta(\bP,\beta)$ counting rational curves in $U$ as follows:

Let $B\coloneqq\set{j | P_j\neq 0}$.
Write $P_j=m_j\nu_j$ for $j\in B$.
Modifying the compactification $U\subset Y$ by a blowup $b\colon (\tY,\tD)\to (Y,D)$, we can assume each $\nu_j$ has divisorial center $D_j\subset\tD$.
Let $H(\bP,\beta)$ be the space of maps
\[f\colon[\bbP^1,(p_1,\dots,p_n,s)]\longrightarrow\tY\]
such that
\begin{enumerate}
	\item $p_1,\dots,p_n,s$ are distinguished marked points of $\bbP^1$,
	\item for each $j\in B$, $f(p_j)$ lies in the open stratum $D_j^\circ$,
	\item $f\inv(\tD)=\sum_{j\in B} m_j p_j$,
	\item $(b\circ f)_*[\bbP^1]=\beta$.
\end{enumerate}
One can show that the map
\[\Phi\coloneqq(\dom,\ev_s)\colon H(\bP,\beta)\longrightarrow \cM_{0,n+1}\times U\]
taking modulus of domain and evaluation at $s$ is finite étale over a Zariski dense open of the target (see \cref{prop:smoothness_all}). 
We define $\eta(\bP,\beta)$ to be the degree of the finite étale map above.
The fiber of $\Phi$ over the generic point, and in particular the degree $\eta(\bP,\beta)$, is independent of the blowup $b$.
Note the counts here are naive counts (as opposed to counts defined using virtual fundamental classes).
\begin{figure}[!ht]
	\centering
	\setlength{\unitlength}{0.35\textwidth}
	\begin{picture} (1,1)
		\put(0,0){\includegraphics[width=\unitlength]{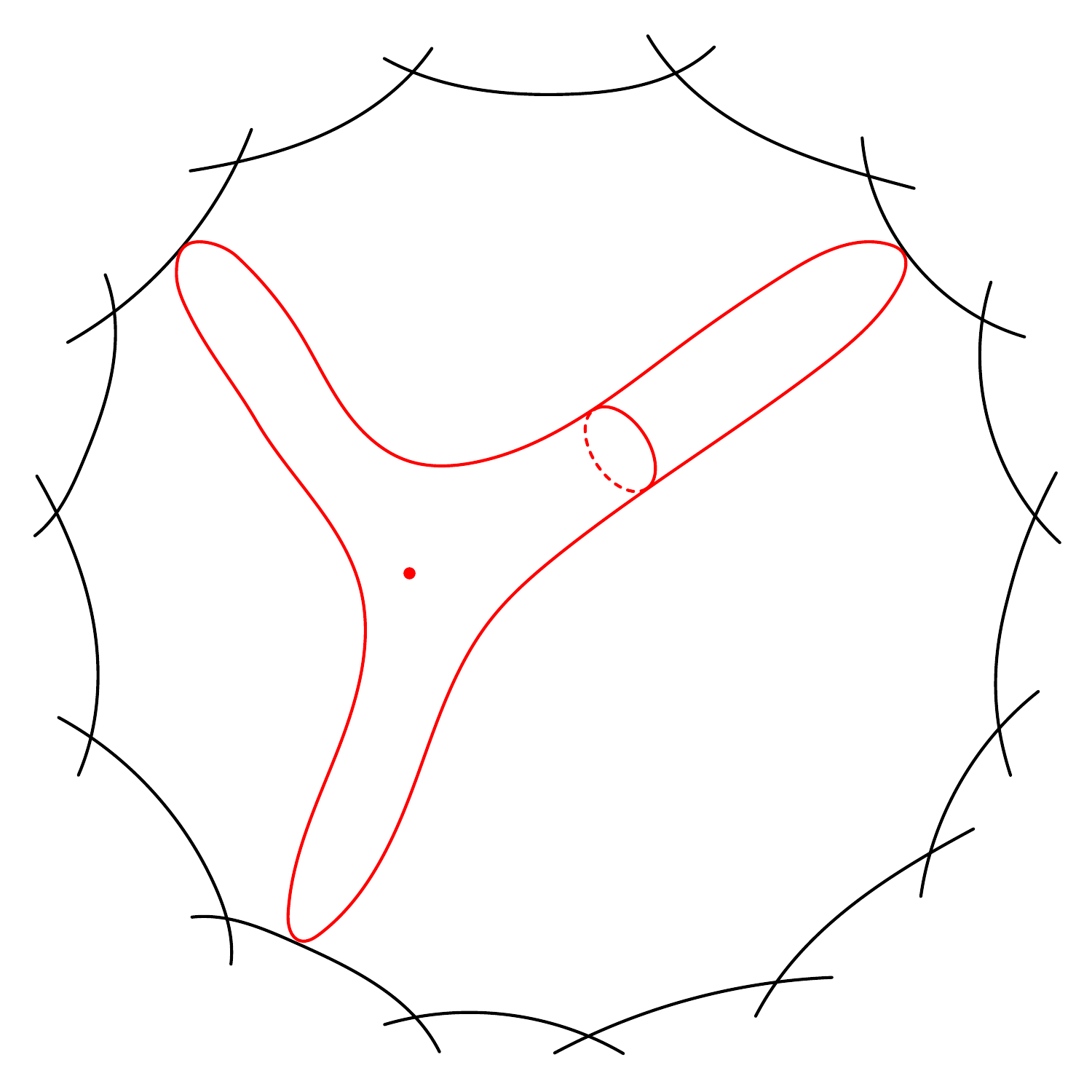}}
		\put(0.09,0.8){$D_1$}
		\put(0.84,0.77){$D_2$}
		\put(0.23,0.07){$D_3$}
		\put(0.38,0.43){$s$}
	\end{picture}
	\caption{}
	\label{fig:Frobenius_form}
\end{figure}
\end{definition}

Next we assemble the numbers $\eta(\bP,\beta)$ into generating series in the following way:
let
\begin{align*}
	R &\coloneqq \bbZ[\NE(Y)] \coloneqq \bigoplus_{\beta \in \NE(Y)} \bbZ \cdot z^{\beta},\quad \text{the monoid ring of $\NE(Y)$ over $\bbZ$},\\
	A &\coloneqq R^{(\Sk(U,\bbZ))} \coloneqq\bigoplus_{P\in\Sk(U,\bbZ)} R\cdot\theta_P,\quad \text{the free $R$-module with basis $\Sk(U,\bbZ)$}.
\end{align*}
Define $\Trace\colon A \to R$ taking coefficient of $\theta_0$.
For $n\ge 2$, let $\braket{,\dots,}_n\colon A^n \to R$ be the $R$-multilinear map with
\[\braket{\theta_{P_1},\dots,\theta_{P_n}}_n = \sum_{\beta \in \NE(Y)} \eta(P_1,\dots,P_n,\beta)z^{\beta},\]
called the \emph{Frobenius multilinear form}.
The affineness of $U$ implies that the sum above is finite (see \cref{lem:finitely_many_beta}).

\begin{theorem}[Frobenius structure theorem] \label{thm:main}
	Assume $U$ contains an open split algebraic torus.
	The following hold:
	\begin{enumerate}
		\item \label{thm:main:non-degeneracy} For every $n\ge 2$, the $R$-multilinear map $\braket{,\dots,}_n\colon A^n\to R$ is non-degenerate, i.e.\ the induced map $A\to \Hom_R(A^{\otimes(n-1)},R)$ is injective.
		\item \label{thm:main:algebra}There exists a unique finitely generated commutative associative $R$-algebra structure on $A$ such that $\theta_0 = 1$
		and 
		\[\braket{a_1,\dots,a_n}_n = \Trace(a_1 a_2\cdots a_n)\quad\text{for every }n\ge 2.\]
				\item \label{thm:main:family} The restriction of the family
		\[\cX\coloneqq\Spec(A)\longto\Spec(R)\]
		over $\bbQ\supset\bbZ$
		is a flat family of affine varieties of same dimension as $U$, each fiber is Gorenstein, semi-log-canonical and $K$-trivial.
		The generic fiber is log canonical and log Calabi-Yau.
	\end{enumerate}
\end{theorem}

\begin{remark} \label{rem:after_main_theorem}
	The statements (1-3) are proved respectively in Sections \ref{sec:non-degeneracy}, \ref{sec:associativity}-\ref{sec:torus_action} and \ref{sec:geometry_of_the_mirror_family}.
	By the non-degeneracy, the algebra $A$ is canonically associated to $U\subset Y$, independent of any choice of torus.
	Moreover, we can remove the dependence on the compactification $Y$ by setting all curves classes to 0: we set $A_U \coloneqq A \otimes_{R} \bbZ$, where $R \to \bbZ$ sends every $z^\gamma$ to $1$, see \cref{rem:A_U}.
	This corresponds to taking the fiber of $\cX$ over the identity point of $T^{N_1(Y)}\coloneqq \Spec(\bbZ[N_1(Y)])\subset\Spec R$.
	
	Throughout the paper we will refer to $A$, as well as its variants, as \emph{mirror algebras}, because we expect $\cX$ to be the mirror to $U$ in the sense of homological mirror symmetry; and in particular, $A$ should be isomorphic to the symplectic cohomology ring $\SH^0(U)$.
	We do not address any symplectic aspects of homological mirror symmetry here.
	Our non-archimedean enumerative approach can be seen as a study of the A-side of homological mirror symmetry from an algebro-geometric viewpoint.
	%	Nevertheless, statements (\ref{thm:main:non-degeneracy}-\ref{thm:main:algebra}) give an indication for the comparison between our algebra $A$ and the symplectic cohomology ring $\SH^0(U)$.
	%	It is shown by Pascaleff \cite{Pascaleff_On_the_symplectic_cohomology} that in dimension two, $\SH^0(U)$ has a basis naturally identified with our $\Sk(U,\bbZ)$.
	%	Thus by the non-degeneracy, to show that $A\simeq\SH^0(U)$, one is reduced to showing that (\ref{thm:main:algebra}) holds in $\SH^0(U)$.
	
	The assumption that $U$ is affine and contains an open torus allows us deeper understandings of the geometry of the mirror variety $\cX$.
	Without the affineness assumption, we lose the convexity and finiteness in \cref{sec:convexity}, and the mirror algebra $A$ will only be formal.
	Moreover, the affineness is necessary for the positivity of structure disk classes (see \cref{prop:structure_disk_interior_divisor}).
	The assumption that $U$ contains an open torus is twofold:
	first, it allows a degeneration of the mirror variety to a toric variety, crucial for the proof of \cref{thm:main}(\ref{thm:main:non-degeneracy}) and (\ref{thm:main:family});
	second, it greatly simplifies the enumerative part of the theory, making it self-contained, see \cref{sec:intro:naive_counts}.
	Removing this assumption would result in a more sophisticated enumerative theory, which is currently under development (see \cite{Porta_Yu_Non-archimedean_quantum_K-invariants}, see also \cite{Abramovich_Punctured_logarithmic_maps} for an intimately related logarithmic approach).
	Note the assumption that $U$ contains an open torus is always satisfied in dimension two (i.e.\ for log Calabi-Yau surfaces) by the classification of surfaces, but not always in higher dimensions.
	We conjecture that the theorem should still hold without this assumption.
	
	\cref{thm:main} is directly inspired by \cite[Conjecture 0.8]{Gross_Mirror_symmetry_for_log_Calabi-Yau_surfaces_I_v1}, where the multilinear pairing $\braket{,\dots,}$ is defined via log Gromov-Witten invariants instead of naive counts of rational curves.
	It is a separate interesting question, whether these log Gromov-Witten invariants are equal to our counts.
	Gross and Siebert are working on the reconstruction problem of mirror symmetry in greater generality, including compact Calabi-Yau manifolds, using the theory of punctured log curves (see \cite{Gross_Intrinsic_mirror_symmetry,Gross_The_canonical_wall_structure}).
	It is not evident whether their mirror algebra (restricted to our context) will be isomorphic to ours, and by the non-degeneracy, this will in fact be a consequence of the previous comparison question.
	%These questions are not obvious because log Gromov-Witten counts are virtual counts, which could include contributions from stable log curves with irreducible components mapping into the boundary $D$.
	
	We prove a bit more than statement (\ref{thm:main:family}):
	we compactify each fiber of $\cX$ to an anti-canonical semi-log-canonical pair, see \cref{prop:geometry_of_fiber}.
	We expect that every fiber of the restriction $\cX_\bbQ|_{T^{N_1(Y)}}$ is normal with canonical singularities; together with Paul Hacking, we prove this in dimension two in \cite{Hacking_Secondary_fan}.
		Statements (\ref{thm:main:non-degeneracy}-\ref{thm:main:algebra}) are proven independently, for X-cluster varieties by Travis Mandel \cite{Mandel_Theta_bases} using a completely different argument.
\end{remark}

\subsection{Structure Constants}  \label{sec:intro:structure_constants}

Given $P_1,\dots,P_n\in\Sk(U,\bbZ)$, we write the product in the mirror algebra $A$ as
\begin{equation} \label{eq:multiplication}
	\theta_{P_1}\cdots\theta_{P_n}=\sum_{Q\in\Sk(U,\bbZ)}\ \sum_{\gamma\in\NE(Y)} \chi(P_1,\dots,P_n,Q,\gamma) z^\gamma\theta_Q.
\end{equation}
Now we give the precise description of the structure constants $\chi(P_1,\dots,P_n,Q,\gamma)$ as counts of non-archimedean analytic disks, building on ideas from
\cite{Yu_Enumeration_of_holomorphic_cylinders_I,Yu_Enumeration_of_holomorphic_cylinders_II}.
We would like to illustrate the simplicity, both conceptually and in technical detail, because we feel that this simple direct geometric construction of the mirror algebra is our most important contribution.

Recall that for an algebraic variety $X$ over any non-archimedean field $K$, we have an associated $K$-analytic space $X^\an$ constructed by Berkovich \cite{Berkovich_Spectral_theory}.
As a set, $X^\an$ consists of pairs $(\xi,\nu)$ where $\xi\in X$ is a scheme-theoretic point and $\nu$ is an absolute value on the residue field $\kappa(\xi)$ extending the given one on $K$.

Here we equip our base field $k$ with the trivial absolute value, i.e.\ $\abs{x}=1$ for all $x\in k\setminus 0$ and $\abs{0}=0$.
Note $\Sk(U,\bbZ)$ is naturally a subset of $U^\an$.
By our assumption, $U$ contains an open split algebraic torus, which we denote by $T_M$, $M$ being the co-character lattice.
This induces a canonical identification $\Sk(U,\bbZ)\simeq\Sk(T_M,\bbZ)\simeq M$.

Let $P_1,\dots,P_n,Q \in \Sk(U,\bbZ)$ and $\gamma\in\NE(Y)$.
Set $Z\coloneqq -Q \in M\simeq\Sk(U,\bbZ)$.
Let $\delta\in\NE(Y)$ be the class of the closure of any general translation of the one-parameter subgroup in $T_M$ given by $Q\in M$.
Let $\bP_Z\coloneqq(P_1,\dots,P_n,Z)$ and $\beta\coloneqq\gamma+\delta$.
We have a moduli space $H(\bP_Z,\beta)$, and a map
\[\Phi\coloneqq(\dom,\ev_s)\colon H(\bP_Z,\beta)\longrightarrow\cM_{0,n+2}\times U\]
as in \cref{def:counting_Frobenius}, where the $n+2$ marked points are labeled as $p_1,\dots,p_n,z,s$.

Let $\mu\in\cM_{0,n+2}^\an$ be the valuation on the generic point of $M_{0,n+2}$ given by the divisor in $\ocM_{0,n+2}$ parameterizing nodal curves where $\{p_1,\dots,p_n\}$ and $\{z,s\}$ are separated by a node.
The points $\mu\in\cM_{0,n+2}^\an$ and $Q\in U^\an$ give a discrete valuation on the generic point of $\cM_{0,n+2}\times U$, and hence a point $\tQ\in(\cM_{0,n+2}\times U)^\an$.
Since $\tQ$ is contained in any Zariski open of $(\cM_{0,n+2}\times U)^\an$, and $\Phi$ is étale over a Zariski dense open of $\cM_{0,n+2}\times U$ (see \cref{prop:smoothness_all}), the fiber $(\Phi^\an)\inv(\tQ)$ is finite (over the complete residue field $\cH(\tQ)$).

A map
\[f\colon[C,(p_1,\dots,p_n,z,s)]\longrightarrow\tY^\an\]
is said to satisfy the \emph{toric tail condition}\footnote{For readers less familiar with the Berkovich geometry of curves, we refer to \cref{rem:toric_tail_condition_algebro-geometric_interpretation} for an algebro-geometric interpretation of the toric tail condition.} if the following holds:
let $\Gamma$ be the convex hull of all the marked points in $C$, $r\colon C\to\Gamma$ the canonical retraction, and $\bbT\coloneqq r\inv([s,z])\subset C$; then we have $f(\bbT\setminus z)\subset T_M^\an$.
Let $F\subset(\Phi^\an)\inv(\tQ)$ denote the subspace satisfying the toric tail condition.

\begin{definition} \label{def:structure_constant}
	We define the structure constant $\chi(P_1,\dots,P_n,Q,\gamma)$ to be the length of $F$ as 0-dimensional $\cH(\tQ)$-variety; in other words, the cardinality of the underlying set after passing to an algebraic closure.
\end{definition}

\begin{theorem} \label{thm:structure_constants}
	The sums in the multiplication rule \eqref{eq:multiplication} are finite, and give the finitely generated commutative associative $R$-algebra structure on $A$ in \cref{thm:main}(\ref{thm:main:algebra}).
	In particular, we have $\chi(P_1,\dots,P_n,0,\gamma)=\eta(P_1,\dots,P_n,\gamma)$.
\end{theorem}

\begin{remark}[Heuristics behind the structure constants, see \cref{fig:structure_constants}] \label{rem:heuristics_structure_constants}
	Suppose $Q\neq 0$.
	Heuristically, the structure constant $\chi(P_1,\dots,P_n,Q,\gamma)$ counts analytic maps
	\[g\colon[\bbD,(p_1,\dots,p_n)]\longrightarrow\tY^\an\]
	such that
	\begin{enumerate}
		\item $\bbD$ is a closed unit disk over a non-archimedean field extension $k'/k$, and $p_1,\dots,p_n$ are distinguished $k'$-points on $\bbD$,
		\item $g(p_j)\in D_j^\circ$ for all $j$ such that $P_j\neq 0$, and $g\inv(\tD)=\sum m_j p_j$,
		\item $g(\partial\bbD)=Q\in U^\an$,
		\item \label{rem:heuristics_structure_constants:derivative} The derivative of $g$ at $\partial\bbD$ is equal to $Q$ (see \cref{rem:tangent_space}),
		\item $(b\circ g)_*(\bbD)=\gamma$, in a limiting sense (see \cref{def:curve_class}).
	\end{enumerate}
\begin{figure}[!ht]
	\centering
	\setlength{\unitlength}{0.4\textwidth}
	\begin{picture} (1,1)
		\put(0,0){\includegraphics[page=1,width=\unitlength]{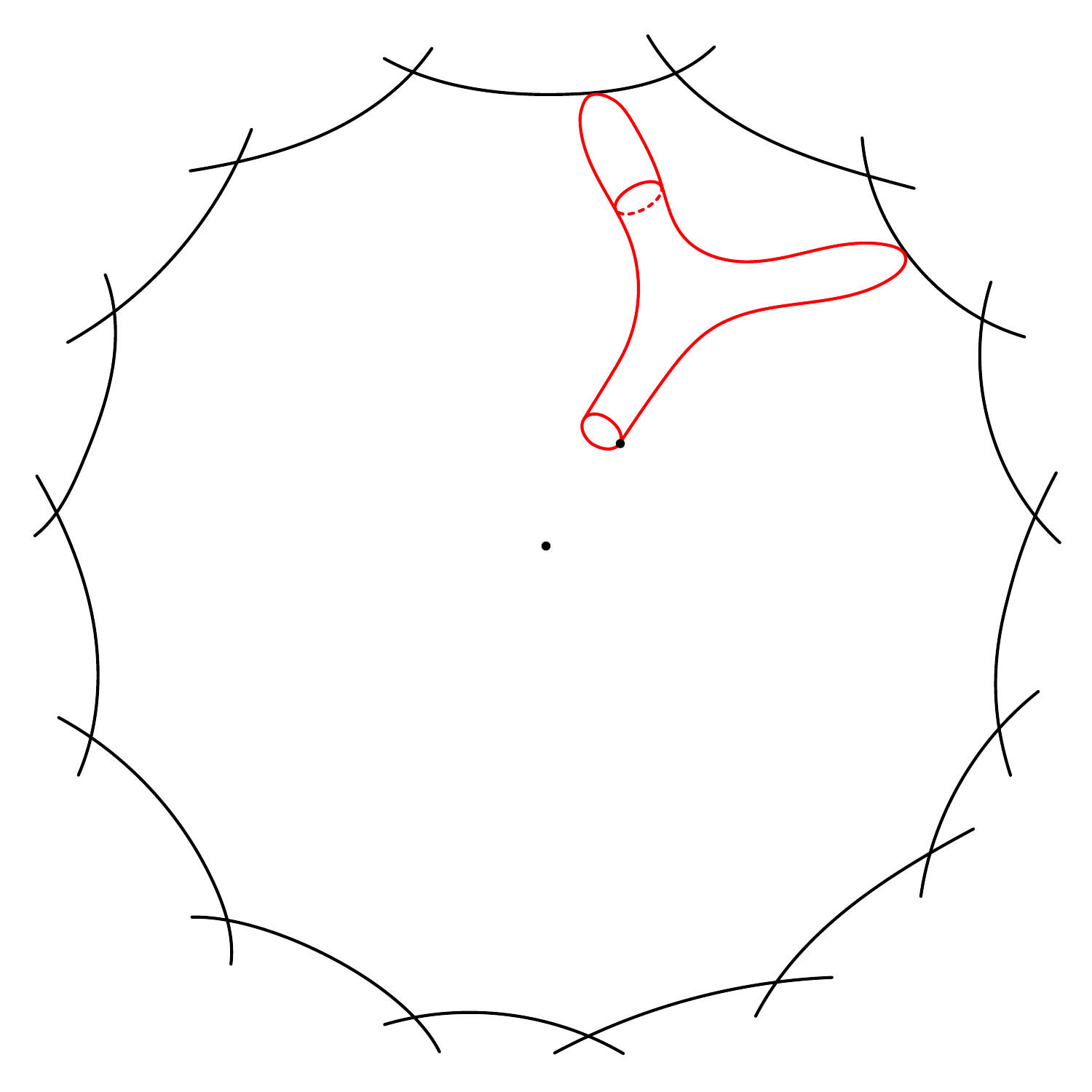}}
		\put(0.5,0.94){$D_1$}
		\put(0.84,0.77){$D_2$}
		\put(0.52,0.47){$0$}
		\put(0.57,0.55){$Q$}
	\end{picture}
	\caption{This is only a heuristic picture.
	It does not represent the underlying topological space of a non-archimedean curve, in particular, the actual non-archimedean boundary consists of a single point rather than a circle.}
	\label{fig:structure_constants}
\end{figure}
	Unfortunately, the space of all such analytic maps is $\infty$-dimensional.
	In order to extract a finite counting number, we use a variant of the strategy in \cite{Yu_Enumeration_of_holomorphic_cylinders_I}, by imposing a regularity condition on the boundary of our disks (see \cref{fig:structure_constants_tail}):
	we ask that the map $g$ can be analytically continued at the boundary $\partial\bbD$ to a map $f$ from a closed rational curve $C$ to $\tY^\an$, such that
	\begin{enumerate}[label=(\roman*), ref=\roman*]
		\item The tail $\bbT\coloneqq C\setminus\bbD^\circ$ intersects the divisor $D_Z$ at one point $z$ with multiplicity $m_Z$, where $Z=m_Z\nu_Z$ and $\nu_Z$ has divisorial center $D_Z\subset\tD$;
		\item The punctured tail $\bbT\setminus z$ maps into the torus $T_M^\an$.
	\end{enumerate}
	\begin{figure}[!ht]
	\centering
	\setlength{\unitlength}{0.4\textwidth}
	\begin{picture} (1,1)
		\put(0,0){\includegraphics[page=2,width=\unitlength]{images/structure_constants}}
		\put(0.5,0.94){$D_1$}
		\put(0.84,0.77){$D_2$}
		\put(0.52,0.47){$0$}
		\put(0.57,0.55){$Q$}
		\put(0.22,0.08){$D_Z$}
	\end{picture}
	\caption{}
	\label{fig:structure_constants_tail}
	\end{figure}
	Hence, the problem of counting analytic disks can be translated into counting special types of closed rational curves, which is exactly the content of \cref{def:structure_constant}.
	%The toric tail condition should only be seen as a tool for imposing a boundary regularity condition in order to achieve a finite count of analytic disks, as we will show in \cref{sec:varying_torus} that the counts are independent of the choice of torus $T_M\subset U$.
	%We expect to work out a more intrinsic boundary regularity condition applicable in greater generalities.
\end{remark}

\begin{remark} \label{rem:retraction}
	Any pluricanonical form $\omega$ on $X$ gives a piecewise linear \emph{skeleton} $\Sk(\omega)\subset X^\an$ (defined by Temkin \cite{Temkin_Metrization_of_differential_pluriforms}, generalizing Kontsevich-Soibelman \cite{Kontsevich_Homological_mirror_symmetry} and Mustata-Nicaise \cite{Mustata_Weight_functions}.)
	We define the \emph{essential skeleton} $\Sk(X)\coloneqq\bigcup\Sk(\omega)\subset X^\an$, union over all nonzero log pluricanonical forms $\omega$ on $X$ (see \cref{def:essential_skeleton}).
	In our context, we have $\Sk(U,\bbZ)\subset\Sk(U)=\Sk(\omega)\subset U^\an$, where $\omega$ is the log volume form on $U$, unique up to scaling.
	When the compactification $U\subset Y$ is minimal, we obtain a natural retraction map $U^\an\to\Sk(U)$, an instance of non-archimedean SYZ fibration (see \cite{Nicaise_Xu_Yu_The_non-archimedean_SYZ_fibration}).
	The retraction induces an integral affine structure on $\Sk(U)$ outside codimension two.
	Note that while the embedding $\Sk(U)\subset U^\an$ is intrinsic to $U$, the retraction $U^\an\to\Sk(U)$, as well as the resulting integral affine structure, depend on the choice of a minimal compactification $U\subset Y$.
	Moreover, snc minimal compaction often does not exist, and the presence of singularities generates technical complications.
	Both the Kontsevich-Soibelman program \cite{Kontsevich_Affine_structures} and the Gross-Siebert program \cite{Gross_From_real_affine_geometry} are based on the integral affine structure above, while our approach is not.
	One key technology we develop in this paper that enables us to work independent of any retraction $U^\an\to\Sk(U)$ is called \emph{skeletal curves}, which we describe below in \cref{sec:intro:skeletal_curves}.
\end{remark}

\begin{remark} \label{rem:structure_constants_vary_point}
	Due to the choice of the specific point $\tQ\in(\cM_{0,n+2}\times U)^\an$, the curves in $F$ responsible for structure constants, though highly generic in the algebraic sense, are in fact very special, i.e.\ non-transverse, from the tropical viewpoint.
	This is convenient for giving a quick definition of structure constants, but impractical for the purpose of proving any properties as in \cref{thm:structure_constants}.
	In order to perturb the curves in $F$ into more tropically transverse positions, we will allow the point $\tQ$ to vary over a subset $V_\cM\times V_Q\subset\Sk(\cM_{0,n+2}\times U)\subset(\cM_{0,n+2}\times U)^\an$.
	We will show that the subspace $F'$ of $(\Phi^\an)\inv(V_\cM\times V_Q)$ satisfying the toric tail condition is a union of connected components.
	As $\Phi^\an$ is finite étale over a neighborhood of $V_\cM\times V_Q$, we deduce that the degree of $\Phi^\an|_{F'}$ is well-defined, and gives the structure constant $\chi(P_1,\dots,P_n,Q,\gamma)$ (see \cref{lem:structure_constants}).
\end{remark}

\begin{remark} \label{rem:toric_tail_condition_algebro-geometric_interpretation}
	For readers less familiar with the Berkovich geometry of curves, here we give an algebro-geometric interpretation of the toric tail condition.
	Let $E\coloneqq\overline{U\setminus T_M}\subset Y$, $Z\coloneqq f\inv(E^\an)\subset C$ and $\Sigma\subset C$ the union of the marked points.
	Say $f$ is defined over a non-archimedean field extension $k'/k$.
	Up to a finite extension of $k'$, we can find a semistable model $\cC$ of $C$ over the ring of integers $(k')^\circ$ such that $Z\cup\Sigma$ extends to the smooth locus of the special fiber $\cC_s$.
	Let $\Gamma_\cC$ be the dual graph (which is a tree) of $\cC_s$.
	Let $v_{p_1},\dots,v_{p_n},v_z,v_s$ be the vertices of $\Gamma_\cC$ corresponding to all the marked points of $C$; they are not necessarily different from each other.
	Let $\Gamma_\Sigma\subset\Gamma_\cC$ be the convex hull of these vertices, and $\rho\colon\Gamma_\cC\to\Gamma_\Sigma$ the unique retraction.
	Let $\cC_s^\bbT\subset\cC_s$ be the union of irreducible components corresponding to $\rho\inv([v_s,v_z])$.
	Then the toric tail condition for $f$ is equivalent to $\cC_s^\bbT\cap\oZ=\emptyset$ (see \cref{lem:toric_tail_equiv}(\ref{lem:toric_tail_equiv:E})).
\end{remark}

\subsection{Skeletal Curves} \label{sec:intro:skeletal_curves}

The curves responsible for structure constants enjoy a special property:
the essential skeleton of the curve maps into the essential skeleton of $U$.
Such curves play an important role in various stages of our theory.
Below is the main theorem concerning such curves:

For this theorem, it is not necessary to assume $k$ has trivial valuation, $U$ is affine or contains an algebraic torus.
Let $D^\ess\subset D$ denote the union of irreducible components of $D$ where $\omega$ has a pole.
Let $\oSk(U)$ denote the closure of $\Sk(U)$ in $Y^\an$.

\begin{theorem}[see \cref{thm:f_in_skeleton}] \label{thm:intro_skeletal_curve}
	Let $k \subset k'$ be a non-archimedean field extension, $C$ a proper rational nodal \kanal curve, and $f\colon C_{k'} \to Y_{k'}^{\an}$ a $k'$-analytic map such that $f^{-1}(D) = f^{-1}(D^{\ess,\sm}) = \sum m_j p_j$ for $m_j\in\bbN_{>0}$ and $p_j \in C^{\sm}(k)$ (where $^\sm$ indicates the smooth locus).
		Consider the composition
	\[f_Y \colon C_{k'} \xrightarrow{\ f\ } Y_{k'}^\an \longrightarrow Y^\an.\]
	Let $\Gamma(k)\subset C_{k'}$ be the convex hull of $C(k)$ in $C_{k'}$.
	Assume $f_Y(x)\in\Sk(U)$ for some $x\in C(k)\subset C_{k'}$.
	Then the image $f_Y(\Gamma(k))$ lies in $\oSk(U)$; in particular, $f_Y$ maps the convex hull of all $p_j$ in $C_{k'}$ into $\oSk(U)$.
\end{theorem}

\begin{definition} \label{def:skeletal_curve}
	We call such $f_Y\colon C_{k'} \to Y^\an$ a \emph{skeletal curve}.
\end{definition}

Following \cref{rem:retraction}, we note that a skeletal curve have a canonical \emph{spine} independent of any retraction $U^\an\to\Sk(U)$.
Given any $Q\subset C(k)$ containing all $p_j$, let $\Gamma(Q)\subset C_{k'}$ be the convex hull.
We define the \emph{spine} of $f$ (with respect to $Q$) to be the restriction
\[\Gamma(Q)\xrightarrow{\ f_Y\ }\oSk(U).\]

%We refer to \cref{thm:f_in_skeleton} for more characterizations of skeletal curves.

The main source of skeletal curves in this paper comes from the following construction:

Fix $J$ a finite set of cardinality $n\ge 3$, and $\bP\coloneqq(P_j)_{j\in J}$ with $P_j\in\Sk(U,\bbZ)$.
Let $B\coloneqq\set{j|P_j\neq0}$, $I\coloneqq\set{j|P_j=0}$.
For each $j\in B$, write $P_j=m_j\nu_j$, and assume each $\nu_j$ has divisorial center $D_j\subset D$.
Let $\cM(U,\bP,\beta)$ denote the moduli stack of $n$-pointed rational stable maps $f\colon[C,(p_j)_{j\in J}]\to Y$ of class $\beta$ such that for each $j\in B$, $p_j$ maps to the open stratum $D_j^\circ$, and $f\inv(D)=\sum_{j\in B}m_j p_j$.
For $i \in I$ let 
\begin{equation} \label{eq:Phi_M0n1}
\Phi_i\coloneqq(\st,\ev_i)\colon \cM(U,\bP,\beta)^\an \longrightarrow \ocM_{0,n}^\an \times U^\an
\end{equation}
be the map taking stabilization of domain and evaluation at $p_i$.

\begin{proposition}[see Lemmas \ref{lem:source_of_skeletal_curve}, \ref{lem:restriction_to_skeleton}] \label{prop:intro_source_of_skeletal_curve}
	The preimage
	\[\ISk \coloneqq \Phi_i\inv\big(\oSk(\cM_{0,n})\times\Sk(U)\big)\]
	consists of skeletal curves.
	On a neighborhood of $\ISk$, $\Phi_i$ is étale and representable (i.e.\ non-stacky); $\Phi_i|_\ISk$ is set-theoretically finite.
	In particular, the structure constant $\chi(P_1,\dots,P_n,Q,\gamma)$ in \cref{def:structure_constant} is a naive count of skeletal curves.
\end{proposition}

\subsection{Naive counts of spines in $\Sk(U)$} \label{sec:intro:naive_counts} 

\cref{prop:intro_source_of_skeletal_curve} suggests a simple definition of counts associated to piecewise affine trees in $\Sk(U)$.
The study of properties of such counts is the main technical foundation of this paper.

Note that $\Sk(U)$ has an intrinsic conical piecewise $\bbZ$-linear structure (see \cref{sec:naive_counts}).
So we can consider a piecewise \Zaffine map $h$ from a stable nodal metric tree $\Gamma$ to $\oSk(U)$.
Assume $\Gamma$ has $n$ 1-valent vertices $(v_j)_{j\in J}$, all of them infinite, and $h\inv(\partial\oSk(U))\subset\bigcup v_j$.
We call $S\coloneqq(\Gamma,h)$ an \emph{extended spine} in $\Sk(U)$.
For each $j$, let $P_j\in\Sk(U,\bbZ)$ be derivative of $h$ at $v_j$ (pointing outwards).
Let $\bP\coloneqq(P_1,\dots,P_n)$.

Let $B\coloneqq\set{j|P_j\neq0}$ and $I\coloneqq\set{j|P_j=0}$.
Assume $\abs{B}\ge 2$, $\abs{I}\ge 1$ and fix $i\in I$.
Let $\Phi_i$ be as in \eqref{eq:Phi_M0n1}.
The closure $\oSk(\cM_{0,n})\subset\ocM_{0,n}^\an$ can be identified with the space of stable extended nodal metric trees with $n$ legs (see \cref{prop:skeleton_of_M0n}).
Thus $\Gamma$ gives a point $\Gamma\in\oSk(\cM_{0,n})\subset\ocM_{0,n}^\an$.
By \cref{prop:intro_source_of_skeletal_curve}, $\Phi_i\inv(\Gamma,h(v_i))$ is finite, and consists of skeletal curves.
Let $F_i(S,\beta)$ be the subspace of $\Phi_i\inv(\Gamma,h(v_i))$ consisting of maps whose spine is equal to $S$.
We define $N_i(S,\beta) \coloneqq \length(F_i(S,\beta))$, counting analytic curves in $U^\an$ of spine $S$ and class $\beta$ (evaluating at $i$).

Following the same heuristics in \cref{rem:heuristics_structure_constants}, we can also count analytic curves associated to a non-extended spine $S=(\Gamma,h)$ in $\Sk(U)$, i.e.\ we allow some 1-valent vertices of $\Gamma$ to be finite vertices.
We extend each finite leg to an infinite leg via the identification $\Sk(U)\simeq M_\bbR$ and obtain an extended spine $\hS$.
The extension of spine induces also an extension of curve class from any $\gamma\in\NE(Y)$ to $\hgamma\in\NE(Y)$.
We apply the paragraph above to $\hS$ and $\hgamma$, and obtain $F_i(\hS,\hgamma)$.
Let $F_i(S,\gamma)\subset F_i(\hS,\hgamma)$ be the subspace consisting of stable maps satisfying the \emph{toric tail condition} for each leg extension, as in \cref{sec:intro:structure_constants}.

\begin{definition}[see \cref{sec:naive_counts} for details]
	We define $N_i(S,\gamma)\coloneqq\length(F_i(S,\gamma))$, the count of analytic curves (with boundaries) in $U^\an$ of spine $S$ and class $\gamma$ (evaluating at $i$).
\end{definition}

The counts above enjoy very nice properties when the spine $S$ is sufficiently general, more precisely, when it is transverse with respect to walls in $\Sk(U)$ (see \cref{def:transverse}).
The fundamental property is deformation invariance, without which the counts are almost useless.

\begin{theorem}[Deformation invariance, see \cref{thm:deformation_invariance_truncated}]
	The count $N_i(S,\gamma)$ is deformation invariant among transverse spines.
\end{theorem}

Other important properties include:
\begin{enumerate}
	\item The symmetry property: the count $N_i(S,\gamma)$ is independent of the choice of $i\in I$, see \cref{thm:forgetting_interior_marked_points}, a generalization of \cite[Theorem 6.3]{Yu_Enumeration_of_holomorphic_cylinders_I} with a fundamentally different proof.
	\item The gluing formula: a product formula when we glue spines at finite vertices, see \cref{thm:gluing_concatenate}, a variant of \cite[Theorem 1.2]{Yu_Enumeration_of_holomorphic_cylinders_II}.
	\item Tail condition with varying torus: the count $N_i(S,\gamma)$ is independent of the choice of torus $T_M\subset Y$, see \cref{thm:count_independent}.
	\item Convexity, positivity and finiteness, see \cref{sec:convexity}.
\end{enumerate}

\subsection{Scattering diagram and comparison with GHKK} \label{sec:intro:scattering}

Both the Kontsevich-Soibelman program \cite{Kontsevich_Affine_structures} and the Gross-Siebert program \cite{Gross_From_real_affine_geometry} for the reconstruction of mirror varieties rely on a combinatorial algorithmic construction of scattering diagram (aka wall-crossing structure, see \cite{Kontsevich_Wall-crossing_structures}).
Our construction of the mirror algebra by counting non-archimedean analytic disks as in \cref{sec:intro:structure_constants} completely bypasses any use of scattering diagram.
Nevertheless, our geometric approach also allows us to give a direct construction of the scattering diagram by counting infinitesimal analytic cylinders, without the step-by-step Kontsevich-Soibelman algorithm\footnote{A different direct construction of scattering diagram was proposed by Gross and Siebert in \cite[\S 2.4]{Gross_Intrinsic_mirror_symmetry} via punctured log curves.}, see \cref{sec:scattering}.
Here is a brief summary.

Fix $T_M\subset U\subset Y$ as before, let $N\coloneqq\Hom(M,\bbZ)$.
Given any $n\in N\setminus 0$, $x\in n^\perp$ generic, $v,w\in M_{\braket{n,\cdot}>0}$ and $\alpha\in\NE(Y)$, let $V_{x,v,w}$ be the infinitesimal spine with two ends of outward derivatives $v$ and $-w$ respectively, and bending at $x$ (see \cref{fig:infinitesimal_spine}).
We obtain an associated count of analytic curves $N(V_{x,v,w},\alpha)$ as in \cref{sec:intro:naive_counts}.
\begin{figure}[!ht]
	\centering
	\setlength{\unitlength}{0.3\textwidth}
	\begin{picture} (1,1)
	\put(0,0){\includegraphics[width=\unitlength]{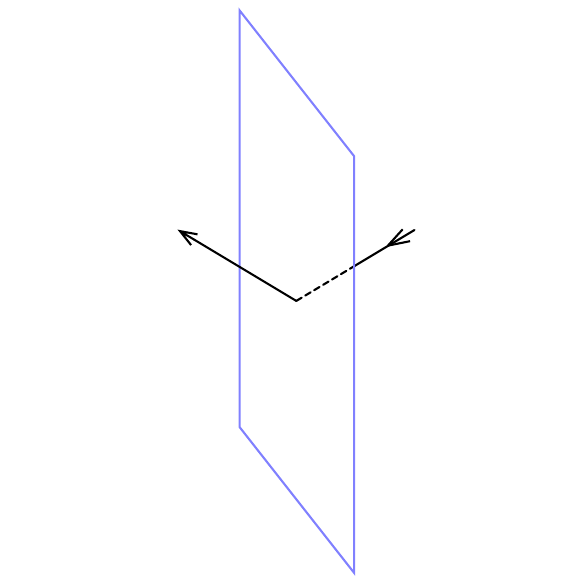}}
	\put(0.62,0.02){$n^\perp$}
	\put(0.49,0.43){$x$}
	\put(0.3,0.52){$v$}
	\put(0.65,0.52){$w$}
	\end{picture}
	\caption{}
	\label{fig:infinitesimal_spine}
\end{figure}
\begin{definition}[Wall-crossing transformation, see \cref{def:wall-crossing_transformation}] \label{def:wall-crossing_intro}
	For any $x\in n^\perp\subset M_\bbR$ generic and $v\in M_{\braket{n,\cdot}>0}$, we define
	\[\Psi_{x,n}(z^v)\coloneqq\sum_{\substack{w\in M_{\braket{n,\cdot}>0}\\ \alpha\in\NE(Y)}} N(V_{x,v,w},\alpha) z^\alpha z^w.\]
\end{definition}

\begin{theorem}[see \cref{thm:wall-crossing_homomorphism}, \cref{prop:wall-crossing_function}]
	Fix a strictly convex toric monoid $Q\supset\NE(Y)$, let $\hR$ be the completion of $\bbZ[Q\oplus M]$ with respect to the maximal monomial ideal.
	The wall-crossing transformation $\Psi_{x,n}$ extends to an automorphism of $\Frac{\hR}$.
	Moreover, there exists $f_x\in\hR$ such that for any $v\in M$, we have
	\[\Psi_{x,n}(z^v)=z^v\cdot f_x^{\braket{n,v}}.\]
\end{theorem}

We call
\[\fD\coloneqq\Set{(x,f_x) | x\in n^\perp\subset M_\bbR\text{ generic for some } n\in N\setminus 0},\]
the \emph{scattering diagram} associated to $U\subset Y$ with respect to $T_M$.
We then prove that $\fD$ has finite polyhedral finite-order approximations (see \cref{prop:wall_decomposition}), and that $\fD$ is theta function consistent (see \cref{prop:theta_function_consistent}).

In \cref{sec:setting_the_curve_classes_to_0}, under additional assumptions, we set all the curve classes to 0 in $\fD$ and obtain a scattering diagram $\fD_U$ independent of the compactification $U\subset Y$.
We show that $\fD_U$ is consistent in the sense of Kontsevich-Soibelman (see \cref{prop:KS_consistent}).
This paves the way for the comparison in the case of cluster varieties with the work of Gross-Hacking-Keel-Kontsevich \cite{Gross_Canonical_bases}, see \cref{sec:cluster_case}.

We prove in \cref{thm:scattering_diagram_comparison} that the (combinatorially defined) scattering diagram $\fD_\GHKK$ of \cite[Theorem 1.12]{Gross_Canonical_bases} is equivalent to our (geometrically defined) scattering diagram $\fD_U$.
From this we deduce the comparison theorems for both A-cluster variety and X-cluster variety, see \cref{thm:comparison_with_GHKK_A-type} and \cref{cor:comparison_with_GHKK_X-type}.
For simplicity, here we only state the latter:

\begin{theorem}[see \cref{cor:comparison_with_GHKK_X-type}] \label{thm:X-cluster_intro}
	Let $\cX$ be a Fock-Goncharov skew-symmetric X-cluster variety (possibly with frozen variables), such that $U \coloneqq \Spec(H^0(X,\cO_\cX))$ is smooth and the canonical map $\cX \to U$ is an open immersion, (e.g.\ double Bruhat cells in semisimple complex Lie groups).
	Then $U$ is a smooth affine log Calabi-Yau variety containing an open split algebraic torus, so our \cref{thm:main} applies.
	Let $\cX^\vee$ be the Fock-Goncharov dual, and let $\can(\cX^\vee)$ be as in \cite[Theorem 0.3]{Gross_Canonical_bases}.
	Let $A_U$ be our mirror algebra as in \cref{rem:after_main_theorem}.
	The following hold:
	\begin{enumerate}
		\item \label{thm:X-cluster_intro:comparison} The (combinatorially defined) structure constants of \cite[Theorem 0.3(1)]{Gross_Canonical_bases} are equal to our (geometrically defined) structure constants.
		Hence they give $\can(\cX^{\vee})$ an algebra structure, equal to our mirror algebra $A_U$.
		\item \label{thm:X-cluster_intro:independence} The mirror algebra $\can(\cX^{\vee}) \simeq A_U$, together with its theta function basis, is independent of the cluster structure; it is canonically determined by the variety $U$.
	\end{enumerate}
\end{theorem}

The comparison theorems have two-fold implications.
First they give a concrete combinatorial way of computing the abstract non-archimedean curve counting in the case of cluster varieties.
Conversely, we obtain geometric interpretations of various combinatorial constructions and answer several conjectures in \cite{Gross_Canonical_bases}:

(1) As our naive counts are always nonnegative integers, we obtain a much more conceptual proof of the positivity of the structure constants, and of the coefficients of the scattering functions $f_x$, which then implies the positivity in the Laurent phenomenon for cluster algebras, see Theorems 1.13 and 4.10 in loc.\ cit..

(2) The (combinatorially defined) broken lines in \cite[\S 3]{Gross_Canonical_bases} are simply the spines of the analytic curves contributing to the local theta functions $\theta_{x,m}$ in \cref{def:local_theta_funcation}.
Then the broken-line convexity conjecture of \cite[Conjecture 8.12]{Gross_Canonical_bases} follows directly from our general convexity lemma, \cref{lem:convexity}.

(3) Thanks to (2), we obtain an algebra structure on $\can(\cX^\vee)$ as in \cref{thm:X-cluster_intro}(\ref{thm:X-cluster_intro:comparison}); while in \cite{Gross_Canonical_bases}, there is only an algebra structure on a vector subspace $\operatorname{mid}(\cX^\vee)\subset\can(\cX^\vee)$, and the algebra structure on $\can(\cX^\vee)$ is obtained under an additional EGM assumption, see \cite[Theorems 0.12(1), 0.17]{Gross_Canonical_bases}.
This is also a step in the direction of the full Fock-Goncharov conjecture of \cite[Conjecture 0.10]{Gross_Canonical_bases}.

(4) \cref{thm:X-cluster_intro}(\ref{thm:X-cluster_intro:independence}) was conjectured in \cite[Remark 0.16]{Gross_Canonical_bases}.
It is shown in \cite{Zhou_Cluster_structures} that a given variety can have more than one cluster structure.

(5)	While the counts $N(V_{x,v,w},\alpha)$ in \cref{def:wall-crossing_intro} are canonically associated to $U\subset Y$, the resulting scattering diagram depends on the choice of $T_M\subset U$.
This gives a geometric explanation of the ad-hoc appearing formula for change of scattering diagram under mutation (see \cite[Definition 1.22]{Gross_Canonical_bases} and \cite[3.21, 3.30]{Gross_Mirror_symmetry_for_log_Calabi-Yau_surfaces_I_v1}).

\subsection{Outline of the paper} \label{sec:outline}

In \cref{sec:log_CY}, we set up the basic objects for the whole paper: log Calabi-Yau varieties, skeletons, toric models and tropicalizations.

In \cref{sec:smoothness}, we set up the basic moduli spaces of stable maps, and prove various smoothness properties.
The smoothness will allow us to obtain positive integral enumerative invariants, bypassing virtual fundamental classes.
Furthermore, the smoothness is also key to the skeletal curves and the deformation invariance for wall-crossing, see Sections \ref{sec:skeletal_curves} and \ref{sec:toric_tail_condition}.

Analytic stable maps in $U$ give rise to piecewise-linear trees in $\bbR^n$ via tropicalization.
In \cref{sec:tropical}, we introduce various combinatorial notions regarding such trees, namely twigs, walls, spines and tropical curves.
The goal is to capture enough combinatorial features from the tropicalization of analytic stable maps, so that we can deduce statements at the combinatorial level, which will then help establish statements at the analytic level.
Two main combinatorial statements are rigidity and transversality, shown in \cref{sec:rigidity_and_transversality}.

In \cref{sec:toric_case}, we describe the various moduli spaces introduced before in the toric case.

In \cref{sec:curve_classes}, we study the classes of analytic curves (with boundaries) mapping into our log Calabi-Yau variety.
We prove a positivity result and give a tropical formula for computing curve classes.

In \cref{sec:skeletal_curves}, we introduce skeletal curves, provide different characterizations, and apply skeletal curve theory to the moduli spaces in the previous sections.

In \cref{sec:naive_counts}, we give the details regarding the naive counts mentioned in \cref{sec:intro:naive_counts}.
Furthermore, we prove that for transverse spines, the naive count is independent of the choice of the marked point at which we evaluate.
We refer to this independence as the symmetry property.

In \cref{sec:deformation_invariance}, we prove the deformation invariance of naive counts associated to transverse extended spines.
The key ingredient is the properness of the spine map restricted to skeletal curves.
The deformation invariance is generalized to truncated (i.e.\ non-extended) spines in \cref{sec:toric_tail_condition}, where we study toric tail conditions in families.
For the associativity of the structure constants and the wall-crossing homomorphism, we furthermore establish a stronger form of deformation invariance involving also non-transverse spines.
This relies on skeletal curve theory developed in \cref{sec:skeletal_curves}.

In \cref{sec:gluing}, based on the deformation invariance, we prove a gluing formula for gluing spines, which roughly says that the count associated to a glued spine is the product of the counts associated to the spines before the gluing.
Similar idea is then used in \cref{sec:varying_torus} to show that the counts of transverse spines are independent of the choice of the torus embedding, hence so is our mirror algebra.

In \cref{sec:associativity}, we establish the associativity of the structure constants.
The first step is to interpret each structure constant as the degree of a finite étale map over a larger base, see \cref{rem:structure_constants_vary_point}.
This follows from the stronger form of deformation invariance.
Then we express the structure constants as sums of counts associated to truncated spines, and we vary the domain metric trees, so as to stretch the edges giving rise to different groupings of the marked points.
Next we cut at such edges and apply the gluing formula, see the beginning of \cref{sec:associativity} for more details.

Note that the associativity alone is not sufficient for obtaining an algebra structure: we must also show that the two sums in the multiplication rule \eqref{eq:multiplication} are finite sums.
This is proved in \cref{sec:convexity}, via the convexity properties of the structure disks.
In \cref{sec:torus_action}, we describe a natural equivariant torus action on the mirror algebra, and prove that the mirror algebra is finitely generated.
This concludes the proof of \cref{thm:main}(\ref{thm:main:algebra}).
In \cref{sec:change_of_snc_compactification}, we study the dependence of the mirror algebra on the snc compactification.

In \cref{sec:non-degeneracy}, we prove the non-degeneracy of the trace map, i.e.\ \cref{thm:main}(\ref{thm:main:non-degeneracy}).
The idea is to use a degeneration to the toric case, which relies on the positivity properties of the structure disks.
In \cref{sec:geometry_of_the_mirror_family}, we study the geometry of the fibers of the mirror family, and prove \cref{thm:main}(\ref{thm:main:family}).
We use the equivariant torus action and a fiberwise compactification.
The compactification relies on the convexity properties of the structure disks.

The final part of the paper aims towards cluster varieties.
In \cref{sec:scattering}, we give a direct geometric construction of a scattering diagram on the essential skeleton by counting infinitesimal analytic cylinders.
In \cref{sec:cluster_case}, we prove the comparison theorems with \cite{Gross_Canonical_bases} in the cluster case, via the scattering diagram. 

We illustrate several of our basic constructions in one simple running example, see \cref{ex:surface}, Figures \ref{fig:wall}, \ref{fig:spine}, and \cref{ex:mirror_algebra}.

\bigskip
\paragraph{\bfseries Conventions}

Frequently in the paper, when we have a map $f\colon X \to Y$ and subsets $S_1,\dots,S_n\subset Y$, we write $X_{S_1,\dots,S_n}\coloneqq f^{-1}(S_1\cap\dots\cap S_n)$.

For a topological space $X$, a stratification on $X$ is given by a partition of $X=\coprod_{i\in I} X_i$ into locally closed subsets $X_i$ and a partial ordering on $I$ such that for each $j\in I$ we have $\oX_j\subset\bigcup_{i\le j} X_i$ (see \cite[09Y1]{Stacks_project}).
We call each $X_i$ an (open) stratum, and each $\oX_i$ a closed stratum.

\bigskip
\paragraph{\bfseries Acknowledgments}
We benefited tremendously from profound, detailed technical discussions with Mark Gross, Paul Hacking, Johannes Nicaise and Maxim Kontsevich.
The beautiful Frobenius structure conjecture is due to Hacking, as is the idea of using degeneration to the toric case to prove non-degeneracy of the trace pairing.
We enjoyed fruitful conversations with M.\ Baker, V.\ Berkovich, M.\ Brown, A.\ Chambert-Loir, F.\ Charles, A.\ Corti, A.\ Durcos, W.\ Gubler, E.\ Mazzon, M.\ Porta, J.\ Rabinoff, D.\ Ranganathan, M.\ Robalo, B.\ Siebert, Y.\ Soibelman, M.\ Temkin and J.\ Xie.
Keel would like to especially thank B.\ Conrad and S.\ Payne for detailed email tutorials on rigid geometry.
Keel was supported by NSF grant DMS-1561632.
T.Y.\ Yu was supported by the Clay Mathematics Institute as Clay Research Fellow.
Much of the research was carried out during the authors' trips to IHES and IAS.

\section{Log Calabi-Yau pairs} \label{sec:log_CY} 

In this section we set up the basic objects for the whole paper: log Calabi-Yau varieties, essential skeletons (\cref{lem:essential_skeleton}), toric models (\cref{lem:toric_model}) and tropicalizations (\cref{nota:Et}).
We also provide a running example (\cref{ex:surface}).

Fix $k_0$ a field of characteristic zero, equipped with the trivial valuation.
Let $k_0\subset k$ be any non-archimedean field extension.
We say that a variety (or a divisor, a function, etc.) is \emph{constant over $k$} if it is isomorphic to the pullback of something over $k_0$.
We introduce this terminology because it will help simplify notations while we frequently make base field extensions.

Now assume $k$ has discrete (possibly trivial) valuation.
Let $U$ be a $d$-dimensional connected smooth affine log Calabi-Yau $k$-variety, constant over $k$, containing a Zariski open split algebraic torus $T_M$ with cocharacter lattice $M$.
Here the log Calabi-Yau condition is equivalent to the condition that the standard volume form on $T_M$ extends to a volume form $\omega$ on $U$ without zeros or poles.

For any snc compactification $U\subset Y$, constant over $k$, let $D\coloneqq Y\setminus U$, and let $\{D_i\}_{i\in I_D}$ denote the set of irreducible components of $D$.
Let $D^\ess$ be the union of irreducible components of $D$ where $\omega$ has a pole.

\begin{definition} \label{def:SigmaYD}
	Let
	\begin{align*}
	\oSigma_{(Y,D)}&\coloneqq\bigg\{\sum_{i\in I_D} a_i \braket{D_i}\ \bigg|\ \bigcap_{a_i>0} D_i\neq\emptyset\bigg\}\subset[0,+\infty]^{I_D},\\
	\Sigma_{(Y,D)}&\coloneqq\oSigma_{(Y,D)}\cap[0,+\infty)^{I_D},
	\end{align*}
	regarded as (extended) simplicial cone complexes.
	Let $\Sigma_{(Y,D)}^\ess\subset\Sigma_{(Y,D)}$ be the sub cone complex spanned by components of $D^\ess$.
\end{definition}

For any $k$-variety $X$, we have the Berkovich analytification $X^\an$, endowed with a canonical morphism of locally ringed spaces $\iota\colon X^\an\to X$ (see \cite[\S 3]{Berkovich_Spectral_theory}).
When $X$ is integral, let $\eta_X$ be the generic point of $X$, and let $X^\bir\coloneqq\iota\inv(\eta_X)\subset X^\an$, the subset of \emph{birational points}.
Moreover, if $X$ is constant over $k$, i.e.\ it is isomorphic to the pullback of a $k_0$-variety $X_0$, let $X^\bir(\bbZ)\subset X^\bir$ be the subset consisting of valuations on $k(X)$ taking integer values on $k_0(X_0)$.

We have canonical embeddings
\[\begin{tikzcd}
\Sigma_{(Y,D)} \rar[hook] \arrow[d,phantom, "\rotatebox{-90}{$\subset$}"] & U^\bir \arrow[r,phantom,"\subset"] &[-1em] U^\an \arrow[d,phantom, "\rotatebox{-90}{$\subset$}"] \\
\oSigma_{(Y,D)} \arrow[rr, hook] & & Y^\an,
\end{tikzcd}\]
and canonical strong deformation retractions from $Y^\an$ to $\oSigma_{(Y,D)}$, and from $U^\an$ to $\Sigma_{(Y,D)}$ (see \cite{Thuillier_Geometrie_toroidale,Gubler_Skeletons_and_tropicalizations}).

The volume form $\omega$ induces an upper semicontinuous function $\norm{\omega}\colon U^\an\to\bbR_{\ge 0}$ via Temkin's Kähler seminorm (see \cite[\S 8]{Temkin_Metrization_of_differential_pluriforms}).
We denote by $\Sk(U)\subset U^\an$ the maximum locus of $\norm{\omega}$, called the \emph{essential skeleton} of $U$.
Let $\Sk(U,\bbZ)\coloneqq\Sk(U)\cap U^\bir(\bbZ)$.

The following lemma gives a concrete description of the essential skeletons.

\begin{lemma} \label{lem:essential_skeleton}
     \begin{enumerate}[leftmargin=*]
        \item \label{lem:essential_skeleton:comparison} The embedding $\Sigma_{(Y,D)}\hookrightarrow U^\an$ induces a homeomorphism $\Sigma_{(Y,D)}^\ess\simeq\Sk(U)$ preserving the integer points.
        \item \label{lem:essential_skeleton:toric} We have $\Sk(T_M,\bbZ)\simeq M$ and $\Sk(T_M)\simeq M_\bbR$.
        \item \label{lem:essential_skeleton:birational} Let $f\colon V \dasharrow W$ be any birational map of log Calabi-Yau varieties such that $f^*(\omega_W) = \omega_V$;         then it induces $V^\bir\simeq U^\bir$ identifying $\Sk(V) \simeq \Sk(U)$ preserving the integer points. 
        \item \label{lem:essential_skeleton:identifications} We have canonical identifications $\Sigma_{(Y,D)}^\ess\simeq\Sk(U)\simeq\Sk(T_M)\simeq M_\bbR$, preserving the integer points.
\end{enumerate} 
\end{lemma}
\begin{proof}
	(\ref{lem:essential_skeleton:comparison}) follows from local computations on standard normal crossing models (see \cref{lem:explicit_weight} for a more general statement).
	(\ref{lem:essential_skeleton:toric}) follows from (\ref{lem:essential_skeleton:comparison}) by taking $(Y,D)$ to be any smooth toric compactification of $T_M$.
	(\ref{lem:essential_skeleton:birational}) is tautological because the essential skeleton consists of only birational points and the definition of $\norm{\omega}$ is local.
	(\ref{lem:essential_skeleton:identifications}) follows from the previous ones.
\end{proof}

\begin{remark} \label{rem:tangent_space}
	The homeomorphism $\Sigma_{(Y,D)}^\ess\simeq\Sk(U)$ in \cref{lem:essential_skeleton}(\ref{lem:essential_skeleton:comparison}) induces an integral simplicial cone complex structure on $\Sk(U)$.
	By the weak factorization theorem (see \cite{Abramovich_Torification_and_factorization}), any two snc compactifications of $U$ are related by a zigzag of simple blowups.
	Therefore, $\Sk(U)$ has an intrinsic conical piecewise $\bbZ$-linear structure\footnote{We omit its formal definition, because it will not play any role in our proofs.}.
	For any point $b\in\Sk(U)\setminus 0$, this structure does not give a well-defined tangent space at $b$.
	Nevertheless, any multiple of the tangent vector at $b$ in the direction of $\overrightarrow{0b}$ is well-defined.
	\end{remark}

\begin{assumption} \label{ass:strata}
	\begin{enumerate}[leftmargin=*]
		\item \label{ass:strata:E} We assume $E\coloneqq\overline{U\setminus T_M}$ contains no strata of $D$.
		\item We assume the identification $\Sigma_{(Y,D)}^\ess\simeq M_\bbR$ gives a smooth toric fan $\Sigma_\rt$ in $M_\bbR$.
	\end{enumerate}
	Note that both assumptions can be achieved by a toric blowup of $(Y,D)$.
		We will first construct the mirror algebra under these assumptions, and then extend the construction to the general case, see \cref{rem:any_compactification}.
	\end{assumption}

\begin{notation} \label{nota:toric_model}
	Let $(Y_\rt,D_\rt)$ be the toric variety associated to the fan $\Sigma_\rt$ in $M_\bbR$.
	We have $\Sigma_{(Y_\rt,D_\rt)}\simeq\Sigma_\rt$ as simplicial cone complexes.
	Denote $\oM_\bbR\coloneqq\oSigma_\rt\coloneqq\oSigma_{(Y_\rt,D_\rt)}$, the canonical embedding
	\[\iota_\rt\colon\oM_\bbR\hookrightarrow Y_\rt^\an,\]
	and the canonical retraction
	\[\tau_\rt\colon Y_\rt^\an\to\oM_\bbR.\]
\end{notation}

Ideally, we want $(Y,D)$ to be a blowup of the toric variety $(Y_\rt,D_\rt)$ along some subvariety of the boundary $D_\rt$.
For maximal flexibility, we will only work with a birational map $\pi\colon Y\dasharrow Y_\rt$ preserving the torus instead of a birational morphism, and we call $\pi$ a \emph{toric model} for $(Y,D)$.
This is necessary for the application to cluster varieties in \cref{sec:cluster_case}, but creates some technical complications regarding the indeterminate locus, which can be ignored for first time reading.

\begin{lemma} \label{lem:toric_model}
	Let $Y^\idt\subset Y$ be the indeterminate locus of the birational map 
	$$
	\pi\colon Y \supset T_M \hookrightarrow Y_\rt.
	$$
	Let $W \subset Y \setminus Y^{\idt}$ be the isomorphism locus. 
	Then
	\begin{enumerate}
		\item \label{lem:toric_model:codim} $E \subset Y$ is pure codimension one, so $Y^{\idt}$ contains no generic point of $E$.
		\item \label{lem:toric_model:W} $W\subset Y$ contains the generic point of every stratum of $D^\ess$,
		$W \subset Y_\rt$ contains the generic point of every stratum of $D_\rt$, and $\pi$ induces a bijection between those generic points.
		\item \label{lem:toric_model:intersection} $E \cap W = \emptyset$, $W \cap U = T_M$.
		\item  \label{lem:toric_model:fiber} The fibers of $\pi\colon E\setminus Y^\idt \to Y_{\rt}$ are positive dimensional.
		\item \label{lem:toric_model:image} $\overline{\pi(E\setminus Y^\idt)}$ is contained in $D_\rt$, and does not contain the generic point of any stratum of $D_\rt$.
	\end{enumerate} 
\end{lemma}
\begin{proof}
	The complement $T_M^c \subset U$ is pure codimension one (this is true for the complement of any affine Zariski open subset of a separated connected normal variety).
	This gives (\ref{lem:toric_model:codim}).

	Let $p\in D^\ess$ be a 0-stratum which is the intersection of $d$ components $D_1,\dots,D_d$ of $D^\ess$.
	Let $D_{\rt,1},\dots,D_{\rt,d}$ be the corresponding components of $D_\rt$ via the identification of simplicial cone complexes $\Sigma_{(Y,D)}^\ess\simeq\Sigma_\rt$, and let $p_\rt\coloneqq D_{\rt,0}\cap\dots\cap D_{\rt,d}$.
	Assume that $D_{\rt,1},\dots,D_{\rt,d}$, as Cartier divisors, are given respectively by functions $f_1,\dots,f_d$ on $T_M$.
	Then $f_i$ has simple zero along $D_i$ for $i=1,\dots,d$.
	Since $p\notin E$ by \cref{ass:strata}, we deduces that $f_1,\dots,f_d$ are regular on a neighborhood of $p$ in $Y$.
	Hence we have $p=p_\rt\in W$.
	Since every closed stratum of $D_\rt$ contains a 0-stratum, so does every closed stratum of $D^\ess$.
	Since $W$ is open, we deduce that $W\subset Y$ contains the generic point of every stratum of $D^\ess$, $W\subset Y_\rt$ contains the generic point of every stratum of $D_\rt$, and $\pi$ induces a bijection between those generic points.
	This gives (\ref{lem:toric_model:W}).
	
	We have  $T_M \subset W \subset Y_\rt$.
	The volume form $\omega$ has a pole on all of $D_\rt=Y_\rt \setminus T_M$, in particular on $W \setminus T_M$.
	It is regular (and nowhere vanishing) on $U$.
	Thus $W \cap U = T_M$.
	Since $W^c \subset Y$ is closed, $E = \overline{U \setminus T_M} \subset W^c$.
	This gives (\ref{lem:toric_model:intersection}).
	(\ref{lem:toric_model:fiber}) follows from (\ref{lem:toric_model:intersection}).
	(\ref{lem:toric_model:image}) follows from (\ref{lem:toric_model:intersection}) and (\ref{lem:toric_model:W}).
\end{proof}

We \emph{tropicalize} Y via the toric model $\pi$.

\begin{notation} \label{nota:Et}
	Let $E_\rt\coloneqq Y_\rt\setminus W$, $E_\rt^\trop\subset\partial\oM_\bbR$ the image of $E_\rt^\an$ under the retraction map $\tau_\rt\colon Y_\rt^\an\to\oM_\bbR$, and $\tau$ the composition
	\[(Y\setminus Y^\idt)^\an\xrightarrow{\ \pi^\an\ } Y_\rt^\an\xrightarrow{\ \tau_\rt\ }\oM_\bbR.\]
\end{notation}

\begin{example} \label{ex:surface}
	The following example can be helpful for understanding the constructions throughout the paper.
	We take the degree 5 del Pezzo surface $Y\simeq\ocM_{0,5}\simeq\Bl_4(\bbP^2)$, and $D=D_1+\dots+D_5\subset Y$ an anti-canonical cycle of $(-1)$-curves.
	Then $U\coloneqq Y\setminus D$ is a smooth affine log Calabi-Yau surface, equal to the spectrum of the $A_2$-cluster algebra.
	Under the standard action of the symmetric group $S_5$ on $Y\simeq\ocM_{0,5}$, there is a cyclic group $\mu_5\subset S_5$ which cyclically permutes the $D_i$.
	There are exactly 5 other $(-1)$-curves on $Y$, which we label as $E_1,\dots,E_5$ so that $E_i\cdot D_j=\delta_{ij}$.
	Let $\pi\colon Y\to Y_\rt$ be the blowdown of $E_1$ and $E_2$ (which are disjoint).
	Then $Y_\rt$ is a smooth toric surface.
	We draw the fan $\Sigma_\rt$ of $Y_\rt$ as in \cref{fig:example_fan}, and we assume that the images of $E_1$ and $E_2$ are on the divisors corresponding to the rays $\rho_1$ and $\rho_2$ respectively.
	The subset $E_\rt^\trop\subset\partial\oM_\bbR$ consists of two points which are the limits of the rays $\rho_1$ and $\rho_2$ at $\partial\oM_\bbR$.
	\begin{figure}[!ht]
		\centering
		\setlength{\unitlength}{0.25\textwidth}
		\begin{picture} (1,1)
			\put(0,0){\includegraphics[width=\unitlength]{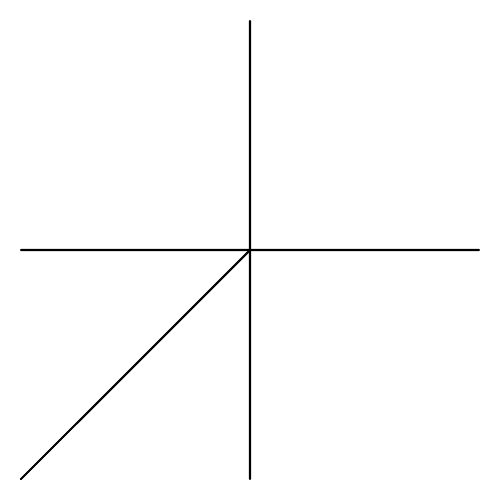}}
			\put(0.88,0.43){$\rho_1$}
			\put(0.52,0.91){$\rho_2$}
			\put(0.05,0.55){$\rho_3$}
			\put(0.05,0.18){$\rho_4$}
			\put(0.39,0.06){$\rho_5$}
		\end{picture}
		\caption{}
		\label{fig:example_fan}
	\end{figure}
	\end{example}

Below are two simple lemmas concerning the affineness of $U$ for later reference.

\begin{lemma} \label{lem:enough_global_functions}
	Let $x_1,\dots,x_n$ be generators of the algebra $H^0(U,\cO_U)$.
	The map
	\[\alpha\coloneqq(\abs{x_1},\dots,\abs{x_n}) \colon U^{\an} \to \bbR_{\geq 0}^n\]
	is a proper continuous map.
\end{lemma}
\begin{proof}
	Since $U$ is affine, the generators $x_1,\dots,x_n$ give a closed immersion $U\hookrightarrow\bbA^n_k$.
	So the lemma follows from the properness of the coordinate-wise norm map $(\bbA^n_k)^\an\to\bbR^n_{\ge 0}$.
\end{proof}

\begin{lemma} \label{lem:ample_divisor}
	There is an ample divisor $F$ on $Y$ such that $-F|_U$ is effective.
\end{lemma}
\begin{proof}
	Let $L$ be any ample line bundle on $Y$.
	Since $U$ is affine, $L^\vee|_U$ is globally generated.
	Choose any nonzero section $s$ of $L^\vee|_U$, viewed as a rational section of $L^\vee$ on $Y$, and we take $-F$ to be the associated Cartier divisor.
\end{proof}

\section{Smoothness of the moduli spaces} \label{sec:smoothness}

In this section, we set up several basic moduli spaces of stable maps for this paper, and prove various smoothness properties.
The main result is \cref{prop:smoothness_all}.
The idea is that although the moduli spaces are generally singular, we have a big smooth locus inside, and we can reach the smooth locus by controlling the evaluation map.
The smoothness will allow us to obtain positive and integral enumerative invariants, bypassing the use of virtual fundamental classes.

Fix a finite set $J$ of cardinality $n\ge 3$.
Fix $\bP\coloneqq(P_j)_{j\in J}$  with $P_j \in \Sk(U,\bbZ)$.
Let
\begin{equation} \label{eq:BI_of_P}
B\coloneqq\set{j|P_j\neq 0}, \quad I\coloneqq\set{j|P_j= 0},
\end{equation}
where $B$ means \emph{boundary} and $I$ means \emph{interior}.
For each $j\in B$, write $P_j=m_j\nu_j$ with $m_j\in\bbN_{>0}$ and $\nu$ a divisorial valuation on $k(U)$.
We assume each $\nu_j$ has divisorial center $D_j\subset D$.

\begin{notation} \label{nota:moduli_spaces}
	Let $\beta\in\NE(Y)$ be a curve class.
	Let $\ocM(Y,\bP,\beta)$ denote the moduli stack of $n$-pointed rational stable maps $f\colon[C,(p_j)_{j\in J}]\to Y$ of class $\beta$ such that for each $j\in B$, $p_j$ maps to $D_j$ with multiplicity greater than or equal to $m_j$.
	For any $j\in J$, we denote
	\[\Phi_j\coloneqq\Phi_{p_j}\coloneqq(\st,\ev_{p_j})\colon\ocM(Y,\bP,\beta)\to\ocM_{0,n}\times Y\]
	taking stabilization of domain and evaluation at $p_j$.
	Let $\cM(U,\bP,\beta)\subset \ocM(Y,\bP,\beta)$ denote the substack where for each $j\in B$, $p_j$ maps to the open stratum $D_j^\circ$ and $f\inv(D)=\sum_{j\in B} m_j p_j$ scheme-theoretically.
	Let $\cM^\sd(U,\bP,\beta)\subset \cM(U,\bP,\beta)$ denote the open substack consisting of stable maps whose domain is a stable $n$-pointed curve.
	As we have only rational curves here, $\cM^\sd(U,\bP,\beta)$ is a variety.
\end{notation}

\begin{notation} \label{nota:Msm}
	Let $\cM^\sm(U,\bP,\beta)\subset\cM^\sd(U,\bP,\beta)$ denote the open subvariety consisting of stable maps $f\colon[C,(p_j)_{j\in J}]\to Y$ satisfying the following conditions:
	\begin{enumerate}
		\item \label{nota:Msm:trivial} The pullback $f^*(T_Y(-\log D))$ is a trivial vector bundle on $C$.
		\item \label{nota:Msm:idt} The image $f(C)\subset Y$ does not intersect $Y^\idt\cup(D^\ess\setminus W)$ (notation as in \cref{lem:toric_model}). 
		\item \label{nota:Msm:E} The pullback $f\inv(E)$ is a finite set of points without multiplicities, disjoint from the nodes and the marked points of $C$.
	\end{enumerate}
\end{notation}

The subvariety $\cM^\sm(U,\bP,\beta)$ consists of the nicest curves that we will work with.
The reason for imposing the first condition is explained in \cref{lem:Msm_smooth}.
We will show that all the conditions can be achieved as long as the evaluation map lands in a general locus of $Y$ (see \cref{lem:Msm_big}).

\begin{notation} \label{nota:moduli_spaces_analytic}
	Similarly, we define the analytic versions
	\[\cM^\sm(U^{\an},\bP,\beta)\subset\cM^\sd(U^\an,\bP,\beta)\subset\cM(U^{\an},\bP,\beta)\subset\ocM(Y^\an,\bP,\beta)\]
	as well as the maps $\Phi_j$ for the analytic moduli spaces. 
	By non-archimedean GAGA principle (see \cite[Theorem 8.7]{Yu_Gromov_compactness}), the analytic moduli spaces above are isomorphic to the analytifications of the respective algebraic moduli spaces.
\end{notation}

\begin{remark} \label{rem:compatible_curve_class}
	A curve class $\beta\in\NE(Y)$ is said to be compatible with $\bP$ if
	\begin{enumerate}		
		\item $\beta\cdot D_j=m_j$ for every $j\in B$,
		\item $\beta\cdot D'=0$ for any other irreducible component $D'$ of $D$.
	\end{enumerate}
	Note that if a curve class $\beta\in\NE(Y)$ is not compatible with $\bP$, then the moduli space $\cM(U,\bP,\beta)$ is empty.
\end{remark}

\begin{lemma} \label{lem:Msm_smooth}
	Let $\mu=[C,(p_j)_{j\in J}]\in\ocM_{0,n}$ be a closed point.
	Let $q\in C$ be any closed point not belonging to $\{p_i\}_{i\in B}$.
	Let $\cM^\sd(U,\bP,\beta)_\mu$ be fiber of the map $\dom\colon\cM^\sd(U,\bP,\beta)\to\ocM_{0,n}$ over $\mu$.
	Let $\nu=[f\colon C\to Y]$ be a closed point of $\cM^\sd(U,\bP,\beta)_\mu$.
	The following are equivalent:
	\begin{enumerate}
		\item \label{lem:Msm_smooth:pullback} The pullback $f^*(T_Y(-\log D))$ is a trivial vector bundle on $C$.
		\item \label{lem:Msm_smooth:surjective} The derivative $d\ev_q$ of the evaluation map $\ev_q\colon\cM^\sd(U,\bP,\beta)_\mu\to Y$ is surjective at $\nu$.
		\item \label{lem:Msm_smooth:smooth_fiberwise} The evaluation map $\ev_q\colon\cM^\sd(U,\bP,\beta)_\mu\to Y$ is smooth at $\nu$.
		\item \label{lem:Msm_smooth:smooth} For any $i\in I$, the map $\Phi_i\coloneqq(\dom,\ev_i)\colon\cM^\sd(U,\bP,\beta)\to\ocM_{0,n}\times Y$ is smooth at $\nu$.
	\end{enumerate}
	Moreover, under the equivalent conditions above, the following hold:
	\begin{enumerate}[label=(\roman*), ref=\roman*]
		\item \label{lem:Msm_smooth:etale} The maps $\ev_q$ and $\Phi_i$ above are in fact étale at $\nu$.
		\item \label{lem:Msm_smooth:boundary} For any $i\in B$, the maps
		\[\ev_i\colon\cM^\sd(U,\bP,\beta)_\mu\to D_i,\]
		\[\Phi_i^\partial=(\dom,\ev_i)\colon\cM^\sd(U,\bP,\beta)\to\ocM_{0,n}\times D_i\]
		are smooth at $\nu$.
	\end{enumerate}
\end{lemma}
\begin{proof}
	Let $V$ denote the vector bundle $f^*(T_Y(-\log D))$.
	The derivative $d\ev_q$ at the point $\nu$ is given by the map
	\[H^0(C,V)\longrightarrow f^*(T_Y)_q.\]
	Since $f(q)\notin D$, we have a natural isomorphism $f^*(T_Y)_q\xrightarrow{\sim} V_q$, and the composite map $H^0(C,V)\to f^*(T_Y)_q\xrightarrow{\sim}V_q$ is the restriction of sections of $V$ at $q$.
	Since $f(C)\cap D\subset D^\ess$,  the restriction of $V$ to any irreducible component of $C$ has degree zero.
	Since $C$ is a nodal rational curve, it follows that $V$ is trivial if and only if it is globally generated at the point $q$ (see \cite[Lemma 5.2]{Yu_Enumeration_of_holomorphic_cylinders_II}).
	This shows the equivalence between (\ref{lem:Msm_smooth:pullback}) and (\ref{lem:Msm_smooth:surjective}).
	
	Now assume (\ref{lem:Msm_smooth:pullback}) and (\ref{lem:Msm_smooth:surjective}).
	This implies that $H^1(C,V)=0$.
	So for both spaces $\cM^\sd(U,\bP,\beta)_\mu$ and $\cM^\sd(U,\bP,\beta)$, the dimension at $\nu$ is equal to the dimension of the Zariski tangent space at $\nu$ (see \cite[Chapter II Theorem 1.7]{Kollar_Rational_curves_on_algebraic_varieties} and \cite[Proposition 5.3]{Keel_Rational_curves}).
	Hence both spaces are smooth at $\nu$.
	Then (\ref{lem:Msm_smooth:surjective}) implies (\ref{lem:Msm_smooth:smooth_fiberwise}).
	Moreover, it implies that the derivative $d\Phi$ is surjective at $\nu$, hence we obtain  (\ref{lem:Msm_smooth:smooth}).
	
	The directions (\ref{lem:Msm_smooth:smooth})$\implies$(\ref{lem:Msm_smooth:smooth_fiberwise})$\implies$(\ref{lem:Msm_smooth:surjective}) are obvious.
	
	Now we assume the equivalent conditions in the lemma.
	By the Riemann-Roch formula, we compute that the dimension of $\cM^\sd(U,\bP,\beta)_\mu$ at $\nu$ is equal to
	\[h^0(C,V)=h^0(C,V)-h^1(C,V)=\rank V + \deg V = \dim Y +0=\dim Y.\]
	Since they are smooth at $\nu$, we deduce that they are étale at $\nu$.
	This shows (\ref{lem:Msm_smooth:etale}).
	
	For (\ref{lem:Msm_smooth:boundary}), note that the natural inclusion of sheaves
	\[T_Y(-\log D)\hookrightarrow T_Y\]
	induces by restriction a map
	\[T_Y(-\log D)|_{D_i^\circ}\to T_Y|_{D_i^\circ},\]
	whose image is $T_{D_i^\circ}$.
	Then $f^*(T_{D_i})_{p_i}$ is isomorphic to a quotient of $V_{p_i}$.
	The derivative $d\ev_i$ at the point $\nu$ is given by the map
	\[H^0(C,V)\to f^*(T_{D_i})_{p_i},\]
	which factors as
	\[H^0(C,V)\to V_{p_i}\twoheadrightarrow f^*(T_{D_i})_{p_i}.\]
	Since $V$ is a trivial vector bundle, the first arrow in the diagram above is also surjective.
	It follows that the derivative $d\ev_i$ is surjective at the point $\nu$, hence the derivative $d\Phi_i^\partial$ is also surjective at $\nu$, completing the proof of (\ref{lem:Msm_smooth:boundary}).
\end{proof}

\begin{lemma} \label{lem:beta_nef}
	If $\cM^\sd(U,\bP,\beta)$ is nonempty, then $\beta\in\NE(Y)$ is nef.
\end{lemma}
\begin{proof}
	Let $F$ be any effective divisor on $Y$.
	Let $\nu=[C,(p_j)_{j\in J}, f\colon C\to Y]$ be a closed point in $\cM^\sd(U,\bP,\beta)$.
	Using \cref{lem:Msm_smooth}(\ref{lem:Msm_smooth:etale}), by choosing $q$ to be any general closed point on each irreducible component of $C$, we see that there is a small deformation of $\nu$ such that no irreducible component of $C$ maps into $F$.
	Hence $\beta\cdot F = f_*[C]\cdot F \ge 0$.
\end{proof}

\begin{lemma} \label{lem:finitely_many_beta}
	Given $\bP=(P_j)_{j\in J}$, there are only finitely many $\beta\in\NE(Y)$ such that $\cM^\sd(U,\bP,\beta)$ is nonempty.
\end{lemma}
\begin{proof}
	Let $F$ be an ample divisor on $Y$ with $-F|_U$ effective (see \cref{lem:ample_divisor}).
	We decompose $F=\overline{F|_U} + F_D$, where $F_D$ is supported on $D$.
	For any $\beta$ such that $\cM^\sd(U,\bP,\beta)$ is nonempty, by \cref{lem:beta_nef}, we have
	\[F\cdot\beta=F_D\cdot\beta + \overline{F|_U}\cdot\beta\le F_D\cdot\beta.\]
	The right hand side is fixed by $\bP$.
	Since $F$ is ample, there are only finitely such $\beta$.	
\end{proof}

\begin{lemma} \label{lem:Msm_big}
	Let $G$ and $Z$ be two closed subvarieties of $Y$.
	Assume that $G$ is reduced and does not containing any irreducible component of $D^\ess$, and that $Z$ is of codimension at least 2.
	Fix $i\in I$.
	Let $\mu=[C,(p_j)_{j\in J}]\in\ocM_{0,n}$ be a closed point.
	Consider the evaluation map
	\[\ev_i\colon\cM^\sd(U,\bP,\beta)_\mu\to U.\]
	There is a Zariski dense open subset $V\subset Y$ such that any stable map $f\in\ev_i\inv(V)$ satisfies the following conditions:
	\begin{enumerate}
		\item \label{lem:Msm_big:trivial} The pullback $f^*(T_Y(-\log D))$ is a trivial vector bundle on $C$.
		\item \label{lem:Msm_big:Z} The image $f(C)\subset Y$ does not intersect $Z$.
		\item \label{lem:Msm_big:G} The pullback $f\inv(G)$ is a finite set of points without multiplicities, disjoint from the nodes and the marked points of $C$.
		In particular, the map $f\inv(G)\mapsto f$ is finite étale over $\ev_i\inv(V)$.
	\end{enumerate}
	As application, let $G\coloneqq E$ as in \cref{ass:strata}, and $Z\coloneqq Y^\idt\cup(D^\ess\setminus W)$ as in \cref{lem:toric_model}.
	We obtain a Zariski dense open subset $V\subset Y$ such that
	\[\ev_i\inv(V)\subset\cM^\sm(U,\bP,\beta)_\mu.\]
\end{lemma}
\begin{proof}
	Let $L\subset\cM^\sd(U,\bP,\beta)_\mu$ be the locus where Condition (\ref{lem:Msm_big:trivial}) is not satisfied.
	By \cref{lem:Msm_smooth}, $L$ is also the locus where the derivative $d\ev_i$ is not surjective.
	So it follows from \cite[Chapter III Proposition 10.6]{Hartshorne_Algebraic_geometry} that the Zariski closure of $\ev_i(L)$ is of codimension at least one in $Y$.
		Then $V\coloneqq\overline{\ev_i(L)}^c\subset Y$ is a Zariski dense open subset, and any stable map $f\in\ev_i\inv(V)$ satisfies Condition (\ref{lem:Msm_big:trivial}).
	
	By \cref{lem:Msm_smooth}(\ref{lem:Msm_smooth:boundary}), for $j\in B$, the map \[\ev_j\colon\cM^\sd(U,\bP,\beta)_\mu\to D_j\]
	is étale over $\ev_i\inv(V)$.
	By assumption, $Z\cap D_j\subset D_j$ is of codimension at least 1.
	Therefore, up to shrinking $V$, we can require that for any map $f\in\ev_i\inv(V)$,  the preimage $f\inv(Z)$ does not contain any marked points in $B$.
	
	Consider the map
	\[\Psi\colon C\times\ev_i\inv(V)\to C\times Y\qquad (q,f)\mapsto (q,f(q)).\]
	Let $C^\circ\coloneqq C\setminus\{p_j\}_{j\in B}$.
	The restriction
	\[\Psi^\circ\colon C^\circ\times\ev_i\inv(V)\to C^\circ\times Y,\]
	viewed as a map of varieties over $C^\circ$, is fiberwise étale by \cref{lem:Msm_smooth}(\ref{lem:Msm_smooth:etale}).
	Note that \cref{lem:Msm_smooth}(\ref{lem:Msm_smooth:etale}) also shows that $\ev_i\inv(V)$ is smooth.
	So both the domain and the target of $\Psi^\circ$ are smooth over $C^\circ$, and we deduce that $\Psi^\circ$ is étale.
	Then the closure $K$ of $(\pr_Y\circ\Psi\circ\Psi\inv)(C^\circ\times Z)$ in $Y$ has codimension at least 1.
	Hence shrinking $V$ by intersecting with $K^c$, Condition (\ref{lem:Msm_big:Z}) is satisfied.
	
	By \cref{lem:Msm_smooth}(\ref{lem:Msm_smooth:boundary}), after shrinking $V$, we can require that for any $f\in\ev_i\inv(V)$,  $f\inv(G)$ does not contain any boundary marked points of $C$ (i.e.\ those indexed by $B$).
	Applying \cref{lem:Msm_smooth}(\ref{lem:Msm_smooth:etale}) to the nodes and the interior marked points of $C$ (i.e.\ those indexed by $I$), after further shrinking $V$, we can require that $f\inv(G)$ does not contain any nodes or internal marked points of $C$.
	So we have
	\[\Psi\inv(C\times G)\subset\Psi\inv[(C^\circ\setminus C^\sing)\times G].\]
	
	Note that the singular locus $G^\sing$ of $G$ has codimension at least 2 in $Y$.
	So the proof for Condition (\ref{lem:Msm_big:Z}) shows that after shrinking $V$, we can achieve $f(C)\cap G^\sing=\emptyset$ for any $f\in\ev_i\inv(V)$.
	Thus we have
	\[\Psi\inv(C\times G)\subset\Psi\inv[(C^\circ\setminus C^\sing)\times(G\setminus G^\sing)].\]
	So $\Psi\inv(C\times G)$ is smooth.
	Then by generic smoothness, $\Psi(C\times G)$ is étale over a Zariski dense open subset $W\subset\ev_i\inv(V)$.
	So replacing $V$ by $V\setminus\overline{\ev_i(W^c)}$, we see that $\Psi\inv(C\times G)$ is étale over $\ev_i\inv(V)$.
	It is in fact finite étale as $G$ is closed in $Y$.
	As we have shown above that $\Psi\inv(C\times G)$ is disjoint from the nodes and the marked points of $C$, this concludes the proof of Condition (\ref{lem:Msm_big:G}).
\end{proof}

In \cref{lem:Msm_big}, we managed to ``bring'' $\cM^\sd(U,\bP,\beta)_\mu$ into $\cM^\sm(U,\bP,\beta)_\mu$ by controlling the evaluation map.
It turns out that we can also ``bring'' $\cM(U,\bP,\beta)_\mu$ into $\cM^\sd(U,\bP,\beta)_\mu$ using the evaluation map, see \cref{lem:stable_domain}.
We need the following lemma describing the stabilization of domain curves in $\cM(U,\bP,\beta)$.

\begin{lemma} \label{lem:A1-curve}
	Assume $\abs{B}\ge 2$, $\abs{I}\ge 1$.
	Let $f\colon[C,(p_j)_{j\in J}]\to Y$ be a stable map in $\cM(U,\bP,\beta)$.
	Let $s\colon C\to C^\st$ be the stabilization of the pointed domain curve.
	Then the exceptional locus of $s$ is a disjoint union of irreducible components, each of which contains exactly two special points: a marked point $p_i$ with $i\in B$ and a node of $C$.
\end{lemma}
\begin{proof}
	Let $E$ be an irreducible component of $C$ containing fewer than 3 special points.
	If such an $E$ exist, the curve $C$ is reducible.
	So $E$ contains at least one node of $C$.
	
	By definition of stable map, $f$ is not constant on $E$.
	Thus, since $U$ is affine, the image $f(E)$ meets the boundary $D$.
	By the definition of $\cM(U,\bP,\beta)$, the pullback $f\inv(D)$ is equal to $\sum_{j\in B} m_j p_j$.
	So $E$ contains one of $\{p_j\}_{j\in B}$.
	Since $E$ contains fewer than 3 special points, and it already contains one node of $C$, $E$ contains exactly one marked point $p_i, i\in B$ and one node of $C$.
	
	Now let $E_1$ and $E_2$ be two such components.
	If $E_1\cap E_2\neq\emptyset$, since each contains exactly one node of $C$, it follows that $C=E_1\cup E_2$.
	This is impossible as $\abs{I}\ge 1$.
	Therefore, by contracting all irreducible components of $C$ containing fewer than 3 special points, we obtain a stable pointed curve.
	So the exceptional locus of $s$ is the disjoint union of all irreducible components of $C$ containing fewer than 3 special points, completing the proof.
\end{proof}

\begin{lemma} \label{lem:stable_domain}
	Assume $\abs{B}\ge 2$, $\abs{I}\ge 1$.
	Fix $i\in I$.
	Let $\mu\in\ocM_{0,n}$ be a closed point.
	There is a Zariski dense open subset $V\subset Y$ such that
	\[\cM(U,\bP,\beta)_{\mu,V}\subset\cM^\sd(U,\bP,\beta)_\mu,\]
	where the subscript $V$ denotes the preimage of $V$ by the evaluation map $\ev_i$.
\end{lemma}
\begin{proof}
	We call a map $f\colon\bbP^1_k\to Y$ an \emph{$\bbA^1$-curve} if $f\inv(D)\subset\bbP^1_k$ is a single set-theoretic point.
	Fix any ample line bundle $L$ on $Y$.
	By log Kodaira dimension, there is a codimension one subvariety $Z\subset Y$ containing the image of every $\bbA^1$-curve whose degree with respect to $L$ is at most $L\cdot\beta$ (see \cite[Lemma 5.11]{Keel_Rational_curves}).
	
	Given any $B'\subset B$, for each $j\in J$, let
	\[P'_j\coloneqq\begin{cases}
	P_j &\text{ if }j\in B',\\
	0 &\text{ otherwise,}
	\end{cases}\]
	and denote $\bP'\coloneqq(P'_j)$.
	By \cref{lem:Msm_big} Condition (\ref{lem:Msm_big:trivial}) and \cref{lem:Msm_smooth}(\ref{lem:Msm_smooth:etale}), there exists a Zariski dense open subset $V\subset Y$ such that for any subset $B'\subset B$, any $\beta'\in\NE(Y)$ with $\beta'\cdot L\le\beta\cdot L$, any $j\notin B'$, and any stable map $f\in\cM^\sd(U,\bP',\beta')_{\mu,V}$, we have $f(p_j)\notin Z$.
	
	Now let $f\colon[C,(p_j)_{j\in J}]\to Y$ be a stable map in $\cM(U,\bP,\beta)_{\mu,V}$.
	Let $s\colon C\to C^\st$ be the stabilization of the pointed domain curve, and let $\cE$ be the exceptional locus.
	By \cref{lem:A1-curve}, there exists a subset $B'\subset B$ such that we can write $\cE=\bigsqcup_{j\in B\setminus B'} E_j$, where each $E_j$ is irreducible and contains exactly two special points: a node of $C$ and the marked point $p_j$;
	moreover, $C^\st$ can be identified with the closure of $C\setminus\cE$ in $C$.
	For each $j\in B\setminus B'$, let $p'_j\coloneqq E_j\cap C^\st$; we have $f(p'_j)\in Z$.
	For each $j\in B'$, let $p'_j\coloneqq p_j$.
	
	Let $\beta'$ be the curve class $f_*[C^\st]\in\NE(Y)$.
	Then the stable map
	\[f|_{C^\st}\colon[C^\st, (p'_j)]\to Y\]
	belongs to $\cM^\sd(U,\bP',\beta')_{\mu,V}$.
	By the choice of $V\subset Y$, for every $j\in B\setminus B'$, we have $f(p'_j)\notin Z$, which contradicts $f(p'_j)\in Z$ above.
	So the marked points $p'_j, j\in B\setminus B'$ cannot exist.
	In other words, we have $B=B'$, and the stabilization map $s\colon C\to C^\st$ is an isomorphism, completing the proof.
\end{proof}

\begin{lemma} \label{lem:smoothness_fiber}
	Given $i\in I$, there is a Zariski dense open subset $V\subset Y$ such that
	\[\cM(U,\bP,\beta)_{\mu,V}\subset\cM^\sm(U,\bP,\beta)_\mu,\]
	where the subscript $V$ denotes the preimage of $V$ by the evaluation map $\ev_i$.
\end{lemma}
\begin{proof}
	This is a combination of Lemmas \ref{lem:stable_domain} and \ref{lem:Msm_big}.
\end{proof}

\begin{proposition} \label{prop:smoothness_all}
	Let $Z \subset \ocM_{0,n}$ be a Zariski closed subvariety (possibly $Z = \ocM_{0,n}$).
	Given $i\in I$, there is a Zariski dense open subset $O\subset Z\times Y$ such that the following hold:
	\begin{enumerate}
		\item \label{prop:smoothness_all:preimage} The preimage of $O$ by the map
		\[\Phi_i\colon\cM(U,\bP,\beta)\to\ocM_{0,n}\times Y\]
		is contained in $\cM^\sm(U,\bP,\beta)$.
		\item \label{prop:smoothness_all:etale} The map $\Phi_i$ is representable (i.e.\ non-stacky) and finite étale over $O$.
	\end{enumerate}  
\end{proposition}
\begin{proof}
	Suppose (\ref{prop:smoothness_all:preimage}) fails.
	Then $\cM(U,\bP,\beta)_Z \setminus\cM^\sm(U,\bP,\beta) \to Z \times Y$ is dominant, and thus its image contains a dense open subset $W \subset Z \times Y$.
	So for any closed point $\mu\in \pi_Z(W) \subset Z$, $\cM(U,\bP,\beta)_{\mu}\setminus\cM^{\sm}(U,\bP,\beta) \to \mu \times Y$ dominant.
	This contradicts \cref{lem:smoothness_fiber}.
	
	Since $\cM^\sm(U,\bP,\beta)\subset\cM^\sd(U,\bP,\beta)$ are varieties, $\Phi_i$ is representable over $O$ by (\ref{prop:smoothness_all:preimage}).
	Moreover, by \cref{lem:Msm_smooth}(\ref{lem:Msm_smooth:etale}), we deduce that $\Phi_i$ is étale over $O$.
	Hence, up to further shrinking $O$, the map $\Phi_i$ becomes finite étale over $O$.
\end{proof}

\section{Trees, walls, spines and tropical curves} \label{sec:tropical}

Given any stable map $[C,(p_j)_{j\in J},f\colon C\to Y^\an]$ in $\cM^\sm(U^\an,\bP,\beta)$, if we take the convex hull of the marked points in $C$, and compose with the projection $\tau\colon(Y\setminus Y^\idt)^\an\to\oM_\bbR$ (see \cref{nota:Et}), we obtain a metric tree mapping into $M_\bbR$.
The study of such trees in $M_\bbR$ provides valuable information about analytic curves in $U^\an$.
While all we care are those trees that come from analytic stable maps, there are no general ways for characterizing them.
All we can do is to find a list of conditions that capture sufficiently many combinatorial features induced by the analytic curves, so that we can deduce statements first at the combinatorial level, which will then help establish statements at the analytic level.

Therefore, we introduce in this section several combinatorial notions: twigs, walls, spines and tropical curves.
In brief, tropicalization of analytic stable maps gives rise to tropical curves in $M_\bbR$, which enjoy the best combinatorial properties, e.g.\ everywhere balanced, see \cref{prop:tropicalization_of_stable_map}.
When we take the convex hull of the marked points in a tropical curve, we obtain a spine in $M_\bbR$, and the branches of the tropical curve off the spine are called twigs.
The definitions of twigs, spines and tropical curves (see Definitions \ref{def:twig}, \ref{def:spine} and \ref{def:tropical_curve} respectively) are designed to capture the combinatorial features from this construction.

By definition, twigs can only meet the boundary $\partial\oM_\bbR$ at the subset $E_\rt^\trop$, i.e.\ the tropicalization of the ``blowup locus''.
Together with the balancing condition, we deduce that the image of the twigs has codimension at least one.
We will explain this via the notion of walls (see \cref{def:wall}) and an inductive procedure (see \cref{const:walls_by_induction}), which can be thought of as growing twigs backwards: we start from the leaves, and determine the branches by the balancing condition.

Walls give a strong constraint on the shape of spines, that is, a spine can only bend at the walls, which has codimension at least one.
Otherwise, we cannot attach twigs to the spine in order to obtain an everywhere balanced tropical curve.
This strong constraint is the key for the rigidity of spines, see \cref{prop:rigidity_spine}.
More precisely, we must restrict to transverse spines (see \cref{def:transverse}) for the rigidity to hold, and we will show in \cref{prop:transversality} that the transverse locus is sufficiently big.

The space of spines has a natural topology induced by the compact open topology on the space of maps, see \cref{def:compact_open_topology}.
We will show that the spine map from the space of analytic stable maps to the space of spines is a continuous map, see \cref{prop:continuity}.
In fact, the tropicalization map from the space of analytic stable maps to the space of tropical curves is also continuous (see \cite[\S 8]{Yu_Tropicalization_of_the_moduli_space_of_stable_maps} and \cite{Ranganathan_Skeletons_of_stable_maps_I}), from which we can deduce the continuity of the spine map.
However, the topology on the space of tropical curves is more complicated to describe, and we have engineered the proofs in later sections without using the topology on the space of tropical curves.
This leads to a simplification of the combinatorics in this section.

Let us begin by setting up the terminology regarding nodal metric trees.
Nodal metric trees arise naturally as convex hull of points in nodal non-archimedean analytic curves.

\begin{definition} \label{def:metric_tree}
	A \emph{metric tree} $\Gamma$ consists of a finite (combinatorial) tree together with an extended metric on its topological realization such that each edge is modeled on an interval $[0,l]$ for $l\in(0,+\infty]$.
	An endpoint $v$ of an edge $e$ of $\Gamma$ is called an \emph{infinite endpoint} if any neighborhood of $v$ in $e$ has infinite length, otherwise it is called a \emph{finite endpoint}.
	A vertex $v$ of $\Gamma$ is called an \emph{infinite vertex} if it is an infinite endpoint of an edge incident to it, otherwise it is called a \emph{finite vertex}.
	A \emph{topological edge} of $\Gamma$ is a simple path in $\Gamma$ whose interior does not contain any vertex of valency greater than 2, and whose endpoints are vertices of valency not equal to 2.
	A \emph{topological edge} incident to a 1-valent vertex is also called the \emph{leg} incident to the vertex.
\end{definition}

\begin{definition} \label{def:nodal}
	A metric tree $\Gamma$ is called \emph{nodal} if it satisfies the following condition:
	if $v$ is an infinite endpoint of an edge $e$, then $v$ is at most 2-valent, and $v$ is an infinite endpoint of each edge attached to $v$.
	A \emph{node} of a nodal metric tree is a 2-valent infinite vertex (think of gluing two copies of $[0,+\infty]$ at $+\infty$).
	A nodal metric tree is called \emph{irreducible} if it contains no node.
	An \emph{irreducible component} of a nodal metric tree $\Gamma$ is a maximal irreducible subtree of $\Gamma$.
\end{definition}

\begin{definition} \label{def:stable_nodal_metric_tree}
	A nodal metric tree with $\Gamma$ is called \emph{stable} if none of its irreducible components is isomorphic to $[-\infty,+\infty]$ (as extended metric space).
	\end{definition}

Let $J$ be a finite set of cardinality $n$.

\begin{definition}
	A \emph{nodal metric tree with $n$ legs} (indexed by $J$) consists of a nodal metric tree $\Gamma$ with 1-valent vertices $(v_j)_{j\in J}$, such that there are no other 1-valent vertices.
	It is called \emph{extended} if every $v_j$ is infinite.
\end{definition}

\begin{definition} \label{def:NT}
	Two nodal metric trees with $n$ legs (indexed by $J$) are considered equivalent if they become isomorphic after some subdivisions of edges,	(or equivalently if they are isomorphic as extended metric spaces preserving the labeling of 1-valent vertices). 
	A nodal metric tree with $n$ legs is called \emph{simple} if it does not contain any finite 2-valent vertex, in other words, it contains the least vertices among all equivalent ones.
	For any $F\subset J$, let $\NT^F_J=\NT^F_n$ denote the set of simple stable  nodal metric trees with $n$ legs (indexed by $J$) whose finite 1-valent vertices are indexed by $F$.
	We will drop the superscript $F$ when $F=\emptyset$.
\end{definition}

\begin{construction} \label{const:topology_NT}
	Fix $F\subset J$.
	Let $T=[\Gamma,(v_j)_{j\in J}]\in\NT^F_J$.
	For any positive real number $\epsilon$, let $U(T,\epsilon)$ be the subset of $\NT^F_J$ consisting of simple stable nodal metric trees with $n$ legs $[\Gamma',(v'_j)_{j\in J}]$ satisfying the following conditions:
	\begin{enumerate}
		\item There is a continuous map $c\colon\Gamma'\to\Gamma$ contracting a subset of topological edges of $\Gamma'$, sending each $v'_j$ to $v_j$, and each node of $\Gamma'$ to a node of $\Gamma$.
				\item The sum of lengths of all edges in $\Gamma'$ contracted by $c$ is less than $\epsilon$.
		\item For each edge $e$ of $\Gamma$, let $e'$ be the edge of $\Gamma'$ such that $c(e')=e$.
		If $e$ has finite length, then the difference between the lengths of $e$ and $e'$ is less than $\epsilon$.
		If $e$ has infinite length, then the length of $e'$ is greater than $1/\epsilon$.
	\end{enumerate}
	Let $U(T,\epsilon)$ be a base of open neighborhoods of $T$ in $\NT^F_J$.
	This gives a topology on $\NT^F_J$.
\end{construction}

\begin{remark} \label{rem:comparison_ACP}
	The notion of extended nodal metric tree with $n$ legs is equivalent to the notion of $n$-pointed genus 0 extended tropical curves in the sense of Abramovich-Caporaso-Payne \cite{Abramovich_The_tropicalization_of_the_moduli_space_of_curves}.
	Our moduli space $\NT^\emptyset_n$ is homeomorphic to their moduli space $\oM_{0,n}^\trop$ of $n$-pointed genus 0 extended tropical curves.
	Recall from \cite[Theorem 1.2.1]{Abramovich_The_tropicalization_of_the_moduli_space_of_curves} that the tropicalization map $\ocM_{0,n}^\an\to\oM_{0,n}^\trop\simeq\NT_n^\emptyset$ is continuous, proper and surjective.
\end{remark}

Next we consider maps from nodal metric trees into $\oM_\bbR$, which will include twigs, spines and tropical curves as special cases.

\begin{definition} \label{def:pointedtree}
	A \emph{pointed tree in $M_\bbR$} consists of a nodal metric tree $\Gamma$, a set of different 1-valent vertices $(v_j)_{j\in J}$ called {\it marked points}, and a continuous map $h\colon\Gamma\to\oM_\bbR$ satisfying the following conditions:
	\begin{enumerate}
		\item The preimage $h\inv(\partial\oM_\bbR)$ is a subset of infinite 1-valent vertices. 
		\item \label{def:pointedtree:Zaffine} The map $h$ is \emph{\Zaffine} in the following sense:
		for each finite vertex $v$ of $\Gamma$ and each edge $e$ incident to $v$, the restriction $h|_{e^\circ}\colon e^\circ\to M_\bbR$ is affine with integer derivative $w_{(v,e)}\in M$, defined via the unit tangent vector pointing from $v$ to $e$.
		We call $w_{(v,e)}$ the \emph{weight vector} of the edge $e$ at $v$.
	\end{enumerate}
\end{definition}

\begin{definition} \label{def:subdivision_of_edges}
	Two pointed trees in $M_\bbR$ are considered equivalent if they become isomorphic after some subdivisions of edges.
	A pointed tree in $M_\bbR$ is called \emph{simple} if it contains the least vertices among all equivalent ones.
	Given any pointed tree $T$ in $M_\bbR$, by removing redundant 2-valent vertices in the domain of $T$, we obtain the unique simple representative in the equivalence class of $T$.
\end{definition}

\begin{definition} \label{def:combinatorial_type}
	The \emph{combinatorial type} of a pointed tree $T=[\Gamma,(v_j)_{j\in J}, h]$ consists of the following data:
	\begin{enumerate}
		\item The underlying combinatorial tree of $\Gamma$, remembering the marked points and the nodes.
		\item For each vertex $v$ of $\Gamma$, the open cell of $\oM_\bbR$ containing $h(v)$ (recall that $\oM_\bbR$ has the structure of an extended simplicial cone complex, see \cref{nota:toric_model}).
		\item For each vertex $v$ of $\Gamma$ and each edge $e$ incident to $v$, the weight vector $w_{(v,e)}\in M$.
	\end{enumerate}
\end{definition}

\begin{definition} \label{def:norm_of_weight_vector}
	The embedding $\Sigma_\rt\hookrightarrow\bbR_{\ge 0}^{I_{D_\rt}}$ induces an embedding $M\hookrightarrow\bbZ_{\ge 0}^{I_{D_\rt}}$.
	Let $\abs{\cdot}$ denote the function on $\bbZ_{\ge 0}^{I_{D_\rt}}$ given by the sum of every coordinate.
	It induces a piecewise linear function on $M$, denoted again by $\abs{\cdot}$, via the embedding $M\hookrightarrow\bbZ_{\ge 0}^{I_{D_\rt}}$.
	We call it the \emph{norm} of vectors in $M$.
\end{definition}

\begin{definition}[see \cref{fig:spine}] \label{def:twig}
	A \emph{twig} in $M_\bbR$ is a pointed tree $[\Gamma,(r,u_1,\dots,u_m),h]$ in $M_\bbR$ for some positive integer $m$, satisfying the following conditions:
	\begin{enumerate}
		\item The vertex $r$ is a 1-valent finite vertex called \emph{root}; the vertices $u_1,\dots,u_m$ are different 1-valent infinite vertices.
		These are the only 1-valent vertices of $\Gamma$.
		\item We have $h\inv(\partial\oM_\bbR)=\{u_1,\dots,u_m\}$.
		\item The image $h(\{u_1,\dots,u_m\})$ is contained in $E_\rt^\trop\subset\partial\oM_\bbR$, see \cref{nota:Et}.
		\item (Balancing condition) The \Zaffine map $h$ is balanced at every vertex of $\Gamma$ of valency greater than 1 (where for a node, being balanced means that the two incident edges are contracted).
	\end{enumerate}
	Given a twig $T=[\Gamma,(r,u_1,\dots,u_m),h]$, for $i=1,\dots,m$, let $e_i$ be the edge of $\Gamma$ incident to $u_i$, and let $u'_i$ be the other endpoint of $e_i$.
	We define the \emph{degree} $\deg T$ to be the sum $\sum_{i=1}^m\abs{w_{(u'_i,e_i)}}$.
	Let $e$ be the edge incident to $r$.
	We call the weight vector $w_{(r,e)}$ the \emph{monomial} of the twig, and $-w_{(r,e)}$ the \emph{direction} of the twig.
\end{definition}

\begin{remark} \label{rem:twig_bound}
	Since $\Gamma$ is a tree, the number of vertices of valency greater than 2 in $\Gamma$ is less than or equal to $m$.
	Since $m$ is less than or equal $\deg T$, we deduce that the number of vertices of valency greater than 2 in $\Gamma$ is less than or equal to $\deg T$.
	Consequently, by the balancing condition, given $A \in\bbN$, there are only finitely many combinatorial types of twigs of degree at most $A$. 
\end{remark}

\begin{definition} \label{def:wall}
	A \emph{wall} in $M _\bbR$ is a pair $(\fw,v)$ consisting of a closed convex rational polyhedral cone $\fw\subset M_\bbR$ of codimension at least one, and a nonzero vector $v\in M$ such that $\fw-v\subset\fw$.
	We call $v$ the \emph{monomial} of the wall, and $-v$ is the \emph{direction} of the wall.
\end{definition}

\begin{construction}[see \cref{fig:wall}] \label{const:initial_walls}
	Let $E_\rt\subset Y_\rt$ and $E_\rt^\trop\subset\partial\oM_\bbR$ be as in \cref{nota:Et}.
	For every closed polyhedral cell $\sigma\subset E_\rt^\trop\subset\partial\oM_\bbR$, define
	\begin{align*}
		\fw_\sigma&\coloneqq\Set{x\in M_\bbR | \exists v\in M,\ \lim_{\lambda\to+\infty} x+\lambda v\in\sigma},\\
		V_\sigma&\coloneqq\Set{v\in M | \exists x\in M_\bbR,\ \lim_{\lambda\to+\infty} x+\lambda v\in\sigma}.
	\end{align*}
	By \cref{lem:toric_model}(\ref{lem:toric_model:W}), $E_\rt$ does not contain any stratum of $D_\rt$, so $E_\rt^\trop$ does not contain any stratum of $\partial\oM_\bbR$.
	Hence every $\fw_\sigma$ has codimension at least one.
	Moreover, as $E_\rt$ is constant over $k$, $\sigma$ is conical in the cell of $\partial\oM_\bbR$ containing $\sigma$, hence $\fw_\sigma$ is conical in $M_\bbR$.
	
	For any $A\in\bbN$, let
	\[W_A^0\coloneqq\Set{(\fw_\sigma,v) | \sigma\text{ closed cell in }E_\rt^\trop,\ v\in V_\sigma,\ \abs{v}\le A}.\]
\end{construction}

\begin{construction}[see \cref{fig:wall}] \label{const:walls_by_induction}
	For $A\in\bbN$, we construct $W_A^n$ by induction on $n$.
	Assume we already have $W_A^n$.
	Then $W_A^{n+1}$ contains all the walls in $W_A^n$ together with the following new walls:
	for every two walls (not necessarily different) $(\fw_1,v_1), (\fw_2,v_2)\in W_A^n$, we add a wall $(\fw,v)\in W_A^{n+1}$ by setting $v\coloneqq v_1+v_2$ and 
	$\fw\coloneqq\set{x-\lambda v|x\in\fw_1\cap\fw_2, \lambda\in\bbR_{\ge 0}}$.
	
	Let $\Wall_A$ be the union of $\fw\subset M_\bbR$ over all $(\fw,v)\in W_A^A$.
	By \cref{rem:twig_bound}, for any twig in $M_\bbR$ of degree at most $A$, its image in $M_\bbR$ lies in $\Wall_A$.
		
	We deduce by induction that $\Wall_A$ is a finite conical rational polyhedral subset of $M_\bbR$ of codimension at least one.
\end{construction}

\begin{figure}[!ht]
	\centering
	\setlength{\unitlength}{0.25\textwidth}
	\begin{picture} (1,1)
		\put(0,0){\includegraphics[width=\unitlength]{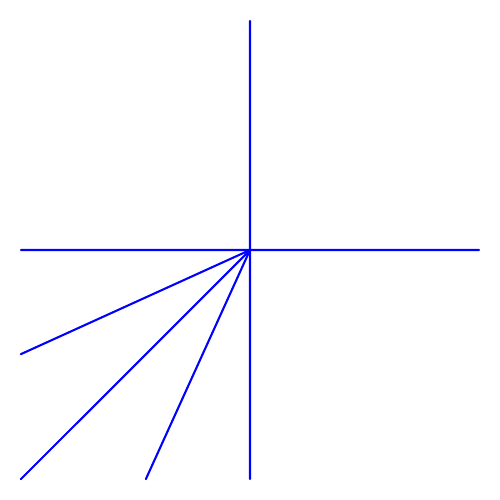}}
	\end{picture}
	\caption{$\Wall_A$ for \cref{ex:surface} after two iterations.
	The horizontal and vertical lines come from the initial walls, i.e.\ from $W_A^0$.
	The skew rays come from the iterations.}
	\label{fig:wall}
\end{figure}

\begin{figure}[!ht]
	\centering
	\setlength{\unitlength}{\textwidth}
	\begin{picture} (1,0.25)
		\put(0,0){\includegraphics[width=\unitlength]{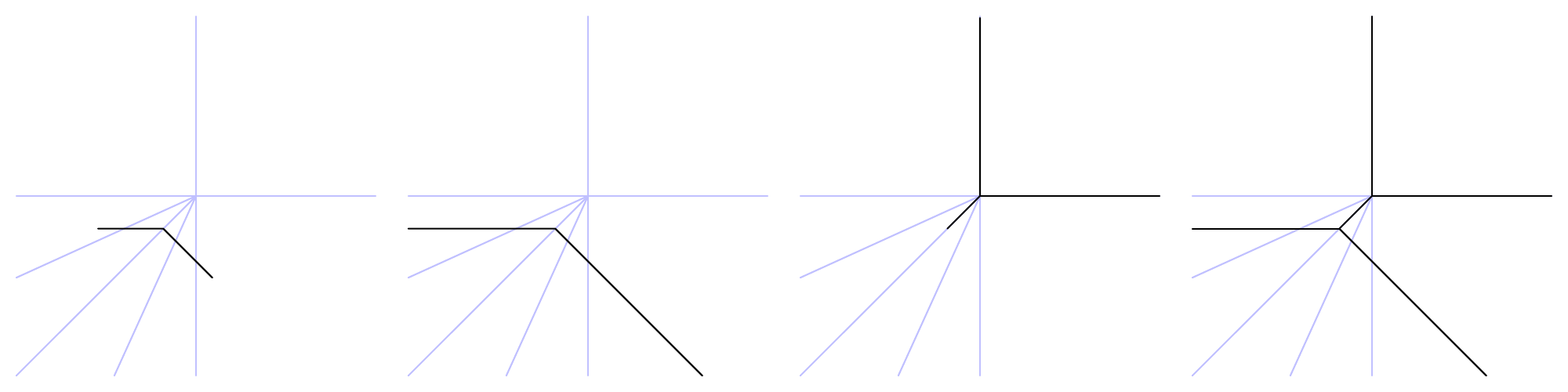}}
	\end{picture}
	\caption{Respectively: a (non-extended) spine, an extended spine, a twig, and a tropical curve for \cref{ex:surface}.}
	\label{fig:spine}
\end{figure}

Fix a finite set $J$ of cardinality $n\ge 2$.
Fix $\bP\coloneqq(P_j)_{j\in J}$  with $P_j \in M$.
Let $J=B\sqcup I$ be as in \eqref{eq:BI_of_P}.
Fix $F\subset J$ and $A\in\bbN$.

\begin{definition}[see \cref{fig:spine}] \label{def:spine}
	A \emph{spine} in $M_\bbR$ of type $\bP^F$ with respect to $\Wall_A$\footnote{We will omit ``with respect to $\Wall_A$'' when $A$ is clear from the context.} is a pointed tree $[\Gamma,(v_j)_{j\in J},h]$ in $M_\bbR$ satisfying the following conditions:
	\begin{enumerate}
		\item \label{def:spine:marked_points} The vertices $v_j, j\in J$ are the only 1-valent vertices of $\Gamma$.
		The weight vector of the edge incident to $v_j$ (pointing outwards) is $P_j$.
		For each $j\in J$, $v_j$ is a finite vertex if and only if $j\in F$.
		\item \label{def:spine:stable} The domain $\Gamma$ is stable in the sense of \cref{def:stable_nodal_metric_tree}.
				\item \label{def:spine:boundary} The preimage $h\inv(\partial\oM_\bbR)$ is equal to $\set{v_j | j\in B\setminus F}$.
		\item \label{def:spine:bending} (Bending condition) 
		For every vertex $v$ of valency greater than 1, if $s\coloneqq -\sum_{e\ni v} w_{(v,e)}$ is nonzero, we call $v$ a \emph{bending vertex}, and we require that $h(v)\in\Wall_A$ and $s$ lies in the tangent cone of $\Wall_A$ at $h(v)$.
			\end{enumerate}
	When $F=\emptyset$, we will drop the superscript $F$ from the notation $\bP^F$;
	in this case the spine is called \emph{extended}.
	Let $\SP(M_\bbR,\bP^F)$ denote the set of simple (see \cref{def:subdivision_of_edges}) spines in $M_\bbR$ of type $\bP^F$ with respect to $\Wall_A$.
\end{definition}

\begin{remark} \label{rem:internal_leg_contracted}
	It follows from the bending condition that for each $i\in I$, if $h(v_i)\notin\Wall_A$, then $h$ is constant on the leg incident to $v_i$.
\end{remark}

\begin{definition} \label{def:compact_open_topology}
	Let $\Gamma,N,M$ be topological spaces and $\pi\colon\Gamma\to N$ a continuous map.
	Let $\Cont(\Gamma/N,M)$ be the set of pairs $(n\in N, f\colon\Gamma_n\to M)$ where $\Gamma_n\coloneqq\pi\inv(n)$ and $f$ is continuous.
	For every open $S\subset N$, $W\subset M$ and every $K\subset \Gamma_S$ proper over $S$, let $V(S,K,W)\subset\Cont(\Gamma/N,M)$ be the set of maps $f\colon\Gamma_n\to M$ with $n\in S$ and $f(K\cap \Gamma_n)\subset W$.
	They generate a topology on $\Cont(\Gamma/N,M)$ called the \emph{compact open topology}.
	
	Now $\SP(M_\bbR,\bP^F)$ is naturally a subset of $\Cont(\Gamma/\NT_J^F, \oM_\bbR)$, where $\Gamma/\NT_J^F$ denotes the universal nodal metric tree.
	We equip $\SP(M_\bbR,\bP^F)$ with the induced topology.
\end{definition}

\begin{definition} \label{def:transverse}
	A spine $[\Gamma,(v_j)_{j\in J},h]$ in $M_\bbR$ of type $\bP^F$ with respect to $\Wall_A$ is called \emph{transverse} if it satisfies the following conditions:
	\begin{enumerate}
		\item \label{def:transverse:image} $h(\Gamma)\cap M_\bbR$ is transverse to $\Wall_A$, in particular, it does not meet any $(d-2)$-dimensional closed strata of $\Wall_A$, and $h(\Gamma)\cap E_\rt^\trop=\emptyset$.
				\item \label{def:transverse:bending} Every vertex of $\Gamma$ whose image lies in $\Wall_A$ is 2-valent.
				\item \label{def:transverse:finite_end} For every $j\in F$, the image $h(v_j)$ does not lie in the codimension-one skeleton of $\Sigma_\rt$.
				\item \label{def:transverse:extension} For every $j\in F$, let $e_j$ denote the edge of $\Gamma$ incident to $v_j$.
		Then the ray starting from $h(v_j)$ in the direction of $-w_{(v_j,e_j)}$ is transverse to $\Wall_A$.
			\end{enumerate}
	It is called \emph{transverse at infinity} if $h(\Gamma)\cap E_\rt^\trop=\emptyset$.
	Let $\SP^\tr(M_\bbR,\bP^F)\subset\SP(M_\bbR,\bP^F)$ denote the subset consisting of transverse spines.
\end{definition}

\begin{remark} \label{rem:spine_factorization}
	In the context of \cref{def:transverse}, let $\Gamma^B$ be the convex hull of $(v_j)_{j\in B}$ in $\Gamma$ and $r\colon\Gamma\to\Gamma^B$ the retraction map.
	Then by \cref{rem:internal_leg_contracted}, Condition (\ref{def:transverse:bending}) implies that $h\colon\Gamma\to\oM_\bbR$ factors through $r$.
\end{remark} 

\begin{definition}[see \cref{fig:spine}] \label{def:tropical_curve}
	A \emph{tropical curve} in $M_\bbR$ of type $\bP$ is a pointed tree $[\Gamma,(v_j)_{j\in J},h]$ in $M_\bbR$ satisfying the following conditions:
	\begin{enumerate}
		\item \label{def:tropical_curve:marked_points} The vertices $(v_j)_{ j\in J}$ are all infinite 1-valent vertices.
		The weight vector of the edge incident to $v_j$ (pointing outwards) is $P_j$.
		\item \label{def:tropical_curve:stable} The domain $\Gamma$ is stable in the sense of \cref{def:stable_nodal_metric_tree}.
				\item \label{def:tropical_curve:boundary} $h(v_j)\in M_\bbR$ for all $j\in I$, and $h(\Gamma\setminus(v_j)_{j\in J})\cap\partial\oM_\bbR$ is contained in $E_\rt^\trop\subset\partial\oM_\bbR$.
		\item (Balancing condition) \label{def:tropical_curve:balancing} For every vertex $v$ such that $h(v)\in M_\bbR$, we have $\sum_{e\ni v}w_{(v,e)}=0$.
	\end{enumerate}
	Given a tropical curve $T=[\Gamma,(v_j)_{j\in J},h]$, its \emph{degree} $\deg(T)$ is by definition the sum of norms of weights of edges incident to all the marked points.
	Let $\Gamma^s$ denote the convex hull of all the marked points in $\Gamma$.
	The restriction of $h$ to the closure of a connected component of $\Gamma\setminus\Gamma^s$ is called a \emph{twig} of $T$.
	It is a twig in $M_\bbR$ in the sense of \cref{def:twig}.
	We denote by $\degtwig(T)$ the sum of degrees of all twigs of $T$.
	Let $\TC(M_\bbR,\bP)$ denote the set of simple tropical curves in $M_\bbR$ of type $\bP$.
\end{definition}

\begin{lemma} \label{lem:associated_extended_spine}
	Let $T=[\Gamma,(v_j)_{j\in J},h]\in \TC(M_\bbR,\bP)$ be a tropical curve  whose twigs have degrees at most $A$.
	Then
	\[\Sp(T)\coloneqq[\Gamma^s,(v_j)_{j\in J},h|_{\Gamma^s}]\]
	is an extended spine in $M_\bbR$ of type $\bP$ with respect to $\Wall_A$, i.e.\ it belongs to $\SP(M_\bbR,\bP)$.
	We call it the extended spine associated to $T$.
\end{lemma}
\begin{proof}
	Let $\Lambda$ be a twig of $T$.
	As its degree is at most $A$, the image of $\Lambda$ is contained in $\Wall_A$.
	This implies the bending condition (see \cref{def:spine} Condition (\ref{def:spine:bending})) for $\Sp(T)$ to be an extended spine in $M_\bbR$ with respect to $\Wall_A$.
	The other conditions are obvious to check.
\end{proof}

\begin{proposition} \label{prop:tropicalization_of_stable_map}
	Let $[C,(p_j)_{j\in J},f\colon C\to Y^\an]$ be a stable map in $\cM^\sm(U^\an,\bP,\beta)$ as in Notations \ref{nota:Msm}, \ref{nota:moduli_spaces_analytic}.
	Consider the composition $\of\colon C\to (Y\setminus Y^\idt)^\an\to Y_\rt^\an$.
	Let $\Gamma \subset C$ be the convex hull of $\of\inv(D_\rt^\an)\cup\{p_j\}_{j\in J}$.
	Let $\Gamma^s\subset \Gamma$ be the convex hull of the marked points.
	Let $h\coloneqq(\tau\circ\of)|_\Gamma$, where $\tau\colon (Y\setminus Y^\idt)^\an\to\oM_\bbR$ as in \cref{nota:Et}.
	Then
	\[\Trop(f)\coloneqq[\Gamma,(p_j)_{j\in J},h]\in\TC(M_\bbR,\bP),\]
	which we call the \emph{tropical curve associated to $f$}.
	We have
	\[\deg\Trop(f)=\pi_*\beta\cdot D_\rt,\]
		\[\degtwig\Trop(f)=\beta\cdot\tE,\]
	where $\tE\coloneqq\pi^*D_\rt-D$ (See \cref{def:projection_formula_rational} regarding $\pi_*$ and $\pi^*$.
	Here $\pi^*D_\rt$ can also be viewed as pullback of Cartier divisor because the indeterminate locus of $\pi$ has codimension at least 2.).
		Assume $A\ge\beta\cdot\tE$, then
	\[\Sp(f)\coloneqq[\Gamma^s,(p_j)_{j\in J},h|_{\Gamma^s}]\in\SP(M_\bbR,\bP),\]
	which we call the \emph{spine associated to $f$}.
\end{proposition}
\begin{proof}
	By taking a semistable model for $\of\colon C\to Y_\rt^\an$, there exists a nodal metric tree $\Gamma'\subset C$ containing $\Gamma$ such that $\tau\circ\of\colon C\to\oM_\bbR$ factors through $C\xrightarrow{\tau_C}\Gamma'\xrightarrow{h'}\oM_\bbR$, where $\tau_C$ is the retraction map (see \cite[\S 5]{Gubler_Skeletons_and_tropicalizations}).
	By \cite[Theorem 6.14]{Baker_Nonarchimedean_geometry}, $h'$ is balanced at every vertex of $\Gamma'$.
	Hence
	\[[\Gamma',(p_j)_{j\in J},h']\in\TC(M_\bbR,\bP).\]
		Let $T$ be a connected component of $\Gamma'\setminus\Gamma$.
	If $h'|_T$ is not constant, then by the balancing condition, $h'(T)\cap\partial\oM_\bbR\neq\emptyset$; thus $T\cap\of\inv(D_\rt^\an)\neq\emptyset$; this contradicts the choice of $T$.
	Therefore $h'$ is constant on $\Gamma'\setminus\Gamma$.
	Then the balancing of $h'$ implies the balancing of $h$, so
	\[[\Gamma,(p_j)_{j\in J},h]\in\TC(M_\bbR,\bP).\]
	
	By \cref{def:norm_of_weight_vector}, we have $\deg\Trop(f)=\of_*[C]\cdot D_\rt=\pi_*\beta\cdot D_\rt$.
	Note that the degree contributed by all $p_j$ is equal to $\beta\cdot D$.
	By the projection formula in \cref{def:projection_formula_rational}, we have $\deg\Trop(f)=\pi_*\beta\cdot D_\rt=\beta\cdot\pi^*D_\rt$.
	Hence we obtain
	\[\degtwig\Trop(f)=\beta\cdot\pi^*D_\rt-\beta\cdot D=\beta\cdot\tE.\]
	
	Finally, under the assumption that $A\ge\beta\cdot\tE$, we conclude from \cref{lem:associated_extended_spine} that
	\[[\Gamma^s,(p_j)_{j\in J},h|_{\Gamma^s}]\in\SP(M_\bbR,\bP).\]
\end{proof}

\begin{remark}
	In general, both $\Trop(f)$ and $\Sp(f)$ depends on the choice of the embedding $T_M\subset U$.
	However, in the case of skeletal curves, $\Sp(f)$ is in fact independent of the choice (see \cref{sec:skeletal_curves}).
\end{remark}

\begin{definition} \label{def:realizable}
	Tropical curves and spines arising from $\Trop(f)$ and $\Sp(f)$ as in \cref{prop:tropicalization_of_stable_map} are called \emph{realizable}.
\end{definition}

\begin{proposition} \label{prop:continuity}
	The spine map $\Sp\colon\cM^\sm(U^\an,\bP,\beta)\longto\Sp(M_\bbR,\bP)$ is continuous.
\end{proposition}
\begin{proof}
	Let $\Gamma\to\oM_{0,n}^\trop$ be the universal nodal metric tree, $S\subset\oM_{0,n}^\trop$ be any open subset, and $K\subset\Gamma_S$ proper over $S$.
	Let $\cK\to\cS$ be the pullback of $K\to S$ by the proper surjective tropicalization map $\ocM_{0,n}^\an\to\oM_{0,n}^\trop$ (see \cref{rem:comparison_ACP}).
	Let $\tcK\to\tcS$ be the further pullback by the map $\cM^\sm(U^\an,\bP,\beta)\to\ocM_{0,n}^\an$ taking domain curves, so $\tcK\to\tcS$ is proper.
	Let $W\subset\oM_\bbR$ be any open subset, and $\cW\subset Y^\an$ the preimage of $W$ by $\tau$ in \cref{nota:Et}.
	Let $\alpha\colon\cC\to\cM^\sm(U^\an,\bP,\beta)$ be the universal curve and $\cf\colon\cC\to Y^\an$ the universal map.
	Let $V(S,K,W)$ be as in \cref{def:compact_open_topology}, a generator of the compact-open topology.
	Then $\Sp\inv(V(S,K,W))\subset\cM^\sm(U^\an,\bP,\beta)$ is equal to $\tS\setminus\alpha(\tcK\setminus\cf\inv(\cW))$.
	Since $\tcK$ is proper over $\tcS$ and $\cf\inv(\cW)$ is open, $\alpha(\tcK\setminus\cf\inv(\cW))$ is closed in $\tcS$.
	Therefore $\Sp\inv(V(S,K,W))=\tS\setminus\alpha(\tcK\setminus\cf\inv(\cW))$ is open in $\tcS$, hence open in $\cM^\sm(U^\an,\bP,\beta)$.
	This shows that the spine map is continuous.
\end{proof}

\begin{definition} \label{def:A_big}
	Motivated by \cref{prop:tropicalization_of_stable_map}, we say that $A \in \bbN$ is \emph{big with respect to} $\beta \in \NE(Y)$ if $A\ge\beta\cdot\tE$.
\end{definition}

The data $\bP=(P_j)_{j\in J}$, $\beta\in\NE(Y)$ and $A\in\bbN$ restrict each other as shown in the following two lemmas:

\begin{lemma} \label{lem:finiteness_given_P}
	Given $\bP=(P_j)_{j\in J}$, there exists $A\in\bbN$ such that for any $\beta\in\NE(Y)$, $f\in\cM^\sm(U^\an,\bP,\beta)$, we have $\deg \Trop(f)\le A$.
	Hence there are only finitely many combinatorial types of $\Trop(f)$ and $\Sp(f)$ for such $f$.
\end{lemma}
\begin{proof}
	It follows from \cref{lem:finitely_many_beta} and \cref{prop:tropicalization_of_stable_map}.
\end{proof}

\begin{lemma} \label{lem:bound_on_class}
	Given $A\in\bbN$, the following hold:
	\begin{enumerate}
		\item \label{lem:bound_on_class:deg} There are at most finitely many $\beta\in\NE(Y)$ such that there exists $\bP$ and $[C,(p_j)_{j\in J},f]\in\cM^\sm(U^\an,\bP,\beta)$ with $\deg\Trop(f)\le A$.
		\item \label{lem:bound_on_class:degtwig} Given $\bP$, there are at most finitely many $\beta \in \NE(Y)$ such that there exists $[C,(p_j)_{j\in J},f]\in\cM^\sm(U^\an,\bP,\beta)$ with $\degtwig\Trop(f)\le A$.
	\end{enumerate}
\end{lemma}
\begin{proof}
	Since $\deg\Trop(f)$ is equal to $\degtwig\Trop(f)$ plus the sum of norm of every component of $\bP$, we see that (\ref{lem:bound_on_class:deg}) implies (\ref{lem:bound_on_class:degtwig}).
	For (\ref{lem:bound_on_class:deg}), let $[C,(p_j)_{j\in J},f]\in\cM^\sm(U^\an,\bP,\beta)$ with $\deg\Trop(f)\le A$.
	By \cref{prop:tropicalization_of_stable_map}, we have $\beta\cdot\pi^*D_\rt=\deg\Trop(f)\le A$.
	Since the indeterminate locus of $\pi$ has codimension at least 2, we view $\pi^*D_\rt$ as pullback of Cartier divisor.
	By the definition of $\cM^\sm(U^\an,\bP,\beta)$, $f(C)$ has no components contained in $\pi^*D_\rt$.
	Thus the intersection number between $\beta$ and every irreducible component of $\pi^*D_\rt$ is non-negative and bounded by $A$.
	Note the support of $\pi^*D_\rt$ is $Y\setminus T_M$.
	We have an exact sequence of Chow groups
	\[\CH^1(Y\setminus T_M)\longrightarrow\CH^1(Y)\longrightarrow\CH^1(T_M)\longrightarrow 0.\]
	Since $\CH^1(T_M)\simeq 0$, we have a surjection $\CH^1(Y\setminus T_M)\twoheadrightarrow\CH^1(Y)$.
	Therefore, the irreducible components of $\pi^*D_\rt$ generate $N^1(Y)$.
	We conclude that there are at most finitely many such $\beta$.
\end{proof}

\section{Rigidity and transversality of spines} \label{sec:rigidity_and_transversality}

In this section, we prove two properties of spines: the rigidity in \cref{prop:rigidity_spine} and the transversality in \cref{prop:transversality}.
Although the underlying geometric ideas are simple, the formal proofs are a bit lengthy.
So first-time readers are advised to skip the proofs.

\begin{proposition} \label{prop:rigidity_spine}
	Let $\SP^\tr(M_\bbR,\bP^F)$ be as in \cref{def:transverse}.
	Let $u\coloneqq v_i$ for some $i\in F\cup I$.
	Let
	\[\Psi_u\coloneqq(\dom,\ev_u)\colon\SP^\tr(M_\bbR,\bP^F)\longrightarrow\NT^F_J\times M_\bbR.\]
	Let $S\in\SP^\tr(M_\bbR,\bP^F)$.
	Then for a sufficiently small neighborhood $V_S$ of $S$ in $\SP^\tr(M_\bbR,\bP^F)$, the restriction of $\Psi_u$ to $V_S$ is a homeomorphism onto its image and is open.
\end{proposition}

\begin{proof}
	We introduce a notational convention:
	when there is no ambiguity, for an object $X$, we denote by $V_X$ a sufficiently small neighborhood of $X$ in the natural moduli space (depending on the context) containing $X$.
	
	Write $S=[\Gamma,(v_j)_{j\in J},h]$.
	Let $\Gamma_\marked$ be $\Gamma$ together with the marked points.
	Let $e$ be a topological edge of $\Gamma$.
	By the stability condition, $e$ contains at most one node.
	If $e$ contains a node, let $x$ be the node of $e$;
	otherwise, let $x$ be a point in the interior of $e$ which is not a bending point.
	We cut $\Gamma$ at $x$ and obtain $\Gamma=\Gamma^1\amalg_x\Gamma^2$.
	Assume $u$ belongs to $\Gamma^2$.
	Let $\Gamma^i_\marked$ ($i=1,2$) be $\Gamma^i$ together with the marked points inherited from $\Gamma_\marked$ plus the marked point $x$.
	Let $h^i\coloneqq h|_{\Gamma^i}$ and $S^i\coloneqq[\Gamma^i_\marked,h^i]$.
	Let $p\colon G \to V_{\Gamma_\marked}$ be the restriction of the universal nodal metric tree to $V_{\Gamma_\marked}$, a sufficiently small neighborhood of $\Gamma_\marked$ in $\NT^F_J$.
	
	Now we define a continuous section $r$ of $p$ with value $x$ at $\Gamma_\marked$ as follows:
	By \cref{const:topology_NT}, for any $[\Gamma',(v'_j)_{j\in J}]$ in $V_{\Gamma_\marked}$, we have a continuous map $c\colon\Gamma'\to\Gamma$ contracting a subset of edges of $\Gamma'$ and sending every $v'_i$ to $v_i$.
	Let $e' \subset \Gamma'$ be the unique topological edge with $c(e') = e$.
	We define the section by picking a point $x'\in e'$ as follows:
	\begin{enumerate}
		\item If $x$ is a node, then $e'$ has either finite length or contains exactly 1 node.
		In the first case, let $x'$ be the midpoint of $e'$; in the second case, let $x'$ be the node of $e'$.
		\item If $x$ is not a node, we distinguish two cases:
		\begin{enumerate}
			\item If $e$ has infinite length, since $e$ does not contain a node, $e$ must contain an 
			infinite 1-valent vertex, which is a marked point.
			So this marked point remains in $e'$ as an infinite 1-valent vertex.
			Let $w$ and $w'$ be respectively the finite endpoint of $e$ and $e'$.
			Let $x'\in e'$ such that the two segments $[w,e]$ and $[w',e']$ have the same lengths.
			\item If $e$ has finite length, than $e'$ also has finite length.
			We let $x'$ be the point of $e'$ that divides $e'$ by the same ratio as the point $x$ divides $e$.
		\end{enumerate}
	\end{enumerate}
	
	Up to shrinking the neighborhoods, we have a natural map
	\[V_{\Gamma^1_\marked}\times V_{\Gamma^2_\marked}\longrightarrow V_{\Gamma_\marked}\]
	given by gluing at the marked point $x$.
	It admits a section $s$ induced by the section $r$ of $p$ constructed above.
	
	Let $V_{S^1}\times_{M_\bbR} V_{S^2}$ be the fiber product of the two evaluation maps at $x$.
	Similarly, up to shrinking the neighborhoods, we have a natural gluing map
	\[V_{S^1}\times_{M_\bbR} V_{S^2} \longrightarrow V_S\]
	which admits a section $t$ induced by the section $r$ of $p$.
	
	Up to further shrinking, we obtain a pullback diagram
	\[\begin{tikzcd}[column sep=7em]
		V_{S^1}\times_{M_\bbR} V_{S^2} \ar{r}{\widetilde\Psi_u\coloneqq(\dom^1,\dom^2,\ev_u)} & V_{\Gamma^1_\marked}\times V_{\Gamma^2_\marked}\times V_{h(u)} \\
		V_S \ar{u}{t} \ar{r}{\Psi_u=(\dom,\ev_u)} & V_{\Gamma_\marked}\times V_{h(u)}. \ar{u}{s}
	\end{tikzcd}\]
	Therefore, in order to show that 
	$\Psi_u|_{V_S}$ is a homeomorphism, it suffices to show that the upper horizontal map $\widetilde\Psi_u$ is a 
	homeomorphism.
	Suppose that \cref{prop:rigidity_spine} holds for the spines $S^1$ and $S^2$, then to up shrinking the 
	neighborhoods, we have homeomorphisms
	\begin{gather*}
		V_{S^1}\xrightarrow{(\dom^1,\ev_x)}V_{\Gamma^1_\marked}\times V_{h(x)},\\
		V_{S^2}\xrightarrow{(\dom^2,\ev_x)}V_{\Gamma^2_\marked}\times V_{h(x)},\\
		V_{S^2}\xrightarrow{(\dom^2,\ev_u)}V_{\Gamma^2_\marked}\times V_{h(u)}.	
	\end{gather*}
	This implies that the map $\widetilde\Psi_u$ is a homeomorphism.
	
	Now by cutting the spine $S$ sufficiently many times, we are reduced to prove the proposition in the following 2 special cases:
	\begin{enumerate}
		\item The spine $S$ exactly 1 bending vertex, and the type $\bP^F$ has $\abs{F}=\abs{B}=2$ and $I=\emptyset$.
		\item The spine $S$ has no bending vertex.
	\end{enumerate}
	
	In the first case, we have $V_{\Gamma_\marked} \subset \bbR$ naturally given by the length of $\Gamma$.
	By \cref{def:transverse}(\ref{def:transverse:image}), a small deformation of $S$ has the same combinatorial type as $S$, with the bending vertex mapping to the same $(d-1)$-dimensional cell of $\Wall_A$.
	So the proposition holds in this case.
	In the second case, as there are no bending vertices, any small deformation of $S$ is uniquely determined continuously by a small deformation of $\Gamma_\marked$ and a small deformation of $h(u)$ in $M_\bbR$.
	Hence the proposition holds in this case too.
	The proof is now complete.
\end{proof}

For the proof of \cref{lem:skeletal_curve_connected_component} we will use the following:

\begin{lemma} \label{lem:spine_injective}
	Let $S=[\Gamma,(v_j)_{j\in J},h]\in\SP^\tr(M_\bbR,\bP^F)$ and $q\in\Gamma$.
	There are open neighborhoods $V_S$ of $S$ in $\SP^\tr(M_\bbR,\bP^F)_\Gamma$ (where the subscript $_\Gamma$ indicates the domain is fixed) and $R$ of $q$ in $\Gamma$ such that
	\begin{align*}
		\Psi\coloneqq(\id,\ev_q)\colon R\times V_S&\longto R\times \oM_\bbR\\
		(r,h')&\longmapsto (r,h'(r))
	\end{align*}
	is injective.
\end{lemma}
\begin{proof}
	Suppose first that $q$ is not a node or a marked point, and $h(q)\not\in\Wall_A$.
	Then the result follows from \cref{prop:rigidity_spine} by gluing a new leg $[0,+\infty]$ at $r$ and extend $h'$ constantly on the new leg.
	
	For the remaining cases, since $S$ is transverse, $q$ is at most two-valent.
	Since a small deformation of $h$ does not change the derivatives of the edges connected to $q$, we can find open neighborhoods $V_S$ of $S$ in $\SP^\tr(M_\bbR,\bP^F)_\Gamma$ and $R$ of $q$ in $\Gamma$, and a point $p\in R$ which is not a node or a marked point, such that $h'(p)\not\in\Wall_A$, and that $h_1(x)=h_2(x)$ implies $h_1|_R=h_2|_R$, for all $h_1,h_2,h'\in V_S$ and $x\in R$.
	
	Now suppose $h_1(x)=h_2(x)$ for some $h_1,h_2\in V_S$ and $x\in R$.
	Then $h_1|_R=h_2|_R$, in particular $h_1(p)=h_2(p)$.
	So it follows from the first paragraph that $h_1=h_2$.
	We conclude that $\Psi$ is injective.
\end{proof}

\begin{lemma} \label{lem:enlarge_Wall}
	Let $W\subset M$ be a finite subset.
	There exists a finite set $H$ of hyperplanes in $M_\bbR$ passing through 0 such that the following hold:
	\begin{enumerate}
		\item Every $(d-1)$-dimensional cell of $\Wall_A$ is contained in a hyperplane in $H$.
		\item For every cell $\sigma$ of $\Wall_A$ of dimension less than $(d-1)$, every point $x\in\sigma$, every vector $w\in W$, the ray starting from $x$ in the direction of $w$ is contained in a hyperplane in $H$.
	\end{enumerate}
\end{lemma}
\begin{proof}
	It follows from the finiteness of the number of cells in $\Wall_A$ and the finiteness of $W$.
\end{proof}

\begin{proposition} \label{prop:transversality}
	Let $\SP(M_\bbR,\bP^F)$ be as in \cref{def:spine}.
	Let $u\coloneqq v_i$ for some $i\in F\cup I$.
	Let $N$ be a natural number, and $W\subset M$ a finite subset.
	Let $\SP(M_\bbR,\bP^F,N,W)\subset\SP(M_\bbR,\bP^F)$ be the subset consisting of spines such that the number of bending vertices is bounded by $N$, and all the weight vectors belong to $W$.
	Let
	\[\Psi_u\coloneqq(\dom,\ev_u)\colon\SP(M_\bbR,\bP^F,N,W)\longrightarrow\NT^F_J\times M_\bbR.\]
	Then there exists a lower dimensional finite polyhedral subset $Z\subset\NT^F_J\times M_\bbR$ such that all the spines in $\Psi_u\inv(Z^c)$ are transverse.
\end{proposition}
\begin{proof}
	We temporarily enlarge $\Wall_A$ by adding all the hyperplanes in \cref{lem:enlarge_Wall}.
	Then Condition (\ref{def:transverse:image}) in \cref{def:transverse} of transversality implies automatically Condition (\ref{def:transverse:extension}), so we can ignore Condition (\ref{def:transverse:extension}) for the proof.
	Now it follows from Conditions (\ref{def:transverse:image}) and (\ref{def:transverse:finite_end}) that the proposition holds when $\abs{J}=2$.
		For $\abs{J}>2$, we can first restrict to the open subset $\Delta_0\subset\NT^F_J$ where no vertex has valency greater than 3, and so can further restrict to any open subset $\Delta\subset\Delta_0$ where the combinatorial type of the domain is fixed.
	Let $E_1,\dots,E_m$ be the topological edges of the domain.
	It follows from \cref{def:transverse} Condition (\ref{def:transverse:bending}) that a spine in $\SP(M_\bbR,\bP^F,N,W)_\Delta$ is transverse if and only if the restriction to each $E_k$ is transverse.
	
	For each $k=1,\dots, m$, let $v^k_1,v^k_2$ be the 2 endpoints of $E_k$.
	Let $P^k_1, P^k_2\in M$ be respectively the weight vector (pointing outwards) of the edge incident to $v^k_1,v^k_2$.
	Let $\bP^k\coloneqq(P^k_1, P^k_2)$.
	Let $F^k\coloneqq\set{ j | v^k_j\text{ is finite}}$.
	In the case where $\abs{F^k}=1$, we consider the map
	\[\Psi_k \colon\SP\big(M_\bbR,(\bP^k)^{F^k},N,W\big)\longrightarrow M_\bbR\]
	taking evaluation at the finite endpoint.
	Then we have a lower dimensional polyhedral subset $Z_k\subset M_\bbR$ such that the preimage $\Psi_k\inv(Z_k^c)$ consists of transverse spines.
	In the case where $\abs{F^k}=2$, the lengths of all such $E_k$ give the modulus of domain of the total spine in $\Delta\subset\NT_J^F$.
	We consider the maps
	\begin{align*}
		\Psi_k &\colon\SP\big(M_\bbR,(\bP^k)^{F^k},N,W\big)\longrightarrow\NT_2^{F^k}\times M_\bbR,\\
		\Psi'_k &\colon\SP\big(M_\bbR,(\bP^k)^{F^k},N,W\big)\longrightarrow\NT_2^{F^k}\times M_\bbR	
	\end{align*}
	taking domain, and evaluation at $v^k_1$ and $v^k_2$ respectively.
	We have lower dimensional polyhedral subsets $Z_k$ and $Z'_k$ of $\NT_2^{F^k}\times M_\bbR$  such that the preimages $\Psi_k\inv(Z_k^c)$ and $\Psi_k'^{-1}(Z_k'^c)$ consist of transverse spines.
	Then by \cref{prop:rigidity_spine}, for any lower dimensional polyhedral subset $Z''$ of $\NT_2^{F^k}\times M_\bbR$, $\Psi_k\big(\Psi_k'^{-1}(Z'')\big)$ and $\Psi'_k\big(\Psi_k^{-1}(Z'')\big)$ are contained in a lower dimensional polyhedral subset of $\NT_2^{F^k}\times M_\bbR$.
	Hence the lower dimensional polyhedral subsets defined with respect to evaluation at various 3-valent vertices of the domain of the total spine can be transported to lower dimensional polyhedral subsets defined with respect to evaluation at the marked point $u$.
	Therefore, we obtain a lower dimensional finite polyhedral subset $Z\subset\NT^F_J\times M_\bbR$ such that all the spines in the preimage $\Psi_u\inv(Z^c)$ as in the statement of the proposition are transverse.
\end{proof}

\section{The toric case} \label{sec:toric_case}

In this case, we describe the various moduli spaces introduced before in the special case where $(Y,D)$ is toric.

\begin{lemma} \label{lem:toric_curve_classes}
	Let $Y$ be a $d$-dimensional smooth projective toric variety with cocharacter lattice $M$.
	Let $D$ be the toric boundary, and $\{D_i\}_{i\in I_D}$ the set of irreducible components of $D$.
	Let $\bbZ^{I_D}\to M$ send each basis element $e_i$ to the first lattice point on the ray corresponding to $D_i$.
	Let $N_1(Y)$ denote the group of numerical equivalence classes of 1-cycles.
	Let $N_1(Y)\to\bbZ^{I_D}$ be given by intersection numbers with every $D_i$.
	Then we have a short exact sequence
	\[0\to N_1(Y)\to\bbZ^{I_D}\to M\to 0.\]
\end{lemma}
\begin{proof}
	We refer to \cite[\S 3.4]{Fulton_Introduction_to_toric_varieties}.
		\end{proof}

\begin{proposition} \label{prop:toric_case}
	Notations as in Sections \ref{sec:smoothness} and \ref{sec:tropical}, we consider the case where $(Y,D)$ is toric, i.e.\ $\pi\colon(Y,D)\to (Y_\rt,D_\rt)$ is an isomorphism.
	The following hold:
	\begin{enumerate}
		\item \label{prop:toric_case:equation} We have
		\[\cM^\sm(U,\bP,\beta)=\cM^\sd(U,\bP,\beta)=\cM(U,\bP,\beta).\]
		\item \label{prop:toric_case:spine} We have
		\[\SP^\tr(M_\bbR,\bP)=\SP(M_\bbR,\bP).\]
		\item \label{prop:toric_case:P} The following are equivalent for any tuple $\bP=(P_j)_{j\in J}\in M^J$:
		\begin{enumerate}
			\item $\sum P_j=0\in M$.
			\item $\SP(M_\bbR,\bP)\neq\emptyset$.
			\item There exists $\beta\in\NE(Y)$ compatible with $\bP$ as in \cref{rem:compatible_curve_class}.
			(If such $\beta$ exists, it is unique.)
		\end{enumerate}
		\item \label{prop:toric_case:diagram} Assume the equivalent conditions in (\ref{prop:toric_case:P}) hold and let $\beta$ be compatible with $\bP$.
		For any $i\in I$, consider the commutative diagram
		\[\begin{tikzcd}
		\cM(U,\bP,\beta) \rar{\Phi_i} \dar{\Sp} & \ocM_{0,n}\times U \dar{\rho} \\
		\SP(M_\bbR,\bP) \rar{\Phi^t_i} & \oM_{0,n}^\trop\times M_\bbR.
		\end{tikzcd}\]
		The following hold:
		\begin{enumerate}
			\item \label{prop:toric_case:Phi} The top horizontal map $\Phi_i$ is an open immersion; and it is an isomorphism over $\cM_{0,n}\times U$.
			\item \label{prop:toric_case:Psi} The bottom horizontal map $\Phi^t_i$ is an open immersion; and it is an isomorphism over $M^\trop_{0,n}\times M_\bbR$.
		\end{enumerate}
	\end{enumerate}
\end{proposition}
\begin{proof}
	Since $\pi\colon(Y,D)\to (Y_\rt,D_\rt)$ is an isomorphism, we have $E=\emptyset$, $Y^\idt=\emptyset$ and $W=Y$, notation as in \cref{sec:log_CY}.
	Moreover, the logarithmic tangent bundle $T_Y(-\log D)$ is trivial.
	Therefore, all the conditions in \cref{nota:Msm} are empty conditions.
	So the first equality in (\ref{prop:toric_case:equation}) holds.
	
	Since there is no non-constant invertible functions on $\bbA^1_k$, for any map $f\colon\bbP^1_k\to Y$ whose image is not contained in $D$, the preimage $f\inv(D)$ contains at least 2 points.
	Therefore, every irreducible component of the domain of any stable map in $\cM(U,\bP,\beta)$ must contain at least three special points.
	This implies the second equality in (\ref{prop:toric_case:equation}).
	
	Note $W=Y=Y_\rt$, so $E_\rt=\emptyset$ and $E_\rt^\trop=\emptyset$ as in \cref{nota:Et}, hence $\Wall_A=\emptyset$ for any $A\in\bbN$ by \cref{const:walls_by_induction}.
	This implies (\ref{prop:toric_case:spine}).
	(\ref{prop:toric_case:P}) follows from \cref{lem:toric_curve_classes} and the balancing condition.
	
	Now let us prove (\ref{prop:toric_case:Phi}).
	By (\ref{prop:toric_case:equation}) and \cref{lem:Msm_smooth}(\ref{lem:Msm_smooth:etale}), the top horizontal map $\Phi_i$ is étale.
	So it suffices to prove that it is a bijection on $k$-rational points over $\cM_{0,n}\times U$ after passing to an algebraic closure of the base field $k$.
	
	Note that for any $k$-variety $X$, we have \[\Hom(X,T_M)\simeq\Hom_{\mathrm{Group}}(M,\cO(X)^*).\]
	So given any $k$-point $x\in X$, we have a canonical isomorphism \[\Hom(X,T_M)\simeq\Hom_{\mathrm{Group}}(M,\cO(X)^*/k^*)\times T_M(k),\]
	where the second factor is given by evaluation of the map $X\to T_M$ at $x\in X$.
	Now let $C=\bbP^1_k\setminus\{p_1,\dots,p_n\}$ with $p_i$ pairwise distinct.
	We have a canonical isomorphism
	\[O(C)^*/k^*\xrightarrow{\ \sim\ }\Set{(z_1,\dots,z_n)\in\bbZ^n | \sum z_i=0}\]
	given by valuation at the punctures.
	Thus $\Hom(M,\cO^*(C)/k^*)$ is canonically identified with the set of tuples $\bP=(P_j)_{j\in J}$ satisfying $\sum P_j=0$.
	The bijection, and thus Statement (\ref{prop:toric_case:Phi}), follow.
	
	Since $\Wall_A=\emptyset$, all spines in $\SP(M_\bbR,\bP)$ are everywhere balanced.
	The balancing condition implies that $\Phi^t_i$ is set-theoretically injective, and it is a bijection over $M^\trop_{0,n}\times M_\bbR$.
		So Statement (\ref{prop:toric_case:Psi}) follows from \cref{prop:rigidity_spine}.
\end{proof}

\section{Curve classes} \label{sec:curve_classes}

In this section we study the classes of analytic curves (with boundaries) mapping into our log Calabi-Yau variety.
We prove the positivity of curve classes, see \cref{prop:interior_divisor}.
This will lead to \cref{prop:structure_disk_interior_divisor}, which is the key to extending the mirror family to larger toric partial compactifications than $\Spec R=\TV(\Nef(Y))$, the toric variety associated to the nef cone.
Such extensions will be used in the proof of non-degeneracy in \cref{sec:non-degeneracy}, and is analogous to Viterbo restriction in symplectic cohomology (see \cite{Viterbo_Functors_and_computations}).
The second part of this section gives a tropical formula for computing curves classes, see \cref{prop:curve_class_formula}.

We follow the setting of \cref{sec:log_CY}, and assume $k$ has nontrivial discrete valuation.
Let $\kc$ denote the ring of integers of $k$ and $\tk$ the residue field.
Recall that $Y$ is constant over $k$, i.e.\ it is isomorphic to the pullback of some $Y_0$ over $k_0$, where $k_0$ has trivial valuation.
Let $Y_{\kc}$ be the base change from $Y_0$ to $\Spec\kc$, and $\hY_{\kc}$ the formal completion along the special fiber.

\begin{definition}[from {\cite[Definition 5.10]{Yu_Enumeration_of_holomorphic_cylinders_I}}] \label{def:curve_class}
	Given a compact quasi-smooth strictly \kanal curve\footnote{In the terminology of classical rigid analytic geometry, $C$ is a quasi-compact smooth rigid $k$-analytic curve.} $C$ and a morphism $f\colon C\to Y^\an$, up to passing to a finite base field extension, we can choose a strictly semistable\footnote{When $k$ has discrete valuation, a semistable formal scheme over $\kc$ is always assumed to be regular.} formal model $\fC$ of $C$ over $\kc$ such that $f\colon C\to Y^\an$ extends to a morphism $\ff\colon\fC\to\hY_{\kc}$.
	Let $\fC_s^\pr$ denote the union of proper irreducible components of the special fiber $\fC_s$ of $\fC$.
	We define the \emph{class} of the map $f$ to be $[f]\coloneqq\ff_{s*}[\fC_s^\pr]\in\NE(Y)$.
	Since two different choices of the model $\fC$ can always be dominated by another model, we see that the class is well-defined.
	%	If the domain curve $C$ is nodal, we make the definition after taking normalization.
\end{definition}

\begin{lemma} \label{lem:curve_class_from_model}
	Let $C$ be a compact quasi-smooth strictly \kanal curve and $f\colon C\to Y^\an$ a map with image not contained in $D^\an$.
	Let $\fC$ be a strictly semistable model of $C$ such that the map $f\colon C\to Y^\an$ extends to $\ff\colon\fC\to\hY_{\kc}$ and $(\fC,\ff\inv(\hD_{\kc}))$ is a formal strictly semistable pair.
	Let $\Gamma\coloneqq\oSigma(\fC,\ff\inv(\hD_{\kc}))$ be the associated extended skeleton and $h\colon\Gamma\to\oSigma_{(Y,D)}$ the piecewise linear map induced by $\ff$ (see \cite{Gubler_Skeletons_and_tropicalizations}).
	Assume the following:
	\begin{enumerate}
		\item $h(\Gamma)\cap\Sigma_{(Y,D)}^{(d-2)}=\emptyset$, where $\Sigma_{(Y,D)}^{(d-2)}\subset\Sigma_{(Y,D)}$ denotes the union of $(d-2)$-dimensional cells, and $d$ is the dimension of $Y$.
		\item $V\coloneqq h\inv\big(\Sigma_{(Y,D)}^{(d-1)}\big)\subset\Gamma$ is a finite set containing only 2-valent vertices of $\Gamma$.
	\end{enumerate}
	Let $v$ be a vertex in $V$, and $e$ an edge incident to $v$.
	Let $\sigma$ be the codimension one cone of $\Sigma_{(Y,D)}$ containing $h(v)$.
	Let $\omega_\sigma$ be a primitive integer volume form on $\sigma$, $w_{(v,e)}$ the derivative of $h$ at $v$ along $e$, and $d_v$ the lattice length of $\omega_\sigma\wedge w_{(v,e)}$.
	Let $\fC_s^v$ be the irreducible component of $\fC_s$ corresponding to $v$.
	Let $Z_v$ be the closed stratum of $D$ corresponding to $\sigma$.
	Then the order of $[Z_v]$ in $\ff_{s*}[\fC_s^v]$ is equal to $d_v$.
	In particular, $d_v$ does not depend on the choice of $e$.
	Summing over all $v\in V$, we obtain 
	\begin{equation} \label{eq:curve_class}
	[f]=\ff_{s*}[\fC_s^\pr]=\sum_{v\in V}\ff_{s*}[\fC_s^v]=\sum_{v\in V}d_v[Z_v]\in\NE(Y).
	\end{equation}
\end{lemma}
\begin{proof}
	Let $v'$ be the other endpoint of $e$.
	Up to an admissible blowup of $\fC$, we can assume that $h(e)$ lies in the interior of a $d$-dimensional cell $\sigma'$.
	Then $\sigma$ is a face of $\sigma'$.
	Let $p$ be the node of $\fC_s$ corresponding to $e$, and $y$ the 0-stratum of $D^\an$ corresponding to $\sigma'$.
	Recall that $\fC$ is strictly semistable and $D\subset Y$ is snc.
	So we can choose standard normal crossing coordinates $x_1,x_2$ of $\fC$ around $p$, and $y_1,\dots,y_d$ of $Y^\an$ around $y$.
	Assume that $x_1$ restricts to a coordinate on $\fC_s^v$, and that $y_1$ restricts to a coordinate on $Z_v$.
	Expressing $f$ in these coordinates, we can write $y_1$ as a Laurent series in $x_1$.
	By \cite[Lemma 9.7.1.1]{Bosch_Non-Archimedean_analysis},
		the series has a dominant power, which is equal to both $d_v$ and the order of $[Z_v]$ in $\ff_{s*}[\fC_s^v]$ (by 9.7.1.1(ii) and (iii) respectively in loc.\ cit.).
	Summing over all $v\in V$, we obtain \eqref{eq:curve_class}.
\end{proof}

\begin{lemma} \label{lem:interior_divisor}
	Let $Z$ be a compact strictly $k$-analytic space, $C$  a compact quasi-smooth strictly \kanal curve, and $f\colon C \to Z$ a morphism with a formal model $\mathfrak f\colon\fC\to\fZ$ where $\fC$ is strictly semistable.
	Let $\mathfrak F$ be a Cartier divisor on $\fZ$ with generic fiber $F$ and $\cO(\fF)$ the associated line bundle.
	Let $\fC_s^\pr$ (resp.\ $\fC_s^\npr$) denote the union of proper (resp.\ non-proper) components of the central fiber $\fC_s$.
	Assume that $\ff_s(\fC_s^{\npr})$ is disjoint from the support of $\mathfrak F$, and that $f(C)$ is not contained in the support of $F$, so that $f\inv(F) \subset C$ is a Cartier divisor.
	We have 
	\[ \deg \big(\ff_{s*}[\fC_s^{\pr}] \cdot \cO(\mathfrak F)|_{\fZ_s}\big) =\deg (f^{-1}(F)).\]
\end{lemma}
\begin{proof}
	Since $f(C)$ is not contained in $\supp F$, we have $G\coloneqq\mathfrak f\inv(\fF)$ as Cartier divisor on $\fC$.
		Decompose $G=G^\ver+G^\hor$ where $G^\ver$ is supported on $\fC_s$ and $G^\hor=\overline{G_\eta}$.
	Both $G^\ver$ and $G^\hor$ are Cartier as $\fC$ is regular.
	
	By assumption $\fC_s^\npr \cap \supp G = \emptyset$.
	So $\deg(\fC_s\cdot G^\ver)$ is well-defined.
	Since $\fC_s\subset\fC$ is a principle Cartier divisor, $\deg(\fC_s\cdot G^\ver)=0$.
	Thus
	\begin{multline*}
	\deg \big(\ff_{s*}[\fC_s^{\pr}] \cdot \cO(\mathfrak F)|_{\fZ_s}\big)=\deg (\fC_s^\pr\cdot G)=\deg (\fC_s\cdot G)=\deg\big(\fC_s\cdot (G^\ver+G^\hor)\big)\\
	=\deg(\fC_s\cdot G^\hor)=\deg(G^\hor|_C)=\deg(f^{-1}(F)),
	\end{multline*}
	where the first equality is by projection formula for varieties over the residue field $\tk$, and the second equality is because $\fC_s^\npr \cap \supp G = \emptyset$.
\end{proof}

\begin{proposition} \label{prop:interior_divisor}
	Let $C$ be a compact quasi-smooth strictly \kanal curve, $\partial C$ its boundary (see \cite[\S 2.5]{Berkovich_Spectral_theory}), and $f\colon C\to Y^\an$ a morphism.
	Let $F\subset Y$ be a Cartier divisor containing no essential boundary strata, i.e.\ those of $D^\ess$.
	Assume $(\tau\circ f)(\partial C)$ does not meet the codimension one skeleton of $\Sigma_\rt$, and $f(C)$ is not contained in the support of $F^{\an}$, so that $f^{-1}(F^{\an})$ is a Cartier divisor on $C$.
	Let $[f]$ be as in \cref{def:curve_class}.
	We have 
	\[[f]\cdot F  =\deg (f^{-1}(F^{\an})).\]
	If moreover $F$ is effective, then
	\[[f]\cdot F  \geq 0 \]
	with equality if and only if $f(C)$ is disjoint from $F^{\an}$. 
\end{proposition}
\begin{proof}
	Consider the constant model $(\hY_{\kc},\hD_{\kc})$ of $(Y,D)$ over $\kc$.
	The assumption on the tropicalization implies that ${\mathfrak f}(\fC_s^{\npr})$ is contained in the union of 0-strata of the special fiber $D_{\tk}$, and thus disjoint from the support of $\mathfrak F$.
	Now we apply \cref{lem:interior_divisor}.
\end{proof}

\begin{lemma} \label{lem:N1_blowup}
	Let $p\colon Y\to Y'$ be a birational morphism, and $E$ the set of exceptional divisors.
	We have a split short exact sequence
	\[\begin{tikzcd}
	0\rar & \bbZ^E\rar & N^1(Y)\arrow[r,shift left, "p_*"]  & N^1(Y') \arrow[l, shift left, "p^*"] \rar & 0.
	\end{tikzcd}\]
	Hence, by duality, the map
	\[N_1(Y)\longrightarrow N_1(Y')\oplus\bbZ^E,\]
	given by pushforward $p_*\colon N_1(Y)\to N_1(Y')$ and intersection with every exceptional divisor, is an isomorphism.
\end{lemma}
\begin{proof}
	We refer to \cite[Example 1.8.1]{Fulton_Intersection_theory}.
\end{proof}

\begin{definition-lemma} \label{def:projection_formula_rational}
	Let $f \colon X \dasharrow Y$ be a rational map between normal projective varieties, with $X$ (resp.\ $Y$) smooth.
	Let $p\colon Z\to X$ be a regular proper birational map such that $q\coloneqq f\circ p$ is regular.
	Then $q_* \circ p^*\colon N_*(X) \to N_*(Y)$ (resp.\ $p_* \circ q^*\colon N_*(Y) \to N_*(X)$)	is independent of the resolution $Z \to X$.
	We define it to be $f_*$ (resp.\ $f^*$).
	
	Suppose $X,Y$ are both smooth, and that we have proper $r\colon X \to W$, $s\colon Y \to W$ such that $r=s \circ f$.
	\[\begin{tikzcd}[column sep=small]
	& Z \ar[ld, "p", swap] \ar[rd, "q"] \\
	X \ar[rr, dashed, "f"] \ar[rd, "r", swap] & & Y \ar[ld, "s"] \\
	& W
	\end{tikzcd}\]
	Then a version of the projection formula holds: for any $\alpha\in N_*(X)$, $\beta\in N_*(Y)$, we have
	\[
	r_*(\alpha \cdot f^*\beta) = s_*(f_*\alpha \cdot \beta).
	\]
\end{definition-lemma}
\begin{proof}
	The independence of the resolution and the projection formula both follow from the usual projection formula (for proper regular maps).
\end{proof} 

\begin{lemma} \label{lem:curve_class_blowup}
	In the context of \cref{prop:interior_divisor}, let $p\colon Y\dasharrow Y'$ be a birational map which is an open embedding outside $\supp F$.
	Assume $f(C) \cap \supp F^\an = \emptyset$.
	Consider the composition $f\colon C \to (Y\setminus \supp F)^\an \to Y'^\an$.
	We have $p_*[f] = [f']$ and $p^*[f'] = [f]$.
\end{lemma}
\begin{proof}
	By projection formula and passing to a resolution it is enough to prove the lemma when $p$ is regular.
	We have $p_*[f] = [f']$ by definition of the class. 
	By \cref{prop:interior_divisor}, we have $[f] \cdot E =0$ for all $p$-exceptional divisor $E$.
	Therefore, it follows from \cref{lem:N1_blowup} that $[f] = p^* p_*[f] = p^*[f']$.
\end{proof}

For the toric variety $(Y_\rt, D_\rt)$, formula \eqref{eq:curve_class} can be encoded via an $N_1(Y_\rt)_\bbR$-valued piecewise-linear function on $\Sigma_\rt$.

\begin{definition} \label{def:curve_class_via_varphi}
	By \cite[Lemma 1.14]{Gross_Mirror_symmetry_for_log_Calabi-Yau_surfaces_I_v1}, there is a $\Sigma_\rt$-piecewise-linear function $\varphi\colon M_\bbR \to N_1(Y_\rt,\bbR)$ whose kink (aka bending parameter) along each codimension one cone  $\sigma \subset \Sigma_\rt$ is the class of the corresponding closed 1-stratum.
	Given any \Zaffine map $l\colon[-\infty,+\infty]\supset [a,b]\to\oM_\bbR$ with $l((a,b))\subset M_\bbR$, let
	\[\odelta_l \coloneqq d(\varphi \circ l)_b - d(\varphi \circ l)_a\in\NE(Y_\rt).\]
	If $a=-\infty$, then $d(\varphi \circ l)_a$ means $\lim_{a'\to-\infty}d(\varphi \circ l)_{a'}$; similarly when $b=+\infty$.
	Let $\delta_l\coloneqq\pi^*\odelta_l\in\NE(Y)$, called the curve class associated to $l$.
	Given any $n$-pointed tree $[\Gamma,(v_j),h]$ in $M_\bbR$ as in \cref{def:pointedtree}, we define $\delta_h$ to be the sum of $\delta_l$ over every domain of affineness $l$ of $h$, called the curve class associated to $h$.
\end{definition}

\begin{remark} \label{rem:curve_class_parallel_perturbation}
	Assume that $l(a)$ lies in an open cell $\sigma_a^\circ\subset\Sigma_\rt$ such that the linear span of $\sigma_a$ contains the derivative $dl_a$, and that the same holds for $l(b)$ with open cell $\sigma_b^\circ$.
	Let $\Star(\sigma_a^\circ)$ and $\Star(\sigma_b^\circ)$ be respectively the open stars of $\sigma_a^\circ$ and $\sigma_b^\circ$.
	Then $\odelta_l$ is invariant under any parallel perturbation as long as $l(a)$ stays in $\Star(\sigma_a^\circ)$ and $l(b)$ stays in $\Star(\sigma_b^\circ)$.
\end{remark}

\begin{proposition} \label{prop:curve_class_formula}
	Let $C$ be a compact quasi-smooth strictly \kanal curve and $f\colon C\to Y^\an$ a map with image contained in $W^\an$, but not in $D^\an$ ($W$ as in \cref{lem:toric_model}).
	Assume $\pi\circ f\colon C\to Y_\rt^\an$ tropicalizes to $h\colon\Gamma\to\oM_\bbR$ satisfying \cref{lem:curve_class_from_model} Conditions (1-2).
	Then we have $[f]=\delta_h\in\NE(Y)$.
\end{proposition}
\begin{proof}
	This follows from Lemmas \ref{lem:curve_class_from_model} and \ref{lem:curve_class_blowup}.
	\end{proof}

\section{Skeletal curves} \label{sec:skeletal_curves}

Consider an analytic curve $C$ in our log Calabi-Yau variety $U^\an$.
If $\dim U\ge 2$, by dimensional reasons, $C$ can never meet the essential skeleton $\Sk(U)\subset U^\an$.
But we can let it happen if we allow $C$ to be defined over a big non-archimedean field extension $k\subset k'$.
The surprise is that as soon as some $k$-point of $C$ touches $\Sk(U)$, the whole skeleton of $C$ will lie inside $\Sk(U)$.
Such curves are called \emph{skeletal curves}.
They play an important role in our theory.
One advantage of skeletal curves is that they have canonical spines independent of any choice of retraction $U^\an\to\Sk(U)$, which will lead to the naive counts associated to spines in \cref{sec:naive_counts}.
Another advantage is that they are analytically generic, and will allow us to obtain deformation invariance across tropically non-generic situations, as in \cref{prop:toric_tail_cc_skeletal}, crucial for perturbing the structure disks (see \cref{rem:structure_constants_vary_point}), and for the wall-crossing homomorphism (see \cref{thm:wall-crossing_homomorphism}).

The main result of this section is \cref{thm:f_in_skeleton}, an equivalence of different characterizations of skeletal curves.
Lemmas \ref{lem:local_skeleton}-\ref{lem:explicit_weight} studies essential skeletons when the analytic space comes from a formal semistable pair (we refer to \cite{Gubler_Skeletons_and_tropicalizations} for the basics of such pairs).
\cref{prop:skeleton_of_product} studies product of skeletons, a main ingredient in the proof of \cref{thm:f_in_skeleton}.
Lemmas \ref{lem:source_of_skeletal_curve}-\ref{lem:skeletal_curve_connected_component} apply skeletal curve theory to the moduli spaces of the previous sections.
\cref{prop:skeleton_of_M0n} identifies the essential skeleton of $\cM_{0,n}$ with the moduli space of tropical curves.
Readers unfamiliar with non-archimedean geometry can skip the technical lemmas on first time reading.

Let $k$ be any non-archimedean field (not necessarily discretely valued), $\kc$ its ring of integers, and $\kcc\subset\kc$ the maximal ideal.
Let $X$ be a quasi-smooth \kanal space of pure dimension $d$.
For any positive integer $l$, let $K_X^{\otimes l}\coloneqq(\wedge^d\Omega_X)^{\otimes l}$, the pluricanonical bundle.
Let $\norm{\cdot}$ denote the geometric Kähler seminorm on $K_X^{\otimes l}$ constructed by Temkin (\cite[6.3.15]{Temkin_Metrization_of_differential_pluriforms}).
For any $\omega\in\Gamma(K_X^{\otimes l})$, we obtain an upper semicontinuous function\footnote{When $k$ has discrete valuation, up to a constant factor, $\norm{\omega}$ coincides with the weight function of Mustata and Nicaise (\cite{Mustata_Weight_functions}, see also Kontsevich-Soibelman \cite{Kontsevich_Homological_mirror_symmetry}) by \cite[Theorem 8.3.3]{Temkin_Metrization_of_differential_pluriforms}.} $\norm{\omega}\colon X\to\bbR_{\ge 0}$.

We denote by $\Sk(\omega)\subset X$ the maximum locus of $\norm{\omega}$ (possibly empty), and call it the \emph{skeleton} of $X$ associated to the pluricanonical form $\omega$.

\begin{lemma} \label{lem:local_skeleton}
	Let $\fS$ be the affine formal scheme over $\kc$ given by
	\[\Spf\big(\kc\braket{s^0_0,\dots,s^0_{m_0}}/(s^0_0\cdots s^0_{m'_0}-a_0)\times\dots\times\Spf\kc\braket{s^n_0,\dots,s^n_{m_n}}/(s^n_0\cdots s^n_{m'_n}-a_n)\big)\]
	for some $n\ge 0$, and $0\le m'_i\le m_i$, $a_i\in\kc\setminus 0$, for $i=0,\dots,n$.
	Let $\fU$ be an affine formal scheme and $\alpha\colon\fU\to\fS$ an étale map.
	Let
	\[h\coloneqq\prod_{i=0}^n\prod_{j=m'_i+1}^{m_i} s^i_j.\]
	Let $\fF\subset\fU$ be the Cartier divisor given by $h$, and $U\coloneqq\fU_\eta\setminus \fF_\eta$.
	Let
	\[e\coloneqq\bigg(\bigwedge_{i=0}^n\bigwedge_{j=1}^{m_i}\frac{d s^i_j}{s^i_j}\bigg)^{\otimes l}.\]
	Let $\psi\in\Gamma(\cO_{\fU_\eta})$, and $\omega\coloneqq\psi e\in\Gamma(K_U^{\otimes l})$.
	Let $M$ be the maximum of the norm $\norm{\omega}$ over $U$.
	Then for any point $u\in U$, the following are equivalent:
	\begin{enumerate}
		\item $u\in\Sk(\omega)$;
		\item $u\in \Sigma{(\fU,\fF)}\subset U$ and $\abs{\psi}$ attains a maximum (over $\fU_\eta$) at $u$ that is equal to $M$, where $\Sigma(\fU,\fF)$ denotes the skeleton associated to the product of formal strictly semistable pairs (see \cite{Gubler_Skeletons_and_tropicalizations}).
	\end{enumerate}
	Moreover, the two functions $\norm{\omega}$ and $\abs{\psi}$ are equal when restricted to the skeleton $\Sigma(\fU,\fF)\subset U$, and $\Sk(\omega)$ is a union of faces of $\Sigma(\fU,\fF)$.
\end{lemma}
\begin{proof}
	First we claim that $\norm{e(u)}\le 1$ over $U$ and the equality holds if and only if $u\in \Sigma{(\fU,\fF)}$.
	By \cite[Lemma 8.2.2]{Temkin_Metrization_of_differential_pluriforms}, this holds when $\alpha$ is identity.
	In general, for every $u\in U$, the complete residue field extension $\cH(u)/\cH(\alpha_\eta(u))$ is unramified by \cite[Lemma 1.6]{Berkovich_Smooth_p-adic_analytic_spaces_are_locally_contractible}, hence it is universally spectral by \cite[Corollary 6.3.8]{Temkin_Metrization_of_differential_pluriforms}.
	Then it follows from \cite[Theorem 6.3.11]{Temkin_Metrization_of_differential_pluriforms} that $\norm{e(u)}=\norm{e(\alpha_\eta(u))}$ for all $u\in U$.
	This shows the claim.
	
	Since $\norm{\omega(u)}=\abs{\psi(u)}\cdot\norm{e(u)}$, we deduce that the two functions $\norm{\omega}$ and $\abs{\psi}$ are equal when restricted to the skeleton $\Sigma(\fU,\fF)\subset U$.
	Furthermore, note that the maximum of $\abs{\psi(u)}$ is attained in the Shilov boundary of $\fU_\eta$ which is contained in $\Sigma{(\fU,\fF)}$, so the claim above implies the equivalence in the lemma.
	
	Now for the last statement of the lemma, it suffices to prove that the maximum locus of $\abs{\psi}$ is a union of faces of $\Sigma(\fU,\fF)$.
	This follows from exactly the same argument in last paragraph of the proof of \cite[Theorem 8.2.4]{Temkin_Metrization_of_differential_pluriforms}.
\end{proof}

\begin{lemma} \label{lem:essential_skeleton_semistable_pair}
	Let $(\fX,\fH)$ be a formal strictly semistable pair and $X\coloneqq\fX_\eta\setminus\fH_\eta$.
	Let $\omega\in\Gamma(K_X^{\otimes l})$ with at worst simple poles along $\fH_\eta$.
	Then the essential skeleton $\Sk(w)$ is a union of faces of $\Sigma(\fX,\fH)$.
	Consequently, $\Sk(w)$ consists of only monomial points (i.e.\ points where the Abhyankar inequality is an equality); in particular $\Sk(\omega)$ is contained in any Zariski dense open subset of $X$.
\end{lemma}
\begin{proof}
	Working Zariski locally on $\fX$, this follows from \cref{lem:local_skeleton}.
	\end{proof}

\begin{remark}
	\cref{lem:essential_skeleton_semistable_pair} is an analog of \cite[Theorem 8.2.4]{Temkin_Metrization_of_differential_pluriforms} and \cite[Theorem 4.5.5]{Mustata_Weight_functions} for formal strictly semistable pairs.
\end{remark}

\begin{notation} \label{nota:closure_of_skeleton}
	In the setting of \cref{lem:essential_skeleton_semistable_pair}, we denote by $\oSk(\omega_X)$ the closure of $\Sk(\omega_X)$ in $\oX\coloneqq\fX_\eta$.
\end{notation}

\begin{lemma} \label{lem:explicit_weight}
	In the setting of \cref{lem:essential_skeleton_semistable_pair}, assume $k$ has discrete valuation and $\fX$ is regular.
	Let $\Iv$ be the set of irreducible components of $\fX_s$, and $\Ih$ the set of irreducible components of $\fH$.
	We have a natural embedding $\Sigma(\fX,\fH)\subset\bbR_{\ge 0}^{\Iv\sqcup\Ih}$.
	For $i\in\Iv\sqcup\Ih$, let $D_i$ denote the corresponding component of $\fX_s\cup\fH$, and $\ord_{D_i}(\omega)$ the order of zero of $\omega$ along $D_i$.
	Let $\Psi$ be the linear function on $\bbR^{\Iv\cup\Ih}$ such that the value of $\Psi$ at the unit vector in the $i$-th direction is $\ord_{D_i}(\omega)+1$.
	Assume the Zariski closure of the zeros of $\omega$ does not contain any strata of $\fX_s\cup\fH$.
	Then the essential skeleton $\Sk(\omega)$ is equal to the minimum locus of the function $\Psi|_{\Sigma(\fX,\fH)}$.
\end{lemma}
\begin{proof}
	Working Zariski locally on $\fX$, this follows from Lemmas \ref{lem:local_skeleton} and \ref{lem:divisor_tropicalization}(\ref{lem:divisor_tropicalization:equality}).
	(In terms of exposition, \cref{lem:divisor_tropicalization} fits better in \cref{sec:convexity}.)
\end{proof}

\begin{remark}
	In \cref{lem:explicit_weight}, if $\fH=\emptyset$, then the condition on the zeros of $\omega$ is not necessary (see \cref{lem:divisor_tropicalization}(\ref{lem:divisor_tropicalization:inequality})).
\end{remark}

\begin{lemma} \label{lem:skeleton_etale_pullback}
	Let $f\colon Y\to X$ be a residually tame quasi-étale map of quasi-smooth \kanal spaces of pure dimension $d$.
	Let $\omega\in\Gamma(K_X^{\otimes l})$.
	We have
	\[\Sk(f^*\omega)=f\inv(\Sk(\omega))\subset Y.\]
\end{lemma}
\begin{proof}
	By \cite[Theorem 6.3.11]{Temkin_Metrization_of_differential_pluriforms}, we have $\norm{f^*\omega}=\norm{\omega}\circ f$.
	Therefore, the maximum locus of $\norm{f^*\omega}$ is equal to the preimage of the maximum locus of $\norm{\omega}$ by $f$.
\end{proof}

\begin{proposition} \label{prop:skeleton_of_product}
	Let $(\fX,\fF)$, $(\fY,\fG)$ be two formal strictly semistable pairs.
	Assume all the strata of $\fX_s\cup\fF$ and $\fY_s\cup\fG$ are geometrically connected.
	Let $X\coloneqq\fX_\eta\setminus\fF_\eta$, $Y\coloneqq\fY_\eta\setminus\fG_\eta$.
	Let $\omega_X\in\Gamma(K_X^{\otimes l})$, $\omega_Y\in\Gamma(K_Y^{\otimes l})$, having at worst simple poles along $\fF_\eta$ and $\fG_\eta$ respectively.
	Let $Z\coloneqq X\times Y$, and $p_X\colon Z\to X$, $p_Y\colon Z\to Y$ the projection maps.
	Let $\omega_Z\coloneqq p_X^*\omega_X\wedge p_Y^*\omega_Y$.
	The following hold:
	\begin{enumerate}
		\item \label{prop:skeleton_of_product:product} We have a commutative diagram
		\[\begin{tikzcd}
		\Sk(\omega_Z) \ar{rr}{\sim} \dar[hook] & & \Sk(\omega_X) \times \Sk(\omega_Y) \dar[hook]\\
		Z \rar{\sim} & X\times Y\rar & \abs{X}\times \abs{Y}
		\end{tikzcd}\]
		where the upper horizontal map is a piecewise-linear homeomorphism, and $\abs{X}$, $\abs{Y}$ denote respectively the underlying topological spaces of the \kanal spaces.
		\item \label{prop:skeleton_of_product:closure}
		Let $\oX\coloneqq\fX_\eta$, $\oY\coloneqq\fY_\eta$ and $\oZ\coloneqq(\fX\times\fY)_\eta$.
		Then the commutative diagram above extends to
		\[\begin{tikzcd}
		\oSk(\omega_Z) \ar{rr}{\sim} \dar[hook] & & \oSk(\omega_X) \times \oSk(\omega_Y) \dar[hook]\\
		\oZ \rar{\sim} & \oX\times \oY\rar & \abs{\oX}\times \abs{\oY}.
		\end{tikzcd}\]
		\item \label{prop:skeleton_of_product:fiber} For $x\in X$, let $Z_x\coloneqq p_X\inv(x)$, and $\omega_{Z_x}$ the restriction of $\omega_Z$ to $Z_x$.
		We have a commutative diagram
		\[\begin{tikzcd}
		\Sk(\omega_{Z_x}) \rar{\sim} \dar[hook] & \Sk(\omega_Y) \dar[hook]\\
		Z_x \rar & Y.
		\end{tikzcd}\]
		\item \label{prop:skeleton_of_product:point}
		Let $z\in Z$ be a point, and put $x\coloneqq p_X(z)$, $y\coloneqq p_Y(z)$.
		The following are equivalent:
		\begin{enumerate}
			\item $z\in\Sk(\omega_Z)$;
			\item $x\in\Sk(\omega_X)$ and $z\in\Sk(\omega_{Z_x})$.
		\end{enumerate}
	\end{enumerate}
\end{proposition}
\begin{proof}
	Working Zariski locally, we can assume there exists an étale map
	\[\alpha\colon\fX\to\fS\coloneqq\Spf\big(\kc\braket{s_0,\dots,s_m}/(s_0\cdots s_{m'}-a)\big)\]
	for some $0\le m'\le m$ and $a\in\kcc\setminus 0$, such that every irreducible component of $\fF$ is given by $\alpha^*(s_j)$ for some $j>m'$; and similarly an étale map
	\[\beta\colon\fY\to\fT\coloneqq\Spf\big(\kc\braket{t_0,\dots,t_n}/(t_0\cdots t_{n'}-b)\big)\]
	for some $0\le n'\le n$ and $b\in\kcc\setminus 0$, such that every irreducible component of $\fG$ is given by $\beta^*(t_j)$ for some $j>n'$.
	
	Let $e\coloneqq\big(\frac{d s_1}{s_1}\wedge\dots\wedge\frac{d s_m}{s_m}\big)^{\otimes l}$,  $f\coloneqq\big(\frac{d t_1}{t_1}\wedge\dots\wedge\frac{d t_n}{t_n}\big)^{\otimes l}$.
	Write $\omega_X=\psi_X e$ for some $\psi_X\in\Gamma(\cO_X)$, and $\omega_Y=\psi_Y f$ for some $\psi_Y\in\Gamma(\cO_Y)$.
	Put $\psi_Z\coloneqq p_X^*\psi_X\wedge p_Y^*\psi_Y$.
	Let $\psi_{Z_x}$ be the pullback of $\psi_Z$ to $Z_x$.
	Since $Z_x = p_X\inv(x)$, we have $Z_x\simeq Y\widehat\otimes\cH(x)$.
	Let $\fY'\coloneqq\fY\widehat\otimes\cH(x)^\circ$, $\fG'\coloneqq\fG\widehat\otimes\cH(x)^\circ$.
	We have $\fY'_\eta\setminus\fG'_\eta\simeq Z_x$.
	
	Since all the strata of $\fX_s\cup\fF$ and $\fY_s\cup\fG$ are assumed to be geometrically connected, by the explicit formula for the embedding of skeletons in polyannuli (see \cite[\S 4.2]{Gubler_Skeletons_and_tropicalizations}), and the fact that the maps $\alpha\colon\fX\to\fS$ and $\beta\colon\fY\to\fT$ are étale, we obtain commutative diagrams
	\[\begin{tikzcd}
	\Sigma{(\fZ,\fH)} \ar{rr}{\sim} \dar[hook] & & \Sigma{(\fX,\fF)} \times \Sigma{(\fY,\fG)} \dar[hook]\\
	Z \rar{\sim} & X\times Y \rar & \abs{X}\times \abs{Y},
	\end{tikzcd}\]
	\[\begin{tikzcd}
	\oSigma{(\fZ,\fH)} \ar{rr}{\sim} \dar[hook] & & \oSigma{(\fX,\fF)} \times \oSigma{(\fY,\fG)} \dar[hook]\\
	\oZ \rar{\sim} & \oX\times \oY \rar & \abs{\oX}\times \abs{\oY},
	\end{tikzcd}\]
	\[\begin{tikzcd}
	\Sigma(\fY',\fG') \rar{\sim} \dar[hook] & \Sigma(\fY,\fG) \dar[hook]\\
	Z_x \rar & Y.
	\end{tikzcd}\]
	Hence, by \cref{lem:local_skeleton}, we deduce Statements (\ref{prop:skeleton_of_product:product}-\ref{prop:skeleton_of_product:fiber}) in the proposition.
	
	Next we turn to Statement (\ref{prop:skeleton_of_product:point}).
	Let $M_X$ be the maximum of $\norm{\omega_X}$ over $X$, and $M_Y$ the maximum of $\norm{\omega_Y}$ over $Y$.
	Since $\omega_X$ and $\omega_Y$ have at worst simple poles along $\fF_\eta$ and $\fG_\eta$ respectively, the maximums $M_X$ and $M_Y$ exist by \cref{lem:local_skeleton}.
	Moreover, we have $\abs{\psi_X}\le M_X$, $\abs{\psi_Y}\le M_Y$, and $\abs{\psi_Z}\le M_X\cdot M_Y$.
	By \cref{lem:local_skeleton}, the condition that $z\in\Sk(\omega_Z)$ is equivalent to the following condition:
	\begin{equation} \tag{$*$}
	\text{$z\in \Sigma{(\fZ,\fH)}$ and $\abs{\psi_Z}$ attains its maximum $M_X\cdot M_Y$ at the point $z$.}
	\end{equation}
	This is moreover equivalent to the following condition:
	\begin{align} \tag{$**$}
	x\in \Sigma{(\fX,\fF)}, z\in \Sigma{(\fY',\fG')}\subset Z_x,\  \abs{\psi_X} \text{ attains its maximum }M_X\\ \text{ at the point } x,
	\text{ and } \abs{\psi_{Z_x}} \text{ attains its maximum }M_Y \text{ at the point } z.\nonumber
	\end{align}
	By \cref{lem:local_skeleton} again, the above condition is furthermore equivalent to the condition that $x\in\Sk(\omega_X)$ and $z\in\Sk(\omega_{Z_x})$, completing the proof.
\end{proof}

\begin{remark}
	When $k$ has discrete valuation, \cref{prop:skeleton_of_product}(\ref{prop:skeleton_of_product:product}) was proved via logarithmic geometry by Brown and Mazzon \cite{Brown_The_essential_skeleton_of_a_product}.
\end{remark}

\begin{notation} \label{nota:H_skeletal}
	Let $k$ be any non-archimedean field of residue characteristic zero.
	Let $U$ be a connected smooth log Calabi-Yau $k$-variety with volume form $\omega$, $U\subset Y$ an snc\footnote{By snc, we assume in particular that all the strata are geometrically connected.} compactification, $D\coloneqq Y\setminus U$ and $D^\ess$ the union of irreducible components of $D$ where $\omega$ has a pole.
	Fix $C$ a proper rational nodal \kanal curve, $p_j\in C^\sm(k)$ and $m_j\in\bbN$ (including 0) for a finite set $J$, where $C^\sm\subset C$ denotes the smooth part of $C$.
	Denote by $B$ the set of $p_j$ with $m_j>0$.
	Fix an irreducible component $D_j\subset D^\ess$ for every $m_j>0$.
	Fix $\beta\in\NE(Y)$.
	Let $H$ be the subspace of $\Hom(C,Y^\an)$ consisting of maps $f\colon C\to Y^\an$ of class $\beta$ such that for each $m_j>0$, $p_j$ maps to $D_j^\circ$ and $f\inv(D)=\sum m_j p_j$.
	Let $V$ denote the analytic log tangent bundle $(T_Y(-\log D))^\an$.
	Let $H^\sm$ denote the Zariski open locus of $H$ where $f^*V$ is trivial.
	
	Let $C^\circ\coloneqq C^\sm\setminus(p_j)_{j\in J}$, $\sC\coloneqq C\times H^\sm$, $\sC^\circ\coloneqq C^\circ\times H^\sm$, $p\colon\sC\to H^\sm$, $p_C\colon\sC\to C$ the projection maps, and 
	$e\colon\sC\to Y^\an$ the universal map.
	Let
	\[\Phi\coloneqq(p_C,e)\colon\sC\to C\times Y^\an.\]
\end{notation}

\begin{lemma} \label{lem:H_smoothness}
	The space $H^\sm$ is smooth.
	The map $\Phi|_{\sC^\circ}\colon\sC^\circ\to C^\circ\times U^\an$ is étale.
	For any $x \in C(k) \setminus B$, the evaluation map $\ev_x\colon H^\sm \to U^\an$ is étale. 
	Moreover, we have $p_*(e^*V)\simeq T_{H^\sm}$.
\end{lemma}
\begin{proof}
	As all the spaces involved are the analytification of analogously defined $k$-varieties, the smoothness and étaleness follow from \cref{lem:Msm_smooth}.
	Since the Zariski tangent space at any point $f\in H$ is given by $H^0(C,f^*V^\an)$, we obtain $p_*(e^*V^\an)\simeq T_{H^\sm}$.
	\end{proof}

Consider the sequence of natural maps:
\begin{multline*}
H^0(Y,K_{Y}(D))\simeq H^0(Y^\an,\wedge^n V^\vee)\longrightarrow H^0(\sC,e^*(\wedge^n V^\vee))\xrightarrow{\ \sim\ }\\
\Hom(\wedge^n e^*V,\cO_\sC)\xrightarrow{\ \sim\ }\Hom(p^* p_*\wedge^n e^*V,\cO_\sC)\xrightarrow{\ \sim\ }\\
\Hom(p_{*}\wedge^n e^*V,\cO_{H^\sm})\xrightarrow{\ \sim\ } H^0(H^\sm,K_{H^\sm}),
\end{multline*}
where we use the fact that $e^*V$ is trivial on every fiber of $p\colon\sC\to H^\sm$.
Let $\omega_H \in H^0(H^\sm,K_{H^\sm})$ be the image of $\omega$ under the composition of the maps above.

\begin{lemma} \label{lem:H1form}
	For any 1-form $\alpha$ on $C^\circ$, we have $p_C^*\alpha\wedge e^*\omega=p_C^*\alpha\wedge p^*\omega_H$.
	
	For any $x \in C(k) \setminus B$, consider the evaluation map $\ev_x\colon H\to U^\an$.
	The following hold:
	\begin{enumerate}
	\item $\ev_x^{-1}(\Sk(U)) \subset H^\sm$.
	\item $\ev_x^*(\omega)|_{H^\sm} = \omega_H$;
	in particular, the left hand side is independent of the choice of $x$, and $\omega_H$ is a log volume form on $H^\sm$.
	\item \label{lem:H1form:skeleton} $\ev_x^{-1}(\Sk(U)) = \Sk(\omega_H)$;
	in particular, the left hand side is independent of the choice of $x$.
	\end{enumerate}
\end{lemma}
\begin{proof}
	For the first paragraph, we can check this after base field extension so may assume $k$ is algebraically closed.
	Then it is enough to check the equality at every $k$-rational point $(c,f)\in C^\circ\times H^\sm$.
	It suffices to check that $e^*\omega$ and $p^*\omega_H$ agree on the horizontal (with respect to $p$) tangent space $T_f H\subset T_{(c,f)}\sC$.
	Note the derivative $d(e)$ restricted to this subspace is the restriction map
	\[T_f H\simeq H^0(C,f^*V)\longrightarrow(f^*V)_c=V_{f(c)}=T_{Y^\an,f(c)}.\]
		Let $v\in\wedge^n T_f H\simeq H^0(C,\wedge^n f^*V)$.
	We have a commutative diagram
	\[\adjustbox{width=\textwidth}{\begin{tikzcd}[column sep=small]
	H^0(C,\wedge^n f^*V)\times H^0(Y^\an,\wedge^nV^\vee) \ar{r}\ar{d} & H^0(C,\wedge^n f^*V)\times H^0(C,\wedge^n f^*V^\vee) \ar{r} & H^0(C,\cO_C) \ar{d}\\
	\wedge^nV_{f(c)}\times\wedge^nV_{f(c)}^\vee\ar{r} & \cO_{Y^\an,f(c)}\ar{r}{\sim} & \cO_{C,c}\simeq k.
	\end{tikzcd}}\]
	The value of $(v,\omega)$ under the composition of the upper horizontal maps and the right vertical map is equal to $(p^*\omega_H)(v)$; and the value of $(v,\omega)$ under the composition of the left vertical map (restriction map) and the lower horizontal maps is equal to $(e^*\omega)(v)$.
	Hence the commutativity of the diagram gives the equality $(e^*\omega)(v)=(p^*\omega_H)(v)$, completing the proof.

	Next we consider the second paragraph:
	By \cref{lem:essential_skeleton_semistable_pair}, $\Sk(U)$ is contained in any Zariski dense open subset of $Y$, so (1) follows from \cref{lem:Msm_big}.
	For (2), we view $x$ as a section $x\colon H^\sm \to \sC=C\times H^\sm$.
	Note the image of the derivative $d(x)$ is in the horizontal tangent subspace.
	So by the above, we obtain $\ev_x^*(\omega)|_{H^\sm} = x^*(e^*\omega) = \omega_H$.
		Finally (1) and (2) imply (3), by \cref{lem:skeleton_etale_pullback}.
\end{proof}

\begin{definition}[Essential skeleton] \label{def:essential_skeleton}
	For any non-archimedean field $k$ of characteristic zero, any smooth $k$-algebraic variety $X$, we define
	\[\Sk(X)\coloneqq\Sk(X^\an)\coloneqq\bigcup_{\substack{\omega\in H^0(Y,K_Y(D)^{\otimes l})\setminus 0\\l\in\bbN_{>0}}}\Sk(\omega)\ \subset\  X^\an\]
	for an snc compactification $X\subset Y$, $D\coloneqq Y\setminus X$.
		Note the vector subspace $H^0(Y,K_Y(D)^{\otimes l})\allowbreak \subset H^0(X,K_X^{\otimes l})$, and thus the definition above, is independent of the compactification.
	When $Y$ is given, we denote by $\oSk(X)$ the closure of $\Sk(X)$ in $Y^\an$.
\end{definition}

\begin{remark}
	The essential skeleton was first proposed by Kontsevich-Soibelman \cite{Kontsevich_Homological_mirror_symmetry} for projective Calabi-Yau varieties over the Laurent power series field $\bbC\llp t\rrp$.
	Generalizations were later studied by Nicaise-Mustata \cite{Nicaise_The_essential_skeleton}, Brown-Mazzon \cite{Brown_The_essential_skeleton_of_a_product} and Mauri-Mazzon-Stevenson \cite{Mauri_Essential_skeletons_of_pairs}.
\end{remark}

Now we continue the setting of \cref{nota:H_skeletal}.

\begin{lemma} \label{lem:skeleton_of_C}
	Let $\Gamma\subset C$ be the convex hull of $C\setminus C^\circ$, and let $\Gamma^\circ\coloneqq\Gamma\cap C^\circ$.
	We have
	\begin{equation} \label{eq:skeleton_of_C}
	\Sk(C^\circ)=\bigcup_{\omega\in H^0(C,K_C(C\setminus C^\circ))\setminus 0}\Sk(\omega)\ =\ \Gamma^\circ\ \subset\  C^\circ.
	\end{equation}
	Moreover, when $\Sk(C^\circ)\neq\emptyset$, there exists $\omega\in H^0(K_C(C\setminus C^\circ))$ such that $\Sk(C^\circ)=\Sk(\omega)$.
\end{lemma}
\begin{proof}
	By reasoning over every irreducible component of $C$, we can assume $C$ is irreducible.
	If $\abs{J}=0$ or $1$, $H^0(C,K_C(\cup p_j))=0$, so $\Sk(C^\circ)=\Gamma^\circ=\emptyset$.
	Now assume $\abs{J}\ge 2$.
	For $i\neq j\in J$, let $\omega_{i,j}$ be the unique section of $H^0(C,K_C(p_i\cup p_j))$ having poles with residues $1, -1$ at $p_i$, $p_j$ and no zeros.
	Then the tensor products of $\omega_{i,j}$ generate $H^0(C,K_C(\cup p_i)^{\otimes l})$ for all $l\in\bbN_{>0}$.
	Hence
	\[\Sk(C^\circ)=\bigcup_{i\neq j\in J}\Sk(\omega_{i,j})\ \subset\ C^\circ,\]
	in particular, the first equality of \eqref{eq:skeleton_of_C} holds.
	
	Let $[p_i,p_j]$ denote the path in $C$ connecting $p_i$ and $p_j$, and let $\Gamma_{i,j}\coloneqq[p_i,p_j]\setminus\{p_i,p_j\}$.
	Choose a coordinate on $C$ such that $p_i=0$, $p_j=\infty$.
	Then we obtain a strictly semistable pair $(\bbP^1_{\kc}, 0+\infty)$ whose associated skeleton is $\Gamma_{i,j}\subset C\setminus\{0,\infty\}$.
	Hence \cref{lem:explicit_weight} implies that $\Sk(\omega_{i,j})=\Gamma_{i,j}$. (Note the discrete valuation assumption of \cref{lem:explicit_weight} is harmless here because $(\bbP^1_{\kc}, 0+\infty)$ is constant over $\kc$ and we can also choose $\omega_{i,j}$ to be constant over $\kc$.)
	Taking union, we prove the second equality of \eqref{eq:skeleton_of_C}.
	
	For the last statement, it suffices to take a general linear combination of $\omega_{i,j}$.
\end{proof}

\begin{lemma} \label{lem:x_in_skeleton}
	Let $f\in H$ and $C_f$ the fiber of $C\times H \to H$ over $f$.
	Let $f\colon C_f\to Y^\an$ denote the restriction of the universal map $C\times H\to Y^\an$.
	For any $x\in C_f$, the following are equivalent:
	\begin{enumerate}
		\item $x\in\Phi\inv(\Sk(C^\circ\times U^\an))$.
		\item $f\in\Sk(\omega_H)\subset H^\sm$ and $x\in\Sk(C^\circ_f)$.
	\end{enumerate}
\end{lemma}
\begin{proof}
	By Lemmas \ref{lem:skeleton_etale_pullback} and \ref{lem:H1form}, we have
	\begin{align*}\label{eq:x_in_skeleton}
	\begin{split}
		\Phi\inv(\Sk(C^\circ\times U^\an))&=\Phi\inv\bigg(\bigcup_{\alpha\in H^0(C,K_C(C\setminus C^\circ))}\Sk(\alpha\wedge\omega)\bigg)\\
		&=\bigcup_{\alpha\in H^0(C,K_C(C\setminus C^\circ))}\Sk\big(\Phi^*(\alpha\wedge\omega)\big)\\
		&=\bigcup_{\alpha\in H^0(C,K_C(C\setminus C^\circ))}\Sk(p_C^*\alpha\wedge e^*\omega)\\
		&=\bigcup_{\alpha\in H^0(C,K_C(C\setminus C^\circ))}\Sk(p_C^*\alpha\wedge p^*\omega_H).
	\end{split}
	\end{align*}
	Note from \cref{lem:skeleton_of_C}, here it is sufficient to consider log volume forms instead of all log pluricanonical forms.
	So we conclude from \cref{prop:skeleton_of_product}(\ref{prop:skeleton_of_product:point}).
\end{proof}

\begin{theorem} \label{thm:f_in_skeleton}
	Notation as in \cref{lem:x_in_skeleton}.
	Let $g\colon C_f\to C\times Y^\an$ denote the product of $C_f\to C$ and $f\colon C_f\to Y^\an$.
	The following are equivalent:
	\begin{enumerate}
		\item $f\in\Sk(\omega_H)\subset H^\sm$.
		\item For some $x \in C(k) \subset C_f$, $f(x) \in \Sk(U)$.
		\item For every $x \in C(k) \subset C_f$, $f(x) \in \Sk(U)$.
		\item \label{thm:f_in_skeleton:g_inv} $g\inv(\Sk(C^{\circ} \times U^\an)) = \Sk(C^\circ_f)$.
		\item \label{thm:f_in_skeleton:g_inv_nonempty} $g^{-1}(\Sk(C^{\circ} \times U^\an)) \neq \emptyset$.
	\end{enumerate}
	Assume these equivalent conditions hold, let $\Gamma(k) \subset C_f$ be the convex hull of $C(k)\subset C_f$; then $f(\Gamma(k)) \subset \oSk(U)$.
\end{theorem}
\begin{proof}
	For any $x\in C(k)\setminus B\subset C_f$, let $\ev_x\colon H\to U^\an$ be the evaluation map.
	We have $f(x)=\ev_x(f)$.
	So the equivalences between the first three statements follow from \cref{lem:H1form}(\ref{lem:H1form:skeleton}).
	The equivalences between (1), (4) and (5) follow from \cref{lem:x_in_skeleton}.
	The last claim follows from (4) because we are free to add extra $k$-rational points to $C$ as marked points in addition to $(p_j)_{j\in J}$ and then apply (4) to $C$ together with the extra marked points.
\end{proof}

\begin{remark}
	Note that in \cref{nota:H_skeletal}, $H$ does not change if we add or drop marked points $p_i \not \in B$.
	Thus \cref{thm:f_in_skeleton} Condition (\ref{thm:f_in_skeleton:g_inv_nonempty}) holds for one choice of marked points containing $B$, if and only if it holds for $B$.
\end{remark} 

Recall from \cref{def:skeletal_curve} that $f\colon C_f\to Y^\an$ satisfying the equivalent conditions in \cref{thm:f_in_skeleton} is called a \emph{skeletal curve}.
Below are the sources of skeletal curves in the context of this paper.

\begin{lemma} \label{lem:source_of_skeletal_curve}
	Notation as in \cref{nota:moduli_spaces_analytic}.
	The following hold:
	\begin{enumerate}
	\item For any $\mu \in \ocM_{0,n}^{\an}$, $\Phi_i^{-1}(\Sk(\mu \times_k U^\an)) \cap \cM(U^\an,\bP,\beta)_{\mu}$ consist of skeletal curves.
	\item $\ISk\coloneqq\Phi_i^{-1}(\oSk(\cM_{0,n}) \times \Sk(U)) \cap \cM(U^\an,\bP,\beta)$ consist of skeletal curves.
	\end{enumerate}
\end{lemma} 
\begin{proof}
	(1) is immediate from the definition of skeletal curve.
	(1) implies (2) by \cref{prop:skeleton_of_product}(\ref{prop:skeleton_of_product:point}).
\end{proof}

\begin{lemma} \label{lem:restriction_to_skeleton}
	Notation as in \cref{lem:source_of_skeletal_curve}.
	The following hold:
	\begin{enumerate}
		\item \label{lem:restriction_to_skeleton:sm} Assume $[C,(p_j)_{j\in J},f]\in\cM(U^\an,\bP, \beta)_\mu$ is skeletal.
		Then it belongs to $\cM^\sm(U^\an,\allowbreak\bP,\allowbreak\beta)$; in particular, we have $\ISk\subset\cM^\sm(U^\an,\bP,\beta)$.
		Moreover, for any closed subvariety $G\subset Y$ not containing any irreducible component of $D^\ess$, the pullback $f\inv(G^\an)$ is a finite set of points without multiplicities and disjoint from the nodes of $C$;
		and for any closed subvariety $Z\subset Y$ of codimension at least 2, the image $f(C)$ does not meet $Z^\an$.
		\item \label{lem:restriction_to_skeleton:etale}
		$\Phi_i$ is representable (i.e.\ non-stacky) and étale on a Zariski open neighborhood of $\ISk$.
		\item \label{lem:restriction_to_skeleton:statum} For any open stratum\footnote{See the convention in the end of \cref{sec:introduction}.} $S\subset\ocM_{0,n}$, the map $\Phi_i^{-1}(\Sk(S \times U)) \to \Sk(S \times U)$ is proper, open and set-theoretically finite (i.e.\ having finite fibers).
		\item \label{lem:restriction_to_skeleton:all} $\Phi_i|_{\ISk} \colon \ISk\to \oSk(\cM_{0,n}) \times \Sk(U)$ is open and set-theoretically finite.
	\end{enumerate}
\end{lemma}
\begin{proof}
	By \cref{lem:essential_skeleton_semistable_pair}, $\Sk(U)$ is contained in any Zariski dense open subset of $Y$, so (\ref{lem:restriction_to_skeleton:sm}) follows from Lemmas \ref{lem:stable_domain} and \ref{lem:Msm_big}.
	Note that a stable pointed rational curve does not have non-trivial automorphisms, so
	(\ref{lem:restriction_to_skeleton:etale}) follows from (\ref{lem:restriction_to_skeleton:sm}) and \cref{lem:Msm_smooth}(\ref{lem:Msm_smooth:etale}).
	(\ref{lem:restriction_to_skeleton:statum}) follows from \cref{prop:smoothness_all} and \cref{lem:essential_skeleton_semistable_pair}.
	For (\ref{lem:restriction_to_skeleton:all}), the finiteness follows from (\ref{lem:restriction_to_skeleton:statum}), and the openness follows from (\ref{lem:restriction_to_skeleton:etale}).
\end{proof}

\begin{lemma} \label{lem:skeletal_curve_contraction}
	Notation as in \cref{lem:source_of_skeletal_curve}.
	Let $(f\colon[C,(p_1,\dots,p_n)] \to Y^{\an}) \in \ISk$.
	Let $\Gamma$ (resp.\ $\Gamma^B$) denote the convex hull in $C$ of the all the marked points (resp.\ all the marked points from $B$).
	Then $f|_\Gamma\colon\Gamma\to {\oSk(U)}\subset Y^\an$ factors through the retraction $\Gamma\to\Gamma^B$.
\end{lemma} 
\begin{proof}
	When $\Sp(f)$ is transverse, the statement follows from \cref{rem:spine_factorization}.
	By \cref{lem:restriction_to_skeleton}(\ref{lem:restriction_to_skeleton:all}), $\Phi_i|_\ISk\colon \ISk \to \oSk(\cM_{0,n}) \times \Sk(U)$ is open.
	Thus by \cref{prop:transversality}, the statement holds for a dense subset of $f \in \ISk$.
	Now the result follows by continuity.
\end{proof}

The following lemma is the key to the proof of \cref{prop:moving_w}.

\begin{lemma} \label{lem:skeletal_curve_connected_component}
	Notation as in \cref{thm:f_in_skeleton}.
	Let $f\colon C_f\to Y^\an$ be a map in $\Sk(\omega_H)\subset H$.
	Let
	\[\Delta\subset\oSk(C^\circ\times U^\an)\simeq\oSk(C^\circ)\times\oSk(U)\simeq\oSk(C^\circ)\times\oM_\bbR\]
	be the graph of the canonical spine $\Sp(f)\colon\oSk(C^\circ)\simeq\oSk(C_f^\circ)\to\oSk(U)$.
		Fix $A\in\bbN$ big with respect to $\beta$ (see \cref{def:A_big}).
	Assume $\Sp(f)$ is transverse with respect to $\Wall_A$.
	Then $\oSk(C_f^\circ)$ is a connected component of $\Phi\inv(\Delta)$.
\end{lemma}
\begin{proof}
	By \cref{lem:x_in_skeleton}, $\oSk(C_f^\circ)$ is equal to the fiber of $\Phi\inv(\oSk(C^\circ\times U^\an))$ over $f\in\Sk(\omega_H)$.
	So $\oSk(C_f^\circ)\subset\Phi\inv(\Delta)$ is closed.
	
	Now let us prove that the inclusion is open.
	Consider any point $q\in\oSk(C_f^\circ)$ and any open connected neighborhood $V$ of $q$ in $\Phi\inv(\Delta)$.
	It suffices to prove that $p_H(V)=\{f\}$ where $p_H\colon C\times H^\sm\to H^\sm$ denotes the projection.
	By \cref{lem:spine_injective}, $\Sp(g)=\Sp(f)$ for all $g\in p_H(V)\subset\Sk(\omega_H)$, because $\Phi(V)\subset\Delta$ implies that $\Sp(g)(x)=\Sp(f)(x)$ for any $x\in V\cap p_H\inv(g)$.
	Therefore, for any fixed $q'\in\Sk(C^\circ)$, we have $\Sp(g)(q')=\Sp(f)(q')$ for all $g\in p_H(V)$.	
	By \cref{lem:H_smoothness}, $\Phi\inv(q,f(q))$ is a finite discrete set.
	We conclude that $p_H(V)$ can only contain the single point $f$, completing the proof.
\end{proof}

The next proposition studies the essential skeleton of $\cM_{0,n}$.
Let $u\colon\ocM_{0,n+1} \to\ocM_{0,n}$ be the forgetful map.
It gives the universal curve over $\ocM_{0,n}$.
Let $p_1,\dots,p_n$ be the universal sections of $u$, and let $C_{0,n} \subset \ocM_{0,n+1}$ be the complement of the union of the universal sections.

\begin{proposition} \label{prop:skeleton_of_M0n}
	The following hold:
	\begin{enumerate}
		\item $\oSk(\cM_{0,n}) = \oSigma(\ocM_{0,n},\partial\ocM_{0,n})\subset \ocM_{0,n}^\an$. 
		\item $\oSk(\cM_{0,n})\simeq \NT^\emptyset_n \simeq \oM_{0,n}^\trop$, (see \cref{rem:comparison_ACP}).
		\item A point $x \in \ocM_{0,n+1}^\an$ lies in $\oSk(\cM_{0,n+1})$ if and only if $u^\an(x)\in\oSk(\cM_{0,n})$ and $x \in \oSk(C^{\circ}_{u^\an(x)})$, where $C^{\circ}_{u^\an(x)}$ denotes the fiber of  $u^\an\colon C_{0,n} \to \ocM_{0,n}^\an$ over $u^\an(x)$.
	\end{enumerate}
\end{proposition}
\begin{proof}
	By \cite{Keel_Equations_for_M0n}, the sheaf $K_{\ocM_{0,n}}(\partial\ocM_{0,n})$ (which is denoted by $\kappa$ in loc.\ cit.) is very ample.
	Hence by \cref{lem:explicit_weight}, we have $\Sk(\cM_{0,n}) = \Sigma(\ocM_{0,n},\partial \ocM_{0,n})\subset \cM_{0,n}^\an$.
		Taking closure, we obtain (1).
	We have $\oSigma(\ocM_{0,n},\partial\ocM_{0,n})\simeq\oM_{0,n}^\trop$ by \cite[Theorem 1.2.1(1)]{Abramovich_The_tropicalization_of_the_moduli_space_of_curves}.
	So (1) implies (2).
	
	Now we consider (3).
	Using the inductive description of the closed strata of $\partial\ocM_{0,n}$ as products of various $\ocM_{0,m}$ for $m<n$, it suffices to prove (3) for $\Sk$ instead of $\oSk$.
	By \cite[2.8]{Keel_Equations_for_M0n}, there is a nice set of generators of $H^0\big(K_{\ocM_{0,n+1}}(\partial \ocM_{0,n+1})\big)$ as follows.
	For each pair $i \neq j \in\{0,\dots,n\}$, there is a canonical $u$-relative log 1-form $\omega_{i,j}$ uniquely characterized by the property that its restriction to each fiber has residues $1,-1$ at $p_i,p_j$, and has no zeros, and no other poles.
	The $\omega_{i,j}$ generate $K_u(P)$
		at every point, where $P$ denotes the sum of the universal sections; moreover their restriction to any fiber of $u$ generate the log volume forms on the fiber.
	Then one obtains a generating set of log volume forms on $\cM_{0,n+1}$ by wedging the $\omega_{i,j}$ with $u^*(\gamma)$ for $\gamma$ inductively constructed log volume forms on $\cM_{0,n}$.
	Let $C_{0,n}^{i,j} \subset C_{0.n}$ be the complement of $p_i \cup p_j$.
	There is a trivialization $C_{0,n}^{i,j}\simeq \cM_{0,n} \times \bbG_m$, such that $\omega_{i,j}$ is pulled back from the $\bbG_m$ factor.
	Since $C_{0,n}^{i,j}$ and $\cM_{0,n+1}$ are birational, $\Sk(\omega_{i,j} \wedge \pi^*(\gamma))\subset \cM_{0,n+1}^\an$ is the same whether we compute using $\cM_{0,n+1}$ or $C^{i,j}_{0,n}$.
	Now the result follows from \cref{prop:skeleton_of_product}(\ref{prop:skeleton_of_product:point}).
\end{proof}

\section{Naive counts and the symmetry property} \label{sec:naive_counts}

Here we give the details regarding the naive counts mentioned in \cref{sec:intro:naive_counts}.
Furthermore, we prove that for transverse spines, the naive count is independent of the choice of the marked point at which we evaluate, see \cref{prop:moving_w} and \cref{thm:forgetting_interior_marked_points}.
This is a generalization of the symmetry theorem in \cite[\S 6]{Yu_Enumeration_of_holomorphic_cylinders_I}.
Our proof here is different from that of loc.\ cit., and gives a stronger statement.

We follow the setting of \cref{sec:log_CY}.

\begin{definition} \label{def:spine_in_U}
	A \emph{spine} in $\Sk(U)$ consists of a stable nodal metric tree $\Gamma$, a set of $n$ different 1-valent vertices $(v_j)_{j\in J}$, and a continuous map $h\colon\Gamma\to\oSk(U)$ satisfying the following conditions:
	\begin{enumerate}
		\item The vertices $v_j$ are the only 1-valent vertices of $\Gamma$.
		\item The preimage $h\inv(\partial\oSk(U))$ is a subset of $(v_j)_{j\in J}$.
		\item The map $h$ is piecewise \Zaffine, i.e.\ each edge of $\Gamma$ maps into a cone of $\oSigma_{(Y,D^\ess)}$ with integer derivative.
	\end{enumerate}
	We denote
	\begin{align*}
		F&\coloneqq\set{ j | v_j\text{ is finite}},\\
		B&\coloneqq\set{ j | h\text{ is not constant near }v_j},\\
		I&\coloneqq\set{ j | h\text{ is constant near }v_j}.
	\end{align*}
	The spine is called \emph{extended} if $F=\emptyset$.
	For each $j$, let $P_j$ be the derivative of $h$ at $v_j$ (pointing outwards).
	Let $\bP\coloneqq(P_j)_{j\in J}$.
	If $j\in I$, we have $P_j=0$.
	If $j\in B\setminus F$, $P_j$ can be identified with a point in $\Sk(U,\bbZ)$.
\end{definition}

We will define the count of analytic curves in $U^\an$ associated to any given spine $S=[\Gamma,(v_j)_{j\in J},h]$ in $\Sk(U)$ and any curve class $\gamma\in\NE(Y)$.
We proceed under several different assumptions on $S$.
We will always assume $\abs{B}\ge 2$.
For Constructions \ref{const:naive_counts_extended} and \ref{const:naive_counts_truncated}, we assume $\abs{I}\ge 1$ and fix $i\in I$.

\begin{construction} \label{const:naive_counts_extended}
	First we assume $S$ is extended.
	Let
	\[\Phi_i\coloneqq(\st,\ev_i)\colon \cM(U^\an,\bP,\gamma) \longrightarrow \ocM_{0,n}^\an \times U^\an\]
	be as in \cref{nota:moduli_spaces_analytic}.
	By \cref{prop:skeleton_of_M0n}, $\Gamma$ gives a point $\Gamma\in\oSk(\cM_{0,n})\subset\ocM_{0,n}^\an$.
	By Lemmas \ref{lem:restriction_to_skeleton} and \ref{lem:source_of_skeletal_curve}, $\Phi_i\inv(\Gamma,h(v_i))$ is finite, and consists of skeletal curves.
	Let $F_i(S,\gamma)$ be the subspace of $\Phi_i\inv(\Gamma,h(v_i))$ consisting of maps whose spine is equal to $S$.
	We define $N_i(S,\gamma) \coloneqq \length(F_i(S,\gamma))$.
	\end{construction}

\begin{construction} \label{const:naive_counts_truncated}
	Now we drop the assumption that $S$ is extended.
		For each $j\in F$, we glue a copy of $l_j\coloneqq[0,\hv_j\coloneqq+\infty]$ to $\Gamma$, along $0$ and $v_j$.
	We extend $h$ affinely to the new leg $l_j$ via the identification $\Sk(U)\simeq M_\bbR$.
	Let $\delta_j\in\NE(Y)$ be the curve class associated to the new leg (see \cref{def:curve_class_via_varphi}).
	Let $\hS=[\hGamma,(\hv_j)_{j\in J},\hh]$ denote the resulting extended spine.
	Let $\hgamma\coloneqq\gamma+\sum_{j\in F}\delta_j$.
	We apply \cref{const:naive_counts_extended} to $\hS$ and $\hgamma$, and obtain $F_i(\hS,\hgamma)$.
	Let $F_i(S,\gamma)\subset F_i(\hS,\hgamma)$ be the subspace consisting of stable maps $[\hC,(p_j)_{j\in J},f]$ satisfying the \emph{toric tail condition}:
	let $r\colon \hC\to\hGamma$ be the canonical retraction;
	for each $j\in F$, let $\bbT^*_j\coloneqq r\inv(l_j\setminus{\hv_j})$;
	then we require that $f(\bbT^*_j)\subset T_M^\an$.
	We define $N_i(S,\gamma)\coloneqq\length(F_i(S,\gamma))$, the count of analytic curves (with boundaries) in $U^\an$ of spine $S$ and class $\gamma$ (evaluating at $i$).
\end{construction}

The counts in the toric case are particularly simple:

\begin{lemma} \label{lem:counts_in_toric_case}
	Assume $(Y,D)$ is toric and $S$ satisfies \cref{lem:curve_class_from_model} Conditions (1-2).
	Then $N_i(S,\gamma)=1$ if $S$ has no bending vertices and $\gamma=\delta_h$ of \cref{def:curve_class_via_varphi}; $N_i(S,\gamma)=0$ otherwise.
\end{lemma}
\begin{proof}
	This follows from Propositions \ref{prop:toric_case} and \ref{prop:curve_class_formula}.
\end{proof}

\begin{proposition} \label{prop:moving_w}
	Assume that under the identification $\Sk(U)\simeq M_\bbR$, $S$ is a transverse spine with respect to $\Wall_A$ for some $A\in\bbN$ big with respect to $\gamma$.
	Let $\ow\in\Gamma\setminus\set{v_j | j\in B\setminus F}$ away from the nodes.
	We glue $[0,w=+\infty]$ to $\Gamma$ along $0$ and $\ow$, extend $h$ constantly on the new leg, and obtain a new spine which we denote by $S_{\ow}$.
	Then the count $N_w(S_{\ow},\gamma)$ is independent of the choice of $\ow\in\Gamma$.
\end{proposition}

\begin{definition} \label{def:naive_count_transverse}
	In virtue of \cref{prop:moving_w}, we define $N(S,\gamma) \coloneqq N_w(S_{\ow},\gamma)$ (for any choice of $\ow$).
\end{definition}

For the proof of \cref{prop:moving_w}, we start with a simple lemma in point-set topology.

\begin{definition} \label{def:proper_extension}
	Let $p\colon M \to V$ be a continuous map between Hausdorff topological spaces.
	A \emph{proper extension} of $p$ consists of an open embedding $M \subset \oM$, with $\oM$ Hausdorff, and a proper map $\op\colon \oM \to V$ such that $\op|_M = p$.
\end{definition}

\begin{lemma} \label{lem:proper_extension_existence}
	Notation as in \cref{def:proper_extension}.
	If there is an open embedding $M\subset M'$ with $M'$ compact and Hausdorff, then a proper extension exists.
	Such $M\subset M'$ exists if and only if $M$ is locally compact.
\end{lemma}
\begin{proof}
	We take for $\oM$ the closure of the graph $\Gamma_p\subset M\times V$ inside $M'\times V$.
	Since $V$ is Hausdorff, $\Gamma_p\subset M\times V$ is closed.
	Thus $\oM\cap(M\times V)=\Gamma_p$.
	Since $M\times V\subset M'\times V$ is open, it follows that $M\simeq\Gamma_p\subset\oM$ is open.
	Hence we obtain a proper extension.
	
	If $M$ is locally compact and noncompact, we can take $M\subset M'$ the one-point compactification.
	The other direction of the last statement is obvious.
\end{proof}

\begin{lemma} \label{lem:topological_extension}
	Let $p\colon M \to V$ be a continuous map between Hausdorff spaces that admits a proper extension.
	Let $R \subset V$ be a closed, locally compact subset and $\tR \subset M_R$ a union of connected components such that $p\colon \tR \to R$ is proper.
	Then there is an open subset $R \subset W \subset V$, and a union of connected components $\tW \subset M_W$ such that $p\colon \tW \to W$ is proper, and $\tW \cap M_R = \tR$.
\end{lemma}
\begin{proof}
	Let $M\subset\oM$ and $\op\colon\oM\to V$ be a proper extension of $p$.
	The problem is local on $R \subset V$, so we can assume $R$, and thus $\tR$, are compact.
	Using this compactness we can find open  neighborhoods $A \supset \tR$ and $B \supset \oM_R \setminus \tR$ with $A \cap B = \emptyset$.
	Then it suffices to set $W \coloneqq \op((A \cup B)^c)^c$, where $^c$ means complement, and $\tW \coloneqq M_W \cap A$.
\end{proof}

\begin{lemma} \label{lem:topological_extension_point}
	Let $p\colon M \to V$ be a continuous map between Hausdorff spaces that admits a proper extension.
	Assume $p$ has finite fibers.
	Let $m\in M$ be a point and $U$ an open neighborhood of $m$.
	Then there is an open neighborhood $W$ of $p(m)$ in $V$ such that the connected component $\tW$ of $M_W$ containing $m$ is contained in $U$, and moreover the restriction $p\colon\tW\to W$ is proper with finite fibers.
\end{lemma}
\begin{proof}
	We apply \cref{lem:topological_extension} with $\tR=m$ and $R=p(m)$, and then shrink $W$ so that $\tW\subset U$.
\end{proof}

\begin{lemma} \label{lem:finite_etale_degree}
	Let
	\[\Phi_i\coloneqq(\st,\ev_i)\colon \cM(U^\an,\bP,\beta) \longrightarrow \ocM_{0,n}^\an \times U^\an\]
	be as in \cref{nota:moduli_spaces_analytic}.
	Let $R \subset \oSk(\cM_{0,n}\times U)$ be any closed connected subset, and $\tR \subset \Phi_i\inv(R)$ any union of connected components.
	Assume $\tR\xrightarrow{\ \Phi_i\ }R$ is (topologically) proper.
	(This assumption is automatic if $R\subset\Sk(\cM_{0,n}\times U)$ by \cref{lem:restriction_to_skeleton}(\ref{lem:restriction_to_skeleton:statum}).)
	Then $\tR\xrightarrow{\ \Phi_i\ } R$ is (topologically) finite and open, and its (analytic) degree over any point $r\in R$ is independent of the choice of $r$.
\end{lemma}
\begin{proof}
	By \cref{lem:restriction_to_skeleton}(\ref{lem:restriction_to_skeleton:etale}), $\Phi_i$ is representable and étale on a Zariski open subset $M$ containing $\tR$.
	Then the image $V$ of $M$ is a Zariski open subset containing $R$.	
	Hence by \cref{lem:topological_extension}, there is a connected neighborhood $W$ of $R$ contained in $V$ and a union of connected components $\tW\subset M_W$ such that $\tW\cap M_R=\tR$, with $\tW\xrightarrow{\ \Phi_i\ } W$ finite étale.
	This implies the lemma.
	%Note this also shows that the map $\tR\xrightarrow{\ \Phi_i\ } R$ is (topologically) finite and open.
\end{proof}

\begin{lemma} \label{lem:hgamma}
	In the context of \cref{const:naive_counts_truncated}, we have $\gamma\cdot\tE=\hgamma\cdot\tE$, $\tE$ as in \cref{prop:tropicalization_of_stable_map}.
	Consequently, for $A\in\bbN$, if $A$ is big with respect to $\gamma$, then $A$ is also big with respect to $\hgamma$.
\end{lemma}
\begin{proof}
	For every $j\in F$, by \cref{const:naive_counts_truncated} and \cref{def:curve_class_via_varphi}, the class $\delta_j\in\NE(Y)$ is the pullback of some class in $\NE(Y_\rt)$.
	Then by the projection formula in \cref{def:projection_formula_rational}, we have $\delta_j\cdot\tE=0$.
	So the lemma follows.
\end{proof}

\begin{proof}[Proof of \cref{prop:moving_w}]
	Let $\hS=[\hGamma,(\hv_j)_{j\in J},\hh]$ be the extension of $S$, and $\hgamma\coloneqq\gamma+\sum\delta_i$ as in \cref{const:naive_counts_truncated}.
	Let $\hGamma^\circ\coloneqq\hGamma\setminus\{\hv_1,\dots,\hv_n\}$ and $\hh^\circ\coloneqq\hh|_{\Gamma^\circ}$.
	The graph $\Delta$ of $\hh^\circ$ is naturally embedded in $M_{0,n+1}^\trop\times M_\bbR\simeq\Sk(\cM_{0,n+1}\times U)$.
	Let $\bP'\coloneqq(P_1,\dots,P_n,0)$, and
	\[\Phi\coloneqq(\st,\ev_w)\colon\cM(U^\an,\bP',\hgamma)\to\ocM_{0,n+1}^\an\times U^\an.\]
	Let $G\subset\Phi\inv(\Delta)$ be the subspace consisting of stable maps $[(C,(p_1,\dots,p_n,w),f]$ such that
	\begin{enumerate}
		\item forgetting the marked point $w$, the associated spine is $\hS$;
		\item the toric tail condition (as in \cref{const:naive_counts_truncated}) is satisfied.
	\end{enumerate}
	By \cref{lem:hgamma}, $A$ is also big with respect to $\hgamma$.
	So we can apply \cref{lem:skeletal_curve_connected_component}, and deduce that $G\subset\Phi\inv(\Delta)$ is a union of connected components, that is homeomorphic to a disjoint union of copies of $\hGamma^\circ$.
	By definition, $N_w(S_{\ow},\hgamma)$ is the degree of $G\xrightarrow{\ \Phi\ }\Delta$ over $(\ow,h(\ow))\in\Delta$; hence it is independent of the choice of $\ow$ by the following \cref{lem:finite_etale_degree}.
\end{proof} 

\begin{lemma} \label{lem:forgetting_points}
	Assume $n\ge 4$.
	Let $i,k\in I$, $i\neq k$.
	Let $\bP'\coloneqq\bP\setminus P_k$.
	Let $V \subset \ocM_{0,n}$ be the subspace consisting of pointed stable curves $[C,(p_j)_{j\in J}]$ which remains stable after forgetting the marked point $p_k$.
	Then $V$ is open and the following is a pullback diagram
	\[\begin{tikzcd}
	\cM^{\sd}(U,\bP,\beta)_V \rar{\Phi_i} \dar & V \times U\dar\\
	\cM^{\sd}(U,\bP',\beta) \rar{\Phi_i} & \ocM_{0,n-1} \times U 
	\end{tikzcd}\]
	where the vertical maps forget the marked point $p_k$.
\end{lemma}
\begin{proof}
	The lemma holds because we are simply forgetting an interior marked point.
\end{proof}

\begin{theorem} \label{thm:forgetting_interior_marked_points}
	Notation as in \cref{const:naive_counts_truncated}.
	Let $\Gamma^B \subset \Gamma$ be the convex hull of the $B$-type marked points. If $N_i(S,\gamma) \neq 0$ then $h$ factors through the canonical retraction $r:\Gamma \to \Gamma^B$.
	
	Now assume $h$ factors through this retraction, and let $S^B$ be the restriction of $S$ to $\Gamma^B$.
	Assume furthermore that $S^B$ is a transverse spine with respect to $\Wall_A$ for some $A\in\bbN$ big with respect to $\gamma$.
	Then $N_i(S,\gamma) = N(S^B,\gamma)$.
	In particular $N_i(S,\gamma)$ is independent of the choice of $i\in I$.
\end{theorem}
\begin{proof}
	The first paragraph follows from \cref{lem:skeletal_curve_contraction}.
	The second paragraph follows from \cref{lem:forgetting_points} and \cref{prop:moving_w}.
\end{proof}

\section{Properness and deformation invariance} \label{sec:deformation_invariance}

In this section, we prove the deformation invariance of naive counts associated to transverse extended spines, see \cref{thm:deformation_invariance_rigid}.
In short, deformation invariance follows from the finite étaleness of some evaluation map.
The étaleness will follow from \cref{sec:smoothness}, but the properness is a delicate issue, because the space $\cM^\sm(U,\bP,\beta)$ of our nicest curves is not proper.
The key properness results are \cref{prop:net_in_Msm} and \cref{cor:properness_SP}.
The heuristic meaning of the statements is that under a limit, as long as we control the spine, there will be no bubbling escaping to infinity.

Fix $\beta\in\NE(Y)$ and $A\in\bbN$ big with respect to $\beta$.
Let $S=[\Gamma,(v_j)_{j\in J},h]$ be a transverse extended spine in $M_\bbR$ of type $\bP$ with respect to $\Wall_A$ (see Definitions \ref{def:spine} and \ref{def:transverse}).
Fix $i\in I$.

Let \[\cM^\sm(U^\an,\bP,\beta)\subset\cM^\sd(U^\an,\bP,\beta)\subset\cM(U^\an,\bP,\beta)\subset\ocM(Y^\an,\bP,\beta),\]
and \[\Phi_i\coloneqq(\st,\ev_i)\colon\ocM(Y^\an,\bP,\beta)\to\ocM_{0,n}^\an\times Y^\an\]
be as in \cref{nota:moduli_spaces_analytic}.
We refer to \cref{sec:tropical} for the notations for moduli spaces of spines and tropical curves.
Let $\SP^{\tr_\infty}(M_\bbR,\bP)\subset\SP(M_\bbR,\bP)$ be the subset consisting of spines that are transverse at infinity.

\begin{proposition} \label{prop:net_in_Msm}
	Let $\Lambda$ be a directed set and $(f_\lambda)_{\lambda\in\Lambda}$ a net in $\cM^\sm(U^\an,\bP,\beta)$.
	If $(\Sp(f_\lambda))$ converges in $\SP^{\tr_\infty}(M_\bbR,\bP)$, then a subnet of $(f_\lambda)$ converges in $\cM^\sd(U^\an,\bP,\beta)$.
\end{proposition}

\begin{corollary} \label{cor:properness_SP}
	Let $\ISk$ be the preimage of $\oSk(\cM_{0,n})\times\Sk(U)$ by $\Phi_i\colon\cM(U^\an,\bP,\beta)\to\ocM_{0,n}^\an\times U^\an$, as in \cref{lem:source_of_skeletal_curve}.
	Let $\ISk^{\tr_\infty}\subset\ISk$ be the subset consisting of stable maps whose associated spines are transverse at infinity.
	The restriction $\Sp|_{\ISk^{\tr_\infty}}\colon\ISk^{\tr_\infty}\to\SP^{\tr_\infty}(M_\bbR,\bP)$ is closed, has finite fibers, and hence proper.
	The image $R\coloneqq\Sp(\ISk^{\tr_\infty})\subset\SP^{\tr_\infty}(M_\bbR,\bP)$ is locally compact.
\end{corollary}
\begin{proof}
	Since $\ISk$ is by definition the preimage of a closed subset, it is closed in $\cM(U^\an,\bP,\beta)$, in particular in $\cM^\sd(U^\an,\bP,\beta)$.
	Therefore, the closedness of $\Sp|_{\ISk^{\tr_\infty}}$ follows from \cref{prop:net_in_Msm}.
	Next, note that the restriction $\Phi_i|_{\ISk} \colon \ISk\to \oSk(\cM_{0,n}) \times \Sk(U)$ factors through $\Sp|_\ISk$, so the finiteness of $\Sp|_{\ISk^{\tr_\infty}}$ follows from \cref{lem:restriction_to_skeleton}(\ref{lem:restriction_to_skeleton:all}).
	Hence $\Sp|_{\ISk^{\tr_\infty}}$ is proper.
	Since $\ISk^{\tr_\infty}$ is locally compact, we deduce that its image by $\Sp$ is also locally compact.
\end{proof}

\begin{proof}[Proof of \cref{prop:net_in_Msm}]
	Since $\ocM(Y^\an,\bP,\beta)$ is proper, after passing to a subnet, we can assume that $(f_\lambda)$ converges to some $f_\infty\in\ocM(Y^\an,\bP,\beta)$.
	Then the stabilizations of the domain curves converge in $\ocM_{0,n}^\an$.
	Let $\oLambda\coloneqq\Lambda\sqcup\{\infty\}$.
	For $\lambda\in\oLambda$, let $[C_\lambda,(p_{j,\lambda})_{j\in J}, f_\lambda\colon C_\lambda\to Y^\an]$ denote the stable map $f_\lambda\in\ocM(Y^\an,\bP,\beta)$, and let $\sigma_\lambda\colon[C_\lambda,(p_{j,\lambda})_{j\in J}]\to[\oC_\lambda,(\op_{j,\lambda})_{j\in J}]$ be the stabilization of the domain curve.
	
	For each $i\in B$, up to shrinking $\oLambda$, the family $(\oC_\lambda)_{\lambda\in\oLambda}$ near the section $\op_{i,\lambda}$ is trivial.
	So we can pick a trivial family of small closed disks $\obbD_{i,\lambda}$ in $\oC_\lambda$ centered at $\op_{i,\lambda}$.
	We denote by $\obbD^\circ_{i,\lambda}$ the associated family of open disks over $\oLambda$.
	Let $\obbB_\lambda\coloneqq\oC\setminus\bigcup_{i\in B}\obbD_{i,\lambda}^\circ$, $\obbB^\circ_\lambda\coloneqq\oC\setminus\bigcup_{i\in B}\obbD_{i,\lambda}$, $\bbD_{i,\lambda}\coloneqq\sigma\inv_\lambda(\obbD_{i,\lambda})$, $\bbB_\lambda\coloneqq\sigma\inv_\lambda(\obbB_\lambda)$,
	and $\bbB_\lambda^\circ\coloneqq\sigma\inv_\lambda(\obbB^\circ_\lambda)$, see \cref{fig:caps}.
	Let $\tau_\rt\colon Y^\an_\rt\to\oM_\bbR$ be as in \cref{nota:toric_model}, and $\tau\colon (Y\setminus Y^\idt)^\an\to\oM_\bbR$ as in \cref{nota:Et}.
	\begin{figure}[!ht]
		\centering
		\setlength{\unitlength}{0.4\textwidth}
		\begin{picture} (1,1)
			\put(0,0){\includegraphics[width=\unitlength]{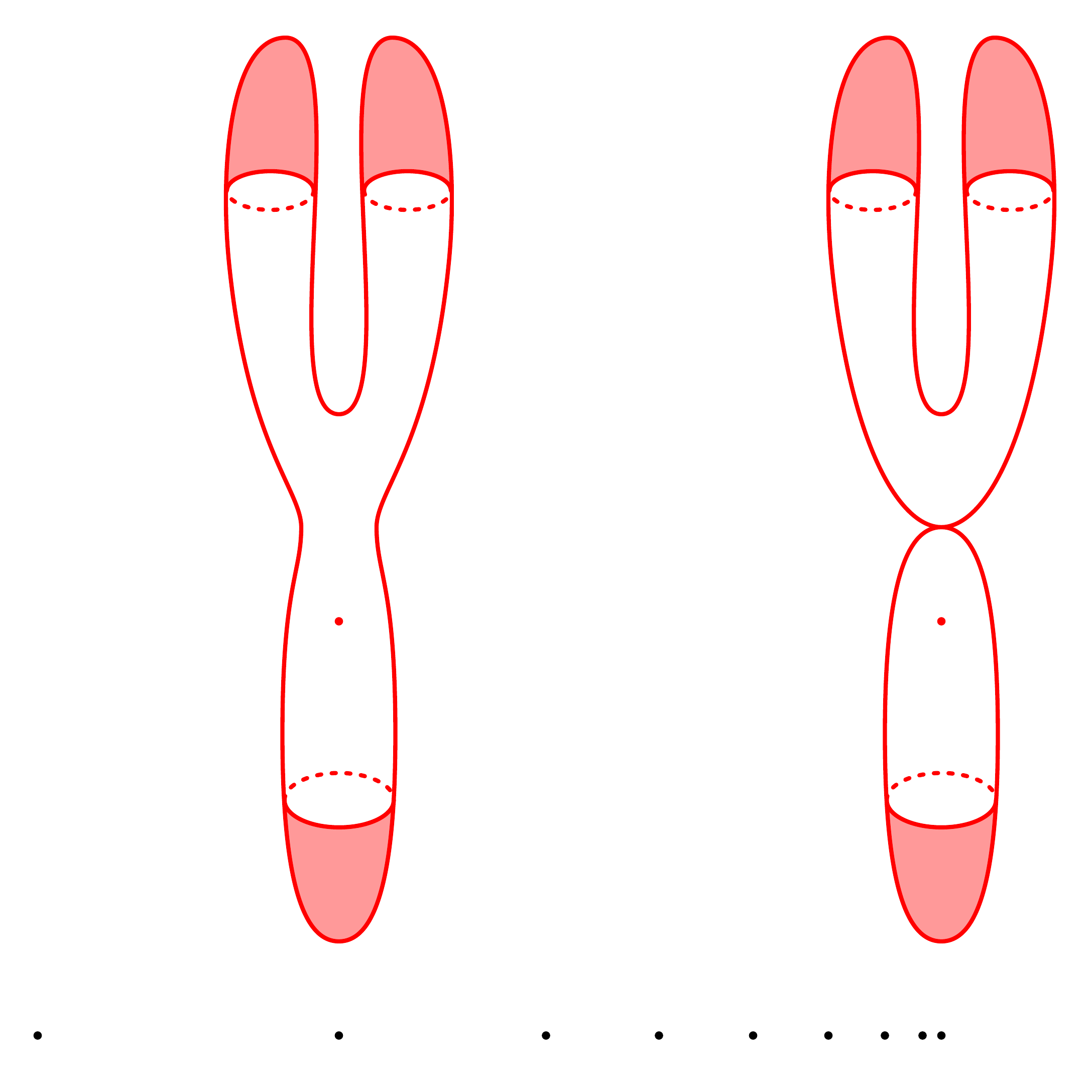}}
			\put(0.88,0.04){$\infty$}
			\put(0.33,0.04){$\lambda$}
			\put(0.18,0.5){$\bbB_\lambda$}
			\put(0.72,0.5){$\bbB_\infty$}
			\put(0.09,0.9){$\bbD_{i,\lambda}$}
			\put(0.63,0.9){$\bbD_{i,\infty}$}
		\end{picture}
		\caption{Cut the domain curves into bodies and caps.}
		\label{fig:caps}
	\end{figure}
	
	\begin{claim} \label{claim:boundary_of_cap}
		Up to shrinking $\oLambda$, there exists a compact subset $K\subset M_\bbR$ such that $f_\lambda(\partial\bbD_{i,\lambda})\in \tau_\rt\inv(K)\subset T_M^\an\subset U^\an$, for all $\lambda\in\oLambda$.
	\end{claim}
	\begin{proof}
		For every $\lambda\in\Lambda$, write $\Sp(f_\lambda)$ as $[\Gamma_\lambda,(p_{j,\lambda})_{j\in J},h_\lambda]$.
		Write the limit of $\Sp(f_\lambda)$ as $[\Gamma_\infty,\allowbreak (p_{j,\infty})_{j\in J},\allowbreak h_\infty]$, which is a priori different from $\Sp(f_\infty)$.
		For every $\lambda\in\oLambda$, let $\oC_\lambda^s$ denote the convex hull of the points $(\op_{j,\lambda})_{j\in J}$ in $\oC_\lambda$.
		For $\lambda\neq\infty$, since $[C_\lambda, (p_{j,\lambda})]$ is stable, $[\Gamma_\lambda,(p_{j,\lambda})_{j\in J}]$ is identified with $[\oC_\lambda^s,(\op_{j,\lambda})_{j\in J}]$.
		Then by the continuity of
		\[\ocM(Y^\an,\bP,\beta)\xrightarrow{\ \st\ }\ocM_{0,n}^\an\longrightarrow\oM_{0,n}^\trop,\]
		they are also identified for $\lambda=\infty$.
		
		Thus for every $\lambda\in\oLambda$, we have a retraction map $r\colon\oC_\lambda\to\Gamma_\lambda$.
		Let $b_{i,\lambda}\coloneqq r(\partial\obbD_{i,\lambda})\in\Gamma_\lambda$.
		By construction, $b_{i,\infty}\in\Gamma_\infty$ is a point close to $p_{i,\infty}$ but different from $p_{i,\infty}$, so $h_\infty(b_{i,\infty})\in M_\bbR$.
		Since $\Sp(f_\lambda)$ converges, up to shrinking $\oLambda$, there exists a compact subset $K\subset M_\bbR$ containing $h_\lambda(b_{i,\lambda})$ for all $\lambda\in\oLambda$.
		For all $\lambda\neq\infty$, by the definition of $\Sp$, we have $(\tau\circ f_\lambda)(\partial\bbD_{i,\lambda})=h_\lambda(b_{i,\lambda})$, hence $f_\lambda(\partial \bbD_{i,\lambda}) \in \tau\inv(K)=\tau_\rt^{-1}(K)$.
		Since $\tau_\rt$ is proper, $\tau_\rt^{-1}(K) \subset T_M^{\an}$ is compact.
		Now by the continuity of the universal stable map $f$, we obtain $f_\infty(\partial \bbD_{i,\infty}) \subset \tau_\rt^{-1}(K)$ as well.
	\end{proof}
	
	\begin{claim} \label{claim:off_caps}
		$f_\infty(\bbB_\infty)$ is disjoint from $D^\an\subset Y^\an$.
	\end{claim}
	\begin{proof}
		By \cref{claim:boundary_of_cap}, up to shrinking $\oLambda$, there exists a compact subset $K\subset M_\bbR$ such that $f_\infty(\partial\bbB_\infty)\in \tau_\rt\inv(K)\subset T_M^\an\subset U^\an$, for all $\lambda\in\oLambda$.
		Let $\alpha\colon U^\an\to\bbR^n_{\ge 0}$ be the proper continuous map in \cref{lem:enough_global_functions}.
		Then $\alpha(\tau_\rt\inv(K))\subset\bbR^n_{\ge0}$ is compact, so it is contained in $[0,N]^n\subset\bbR^n_{\ge 0}$ for some positive number $N$.
		By the maximum modulus principle, for any $\lambda\in\Lambda$, $(\alpha\circ f_\lambda)(\bbB_\lambda)$ is contained in $[0,N]^n$, so $f_\lambda(\bbB_\lambda)$ is contained in $\alpha\inv([0,N]^n)\subset U^\an$.
		The properness of $\alpha$ implies that $\alpha\inv([0,N]^n)$ is compact.
		So we deduce by the continuity of the universal stable map $f$ that $f_\infty(\bbB^\circ_\infty)$ is contained in $\alpha\inv([0,N]^n)\subset U^\an$.
				Now choose any $0<\epsilon'<\epsilon$, and denote $\bbB'_\lambda$ and $\bbB'^\circ_\lambda$ similarly.
		Note that $\bbB_\lambda$ is contained in $\bbB'^\circ_\lambda$.
		Applying the above argument to $\epsilon'$, we conclude that $f_\infty(\bbB_\infty)\subset f_\infty(\bbB'^\circ_\infty)$ is contained in $U^\an$, i.e.\ disjoint from $D^\an\subset Y^\an$.
			\end{proof}
	
	\begin{claim} \label{claim:cap}
		$\bbD^\circ_{i,\infty}$ has no bubbles.
	\end{claim}
	\begin{proof}
		Since $(\Sp(f_\lambda))$ converges in $\SP^{\tr_\infty}(M_\bbR,\bP)$ by assumption, $h_\infty(p_{i,\infty})\not\in E_\rt^\trop$.
		Up to shrinking $\epsilon$, we can assume that $h_\infty([b_{i,\infty},p_{i,\infty}]) \subset \oM_\bbR$ is disjoint from $\Wall_A \cup E_\rt^\trop$.
		Then up to shrinking $\oLambda$, we can pick a compact convex polyhedral subset $V \subset \oM_\bbR\setminus(\Wall_A\cup E_\rt^\trop)$ containing all $h_\lambda([b_{i,\lambda},p_{i,\lambda}])$.
		
		Since $V \cap E_{\rt}^{\trop} = \emptyset$, by \cref{nota:Et} we have $\tV \coloneqq\tau_\rt^{-1}(V) \subset W^\an$.
		Since $V \cap \Wall_A = \emptyset$, for all $\lambda\in\Lambda$, we have $(\tau \circ f)(\bbD_{i,\lambda}) = h_\lambda([b_{i,\lambda},p_{i,\lambda}]) \subset V$, and thus $f(\bbD_{i,\lambda})\subset\tV$.
		By the continuity of the universal stable map and the compactness of $\tV$, we deduce that $f(\bbD^\circ_{i,\infty})\subset\tV$.
		Up to shrinking $V$, we can assume $\tV$ to be affinoid, then $f_\infty\colon \bbD^\circ_{i,\infty} \to \tV$ contracts all possible bubbles.
		Since $p_{i,\lambda}$ is the only marked point in $\bbD^\circ_{i,\infty}$, the stability condition for $f_\infty$ implies that $\bbD^\circ_{i,\infty}$ cannot have any bubbles.
	\end{proof}
	
	Claims \ref{claim:off_caps} and \ref{claim:cap} imply that $f_\infty^{-1}(D^\an)$ is supported on a finite set.
	Since $(f_\lambda)_{\lambda\in\Lambda}$ lies in $\cM^\sm(U^\an,\bP,\beta)$, the latter is in particular nonempty, thus $\beta$ is compatible with $\bP$ by \cref{rem:compatible_curve_class}.
	As $\deg f_\infty^{-1}(D^\an)=\beta\cdot D$, we deduce that for every $i\in B$, $f_\infty(p_{i,\infty})\in (D_i^\circ)^\an$ and $f_\infty\inv(D^\an)=\sum_{i\in B} m_i p_{i,\infty}$; in other words, we have $f_\infty\in\cM(U^\an,\bP,\beta)$.
	It remains to show that the domain curve $[C_\infty,(p_{j,\infty})_{j\in J}]$ is stable.
	There are no unstable components attached to $\bbB_\infty$ by \cref{claim:off_caps} and the affineness of $U$, and no unstable components attached to any $\bbD^\circ_{i,\infty}$ by \cref{claim:cap}.
	This completes the proof.
\end{proof}

Consider the commutative diagram
\[\begin{tikzcd}
	\cM^\sm(U^\an,\bP,\beta) \rar{\Phi_i} \dar{\Sp} & \ocM_{0,n}^\an\times U^\an \dar \\
	\SP(M_\bbR,\bP) \rar{\Phi_i^\trop} & \oM_{0,n}^\trop\times\oM_\bbR
\end{tikzcd}\]
where $\Phi_i^\trop$ takes the domain nodal metric tree and evaluation at the $i$-th marked point.

\begin{proposition} \label{prop:deformation_invariance}
	The space of spines transverse at infinity $\SP^{\tr_\infty}(M_\bbR,\bP)$ has a base of open subsets $\{V\}$ such that the function
	\begin{equation} \label{eq:sum_of_counts}
	T\longmapsto \sum_{T'\in V,\ \Phi_i^\trop(T')=\Phi_i^\trop(T)} N(T',\beta)
	\end{equation}
	is constant on every $V$.
\end{proposition}

\begin{definition}
	We call a spine $S\in\SP(M_\bbR,\bP)$ \emph{rigid} if $\Phi_i^\trop$ is injective near $S$, i.e.\ if there exists an open neighborhood $V_S$ of $S$ such that $\Phi_i^\trop|_{V_S}$ is injective.
	Note that the locus of rigid spines is open.
	Transverse spines are rigid by \cref{prop:rigidity_spine}.
\end{definition}

\begin{theorem} \label{thm:deformation_invariance_rigid}
	The count $N_i(S,\beta)$ is locally constant on the locus of rigid spines of $\SP^{\tr_\infty}(M_\bbR,\bP)$.
	In particular, it is locally constant on the transverse locus $\SP^\tr(M_\bbR,\bP)$.
\end{theorem}
\begin{proof}
	This is a special case of \cref{prop:deformation_invariance}.
\end{proof}

\begin{corollary} \label{cor:vary_lengths}
	For a transverse spine $S=[\Gamma,(v_j)_{j\in J},h]$, the count $N_i(S,\beta)$ does not change when we vary the lengths of the edges of $\Gamma$ on which $h$ is constant.
\end{corollary}

\begin{proof}[Proof of \cref{prop:deformation_invariance}]
	Let $S\in\SP^{\tr_\infty}(M_\bbR,\bP)$ and $V_S^0$ any open neighborhood of $S$.
	We need to find an open neighborhood $V_S\subset V_S^0$ of $S$ satisfying the condition in \cref{prop:deformation_invariance}.
	By \cref{cor:properness_SP}, the subset of realizable spines $\RSP\subset\SP^{\tr_\infty}(M_\bbR,\bP)$ is closed, and locally compact.
	Therefore, if $S\not\in\RSP$, the sum in \eqref{eq:sum_of_counts} is constantly zero on any open neighborhood of $S$ disjoint from $\RSP$.
	
	Next we assume $S\in\RSP$.
	Consider the commutative diagram
	\[\begin{tikzcd}
		\ISk \rar{\Phi_i} \dar[swap]{\Sp} & \oSk\coloneqq\oSk(\cM_{0,n}) \times \Sk(U) \\
		\RSP\subset\SP(M_\bbR,\bP) \urar[swap]{\Phi_i^\trop}
	\end{tikzcd}\]
	where $\Phi_i^\trop$ takes the domain nodal metric tree and evaluation at the $i$-th marked point.
	By \cref{lem:restriction_to_skeleton}(\ref{lem:restriction_to_skeleton:all}), $\Phi_i$ is set-theoretically finite.
	So $\Phi_i^\trop|_\RSP\colon\RSP\to\oSk$ is also finite.
	We apply \cref{lem:topological_extension_point} to $\Phi_i^\trop|_\RSP$ and $S\in V_S^0\cap \RSP\subset \RSP$.
	We obtain $\Phi_i^\trop(S)\in W\subset\oSk$ such that the connected component $\tW$ of $\RSP_W$ containing $S$ is contained in $V_S^0\cap \RSP$, and moreover the restriction $\Phi_i^\trop\colon\tW\to W$ is proper with finite fibers.
	(We use $\RSP$ in the proof instead of just working with $\SP(M_\bbR,\bP)$ because we do not know the local compactness of the larger space and we want to apply \cref{lem:topological_extension_point}.)
	
	Now we choose an open neighborhood $V_S$ of $S$ in $V_S^0$ such that $V_S\cap \RSP=\tW$ and $\Phi_i^\trop(V_S)=W$.
	Let $R\coloneqq W$ and $\tR\coloneqq(\Sp|_\ISk)\inv(\tW)$.
	Since $\tW$ is a connected component of $R_W$, by the continuity of $\Sp$ (see \cref{prop:continuity}), $\tR$ is a union of connected components of $\Phi_i\inv(R)$.
	We conclude the proof using \cref{lem:finite_etale_degree}.
\end{proof}

\section{Toric tail condition in families} \label{sec:toric_tail_condition}

In this section, we study toric tail condition in families, and prove that it cuts out connected components in the moduli spaces of analytic stable maps under two different transversality assumptions, see Proposition \ref{prop:toric_tail_cc} and \ref{prop:toric_tail_cc_skeletal}.
Then we deduce a generalization of the deformation invariance of the previous section, including also truncated (i.e.\ non-extended) spines, see \cref{thm:deformation_invariance_truncated}.

We follow Notations \ref{nota:moduli_spaces} and \ref{nota:moduli_spaces_analytic}, fix $\bP,\beta$, and $A\in\bbN$ big with respect to $\beta$.
We assume $\abs{I}\ge 1$, and pick $s\in I$ and $e\in B$, which will mean respectively the \emph{start} and the \emph{end} of a tail.

\begin{notation} \label{nota:toric_tail}
	Let $\Theta\subset\NT_{B\cup\{s\}}$ be the subspace of nodal metric trees whose $s$-leg and $e$-leg are incident to a single 3-valent vertex.
	Let $\cM^\sm(U^\an,\bP,\beta)_\Theta$ be the preimage of $\Theta$ by the composite map
	\[\cM^\sm(U^\an,\bP,\beta)\xrightarrow{\dom}\ocM_{0,n}^\an\longto\NT_n\longto\NT_{B\cup\{s\}},\]
	where the last arrow forgets all the legs in $I\setminus s$.
	
	Given $[(C,(p_j)_{j\in J},f]\in\cM^\sm(U^\an,\bP,\beta)_\Theta$, let $\Trop(f)=[\Gamma,(p_j)_{j\in J},h]$ be as in \cref{prop:tropicalization_of_stable_map}, $\Gamma'\subset\Gamma$ the convex hull of $\set{p_j | j\in B}\cup p_s$, $\op_s\in\Gamma'$ the root of the $p_s$-leg of $\Gamma'$, and $[p_s,p_e]\subset\Gamma'$ the path connecting $p_s$ and $p_e$ (see \cref{fig:tail}).
	Let $C\xrightarrow{r}\Gamma\xrightarrow{r'}\Gamma'$ be the retraction maps, $\Delta\coloneqq r'^{-1}([p_s,p_e])\subset\Gamma$, $\bbT\coloneqq r\inv(\Delta)\subset C$, and $\bbT^*\coloneqq\bbT\setminus\{p_e\}$.
	The definition of $\Theta$ implies that $\bbT\subset C$ is a closed disk, and $\bbT^*$ does not contain any $p_j$ for $j\in B$.
	\begin{figure}[!ht]
		\centering
		\setlength{\unitlength}{0.3\textwidth}
		\begin{picture} (1,1)
			\put(0,0){\includegraphics[width=\unitlength]{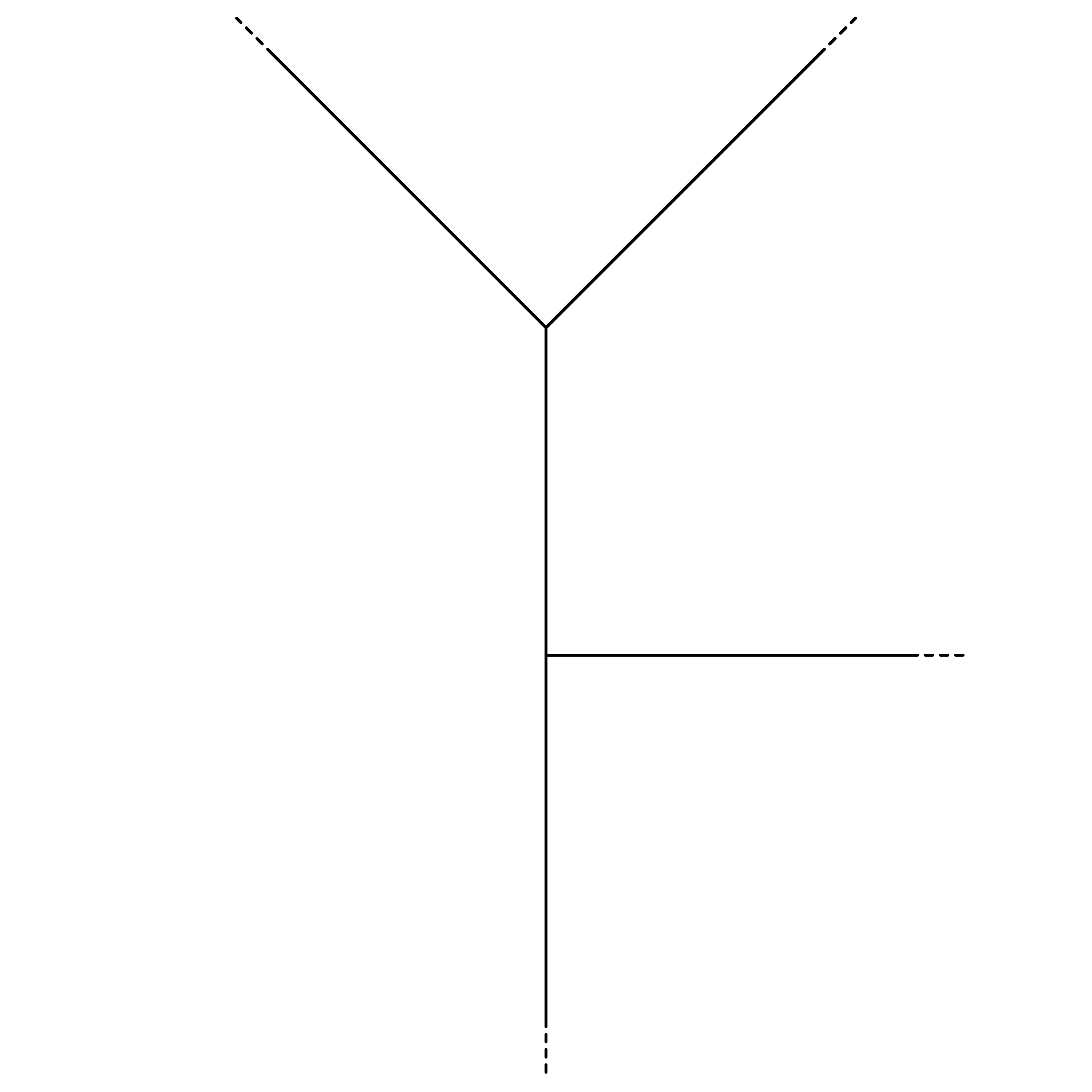}}
			\put(0.41,0.02){$p_e$}
			\put(0.41,0.38){$\op_s$}
			\put(0.85,0.34){$p_s$}
		\end{picture}
		\caption{The tree $\Gamma'$, which is a subtree of $\Gamma$.}
		\label{fig:tail}
	\end{figure}
\end{notation}

\begin{definition} \label{def:toric_tail}
	Let $\cM_\tail^\sm(U^\an,\bP,\beta)_\Theta \subset \cM^{\sm}(U^\an,\bP,\beta)_\Theta$ be the subspace of stable maps such that $f(\bbT^*) \subset T_M^{\an}$.
\end{definition}

\begin{lemma} \label{lem:toric_tail_equiv}
	For $f \in \cM^{\sm}(U^\an,\bP,\beta)_\Theta$, the following are equivalent:
\begin{enumerate}
	\item \label{lem:toric_tail_equiv:T_M} $f(\bbT^*) \subset T_M^{\an}$.
	\item \label{lem:toric_tail_equiv:W} $f(\bbT) \subset W^{\an}$, ($W$ as in \cref{lem:toric_model}).
	\item \label{lem:toric_tail_equiv:E} $f(\bbT) \cap E^{\an} = \emptyset$.
	\item \label{lem:toric_tail_equiv:twig} $h|_\Delta$ factors through the retraction $\Delta\to[\op_s,p_e]$, i.e.\ there are no twigs (in the sense of \cref{def:twig}) attached to $[\op_s,p_e]$.
\end{enumerate}
\end{lemma}
\begin{proof}
	By Notations \ref{nota:toric_tail} and \ref{nota:Msm}, we always have $f(\bbT^*)\subset U^\an$ and $f(p_e)\in(D^\ess\cap W)^\an$.
	Hence the equivalences between (\ref{lem:toric_tail_equiv:T_M}-\ref{lem:toric_tail_equiv:E}) follow from \cref{lem:toric_model}(\ref{lem:toric_model:intersection}).
	The equivalence between (\ref{lem:toric_tail_equiv:T_M}) and (\ref{lem:toric_tail_equiv:twig}) follows from the balancing condition.
\end{proof}

\begin{proposition} \label{prop:toric_tail_cc}
	Let $\cN\subset \cM^{\sm}(U^\an,\bP,\beta)_\Theta$ be any subspace.
	Then \[\cN\cap\cM^\sm_\tail(U^\an,\bP,\beta)_\Theta\subset\cN\]
	is open.
	If furthermore $h(p_s)\notin\Wall_A$ for all $f\in\cN$, then the inclusion above is a union of connected components.
\end{proposition}
\begin{proof}
	Since \cref{lem:toric_tail_equiv}(\ref{lem:toric_tail_equiv:W}) is an open condition, we deduce the openness of the inclusion.
	
	Next we prove the closedness under the assumption that $h(p_s)\notin\Wall_A$ for all $f\in\cN$.
	Let $f\in\cN$ be the limit of a net $(f_\lambda)_{\lambda\in\Lambda}$ in $\cN\cap\cM^\sm_\tail(U^\an,\bP,\beta)_\Theta$.
	In order to show $f\in\cM^\sm_\tail(U^\an,\bP,\beta)_\Theta$, it suffices to prove that $f$ satisfies \cref{lem:toric_tail_equiv}(\ref{lem:toric_tail_equiv:twig}).
	Suppose the contrary, then there is a twig attached to $[\op_s,p_e]$.
	By the continuity of tropicalization of curves (see \cref{rem:comparison_ACP}), the twig is necessarily attached at the point $\op_s$.
	So $h(\op_s)\in\Wall_A$.
	
	For any $\lambda\in\Lambda$, we denote $\Trop(f_\lambda)=[\Gamma_\lambda,(p_{j,\lambda})_{j\in J},h_\lambda]$ and $\op_{s,\lambda}$ as in \cref{nota:toric_tail}.
	By \cref{lem:toric_tail_equiv}(\ref{lem:toric_tail_equiv:twig}), $h_\lambda(\op_{s,\lambda})=h_\lambda(p_{s,\lambda})$.
	Hence $h(\op_s)=h(p_s)$ by continuity.
		Recall $h(\op_s)\in\Wall_A$ shown above, we deduce that $h(p_s)\in\Wall_A$.
	This contradicts the assumption that $h(p_s)\notin\Wall_A$, completing the proof.
\end{proof}

For skeletal curves we have a stronger connected component
statement.
Fix $i\in I$, and let $\ISk$ be the preimage of $\oSk(\cM_{0,n}) \times \Sk(U)$ by the map
\[\Phi_i=(\dom,\ev_i)\colon\cM^\sm(U^\an,\bP,\beta)\to\ocM_{0,n}^\an\times U^\an.\]

\begin{proposition} \label{prop:toric_tail_cc_skeletal}
	Let $\cN\subset\cM^{\sm}(U^\an,\bP,\beta)_\Theta\cap\ISk$ be an open subspace such that for all $f\in\cN$, if $h(p_s)\in\fd$ for some polyhedral cell $\fd\subset\Wall_A$, then the linear span of $\fd$ contains $P_e$.
		In this case, $\cN\cap\cM^{\sm}_\tail(U^\an,\bP,\beta)\subset \cN$ is a union of connected components.
\end{proposition}
\begin{proof}
	The openness follows from \cref{prop:toric_tail_cc}.
	For the proof of closedness, let $f\in\cN$ be a point in the closure.
	By \cref{lem:restriction_to_skeleton}(\ref{lem:restriction_to_skeleton:all}), the restriction of $\Phi_i$ to $\cN$ is open, so the restriction of $\Phi_i$ to $\cN\cap\cM^\sm_\tail(U^\an,\bP,\beta)$ is also open.
	Thus by \cref{prop:transversality}, we can find a net $(f_\lambda)_{\lambda\in\Lambda}$ in $\cN\cap\cM^{\sm}_\tail(U^\an,\bP,\beta)$ converging to $f$ such that $\Trop(f_\lambda)\in\TC^\tr(M_\bbR,\bP)$ for all $\lambda$.
	We denote $\Trop(f_\lambda)=[\Gamma_\lambda,(p_{j,\lambda})_{j\in J},h_\lambda]$, $\Delta_\lambda$ and $\op_{s,\lambda}$ as in \cref{nota:toric_tail}.
	Note $\Trop(f_\lambda)\in\TC^\tr(M_\bbR,\bP)$ implies that $h_\lambda(\op_{s,\lambda})\notin\Wall_A$.	
	
	Let $[\hGamma_\lambda,(p_{j,\lambda})_{j\in B\cup\{s\}},\hh_\lambda]$ be the tropicalization of $f_\lambda$ after forgetting the marked points in $I\setminus s$.
	Let $\sigma_\lambda$ be the topological edge of $\hGamma$ containing $\op_{s,\lambda}$ which does not intersect $[p_{s,\lambda},p_{e,\lambda}]\setminus\{\op_{s,\lambda}\}$ (see \cref{fig:toric_tail}).
	The uniqueness of such a topological edge follows from the definition of $\Theta$ in \cref{nota:toric_tail}, and the fact that $\hh_\lambda(\op_{s,\lambda})=h_\lambda(\op_{s,\lambda})\notin\Wall_A$.
	\begin{figure}[!ht]
	\centering
	\setlength{\unitlength}{0.3\textwidth}
	\begin{picture} (1,1)
	\put(0,0){\includegraphics[width=\unitlength]{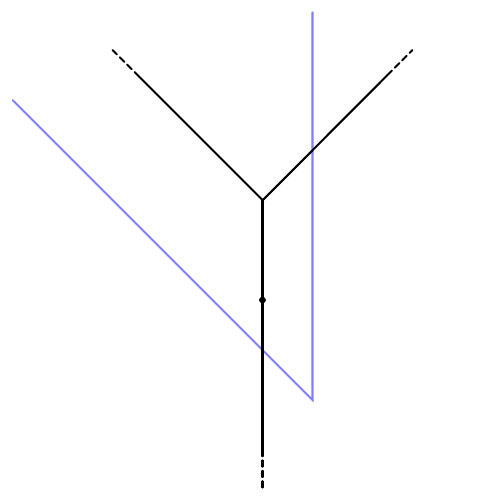}}
	\put(0.54,0.02){$p_{e,\lambda}$}
	\put(0.54,0.38){$\op_{s,\lambda}$}
	\put(0.43,0.50){$\sigma_\lambda$}
	\put(0.64,0.50){$\mathfrak{d}$}
	\end{picture}
	\caption{A example of the image of $\Trop(f_\lambda)$ in $M_\bbR$.}
	\label{fig:toric_tail}
	\end{figure}
	\begin{claim} \label{claim:length_of_sigma}
		There exist $\epsilon > 0$ and $\lambda_0\in\Lambda$ such that for all $\lambda>\lambda_0$, $h_\lambda\inv(\Wall_A)$ has distance (in the metric of $\Gamma_\lambda$) at least $\epsilon$ from $\op_{s,\lambda}$;
		in other words, any twig of $h_\lambda$ has distance at least $\epsilon$ from $\op_{s,\lambda}$.
	\end{claim}
	\begin{proof}
		For any $\lambda\in\Lambda$, since $f_\lambda\in\cN\cap\cM^\sm_\tail(U^\an,\bP,\beta)$, by \cref{lem:toric_tail_equiv}(\ref{lem:toric_tail_equiv:twig}), $h_\lambda|_{\Delta_\lambda}$ factors through the retraction $\Delta_\lambda\to[\op_{s,\lambda},p_{e,\lambda}]$, and $h_\lambda$ is affine on $[\op_{s,\lambda},p_{e,\lambda}]$ with derivative $P_e$.
		Then by the balancing condition, $h_\lambda$ is affine on $\sigma_\lambda\cup[\op_{s,\lambda},p_{e,\lambda}]$ with derivative $P_e$.
		
		Recall our assumption on $\cN$ that for any polyhedral cell $\fd\subset\Wall_A$ containing $h(p_s)$, the linear span of $\fd$ contains $P_e$.
		As $h_\lambda(\op_{s,\lambda})\notin\Wall_A$ and $h_\lambda$ has transverse spine, we deduce that $h_\lambda\big(\sigma_\lambda\cup[\op_{s,\lambda},p_{e,\lambda}]\big)$ does not meet any such $\fd$.
		Hence for sufficiently small $\epsilon >0$, the distance in question can be checked using only cells of $\Wall_A$ that do not contain $h(p_s)$.
		So the claim follows from the continuity of tropicalization of curves (see \cref{rem:comparison_ACP}).
\end{proof}

It follows from \cref{claim:length_of_sigma} and the continuity that any twig of $h$ has distance at least $\epsilon$ from $\op_s\in\Gamma$.
Suppose to the contrary that $f\notin\cM^\sm_\tail(U^\an,\bP,\beta)$, then there is a twig (in the sense of \cref{def:twig}) attached to $[\op_s,p_e]$.
By the continuity, the twig is necessarily attached at the point $\op_s$, contradicting the statement about the distance, completing the proof.
\end{proof} 

Our next goal is \cref{thm:deformation_invariance_truncated}.
We follow the setting of \cref{const:naive_counts_truncated}.
We begin by adding more internal legs to $\hS$.

\begin{construction} \label{const:add_legs_to_tails}
	For each $j\in F$, we glue a copy of $[0, s_j\coloneqq+\infty]$ to $\hGamma$, along $0$ and $v_j$.
	We extend $\hh$ constantly along the new legs.	
	Let $\oS=[\oGamma,(\ov_j)_{j\in\oJ},\oh]$ denote the resulting spine, where $\oJ\coloneqq J\sqcup F=J\sqcup (j')_{j\in F}$, $\ov_j\coloneqq\hv_j$ for $j\in J$, and $\ov_{j'}\coloneqq s_j$ for $j\in F$.
	Then we have the associated $\oF, \oB, \oI$ and $\obP$ as in \cref{def:spine_in_U}.
	We have two injective maps
	\begin{align*}
		s\colon F \to \oI, \quad j\mapsto j'\\
		e\colon F\to \oB, \quad j\mapsto j
	\end{align*}
	Note that for each $j\in F$, the $s(j)$-leg and the $e(j)$-leg of $\oGamma$ are incident to a single 3-valent vertex.
\end{construction}

\discussion{Hi Sean, you formulated the notation of \emph{tail data}, but it was never used later in the paper:
	By {\it Tail data} we mean finite sets $F,B,I$ and injective maps $e: F \to B\cup I$, $s:F \to I$.
	We consider $(S,B \sqcup I) \in \oM_{0,B\sqcup I}$ (where the notation means the valence one vertices are labeled by $B \sqcup I$ and we use the same symbols for the corresponding points on the tree $S$) satisfies the associated {\it tail condition} if for each $t \in F$ the simple path $[s(t),e(t)] \subset S$ contains a unique vertex, $v_t$, of valence greater than two, and the valence of $v_t$ is three, and moreover if these paths are pairwise disjoint.
	Observe: Now note that an extended spine
  with $\oP,B,\oI$ as above whose domain satisfies the tail condition
  $s:F \subset \oI$, $e: F \subset B \sqcup \oI$ above, and whose
  restriction to each $[v(t),e(t)]$ is affine, and each $[v(t),s(t)]$
  is trivial, is exactly the same as a $\bP,B,I,F$ truncated spine, by
  the above construction.}

\begin{construction} \label{const:tail_condition}
	Let $\Theta\subset\NT_{\oJ}$ be the subspace of nodal metric trees whose $s(j)$-leg and $e(j)$-leg are incident to a single 3-valent vertex for each $j\in F$.
	Given $[C, (p_j)_{j\in\oJ},f]\in\cM^\sm(U^\an,\obP,\beta)_\Theta$, let $\Gamma^s$ be the convex hull of all the marked points and $r\colon C\to\Gamma^s$ the retraction map.
	For each $j\in F$, let $[p_{s(j)},p_{e(j)}]\subset\Gamma^s$ be the path connecting $p_{s(j)}$ and $p_{e(j)}$; let $\bbT_j\coloneqq r\inv([p_{s(j)},p_{e(j)}])$ and $\bbT^*_j\coloneqq\bbT_j\setminus\{p_{e(j)}\}$.
	Let $\cM^\sm_\tail(U^\an,\obP,\beta)_\Theta\subset\cM^\sm(U^\an,\obP,\beta)_\Theta$ be the subspace of stable maps such that $f(\bbT^*_j)\subset T_M^\an$ for every $j\in F$.
\end{construction}

\begin{theorem} \label{thm:deformation_invariance_truncated}
	The count $N(S,\gamma)$ in \cref{def:naive_count_transverse} is deformation invariant among transverse spines.
\end{theorem}
\begin{proof}
	Let $\gamma,\hgamma\in\NE(Y)$ be as in \cref{const:naive_counts_truncated}.
	Fix $A\in\bbN$ big with respect to $\gamma$, and assume $S$ is transverse with respect to $\Wall_A$.
	We apply Constructions \ref{const:add_legs_to_tails} and \ref{const:tail_condition}.
	Then the spine $\oS$ is also transverse with respect to $\Wall_A$.
	Let $W$, $\tW$ and $\Phi_i\colon\tR\to R$ be as in the proof of \cref{prop:deformation_invariance} applied to $\oS$ and $\hgamma$.
	Up to shrinking $R$, we can assume that it is contained in the preimage of $\Theta$, and that all the spines in $\tW$ are transverse by \cref{prop:transversality}.
	Let $\tR^\tail\coloneqq\tR\cap\cM^\sm_\tail(U^\an,\obP,\beta)_\Theta$.
	Then by \cref{prop:toric_tail_cc}, $\tR^\tail\subset\tR$ is a union of connected components.
	By \cref{lem:finite_etale_degree}, the degree of $\tR^\tail\xrightarrow{\ \Phi_i\ } R$ over any point $r\in R$ is independent of the choice of $r$.
	Now the theorem follows from the definition of the count $N(S,\gamma)$ (see \cref{def:naive_count_transverse} and \cref{prop:moving_w}).
\end{proof}

Below is an application of \cref{prop:toric_tail_cc_skeletal} that we will use in the proof of \cref{thm:wall-crossing_homomorphism}.

\begin{proposition} \label{prop:deformation_invariance_skeletal} 
	Notations as in \cref{def:spine}.
	Fix $i\in I$, $\gamma\in\NE(Y)$ and $A\in\bbN$ big with respect to $\gamma$.
	Consider
	\[\Psi_i\coloneqq(\dom,\ev_i)\colon\SP(M_\bbR,\bP^F) \longrightarrow \NT_J^F\times M_\bbR.\]
	For each $j\in F$, we fix an open cell $\sigma_j^\circ\subset\Sigma_\rt$.
	Let $\NT_J^{F,\irr}\subset\NT_J^F$ be the subset consisting of irreducible trees (i.e.\ without nodes).
	Let $Q\subset\NT_J^{F,\irr}\times M_\bbR$ be a connected open subset such that the following hold for every spine $S=[\Gamma, (v_j)_{j\in J}, h] \in \SP(M_{\bbR},\bP^F)_Q$
	\begin{enumerate}
	\item For each $j\in F$, the image $h(v_j)$ lies in the open star of $\sigma_j^\circ$, and the linear span of $\sigma_j$ contains the vector $P_j$.
	(This ensures that the curve class $\hgamma\coloneqq\gamma+\sum_{j\in F}\delta_j$ in \cref{const:naive_counts_truncated} is independent of $S$ by \cref{rem:curve_class_parallel_perturbation}.)
	\item For each $j \in F$, if $h(v_j)\in \fd$ for some polyhedral cell $\fd\subset\Wall_A$ then the linear span of $\fd$ in $M_\bbR$ contains the vector $P_j$.
	\end{enumerate} 
	Then 
	\[\sum_{S \in \SP(M_{\bbR},\bP^F)_q} N(S,\gamma)\]
	is independent of $q\in Q$.
\end{proposition}
\begin{proof}
	\cref{const:add_legs_to_tails} gives a commutative diagram
	\[\begin{tikzcd}[column sep=huge]
	\SP(M_\bbR,\bP^F) \rar{\Psi_i=(\dom,\ev_i)} \dar & \NT_J^F\times M_\bbR \dar{\wr} \\
	\SP(M_\bbR,\obP) \rar{\oPsi_i=(\dom,\ev_i)} & \NT_{\oJ}\times M_\bbR.
	\end{tikzcd}\]
	Let $\tQ$ be the image of $Q$ under
	\[\NT_J^F\times M_\bbR\xrightarrow{\sim}\NT_{\oJ}\times M_\bbR\xrightarrow{\sim}\oSk(\cM_{0,\oJ})\times\Sk(U)\hookrightarrow\ocM_{0,\oJ}^\an\times U^\an.\]
	Consider
	\[\Phi_i=(\dom,\ev_i)\colon\cM^\sm(U^\an,\bP,\hgamma)\longrightarrow\ocM_{0,\oJ}^\an\times U^\an.\]
	Let $\cM^\sm_\tail(U^\an,\bP,\hgamma)_{\tQ} \subset\cM^\sm(U^\an,\bP,\hgamma)_{\tQ}$ be as in \cref{const:tail_condition}.
	By \cref{prop:toric_tail_cc_skeletal}, this is a union of connected components.
	By \cref{const:naive_counts_truncated}, the degree of
	\[\cM^\sm_\tail(U^\an,\bP,\hgamma)_q\xrightarrow{\ \Phi_i\ }q\] is equal to
	\[\sum_{S \in \SP(M_{\bbR},\bP^F)_q} N(S,\gamma).\]
	Therefore, we can conclude the proof by \cref{lem:finite_etale_degree}.
\end{proof} 

\section{Gluing formula} \label{sec:gluing}

In this section, we prove a gluing formula for gluing two transverse spines inside the edges (\cref{thm:gluing_inside}), and a gluing formula for concatenating two transverse spines (\cref{thm:gluing_concatenate}).
The latter is an analog of \cite[Theorem 1.2]{Yu_Enumeration_of_holomorphic_cylinders_II} in the present context.
The proofs here use the deformation invariance of \cref{sec:deformation_invariance}, which is directly inspired by the approach in \cite{Yu_Enumeration_of_holomorphic_cylinders_II}.

We begin with a gluing lemma without any transversality assumptions.
It will be used in the proof of \cref{thm:gluing_inside}.

\begin{lemma}[see \cref{fig:gluing1}] \label{lem:gluing_decomposition}
Let $S^i=[\Gamma^i,(v^i_j)_{j\in J^i},h^i]$, $i=1,2$ be two spines in $\Sk(U)$, as in \cref{def:spine_in_U}.
Suppose we have infinite 1-valent vertices $w^i \in \Gamma^i$ with $h^1(w^1) = h^2(w^2) \in \Sk(U)$.
Let $[\Delta,(w^1,w^2,w)]\in\oM_{0,3}^\trop$ be the unique element, i.e.\ the metric tree
\[
[0,w^1=+\infty] \sqcup_0 [0,w^2=+\infty] \sqcup_0 [0,w=+\infty].
\]
Let $\Gamma$ be obtained by gluing $\Delta$ to $\Gamma_1 \sqcup \Gamma_2$ at the $w^i$ points, and let $h\colon\Gamma\to\oSk(U)$ be obtained from $h^1 \sqcup h^2$ by extending constantly over $\Delta$.
Let
\[S\coloneqq\big[\Gamma,\big((v^1_j)_{j\in J^1}\setminus\{w^1\}\big)\sqcup\big((v^2_j)_{j\in J^2}\setminus\{w^2\}\big)\sqcup\{w\},h\big].\]
For any curve class $\gamma\in\NE(Y)$, we have
\[N_w(S,\gamma) = \sum_{\gamma^1 + \gamma^2 = \gamma} N_{w^1}(S^1,\gamma^1)\cdot N_{w^2}(S^2,\gamma^2).\]
\end{lemma}
\begin{figure}[!ht]
	\centering
	\setlength{\unitlength}{0.8\textwidth}
	\begin{picture} (1,0.3)
		\put(0,0){\includegraphics[width=\unitlength,page=1]{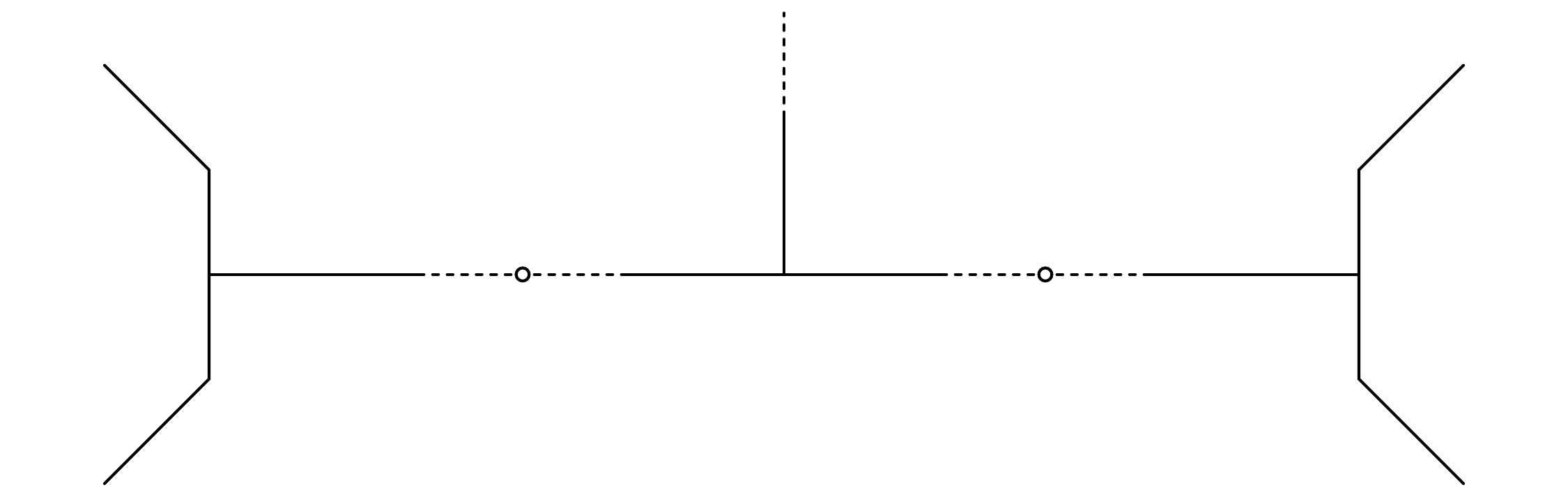}}
		\put(0.49,0.1){$\Delta$}
		\put(0.14,0.06){$\Gamma^1$}
		\put(0.83,0.06){$\Gamma^2$}
	\end{picture}
	\caption{The domain $\Gamma$ of the glued spine $S$, where the tiny circles denote the nodes.}
	\label{fig:gluing1}
\end{figure}
\begin{proof}
	Let $F_w(S,\gamma)$ and $F_{w^i}(S^i,\gamma^i)$ be as in \cref{const:naive_counts_truncated}.
	We make a sufficiently big base field extension $k\subset k'$ so that the spaces $F_w(S,\gamma)$ and $F_{w^i}(S^i,\gamma^i)$ split into disjoint unions of $k'$-rational points.
	For any $[f\colon C\to Y^\an]\in F_w(S,\gamma)$, let $C_\Delta$ be the irreducible component of $C$ corresponding to $\Delta$.
	Since $h(\Delta)\in\Sk(U)$, we have $f(C_\Delta)\subset U^\an$.
	Since $U$ is affine, $f$ is necessarily constant on $C_\Delta$.
	Therefore, we obtain a set-theoretic decomposition
	\[F_w(S,\gamma) = \bigsqcup_{\gamma^1+\gamma^2=\gamma} F_{w^1}(S^1,\gamma^1) \times F_{w^2}(S^2,\gamma^2).\]
	The lemma follows by taking cardinality of the sets.
\end{proof}

\begin{theorem}[see \cref{fig:gluing2}] \label{thm:gluing_inside}
	Let $A\in\bbN$ and $\gamma\in\NE(Y)$.
	Let $S^i=[\Gamma^i,(v^i_j)_{j\in J^i},h^i]$, $i=1,2$ be two spines in $M_\bbR$ transverse respect to $\Wall_A$.
	Let $p^i\in\Gamma^i$ be points in the interiors of edges such that $h^1(p^1)=h^2(p^2)\in M_\bbR\setminus\Wall_A$.
	Let $S$ be the spine obtained by gluing $S^i$ at $p^i$.
	Assume that for any decomposition $\gamma=\gamma^1+\gamma^2$, $\gamma^i\in\NE(Y)$, $A$ is big with respect to both $\gamma^i$.
	Then we have
	\[
	N(S,\gamma)=\sum_{\gamma^1+\gamma^2=\gamma} N(S^1,\gamma^1)\cdot N(S^2,\gamma^2).
	\]
\end{theorem}
\begin{figure}[!ht]
	\centering
	\setlength{\unitlength}{0.8\textwidth}
	\begin{picture} (1,0.67)
		\put(0,0){\includegraphics[width=\unitlength,page=2]{images/gluing_formula}}
	\end{picture}
	\caption{Respectively: $\Gamma^1$, $\Gamma^2$, $\Gamma$, $\Gamma'$ and $\overline\Gamma$.}
	\label{fig:gluing2}
\end{figure}
\begin{proof}
	Let $\Gamma$ be the domain of $S$, and let $\Gamma'$ be the gluing of $\Gamma$ and $[0,w=+\infty]$ along $p^1=p^2$ and $0$.
	We extend $h$ constantly over the new leg and let $S'$ be the resulting spine.
	For $i=1,2$, let $\oGamma^i$ be the gluing of $\Gamma^i$ and $[0,w^i=+\infty]$ along $p^i$ and $0$.
	Let $\Delta$ be as in \cref{lem:gluing_decomposition}.
	Let $\oGamma$ be the gluing of $\Delta$ and $\oGamma^1\sqcup\oGamma^2$ along $w^i$.
	Let $h\colon\oGamma\to\oM_\bbR$ be obtained from $h^1\sqcup h^2$ by extending constantly over the new parts.
	Let $\oS$ be the resulting spine.
	Note that we can deform $\oS$ into $S'$ by shrinking the two new topological edges with nodes, and all the spines during deformation are transverse with respect to $\Wall_A$.
	Hence by \cref{cor:vary_lengths}, we have $N_w(S',\gamma)=N_w(\oS,\gamma)$.
	By \cref{lem:gluing_decomposition}, we have
	\[N_w(\oS,\gamma) = \sum_{\gamma^1 + \gamma^2 = \gamma} N_{w^1}(\oS^1,\gamma^1)\cdot N_{w^2}(\oS^2,\gamma^2).\]
	By \cref{def:naive_count_transverse}, we have $N(S,\gamma)=N_w(S',\gamma)$ and  $N(S^i,\gamma^i)=N_{w^i}(\oS^i,\gamma^i)$ for $i=1,2$.
	Combining all the equalities above, we achieve the proof.
\end{proof}

\begin{lemma} \label{lem:count_balanced_spines}
	Let $\gamma\in\NE(Y)$ and $A\in\bbN$ big with respect to $\gamma$.
	Let $\Sigma_\rt^{d-1}\subset\Sigma_\rt$ be the codimension-one skeleton of the fan $\Sigma_\rt$.
	Let $S=[\Gamma,(v_j)_{j\in J},h]$ be a spine in $M_\bbR$ without bending vertices, and whose image is disjoint from $\Wall_A\cup\Sigma_\rt^{d-1}$.
	Then 
	\[N(S,\gamma)=\begin{cases}
	1 &\text{ if }\gamma=0,\\
	0 &\text{ otherwise.}
	\end{cases}\]
\end{lemma}
\begin{proof}
	Let $B, I, F$ be as in \cref{def:spine}.
	By \cref{thm:forgetting_interior_marked_points}, we can assume $\abs{I}\ge 1$; fix any $i\in I$, and we have $N(S,\gamma)=N_i(S,\gamma)$.
	Let $\hS=[\hGamma,(\hv_j)_{j\in J},\hh]$ and $F_i(S,\gamma)$ be as in \cref{const:naive_counts_truncated}.
	Then for any $f \in F_i(S,\gamma)$, we have $\Sp(f) = \hS$.
	Since $h(\Gamma)$ is disjoint from $\Wall_A$, there are no twigs of $\Trop(f)$ along $\Gamma$;
	moreover, by the toric tail condition and \cref{lem:toric_tail_equiv}(\ref{lem:toric_tail_equiv:twig}), there are no twigs along $\hGamma\setminus\Gamma$.
	So $\Trop(f)=\Sp(f)=\hS$, which implies that $f$ has image in
	$W^\an\subset Y^\an$ (as in \cref{lem:toric_model}).
	Since $W^\an$ is also contained in $Y_\rt^\an$, we deduce that the space $F_i(S,\gamma)$ for $(Y,D)$ is isomorphic to the space $F_i(S,\pi_*\gamma)$ for $(Y_\rt, D_\rt)$.
	Now we conclude from \cref{lem:counts_in_toric_case}.
\end{proof}

\begin{theorem}[see \cref{fig:gluing3}] \label{thm:gluing_concatenate}
	Let $A\in\bbN$ and $\gamma\in\NE(Y)$.
	Let $S^i=[\Gamma^i,(v^i_j)_{j\in J^i},h^i]$, $i=1,2$ be two spines in $M_\bbR$ transverse respect to $\Wall_A$.
	Let $p^i\in\Gamma^i$ be a finite 1-valent vertex, and $e^i$ the edge incident to $p^i$.
	Assume $h^1(p^1)=h^2(p^2)$ and $w_{(p^1,e^1)}+w_{(p^2,e^2)}=0$.
	So we can concatenate $S^1$ and $S^2$ at the vertices $p^1$ and $p^2$, and form a transverse spine which we denote by $S$.
	Moreover, assume that for any decomposition $\gamma=\gamma^1+\gamma^2$, $\gamma^i\in\NE(Y)$, $A$ is big respect to both $\gamma^i$.
	Then we have
	\[
	N(S,\gamma)=\sum_{\gamma^1+\gamma^2=\gamma} N(S^1,\gamma^1)\cdot N(S^2,\gamma^2).
	\]
\end{theorem}
\begin{figure}[!ht]
	\centering
	\setlength{\unitlength}{0.8\textwidth}
	\begin{picture} (1,0.17)
		\put(0,0){\includegraphics[width=\unitlength,page=3]{images/gluing_formula}}
		\put(0.13,0.06){$S^1$}
		\put(0.44,0.06){$S^2$}
		\put(0.22,0.03){$p^1$}
		\put(0.285,0.03){$p^2$}
		\put(0.76,0.06){$S$}
	\end{picture}
	\caption{The concatenation of $S^1$ and $S^2$ at the vertices $p^1$ and $p^2$.}
	\label{fig:gluing3}
\end{figure}
\begin{proof}
	We first make a small straight extension of $S^i$ at $p^i$ to $\hS^i$.
	This does not change the counts by \cref{thm:deformation_invariance_truncated}.
	Let $\hS$ be the gluing of $\hS^i$ at $p^i$.
	By \cref{thm:gluing_inside}, we have
	\begin{equation} \label{eq:gluing1}
	N(\hS,\gamma)=\sum_{\gamma^1+\gamma^2=\gamma} N(\hS^1,\gamma^1)\cdot N(\hS^2,\gamma^2)=\sum_{\gamma^1+\gamma^2=\gamma} N(S^1,\gamma^1)\cdot N(S^2,\gamma^2).
	\end{equation}
	Note that $\hS$ can also be viewed as the gluing of $S$ with a small straight segment $L$.
	Hence by \cref{thm:gluing_inside} again, we have
	\begin{equation} \label{eq:gluing2}
	N(\hS,\gamma)=\sum_{\beta^1+\beta^2=\gamma} N(\hS,\beta^1)\cdot N(L,\beta^2).
	\end{equation}
	By \cref{lem:count_balanced_spines},
	\[N(L,\beta^2)=\begin{cases}
	1 &\text{ if }\beta^2=0,\\
	0 &\text{ otherwise.}
	\end{cases}\]
	Hence equation \eqref{eq:gluing2} implies
	\[N(\hS,\gamma)=N(S,\gamma).\]
	Combining with equation \eqref{eq:gluing1}, we achieve the proof.
\end{proof}

\section{Tail condition with varying torus} \label{sec:varying_torus}

In this section, we will vary the embedding of the algebraic torus $T_M\subset U$, and prove that the counts of transverse spines are independent of the choice of embedding, see \cref{thm:count_independent}.
This implies in particular that our mirror algebra $A$ is independent of the choice of torus, see \cref{prop:structure_constants_varying_tori}.

Let $S=[\Gamma,(v_j)_{j\in J},h]$ be a spine in $\Sk(U)$ as in \cref{def:spine_in_U}, and $\gamma\in\NE(Y)$.
Assume $\abs{I}\ge 1$ and fix $i\in I$.
Suppose for each $j\in F$, we are given a Zariski open split algebraic torus $T_j\subset U$ with cocharacter lattice $M_j$.
Denote $\cT\coloneqq (T_j)_{j\in F}$.
We can generalize \cref{const:naive_counts_truncated} by requiring individual toric tail condition for each $j\in F$ as follows:

\begin{construction} \label{const:count_varying_torus}
	For each $j\in F$, we glue a copy of $l_j\coloneqq[0,\hv_j\coloneqq+\infty]$ to $\Gamma$, along $0$ and $v_j$.
	We extend $h$ affinely to the new leg $l_j$ via the identification $\Sk(U)\simeq M_{j,\bbR}$.
	Let $\delta_j\in\NE(Y)$ be the curve class associated to the new leg.
	Let $\hS=[\hGamma,(\hv_j)_{j\in J},\hh]$ denote the resulting extended spine, and $\hgamma\coloneqq\gamma+\sum_{j\in F}\delta_j$.
	We apply \cref{const:naive_counts_extended} to $\hS$ and $\hgamma$, and obtain $F_i(\hS,\hgamma)$.
	Let $F_i(S_\cT,\gamma)\subset F_i(\hS,\hgamma)$ be the subspace consisting of stable maps $[C,(p_j)_{j\in J},f]$ satisfying the \emph{toric tail condition with respect to $\cT$}:
	let $r\colon C\to\hGamma$ be the canonical retraction;
	for each $j\in F$, let $\bbT^*_j\coloneqq r\inv(l_j\setminus{\hv_j})$;
	then we have $f(\bbT^*_j)\subset T_j^\an$.
	We define $N_i(S_\cT,\gamma)\coloneqq\length(F_i(S_\cT,\gamma))$.
	We say that $A\in\bbN$ is big with respect to $(S_\cT,\gamma)$ if it is big with respect to $\hgamma\coloneqq\gamma+\sum\delta_j$.
\end{construction}

Given $A\in\bbN$, for each $T_j\subset U$, \cref{const:walls_by_induction} gives a polyhedral subset $\Wall_A^j\subset M_\bbR$.
We enlarge $\Wall_A$ by adding all $\Wall_A^j$.
Then \cref{prop:moving_w} carries over for counts satisfying the toric tail condition with respect to $\cT$.
So we can define $N(S_\cT,\gamma)$ without specifying $i\in I$ as in \cref{def:naive_count_transverse}.
Moreover, \cref{thm:deformation_invariance_truncated,thm:gluing_inside} carry over without change.

\begin{theorem}\label{thm:count_independent}
	Let $S=[\Gamma,(v_j)_{j\in J},h]$ be a spine in $\Sk(U)$ as in \cref{def:spine_in_U}, $\gamma\in\NE(Y)$, and $A\in\bbN$.
	Let $\cT$ and $\cT'$ be two sets of choices of tori for every $j\in F$.
	We enlarge $\Wall_A$ by adding the walls induced by all the tori in $\cT$ and $\cT'$.
	Assume $S$ is transverse with respect to $\Wall_A$, and $A$ is big with respect to both $(S_\cT,\gamma)$ and $(S_{\cT'},\gamma)$.
	Then $N(S_\cT,\gamma)=N(S_{\cT'},\gamma)$.
\end{theorem}
\begin{proof}
	It suffices to prove the case where $\cT$ and $\cT'$ differ at a single $j\in F$.
	Denote $v\coloneqq v_j$.
	Pick a small interval $[w,v]\subset\Gamma$ whose image is disjoint from $\Wall_A$.
	Let $L$ be the restriction $[w,v]\xrightarrow{\ h\ }\Sk(U)$.
	Let $L_\cT$ be $L$ together with the choice of tori that assign $T_j$ to both $w$ and $v$.
	Let $L_{\cT'}$ be $L$ together with the choice of tori that assign $T_j$ to $w$, and $T'_j$ to $v$.
	Pick any $x\in(w,v)$.
	Observe
	\[S_\cT \sqcup_x L_{\cT'} = S_{\cT'} \sqcup_x L_{\cT},\]
	where $\sqcup_x$ means gluing at $x$.
	By \cref{thm:gluing_inside} (extended to \cref{const:count_varying_torus}), we obtain
	\begin{equation} \label{eq:gluing_varying_tori}
	\sum_{\beta+\delta=\gamma} N(S_\cT,\beta)\cdot N(L_{\cT'},\delta) = \sum_{\beta+\delta=\gamma} N(S_{\cT'},\beta)\cdot N(L_\cT,\delta).
	\end{equation}
	\cref{lem:count_balanced_spines} implies that
	\[N(L_\cT,\delta)=\begin{cases}
	1 &\text{ if }\delta=0,\\
	0 &\text{ otherwise.}
	\end{cases}\]
	By \cite[Prop. 6.5]{Yu_Enumeration_of_holomorphic_cylinders_II}, \cref{lem:count_balanced_spines} also implies that
	\[N(L_{\cT'},\delta)=\begin{cases}
	1 &\text{ if }\delta=0,\\
	0 &\text{ otherwise.}
	\end{cases}\]
	Substituting into equation \eqref{eq:gluing_varying_tori}, we achieve the proof.
\end{proof}

\begin{proposition} \label{prop:structure_constants_varying_tori}
	The structure constants $\chi(P_1,\dots,P_n,Q,\gamma)$, and thus the multiplication rule on the mirror algebra, is independent of the choice of torus $T_M \subset U$.
\end{proposition}
\begin{proof}
	By the definition, the structure constant is a count of spines with finite vertex $q$ mapping to a (fixed) point of $V_Q$.
	By \cref{prop:transversality} we can pick this fixed point so that all the spines are transverse.
	Now apply \cref{thm:count_independent}.
\end{proof}

\begin{remark}
	\cref{prop:structure_constants_varying_tori} can also be deduced from the non-degeneracy of the Frobenius pairing in \cref{thm:main}.
\end{remark}

\section{Structure constants and associativity} \label{sec:associativity}

In this section, we prove the associativity of the multiplication rule in \eqref{eq:multiplication}, see \cref{thm:associativity}.
As explained in \cref{rem:structure_constants_vary_point}, first we need to interpret each structure constant as the degree of a finite étale map over a larger base containing the point $\tQ\in(\cM_{0,n+2}\times U)^\an$, see \cref{lem:structure_constants}.
We relate the structure constants to the Frobenius multilinear form in \cref{lem:structure_constants_trace}.
After that, we study the curve classes of the analytic disks contributing to the structure constants, proving a positivity result in \cref{prop:structure_disk_interior_divisor}, which will be used in \cref{sec:non-degeneracy}.
For the proof of the associativity, we decompose each structure constant into a sum of counts of truncated spines in \cref{lem:structure_constants_decomposition}.
The key associativity equation is \eqref{eq:associativity_expanded}.
After expressing the right hand side as a sum of counts of truncated spines, we vary the modulus of the domain metric trees, in order to stretch the edge separating the legs $p_1$ and $p_2$ from $p_3$.
Then we cut at this edge, and apply the gluing formula to compute the contribution to the left hand side of \eqref{eq:associativity_expanded}.
The cut and glue procedure is described in Constructions \ref{const:associativity_domain}, \ref{const:cut_and_glue} and \cref{fig:associativity_domain}.

\medskip
We follow the notations of \cref{sec:intro:structure_constants}, assuming $Y=\tY$.
We consider $M^\trop_{0,n+2}$, where we label the marked points
$p_1,\dots,p_n,z,s$. 
Let $V_\cM\subset M^\trop_{0,n+2}$ be the subset consisting of metric trees whose $z$-leg and $s$-leg are incident to a single 3-valent vertex.

Fix $A\in\bbN$ big with respect to $\gamma$.
Let $\Sigma'_\rt$ be a simplicial conical subdivision of $\Sigma_\rt$ induced by $\Wall_A$, i.e.\ we ask that every cell of $\Wall_A$ is a union of cells of $\Sigma'_\rt$.
Let $V_Q$ be the union of open cells in $\Sigma'_\rt$ whose closure contains $Q$.
(Note that $V_Q=M_\bbR$ if $Q=0$.)

Let
\begin{equation} \label{eq:Phi_structure_constants}
	\Phi\coloneqq(\dom,\ev_s)\colon H(\bP_Z,\beta)\longrightarrow\cM_{0,n+2}\times U
\end{equation}
be as in \cref{sec:intro:structure_constants}.
Let $H(P_Z,\beta)^\an_{V_\cM\times V_Q}$ be the preimage by $\Phi^\an$ of
\[V_\cM\times V_Q\subset M_{0,n+2}^\trop\times M_\bbR\simeq\Sk(\cM_{0,n+2}\times U)\subset(\cM_{0,n+2}\times U)^\an.\]

Given any stable map $[C,(p_1,\dots,p_n,z,s),f]\in H(\bP_Z,\beta)^\an$, let $\Gamma$ be the convex hull of all the marked points, $r\colon C\to\Gamma$ the retraction, $\bbT\coloneqq r\inv([z,s])\subset C$, and $\bbD\coloneqq r\inv(\overline{\Gamma\setminus z\text{-leg}})$.
The stable map $f$ is said to satisfy the \emph{toric tail condition} if $f(\bbT\setminus z)\subset T_M^\an$.

Let $\cN(P_1,\dots,P_n,Q,\gamma)\subset H(\bP_Z,\beta)^\an_{V_\cM\times V_Q}$ be the subspace consisting of stable maps satisfying the toric tail condition.
By \cref{prop:toric_tail_cc_skeletal}, the inclusion above is a union of connected components.
Hence by \cref{lem:finite_etale_degree}, the degree of the restriction
\begin{equation} \label{eq:structure_constants}
\cN(P_1,\dots,P_n,Q,\gamma)\xrightarrow{\ \Phi^\an\ } V_\cM\times V_Q
\end{equation}
is well-defined.

\begin{lemma} \label{lem:structure_constants}
	The degree of \eqref{eq:structure_constants} is equal to the structure constant $\chi(P_1,\dots,P_n,\allowbreak Q,\allowbreak \gamma)$.
\end{lemma}
\begin{proof}
	The fiber of the map \eqref{eq:structure_constants} at $\tQ$ is exactly the space $F$ in \cref{sec:intro:structure_constants}, so the lemma follows.
\end{proof}

\begin{lemma} \label{lem:structure_constants_trace}
	If $Q$=0, we have $\chi(P_1,\dots,P_n,Q,\gamma)=\eta(P_1,\dots,P_n,\gamma)$, $\eta$ as in \cref{def:counting_Frobenius}.
\end{lemma}
\begin{proof}
	Pick any $(\mu,b)\in V_\cM\times M_\bbR$ such that $b\notin\Wall_A$.
	Let $[C,(p_1,\dots,p_n,z,s),f]\in H(\bP_Z,\beta)^\an_{\mu,b}$.
	Denote $\Sp(f)=[\Gamma,(p_1,\dots,p_n,z,s),h]$.
	By \cref{lem:skeletal_curve_contraction}, $h$ is constant on the $z$-leg and the $s$-leg of $\Gamma$.
	Since $h(s)=b\notin\Wall_A$, there are no twigs of $\Trop(f)$ along the $z$-leg nor the $s$-leg.
	So $f$ satisfies the toric tail condition by \cref{lem:toric_tail_equiv}.
	In other words, we have shown that
	\[\cN(P_1,\dots,P_n,0,\gamma)_{\mu,b}=H(\bP_Z,\beta)^\an_{\mu,b}.\]
	The length of the left hand side is $\chi(P_1,\dots,P_n,0,\gamma)$, while the length of the right hand side is $\eta(P_1,\dots,P_n,\gamma)$, completing the proof.
\end{proof}

\begin{definition} \label{def:structure_disk}
	For any $f\in\cN(P_1,\dots,P_n,Q,\gamma)$, we call the restriction
	\[f|_\bbD\colon[\bbD,(p_1,\dots,p_n,s)]\longrightarrow Y^\an\]
	a \emph{structure disk} responsible for the structure constant $\chi(P_1,\dots,P_n,Q,\gamma)$.
	Note that $f$ is a skeletal curve by \cref{lem:source_of_skeletal_curve}.
\end{definition}

\begin{lemma} \label{lem:structure_disk_class}
	If $\Phi^\an(f)$ is a general rational point in $V_\cM\times V_Q$, then $f$ is defined over a field with discrete valuation.
	In this case, the class of the structure disk above in the sense of \cref{def:curve_class} is equal to $\gamma\in\NE(Y)$.
\end{lemma}
\begin{proof}
	If $\Phi^\an(f)$ is a rational point in $V_\cM\times V_Q\subset Y^\an$, since $Y^\an$ is defined over a field with discrete (possibly trivial) valuation (see \cref{sec:log_CY}), the complete residue field of $\Phi^\an(f)$ has discrete valuation.
	Thus by \cref{lem:restriction_to_skeleton}(\ref{lem:restriction_to_skeleton:etale}), $f$ is defined over a field with discrete valuation.
	Recall from \cref{sec:intro:structure_constants} that $\delta=\beta-\gamma\in\NE(Y)$ is the class of the closure $g\colon\bbP^1\to Y$ of any general translation of the one-parameter subgroup in $T_M$ given by $Q\in M$.
	Denote $\Sp(f)=[\Gamma,(p_1,\dots,p_n,z,s),h]$.
	Let $l\colon[0,+\infty]\to\oM_\bbR$ be the restriction of $h$ to the $z$-leg of $\Gamma$, and let $\ol\colon[-\infty,+\infty]\to\oM_\bbR$ be its affine extension.
	Let $\delta_l, \delta_{\ol}\in\NE(Y)$ the curve classes associated to $l$ and $\ol$ respectively as in \cref{def:curve_class_via_varphi}.
	Since $l(0)\in V_Q$, $\ol([-\infty,0])$ does not meet any codimension-one cone of $\Sigma_\rt$.
	Hence $\delta_l=\delta_{\ol}$.
	Up to translating the map $g$, we may assume that $g$ tropicalizes to $\ol$.
	By \cref{prop:curve_class_formula}, we have $[g]=\delta_{\ol}$ and $[f|_\bbT]=\delta_l$.
	Combining with $[f]=\beta=\gamma+\delta$ and $\delta_l=\delta_{\ol}$, we deduce that $[f|_\bbD]=\gamma$.
\end{proof}

\begin{lemma} \label{lem:curve_class_zero}
	Let $\bbD^\circ\coloneqq\bbD\setminus\{p_1,\dots,p_n\}$.
	If $\Phi^\an(f)$ is a general rational point in $V_\cM\times V_Q$, then the following are equivalent:
	\begin{enumerate}
		\item $f(\bbD)\subset W^\an$ ($W$ as in \cref{lem:toric_model}), and $(\tau\circ f)(\bbD^\circ)$ lies in an open maximal cone of $\Sigma_\rt$.
		\item The class $[f|_\bbD]=0\in\NE(Y)$.
	\end{enumerate}
\end{lemma}
\begin{proof}
	Assume (1).
	We have $[\pi\circ f|_\bbD]=0\in\NE(Y_\rt)$ by \cref{def:curve_class}.
	Hence $[f|_\bbD]=\pi^*[\pi\circ f|_\bbD]=0\in\NE(Y)$ by \cref{lem:curve_class_blowup}.
	
	Now assume (2).
	By Lemmas \ref{lem:restriction_to_skeleton}(\ref{lem:restriction_to_skeleton:sm}) and \ref{lem:interior_divisor}, we have $f(\bbD)\subset W^\an$.
	Moreover, $[\pi\circ f|_\bbD]=\pi_*[f|_\bbD]=0\in\NE(Y_\rt)$.
	So $(\tau\circ f)(\bbD^\circ)$ lies in an open maximal cone of $\Sigma_\rt$ by \cref{lem:curve_class_from_model}.
\end{proof}

\begin{proposition} \label{prop:structure_disk_interior_divisor}
	Let $F\subset Y$ be an effective Cartier divisor containing no 0-stratum of $D^\ess$.
	Let 
	\[f|_\bbD\colon[\bbD,(p_1,\dots,p_n,s)]\longrightarrow Y^\an\]
	be a structure disk responsible for $\chi(P_1,\dots,P_n,Q,\gamma)$ such that $\Phi^\an(f)$ is a general rational point in $V_\cM\times V_Q$.
	We have $\gamma\cdot F=\deg(F^\an|_{\bbD})$.
	In particular, $\gamma\cdot F\ge 0$, with equality if and only if $f(\bbD)$ is disjoint from $F^\an$.
\end{proposition}
\begin{proof}
	By \cref{lem:restriction_to_skeleton}(\ref{lem:restriction_to_skeleton:sm}), $f(\bbD)$ is not contained in the support of $F^\an$.
	By \cref{lem:structure_disk_class}, the curve class $[f]\in\NE(Y)$ is equal to $\gamma$.
	Hence we conclude from \cref{prop:interior_divisor}.
\end{proof}

\begin{lemma} \label{lem:structure_constants_decomposition}
	Say $\bP_Z=(P_1,\dots,P_n,Z)$ is indexed by $J\coloneqq\{1,\dots,n+1\}$, and let $F=\{n+1\}\subset J$.
	Fix $[\Gamma,(p_1,\dots,p_n,z,s)]\in V_\cM$ and $b\in V_Q$.
	Let $\os$ be the finite endpoint of the $s$-leg of $\Gamma$.
	Let $\oGamma\subset\Gamma$ be the convex hull of $p_1,\dots,p_n,\os$.
	Let $\SP(M_\bbR,\bP_Z^F)$ be as in \cref{def:spine}.
	Consider
	\[\Psi_{\os}\coloneqq(\dom,\ev_{n+1})\colon\SP(M_\bbR,\bP_Z^F)\longrightarrow\NT_J^F\times M_\bbR.\]
	Let $\SP(M_\bbR,\bP_Z^F)_{\oGamma,b}$ be the fiber over $(\oGamma,b)\in\NT_J^F\times M_\bbR$.
	We have
	\[\chi(P_1,\dots,P_n,Q,\gamma)=\sum_{S\in\SP(M_\bbR,\bP_Z^F)_{\oGamma,b}} N(S,\gamma).\]
\end{lemma}
\begin{proof}
	Let $P'_Z\coloneqq(P_1,\dots,P_n,Z,0)$ indexed by $J'=\{0,\dots,n+2\}$.
	Consider
	\[\Psi_s\coloneqq(\dom,\ev_{n+2})\colon\SP(M_\bbR,\bP'_Z)\longrightarrow\NT_{J'}\times M_\bbR.\]
	Let $\SP(M_\bbR,\bP'_Z)_{\Gamma,b}$ be the fiber over $(\Gamma,b)\in\NT_{J'}\times M_\bbR$.
	Let $\SP(M_\bbR,\bP'_Z)_{\Gamma,b}^\tail\subset\SP(M_\bbR,\bP'_Z)_{\Gamma,b}$ be the subspace consisting of $[\Gamma',(p'_1,\dots,p'_n,z',s'),h']$ such that $h'$ is constant on the $s'$-leg and affine on the $z'$-leg with derivative $Z$.
	We have a natural bijection of finite sets
	\begin{equation} \label{eq:SP_bijection}
		\SP(M_\bbR,\bP'_Z)_{\Gamma,b}^\tail\xrightarrow{\ \sim\ }\SP(M_\bbR,\bP_Z^F)_{\Gamma,b}
	\end{equation}
	by forgetting the $s$-leg and the $z$-leg.
	
	Let $H(P_Z,\beta)^\an_{\Gamma,b}$ be the fiber of $\Phi^\an$ in \eqref{eq:Phi_structure_constants} at
	\[(\Gamma,b)\in V_\cM\times V_Q\subset M_{0,n+2}^\trop\times M_\bbR\simeq\Sk(\cM_{0,n+2}\times U)\subset(\cM_{0,n+2}\times U)^\an.\]
	It is finite by \cref{lem:restriction_to_skeleton}(\ref{lem:restriction_to_skeleton:all}).
	By Lemmas \ref{lem:skeletal_curve_contraction} and \ref{lem:toric_tail_equiv}, the map
	\[\Sp\colon H(\bP_Z,\beta)^\an_{\Gamma,b}\longrightarrow\SP(M_\bbR,\bP'_Z)_{\Gamma,b}\]
	has image in $\SP(M_\bbR,\bP'_Z)_{\Gamma,b}^\tail$.
	Hence by the bijection \eqref{eq:SP_bijection}, we have a decomposition
	\[H(\bP_Z,\beta)^\an_{\Gamma,b}=\bigsqcup_{S\in\SP(M_\bbR,\bP'_Z)_{\Gamma,b}^\tail} H(\bP_Z,\beta)^\an_{\Gamma,b,S} = \bigsqcup_{S\in\SP(M_\bbR,\bP_Z^F)_{\oGamma,b}} H(\bP_Z,\beta)^\an_{\Gamma,b,S}.\]
	By \cref{lem:structure_constants}, the length of $H(\bP_Z,\beta)^\an_{\Gamma,b}$ is equal to $\chi(P_1,\dots,P_n,Q,\gamma)$.
	By \cref{const:naive_counts_truncated}, for each $S\in\SP(M_\bbR,\bP_Z^F)_{\oGamma,b}$, the length of $H(\bP_Z,\beta)^\an_{\Gamma,b,S}$ is equal to $N(S,\gamma)$.
	Hence we achieve the proof.
\end{proof}

With the preparations above, we are ready to prove the associativity of the multiplication rule.

\begin{theorem} \label{thm:associativity}
	The multiplication rule in \eqref{eq:multiplication} is commutative and associative.
\end{theorem}

The construction of $\chi(P_1,\dots,P_n,Q,\gamma)$ is symmetric with respect to $P_i$, so the commutativity is obvious.
For associativity, let us prove the following equality
\begin{equation} \label{eq:associativity}
	(\theta_{P_1}\cdot\theta_{P_2})\cdot\theta_{P_3}=\theta_{P_1}\cdot\theta_{P_2}\cdot\theta_{P_3}\quad\text{for all }P_1,P_2,P_3\in\Sk(U,\bbZ).
\end{equation}
The same argument will show that
\[\theta_{P_1}\cdot\theta_{P_2}\cdot\theta_{P_3}=\theta_{P_1}\cdot(\theta_{P_2}\cdot\theta_{P_3}),\]
and more generally, the product $\theta_{P_1}\cdot\theta_{P_2}\cdots\theta_{P_n}$ can be computed by 
adding arbitrary parenthesis.

Write
\begin{align*}
(\theta_{P_1}\cdot\theta_{P_2})\cdot\theta_{P_3}&=\bigg(\sum_{R\in M}\sum_{\eta\in\NE(Y)}\chi(P_1,P_2,R,\eta)z^\eta\theta_R\bigg)\cdot\theta_{P_3}\\
&=\sum_{R,Q\in M}\sum_{\eta,\phi\in\NE(Y)}\chi(P_1,P_2,R,\eta)z^\eta\,\chi(R,P_3,Q,\phi)z^\phi\theta_Q,
\end{align*}
and
\[\theta_{P_1}\cdot\theta_{P_2}\cdot\theta_{P_3}=\sum_{Q\in M}\sum_{\gamma\in\NE(Y)}\chi(P_1,P_2,P_3,Q,\gamma)z^\gamma\theta_Q.\]
So equation \eqref{eq:associativity} is equivalent to the following equality for every $Q\in M$ and $\gamma\in\NE(Y)$,
\begin{equation} \label{eq:associativity_expanded}
	\sum_{R\in M}\sum_{\eta+\phi=\gamma}\chi(P_1,P_2,R,\eta)\cdot\chi(R,P_3,Q,\phi)=\chi(P_1,P_2,P_3,Q,\gamma).
\end{equation}

\begin{construction} \label{const:associativity_finiteness}
	Let $\delta_Q\in\NE(Y)$ be the class of the closure of any general translation of the one-parameter subgroup in $T_M$ given by $Q\in M$.
	Let $A_0\coloneqq\pi_*(\gamma+\delta_Q)\cdot D_\rt$.	
	By the balancing condition, there exists $N\in\bbN$ and a finite subset $W\subset M$ such that for any tropical curve $[\Gamma,(p_1,p_2,p_3,z,s),h]$ in $M_\bbR$ with degree $A_0$, the weight vectors of all edges of $\Gamma$ belong to $W$, and the number of bending vertices on the associated spine $\Gamma^s$ is at most $N$.
	Fix $A\in\bbN$ big with respect to $\gamma^i$ for any decomposition $\gamma=\gamma^1+\gamma^2$ with $\gamma^i\in\NE(Y)$.
	Up to a toric blowup of $(Y,D)$, we can assume every $w\in W\setminus 0$ lies in a 1-dimensional cone of $\Sigma_\rt$.
\end{construction}

\begin{assumption} \label{ass:associativity_finiteness}
	From now on to the end of this section, all the spines in $M_\bbR$ we will consider are with respect to $\Wall_A$, and required to satisfy the following conditions:
	\begin{enumerate}
		\item The number of bending vertices is bounded by $N$;
		\item The weight vectors of all the edges belong to the finite subset $W\subset M$.
	\end{enumerate}
\end{assumption}

Let $\Sigma'_\rt$ be a simplicial conical subdivision of $\Sigma_\rt$ induced by $\Wall_A$.
For each $R\in M$, let $V_R$ be the union of open cells in $\Sigma'_\rt$ whose closure contains $R$.

\begin{lemma} \label{lem:stretch}
	Given $c\in M_\bbR\setminus\Wall_A$, there exists a positive real number $\lambda$ having the following property:
	
	Let $[L,(v_1,v_2),h]$ be any spine in $M_\bbR$ satisfying the following conditions:
	\begin{enumerate}
		\item \cref{ass:associativity_finiteness};
		\item $v_1, v_2$ are finite vertices;
		\item transverse with respect to $\Wall_A$;
		\item \label{lem:stretch:v_1} $h(v_1)=c$;
		\item \label{lem:stretch:length} the length of $L$ is greater than or equal to $\lambda$.
	\end{enumerate} 
	Let $e_2$ be the edge of $L$ incident to $v_2$, and $R\coloneqq-w_{(v_2,e_2)}$.
	Then $h(v_2)\in V_R\subset M_\bbR$.
\end{lemma}
\begin{proof}
	\cref{ass:associativity_finiteness} implies that there are only finitely many combinatorial types of such spines.
	By transversality, either $h$ is constant, or its derivative is nowhere zero.
	If $h$ is constant, then $h(v_2)=h(v_1)=c$ and $R=0$; so we always have $h(v_2)\in V_R=M_\bbR$.
	Now assume $h$ is immersed.
	Once we fix the combinatorial type, since the image $h(v_1)=c$ is fixed, when we increase the length of $L$, we are only increasing the length of the domain of affineness of $L$ containing $v_2$.
	Hence by increasing $\lambda$, we can push the  image $h(v_2)$ as far as we want, until we have $h(v_2)\in V_R\subset M_\bbR$. 
\end{proof}

\begin{construction} \label{const:associativity_domain}
	We define three metric trees $[\Gamma,(p_1,p_2,p_3, v)]$, $[H,(p_1,p_2,u)]$ and $[F,(u,p_3,\allowbreak v)]$ as in \cref{fig:associativity_domain}.
	The dots indicate infinite legs.
	The length of each finite edge is given next to the edge.
	We choose the letters $H$ and $F$ to suggest \emph{head} and \emph{foot}, when we cut $\Gamma$ at $u$.
	\begin{figure}[!ht] \label{figure:associativity_domain}
		\centering
		\setlength{\unitlength}{0.95\textwidth}
		\begin{picture}(1,0.3)
			\put(0,0){\includegraphics[width=\unitlength]{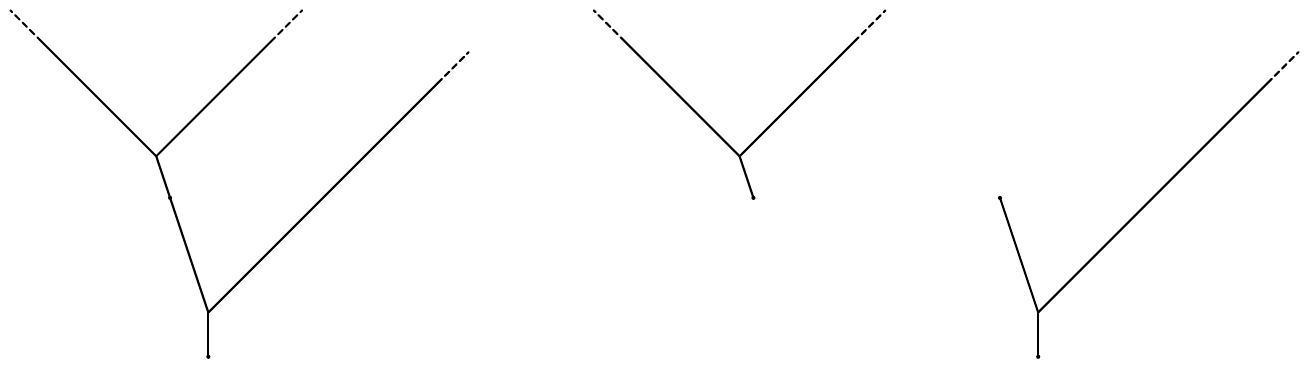}}
			\put(-0.02,0.26){$p_1$}
			\put(0.425,0.26){$p_1$}
			\put(0.234,0.26){$p_2$}
			\put(0.68,0.26){$p_2$}
			\put(0.358,0.22){$p_3$}
			\put(0.99,0.22){$p_3$}
			\put(0.107,0.115){$u$}
			\put(0.55,0.115){$u$}
			\put(0.74,0.115){$u$}
			\put(0.137,0.0){$v$}
			\put(0.77,0.0){$v$}
			\put(0.13,0.138){$1$}
			\put(0.574,0.138){$1$}
			\put(0.165,0.013){$1$}
			\put(0.797,0.013){$1$}
			\put(0.15,0.085){$\lambda$}
			\put(0.78,0.085){$\lambda$}
		\end{picture}
		\caption{Metric trees $\Gamma$, $H$ and $F$.}
		\label{fig:associativity_domain}
	\end{figure}
\end{construction}

\begin{construction} \label{const:cut_and_glue}
	Fix $b\in V_Q$.
	Let $\SP_\Gamma$ denote the set of spines in $M_\bbR$ with domain $\Gamma$, outgoing weight vectors $P_1,P_2,P_3,-Q$ at the vertices $p_1,p_2,p_3,v$ respectively, and sending $v$ to $b$.
	For each $R\in M$, let $\SP_{F,R}$ denote the set of spines in $M_\bbR$ with domain $F$, outgoing weight vectors $R,P_3,-Q$ at the vertices $u,p_3,v$ respectively, and sending $v$ to $b$.
	For each $R\in M$ and $a\in V_R$, let $\SP_{H,R,a}$ denote the set of spines in $M_\bbR$ with domain $H$, outgoing weight vectors $P_1,P_2,-R$ at the vertices $p_1,p_2,u$ respectively, and sending $u$ to $a$.
	Under \cref{ass:associativity_finiteness}, the three sets above are all finite.

	Let
	\[\Pairs\coloneqq\Set{(f,h) | f\in\SP_{F,R},\ h\in\SP_{H,R,f(u)} \text{ for some }R\in M}.\]
	We have obvious inverse bijections
	\[\Cut\colon\SP_\Gamma\to\Pairs,\quad \Glue\colon\Pairs\to\SP_\Gamma\]
	given by cutting and gluing at $u$.
\end{construction}
	
Since these sets are finite, the following lemma follows from \cref{prop:transversality}.
	
\begin{lemma} \label{lem:associativity_generic}
	For general choice of $b \in V_Q$ and $\lambda > 0$, all spines in $\SP_\Gamma$ and $\Pairs$ are transverse with respect to $\Wall_A$.
\end{lemma}

Now we assume $b\in V_Q$ general and let $c\coloneqq b+Q\in M_\bbR$.
Apply \cref{lem:stretch} and obtain $\lambda>0$.
We increase $\lambda$ if necessary to make it general in the sense of \cref{lem:associativity_generic}.

\begin{lemma} \label{lem:associativity_decomposition}
	For every $R\in M$, $\phi,\eta\in\NE(Y)$ and $a\in V_R$, we have the following equalities
	\begin{enumerate} 
		\item \label{lem:associativity_structure_constants:Gamma} $\chi(P_1,P_2,P_3,Q,\gamma)=\sum_{g \in \SP_\Gamma} N(g,\gamma)$.
		\item \label{lem:associativity_structure_constants:H} $\chi(P_1,P_2,R,\eta) =\sum_{h \in \SP_{H,R,a}} N(h,\eta)$.
		\item \label{lem:associativity_structure_constants:F} $ \chi(R,P_3,Q,\phi)=\sum_{f \in \SP_{F,R}} N(f,\phi)$.
	\end{enumerate}
\end{lemma} 
\begin{proof}
	(\ref{lem:associativity_structure_constants:Gamma}) and (\ref{lem:associativity_structure_constants:H}) follow directly from \cref{lem:structure_constants_decomposition}.
	Now we consider (\ref{lem:associativity_structure_constants:F}).
	For each $[F,(u,p_3,v),f]\allowbreak\in\SP_{F,R}$, we glue a copy of $e\coloneqq[0,\hat u=+\infty]$ to $F$ along $0$ and $u$, extend $f$ affinely to $e$, and denote the resulting spine by $f'$.
	By \cref{lem:stretch}, $f'(e\setminus\hat u)\subset V_R$; in particular, $f'(e)\cap\Wall_A=\emptyset$.
	Therefore, in \cref{const:naive_counts_truncated}, the toric tail condition at $u$ is automatically satisfied, by \cref{lem:toric_tail_equiv}.
	Moreover, $f'(e\setminus\hat u)\subset V_R$ implies that the curve class associated to $f'|_e$ is 0.
	Hence we have $N(f,\phi)=N(f',\phi)$.
	Therefore (\ref{lem:associativity_structure_constants:F}) also follows \cref{lem:structure_constants_decomposition}.
\end{proof}

\begin{proof}[Proof of equation \eqref{eq:associativity_expanded}]
	We have 
	\begin{eqnarray*}
		\lefteqn{\chi(P_1,P_2,P_3,Q,\gamma)= \sum_{g \in \SP_\Gamma} N(g,\gamma)}\\
		&=&\sum_{g \in \SP_\Gamma}\ \sum_{\phi+\eta=\gamma} N(\Cut(g)^f,\phi)\cdot N(\Cut(g)^h,\eta) \\
		&=&\sum_{(f,h)\in\mathit{Pairs}}\ \sum_{\phi+\eta=\gamma}  N(f,\phi)\cdot N(h,\eta)\\
		&=&\sum_{R\in M}\ \sum_{f \in \SP_{F,R}}\ \sum_{h \in \SP_{H,R,f(u)}}\ 
		\sum_{\phi+\eta=\gamma} N(f,\phi)\cdot N(h,\eta)\\
		&=&\sum_{R\in M}\ \sum_{f \in \SP_{F,R}}\ \sum_{\phi}\bigg(\sum_{\phi+\eta = \gamma}\ 
		\sum_{h\in \SP_{H,R,f(u)}} N(h,\eta) \bigg) N(f, \phi) \\
		&=& \sum_{R\in M}\ \sum_{\phi}\ \sum_{f \in \SP_{F,R}}\ \bigg(\sum_{\phi+\eta = \gamma} 
		\chi(P_1,P_2,R,\eta) \bigg) N(f, \phi) \\
		&=&\sum_{R\in M}\ \sum_{\phi}\ \bigg(\sum_{\phi+\eta = \gamma}\ \sum_{f \in \SP_{F,R}} N(f,\phi) \bigg) 
		\chi(P_1,P_2,R,\eta) \\
		&=& \sum_{R \in M}\ \sum_{\phi+\eta = \gamma}  \chi(P_1,P_2,R,\eta)\cdot\chi(R,P_3,Q,\phi)
	\end{eqnarray*}
where the first equality is by \cref{lem:associativity_decomposition}, the second by \cref{thm:gluing_concatenate} (the gluing formula), the third by the $\Cut,\Glue$ bijection, the fourth by the definition of $\Pairs$,
the fifth by reorganizing the sum, the sixth by \cref{lem:associativity_decomposition}, the
seventh by reorganizing the sum, and the last by \cref{lem:associativity_decomposition}.
\end{proof} 

\section{Convexity and finiteness} \label{sec:convexity}

In this section, we prove the convexity property for structure disks, see \cref{thm:Ftrop_structure_constants}.
Then we deduce that the two sums in the multiplication rule \eqref{eq:multiplication} are finite sums, see \cref{cor:R-algebra_structure}.
We assume throughout this section that $k$ has nontrivial discrete valuation.

\begin{construction} \label{const:divisor_tropicalization}
	Let $(\fX,\fH)$ be a formal strictly semistable pair, $X\coloneqq\fX_\eta\setminus\fH_\eta$,  $\Sigma(\fX,\fH)\subset X$ the associated skeleton, and $\tau\colon X\to\Sigma(\fX,\fH)$ the retraction map.
	Let $\fF$ be a Cartier divisor on $\fX$.
	It is locally given by a rational function $f$ up to multiplication by invertible functions.
	So $\val f$ is a well-defined continuous $\bbR$-valued function on $\fX_\eta\setminus\supp\fF_\eta$, which we denote by $\val\fF$.
		Let $\fF^\trop\coloneqq(\val\fF)|_{\Sigma(\fX,\fH)}$, which we call the \emph{tropicalization} of $\fF$.
	Since $\Sigma(\fX,\fH)$ is disjoint from any closed analytic subspace of $\fX_\eta$, (in particular $\supp\fF_\eta$), $\fF^\trop$ is well-defined everywhere.
\end{construction}

\begin{lemma} \label{lem:divisor_tropicalization}
	In the setting of \cref{const:divisor_tropicalization}, let $\Iv$ be the set of irreducible components of $\fX_s$, and $\Ih$ the set of irreducible components of $\fH$.
	We have a natural embedding $\Sigma(\fX,\fH)\subset\bbR_{\ge 0}^{\Iv\sqcup\Ih}$.
	Let $\Psi$ be the linear function on $\bbR^{\Iv\sqcup\Ih}$ such that the value of $\Psi$ at the unit vector in the $i$-th direction is equal to the order of zero of $\fF$ along $D_i$, the corresponding component of $\fX_s\cup\fH$.
	The following hold:
	\begin{enumerate}
		\item \label{lem:divisor_tropicalization:inequality} If $\fF$ is effective, then
		\[\fF^\trop\ge\Psi|_{\Sigma(\fX,\fH)}.\]
		Moreover, they are equal at the vertex of $\Sigma(\fX,\fH)$ corresponding to every irreducible component of $\fX_s$.
		\item \label{lem:divisor_tropicalization:equality} Let $\fF'$ be the Zariski closure in $\fX$ of the restriction $\fF_\eta|_X$.
		If $\supp \fF'$ does not contain any strata of $\fX_s\cup \fH$, then
		\[\fF^\trop=\Psi|_{\Sigma(\fX,\fH)}.\]
		\item \label{lem:divisor_tropicalization:fiber}
		Under the assumption in (\ref{lem:divisor_tropicalization:equality}), $\val\fF$ is constant on the fibers of $\tau\colon X\to\Sigma(\fX,\fH)$ over every open top-dimensional cell of $\Sigma(\fX,\fH)$.
	\end{enumerate}
	\end{lemma}
\begin{proof}
	Write $\Sigma\coloneqq\Sigma(\fX,\fH)$ to simplify notation.
	Working Zariski locally on $\fX$, we may assume $\fX=\Spf A$ and that there exists an étale map
	\[\alpha\colon\fX\to\Spf\big(\kc\braket{x_0,\dots,x_d}/(x_0\cdots x_r-\varpi)\big),\]
	for some $0\le r\le d$ and $\varpi$ a uniformizer of $k$, such that every irreducible component $D$ of $\fH$ is defined by $\alpha^*(x_j)$ for some $j>r$.
	The Cartier divisor $\fF$ is assumed effective in (\ref{lem:divisor_tropicalization:inequality}).
	For (\ref{lem:divisor_tropicalization:equality}), by writing $\fF$ as the difference of two effective divisors, we can also assume it to be effective.
	Then $\fF$ is given by some $f\in A$.
	Let $\psi\in A$ be given by a monomial in $x_0,\dots,x_d$ such that it has the same order of zero as $f$ along $D_i$ for every $i\in\Iv\sqcup\Ih$.
	Then $\psi^\trop=\Psi|_\Sigma$.
	Write $f=\psi g$ with $g\in A$.
	By \cite[Proposition 1.4]{Berkovich_Smooth_p-adic_analytic_spaces_are_locally_contractible}, the spectral seminorm of $g$ is at most 1; in other words, $\abs{g(x)}\le 1$ for all $x\in\fX_\eta$.
	Hence
	\[\fF^\trop=f^\trop=\psi^\trop+g^\trop\ge\psi^\trop=\Psi|_\Sigma.\]
	
	Moreover, let $C$ be any irreducible component of $\fX_s$, $s\in C$ its generic point and $v\in\Sigma$ the corresponding vertex.
	Let $r\colon\fX_\eta\to\fX_s$ be the reduction map.
	We have $r\inv(s)=\{v\}$ by the construction of the embedding $\Sigma\subset\fX_\eta$.
	Since $g$ is not zero on $C$, we have $g(s)\neq 0$.
	Since $s=r(v)$, we deduce that $\abs{g(v)}=1$, and thus $\fF^\trop(v)=\Psi|_\Sigma(v)$.
	
	Now assume that $\supp\fF'$ does not contain any strata of $\fX_s\cup\fH$.
	Then the zeros of $g$ does not contain any strata of $\fX_s\cup\fH$.
	Let $\sigma$ be any open cell of $\Sigma$, $S\subset\fX_s$ the corresponding closed strata, and $s$ the generic point of $S$.
	We have $\sigma\subset r\inv(s)$ by the construction of $\Sigma\subset\fX_\eta$.
	Since $g$ is not zero on $S$, we have $g(s)\neq 0$.
	Hence $\abs{g(x)}=1$ for any $x\in r\inv(s)$, in particular for any $x\in\sigma$.
	This shows (\ref{lem:divisor_tropicalization:equality}).
	Specializing to the case where $\sigma$ is an open top-dimensional cell of $\Sigma$, i.e.\ $S$ is a point, we obtain (\ref{lem:divisor_tropicalization:fiber}).
\end{proof}

\begin{lemma} \label{lem:dFtrop}
	Let $\fC$ be a strictly semistable formal scheme over $\kc$, $C\coloneqq\fC_\eta$ its generic fiber, and $\Sigma_\fC\subset C$ the associated skeleton.
	Let $\fF\subset\fC$ be a Cartier divisor and $\fF^\trop\colon \Sigma_\fC\to\bbR$ its tropicalization.
	Decompose $\fF=\fF^\ver+\fF^\hor$ where $\fF^\ver$ is supported on $\fC_s$ and $\fF^\hor=\overline{\fF_\eta}$.
	Let $\fC_s^\pr$ (resp.\ $\fC_s^\npr$) denote the union of proper (resp.\ non-proper) irreducible components of $\fC_s$.
	Assume $\supp \fF^\hor \cap \fC_s^\npr=\emptyset$.
	Then we have
	\begin{equation} \label{eq:dFtrop}
	\sum_{v\in \partial C} \sum_{e\ni v} -d_{ve} \fF^\trop = \fF\cdot\fC_s^\pr-\deg(\fF|_C),
	\end{equation}
	where $d_{ve} \fF^\trop$ denotes the derivative of $\fF^\trop$ at $v$ in the direction of the edge $e$ containing $v$. 
\end{lemma}
\begin{proof}
	Up to a finite base field extension and a modification of $\fC$, we can assume that $\supp \fF^\hor$ does not contain any node of $\fC_s$, and that there is no edge of $\Sigma_\fC$ containing two points of $\partial C$.
	Let $V^\pr$ be the set of vertices of $\Sigma_\fC$ minus $\partial C$.
	For $v\in V^\pr$, write $\fC_s^v$ the corresponding component of $\fC_s^\pr$.
	We can decompose
	\[\sum_{v\in\partial C}\sum_{e\ni v} -d_{ve} \fF^\trop=\sum_{v\in V^\pr}\sum_{e\ni v} d_{ve} \fF^\trop,\]
	\[\fF^\ver\cdot\fC_s^\pr=\sum_{v\in V^\pr} \fF^\ver\cdot\fC_s^v.\]
	
	Fix $v\in V^\pr$.
	Let $E_0\coloneqq\fC_s^v$, and $E_1,\dots,E_n$ the irreducible components of $\fC_s$ intersecting $E_0$ at a point.
	Since the intersection number $E_0\cdot\fC_s=0$, we deduce that $E_0^2=-n$.
	Let $a_i$ be the order of $E_i$ in $\fF^\ver$ for $i=0,\dots,n$.
	For computing $\fF^\trop$, we apply \cref{lem:divisor_tropicalization}(\ref{lem:divisor_tropicalization:equality}) to $\fX=\fC$, $\fH=\emptyset$ and $\fF$, and obtain
	\[\sum_{e\ni v} d_{ve} \fF^\trop=\sum_{i=1}^n (a_i-a_0)=\sum_{i=1}^n a_i-a_0 n=\sum_{i=1}^n a_i E_i\cdot E_0 + a_0 E_0\cdot E_0 = \fF^\ver\cdot E_0=\fF^\ver\cdot \fC_s^v.\]
	
	Summing over $v\in V^\pr$, we obtain
	\[\sum_{v\in \partial C} \sum_{e\ni v} -d_{ve} \fF^\trop = \sum_{v\in V^\pr}\sum_{e\ni v} d_{ve} \fF^\trop = \fF^\ver\cdot\fC_s^\pr.\]
	By the assumption that $\supp \fF^\hor\cap\fC_s^\npr=\emptyset$, we have
	\[\deg(\fF|_C)=\fF^\hor\cdot\fC_s^\pr.\]
	Hence
	\[\sum_{v\in \partial C} \sum_{e\ni v} -d_{ve} \fF^\trop = \fF^\ver\cdot\fC_s^\pr = (\fF-\fF^\hor)\cdot\fC_s^\pr = \fF\cdot\fC_s^\pr-\deg(\fF|_C),\]
	completing the proof.
\end{proof}

Now we follow the notations in \cref{sec:log_CY}.
Let $F$ be a divisor on $Y$ that is constant over $k$.
We take constant formal model $(\hY_{\kc},\hD_{\kc})$ for $(Y,D)$ and $\hF_{\kc}$ for $F$ over the ring of integers $\kc$, apply \cref{const:divisor_tropicalization}, and obtain $\val F\colon (U\setminus\supp F)^\an\to\bbR$ and $F^\trop\colon\Sigma_{(Y,D)}\to\bbR$.
The restriction of $F^\trop$ to $\Sigma_\rt\simeq\Sigma_{(Y,D)}^\ess\subset\Sigma_{(Y,D)}$ will also be denoted by $F^\trop$.

\begin{definition} \label{def:corner_locus}
	Let $(Y',D')$ be a toric blowup of $(Y,D)$ such that $\supp\overline{F|_U}$ does not contain any strata of $D$.
		Let $\Sigma^F$ be the subdivision of $\Sigma_{(Y,D)}$ induced by the toric blowup, and $\Sigma^F_\rt$ the subdivision of $\Sigma_\rt$ induced by $\Sigma^F$.
	By \cref{lem:divisor_tropicalization}, $F^\trop$ is linear on every closed cone of $\Sigma^F$ and $\Sigma^F_\rt$.
	Let $(\Sigma^F_\rt)^{d-1}$ denote the union of codimension-one cells in $\Sigma^F_\rt$.
\end{definition}

The following lemma will serve the proof of \cref{thm:Ftrop_structure_constants}.
It also implies the broken-line convexity conjecture of \cite[8.11]{Gross_Canonical_bases} in our setting.

\begin{lemma} \label{lem:convexity}
	Let $C$ be a compact strictly \kanal curve, $S\subset C$ the convex hull of $\partial C$, $F$ a divisor on $Y$ constant over $k$, and $f\colon C\to U^\an$ a morphism whose image is not contained in $(\supp F)^\an$.
	Let $\phi$ be the restriction of $(\val F\circ f)$ to $S$.
	Assume $(\tau\circ f)(\partial C)$ is disjoint from $(\Sigma^F_\rt)^{d-1}\subset\Sigma_\rt$,
	and that for every $v\in\partial C$, the map $(\tau\circ f)$ is constant on every connected component of $C\setminus v$ not containing $S$.
	We have
	\[\sum_{v\in\partial C} -d_v \phi=F\cdot[f]-\deg(F^\an|_C),\]
	where $d_v\phi$ denotes the derivative of $\phi$ at $v$ in the direction of the unique edge of $S$ containing $v$.
	As a result, if $F$ is nef and $-F|_U$ is effective, the sum on the left hand side is non-negative.
	\end{lemma}
\begin{proof}
	Up to a finite base field extension, we can choose a strictly semistable formal model $\fC$ of $C$ such that the map $f\colon C\to U^\an\subset Y^\an$ extends to $\ff\colon\fC\to\hY_{\kc}$, and that $S$ is contained in the skeleton $\Sigma_\fC\subset C$.
	Let $\fF\coloneqq\ff\inv\hF_{\kc}$.
		Decompose $\fF=\fF^\ver+\fF^\hor$ as in \cref{lem:dFtrop}.
	The assumption that $(\tau\circ f)(\partial C)$ does not meet $(\Sigma^F_\rt)^{d-1}$ implies that $\supp\fF^\hor\cap\fC_s^\npr\neq\emptyset$.
		Therefore, by \cref{lem:dFtrop}, we have
	\[\sum_{v\in\partial C}\sum_{e\ni v} -d_{ve}\fF^\trop = \fF\cdot\fC_s^\pr-\deg(\fF|_C).\]
	By \cref{def:curve_class}, the right hand side is equal to $F\cdot[f]-\deg(F^\an|_C)$.
	By construction we have $\fF^\trop=(\val F\circ f)|_{\Sigma_\fC}$.
	By \cref{lem:divisor_tropicalization}(\ref{lem:divisor_tropicalization:fiber}), it follows from the assumptions that the function $\val F\circ f$ is constant on every connected component of $C\setminus v$ not containing $S$.
	Therefore,
	\[\sum_{v\in\partial C}\sum_{e\ni v} -d_{ve}\fF^\trop = \sum_{v\in\partial C}\sum_{e\ni v} -d_{ve}(\val F\circ f)|_{\Sigma_\fC} = \sum_{v\in\partial C} -d_v \phi.\]
	Combining the equalities above, we conclude the proof.
\end{proof}

\begin{theorem} \label{thm:Ftrop_structure_constants}
	Let $F$ be a divisor on $Y$.
	Let $V_\cM, V_Q$ be as in \cref{sec:associativity}, $\Sigma^F_\rt$ as in \cref{def:corner_locus}, and $V^F_Q$ the intersection of $V_Q$ with the union of open cells in $\Sigma^F_\rt$ whose closure contains $Q$.
	Let
	\[f|_\bbD\colon[\bbD,(p_1,\dots,p_n,s)]\longrightarrow Y^\an\]
	be a structure disk responsible for $\chi(P_1,\dots,P_n,Q,\gamma)$ as in \cref{def:structure_disk} such that $\Phi^\an(f)$ is a general rational point in $V_\cM\times V^F_Q$.
	Let $\bbD^\circ\coloneqq\bbD\setminus\{p_1,\dots,p_n\}$.
	The following hold:
	\begin{enumerate}
		\item \label{thm:Ftrop_structure_constants:equality} We have
		\begin{equation} \label{eq:Ftrop_structure_constants}
		\sum F^\trop(P_i) - F^\trop(Q) = F \cdot \gamma - \deg(F^\an|_{\bbD^\circ}).
		\end{equation}
		\item \label{thm:Ftrop_structure_constants:nef}
		Assume $F$ is nef and $-F|_U$ is effective.
		Then
		\begin{equation} \label{eq:Ftrop_inequality}
			F^\trop(Q)\le\sum F^\trop(P_i).
		\end{equation}
		\item \label{thm:Ftrop_structure_constants:ample}
		Assume $F$ is ample and $-F|_U$ is effective.
		If \eqref{eq:Ftrop_inequality} is an equality, then $f(\bbD)\subset W^\an$, ($W$ as in \cref{lem:toric_model}), and $(\tau\circ f)(\bbD^\circ)$ lies in an open maximal cone of $\Sigma_\rt$.
	\end{enumerate}
\end{theorem}

\begin{proof}[Proof of \cref{thm:Ftrop_structure_constants}]
	Let $b\coloneqq\partial\bbD$, $\Gamma\subset\bbD$ the convex hull of $\{p_1,\dots,p_n,b\}$, $r\colon \bbD\to\Gamma$ the retraction map, and $h\coloneqq\tau\circ f|_\Gamma\colon\Gamma\to\oSigma_\rt$.
	For $i=1,\dots,n$, let $p'_i\in\Gamma$ be a rational point close to but not equal to $p_i$.
	Let $\Gamma'\subset\Gamma$ be the convex hull of $p'_1,\dots,p'_n,b$, and $\bbD'\coloneqq r\inv(\Gamma')$.
	Recall from \cref{def:structure_disk} that $f$ is a skeletal curve.
	So by \cref{thm:f_in_skeleton}(\ref{thm:f_in_skeleton:g_inv}), $f(\Gamma')$ lies in $\Sk(U)\subset U^\an$.
	By \cref{lem:essential_skeleton_semistable_pair}, $\Sk(U)$ is contained in any Zariski dense open subset of $U^\an$, so $f(\Gamma')$ is disjoint from $(\supp F)^\an\subset Y^\an$.
	Let $\phi\coloneqq(\val F\circ f)|_{\Gamma'}$.
	We choose $p'_i$ sufficiently close to $p_i$ so that
	\[\deg(F^\an|_{\bbD^\circ})=\deg(F^\an|_{\bbD'}),\]
	and that for each $P_i\neq 0$, the image $h(p'_i)$ is contained in a closed top-dimensional cell of $\Sigma^F_\rt$ containing $P_i$.
	Note that $h$ has outgoing derivative $P_i$ at $p'_i$.
	Therefore, by \cref{lem:divisor_tropicalization}, we obtain
	\[-d_{p_i'} \phi = F^\trop(P_i).\]
	Furthermore, by the toric tail condition (see \cref{lem:toric_tail_equiv}(\ref{lem:toric_tail_equiv:twig})), $h$ has ingoing derivative $Q$ at $b$, and $h(b)=(\tau\circ f)(q)\in V^F_Q$.
	So \cref{lem:divisor_tropicalization} also implies
	\[d_b \phi = F^\trop(Q).\]

	Since $\Phi^\an(f)$ is a general point in $V_\cM\times V_Q$, by \cref{prop:transversality}, $(\tau\circ f)(\partial\bbD')$ will not meet the codimension-one skeleton $(\Sigma^F_\rt)^{d-1}$ of $\Sigma_\rt^F$, and will not meet the walls.
	So we can apply \cref{lem:convexity} to $f|_{\bbD'}\colon \bbD'\to U^\an$, and obtain
	\[\sum_{v\in\partial \bbD'} -d_v \phi = F\cdot [f|_{\bbD'}] - \deg(F^\an|_{\bbD'}),\]
	where $[f|_{\bbD'}]=\gamma$ by \cref{lem:structure_disk_class}.
	Combining all the equalities above, we obtain equation \eqref{eq:Ftrop_structure_constants}.
	
	When $F$ is nef and $-F|_U$ is effective, we have
	\[F\cdot\gamma-\deg(F^\an|_{\bbD^\circ})\ge 0.\] 
	Hence, equation \eqref{eq:Ftrop_structure_constants} implies inequality \eqref{eq:Ftrop_inequality}.
	If moreover $F$ is ample and \eqref{eq:Ftrop_inequality} is an equality, then $\gamma=0$, so we conclude by \cref{lem:curve_class_zero}. 
\end{proof}

\begin{proposition} \label{prop:finiteness_Q_gamma}
	For fixed $P_1,\dots,P_n \in\Sk(U,\bbZ)$, there are at most finitely many pairs $(Q,\gamma)$, $Q \in\Sk(U,\bbZ)$, $\gamma \in\NE(Y)$ such that $\chi(P_1,\dots,P_n,Q,\gamma) \neq 0$. 
\end{proposition}
\begin{proof}
	By \cref{lem:enough_global_functions}, there exist regular functions $x_1,\dots,x_l$ on $U$ such that for any $c\in\bbR$, the set
	\[\Set{b\in \Sk(U) | \abs{x_i(b)}\le c,\ i=1,\dots,l}\] is bounded; in particular, the subset of integer points inside is finite.
	If $\chi(P_1,\dots,P_n,\allowbreak Q,\allowbreak \gamma)\allowbreak\neq 0$, by \cref{thm:Ftrop_structure_constants}(\ref{thm:Ftrop_structure_constants:nef}), for $i=1,\dots,l$, we have
	\[
	\abs{x_i(Q)} \leq \prod_j \abs{x_i(P_j)}.
	\]
	Thus given $P_1,\dots,P_n$, there are at most finitely many $Q$ such that  $\chi(P_1,\dots,P_n,Q,\allowbreak \gamma)\allowbreak \neq 0$ for some $\gamma\in\NE(Y)$.
	
	Now let $F$ be an ample divisor on $Y$ with $-F|_U$ effective (see \cref{lem:ample_divisor}).
	If $\chi(P_1,\dots,P_n,Q,\allowbreak\gamma)\allowbreak\neq 0$, by \cref{thm:Ftrop_structure_constants}(\ref{thm:Ftrop_structure_constants:equality}), we have
	\[
	F\cdot\gamma=\sum F^\trop(P_i) - F^\trop(Q) + \deg(F^\an|_{\bbD^\circ}) \le \sum F^\trop(P_i) - F^\trop(Q).
	\]
	Therefore, given $P_1,\dots,P_n,Q$, we see that $F\cdot\gamma$ is bounded, so there are at most finitely many $\gamma$ such that $\chi(P_1,\dots,P_n,Q,\gamma)\neq 0$.
	
	Combining the two paragraphs above, we conclude the proof.
\end{proof}

\begin{corollary} \label{cor:R-algebra_structure}
	The two sums in the multiplication rule \eqref{eq:multiplication} are finite sums.
	Therefore, combining \cref{thm:associativity}, the multiplication rule makes $A$ into a commutative associative $R$-algebra.
\end{corollary}

\begin{proposition} \label{prop:algebra_modulo_m}
	Let $\fm\subset R$ be the maximal monomial ideal (i.e.\ generated by all $z^\gamma$, $\gamma\in\NE(Y)\setminus 0$).
	Then $A \otimes_R R/\fm$ is isomorphic to the Stanley-Reisner ring for the fan $\Sigma_\rt$;		
	in other words, modulo $\fm$, $\theta_{P_1} \cdot \theta_{P_2} = \theta_{P_1 + P_2}$ if $P_1,P_2$ lie in a same cone of $\Sigma_\rt$, and $\theta_{P_1} \cdot \theta_{P_2}=0$ otherwise.
	\end{proposition}
\begin{proof}
	Note $z^\gamma\not\in\fm$ if and only if $\gamma=0$.
	In this case the contributing structure discs are described by \cref{lem:curve_class_zero}, and hence the structure constants can be computed as in the toric case by \cref{lem:counts_in_toric_case}, from which we conclude the proof.
\end{proof}

\section{Torus action and finite generation} \label{sec:torus_action}

In this section, we describe a natural torus action on the mirror algebra (see \cref{thm:torus_action}).
Using the torus action, we prove that the mirror algebra is finitely generated (see \cref{thm:finite_generation}).

We follow the setting of \cref{sec:log_CY}.
We have an embedding
\[\Sk(U,\bbZ)\simeq\Sigma^\ess_{(Y,D)}(\bbZ)\subset\Sigma_{(Y,D)}(\bbZ)\subset\bbZ^{I_D}=\Hom(I_D,\bbZ)\]
which we denote by $w$.
We also denote by $w$ the map
\[w\colon N_1(Y)\longrightarrow\bbZ^{I_D}, \quad \gamma\longmapsto(\gamma\cdot D_i)_{i\in I_D}.\]
Let
$T_D\coloneqq\Spec(\bbZ[\bbZ^{I_D}])$ be the split torus with character group $\bbZ^{I_D}$.
Then $w\colon N_1(Y)\longrightarrow\bbZ^{I_D}$ induces a canonical homomorphism $T_D \to T^{N_1(Y)}$, and thus an action of $T_D$ on $\Spec(R)=\Spec\bbZ[\NE(Y)]$.

\begin{lemma} \label{lem:weight}
	Assume $\chi(P_1,\dots,P_n,Q,\gamma) \neq 0$.
	Then
	\[w(Q) + w(\gamma)  = \sum_{j =1}^n w(P_j).\]
\end{lemma}
\begin{proof}
	For each irreducible component $F\subset D$, let $w^F$ denote the corresponding component of $w$.
	Over $\Sk(U,\bbZ)$, we have $w^F=F^\trop$; over $N_1(Y)$, we have $w^F(\gamma)=F\cdot\gamma$.
	So we conclude from \cref{thm:Ftrop_structure_constants}(\ref{thm:Ftrop_structure_constants:equality}).
\end{proof}

Note the mirror algebra $A$, as an abelian group, is free with basis $z^{\gamma} \theta_P$, for $(\gamma,P) \in \NE(Y) \times \Sk(U,\bbZ)$.
\cref{lem:weight} implies the following theorem:

\begin{theorem} \label{thm:torus_action}
	Let $T_D$ act diagonally on the mirror algebra $A$ (viewed as a free abelian group) with weight $w(P) + w(\gamma)$ on the basis vector $z^\gamma\theta_P$.
	This gives an equivariant action of $T_D$ on $\Spec(A) \to \Spec(R)$. 
\end{theorem}

\begin{definition} \label{def:monoid_of_definition}
	Let $\DC(Y)\subset\NE(Y)$ (the notation stands for disk classes) be the submonoid generated by the curve classes of all structure disks, called the \emph{monoid of definition}.
	Let $R_\DC\coloneqq\bbZ[\DC(Y)]\subset\bbZ[\NE(Y)]=R$, and
	\[A_\DC\coloneqq R_\DC^{(\Sk(U,\bbZ))} \coloneqq\bigoplus_{P\in\Sk(U,\bbZ)} R_\DC\cdot\theta_P,\]
	the free $R_\DC$-module with basis $\Sk(U,\bbZ)$.
	We endow $A_\DC$ with an $R_\DC$-algebra structure using the same structure constants as in \eqref{eq:multiplication}.
	We have naturally
	\[ A \simeq A_\DC\otimes_{R_\DC} R.\]
\end{definition}

\begin{lemma} \label{lem:finiteness_fixed_weight}
	Assume there exists an ample divisor $F$ on $Y$ such that $-\overline{F|_U}$ is effective and contains no essential boundary strata (i.e.\ those of $D^\ess$).
	Given $z\in\bbZ^{I_D}$, there are only finitely many $\gamma\in\DC(Y)$ such that $w(\gamma)=z$.
\end{lemma}
\begin{proof}
	Write $F=\overline{F|_U}+F_D$.
	By \cref{prop:interior_divisor}, for any $\gamma\in\DC(Y)$, we have
	\[F\cdot\gamma=F_D\cdot\gamma+\overline{F|_U}\cdot\gamma\le F_D\cdot\gamma.\]
	Note $F_D\cdot\gamma$ is determined by $w(\gamma)\in\bbZ^{I_D}$.
	Therefore, by the ampleness of $F$, there are only finitely many $\gamma\in\DC(Y)$ with given $w(\gamma)=z$.
\end{proof}

The following lemma is an analog of the Krull intersection theorem.

\begin{lemma} \label{lem:Krull_intersection}
	Let $\fm\subset R_\DC$ be the maximal monomial ideal.
	Then $\bigcap_{i>0}\fm^i=0$.
	Consequently, for any free $R_\DC$-module $A$, we have $\cap_{i>0}\fm^i A=0$.
\end{lemma}
\begin{proof}
	The equality $\bigcap_{i>0}\fm^i=0$ is equivalent to the statement that given any $\gamma\in\DC(Y)$, there exists an integer $N(\gamma)$ such that if $\gamma=\gamma_1+\dots+\gamma_n$ for $n>N(\gamma)$ with $\gamma_i\in\DC(Y)$, then $\gamma_j=0$ for some $j$.
	This is true for any submonoid of $\NE(Y)$, by intersecting with an ample class on $Y$.
\end{proof}

\begin{lemma} \label{lem:finite_generation}
	Assume there exists an ample divisor $F$ on $Y$ such that $-\overline{F|_U}$ is effective and contains no essential boundary strata.
	If a set of $\theta_P$ generates $A \otimes_R R/\fm$ as an $R/\fm$-algebra, it also generates $A$ as an $R$-algebra.
\end{lemma}
\begin{proof}
	By \cref{def:monoid_of_definition}, it suffices to prove the theorem for $R_\DC$ and $A_\DC$.
	Hence in the proof, in order to simplify notations, we will temporarily write $R\coloneqq R_\DC$ and $A\coloneqq A_\DC$.
	
	Let $g\in H^0(U,\cO_U)$ such that $g$ has a pole at every irreducible component of $D$, and $-G$ the associated principle Cartier divisor.
	Then $G^\trop\colon\Sk(U)\to\bbR$ is strictly positive away from 0.
	Consider the filtration on $A$ with $A_{\le n}$ having basis $\theta_P$, $G^\trop(P)\le n$.
	By \cref{thm:Ftrop_structure_constants}(\ref{thm:Ftrop_structure_constants:nef}), we have $A_{\le m}\cdot A_{\le n}\subset A_{\le m+n}$.
	
	Let $\Theta\subset\set{\theta_P | P\in\Sk(U,\bbZ)}$ be a subset which generates $A \otimes_{R} R/\fm$.
	Let $A' \subset A$ be the sub-$R$-algebra generated by $\Theta$.
	We have
	\[ A = A' + \fm A.\]
	Iterating the equality above, we obtain
	\begin{equation} \label{eq:AandAprime}
	A = A' + \fm^i A, \text{ for all } i>0.
	\end{equation}
	By \cref{thm:torus_action}, $A$ is a $\bbZ^{I_D}$-graded ring.
	Then $A'$ is also a $\bbZ^{I_D}$-graded ring.
	
	Write $A=\bigoplus_{z\in\bbZ^{I_D}} A_z$ and $A'=\bigoplus_{z\in\bbZ^{I_D}} A'_z$.
	The filtration $A_{\le n}$ is compatible with the $\bbZ^{I_D}$-grading, i.e.\ $A_{\le n}\subset A$ is also $\bbZ^{I_D}$-graded.
	The same holds for $A'$.
	So in order to show that $A=A'$, it suffices to show that $A_{z,\le n}=A'_{z,\le n}$ for every $z\in\bbZ^{I_D}$ and $n\ge 0$.
	
	Now fix $z$ and $n$.
	Equation \eqref{eq:AandAprime} implies that for all $i>0$ we have
	\[A_{z,\le n} = A'_{z,\le n} + \fm^i A \cap A_{z,\le n}.\]
	Since $G^\trop\colon\Sk(U)\to\bbR$ is proper, by \cref{lem:finiteness_fixed_weight}, the set
	\[\Set{(P,\gamma)\in\Sk(U,\bbZ)\times\DC(Y) | G^\trop(P)\le n,\ w(P)+w(\gamma)=z}\]
	is finite.
	In other words, $A_{z,\le n}$ is finitely generated as abelian group.
	Therefore, \cref{lem:Krull_intersection} implies $(\fm^i A)\cap A_{z,\le n}=0$ for $i$ sufficiently large.
	We conclude that $A_{z,\le n}=A'_{z,\le n}$ for all $z\in\bbZ^{I_D}$ and $n\ge 0$, completing the proof.
\end{proof}

\begin{theorem} \label{thm:finite_generation}
	The mirror algebra $A$ is a finitely generated $R$-algebra.
\end{theorem}
\begin{proof}
	Let $(\tY,\tD)$ be another snc compactification of $U$ with a regular map $p\colon\tY\to Y$.
	Let $\tR\coloneqq\bbZ[\NE(\tY)]$, and let $\tA$ be the mirror $\tR$-algebra constructed from the pair $(\tY,\tD)$.
	We have a natural surjection $\tR\twoheadrightarrow R$ induced by $p_*\colon\NE(\tY)\twoheadrightarrow\NE(Y)$.
	This gives an $R$-algebra isomorphism $\tA\otimes_{\tR} R\xrightarrow{\sim} A$ by \cref{prop:blowup}(\ref{prop:blowup:restriction}) (whose proof does not rely on this section).
	Therefore, if $\tA$ is finitely generated as $\tR$-algebra, then $A$ is finitely generated as $R$-algebra.
	Thus, for the purpose of proving finite generation, we are free to make blowups outside $U$.
	Let $F$ be an ample divisor on $Y$ with $-F|_U$ effective (see \cref{lem:ample_divisor}).
	Up to a toric blowup, we can assume $-\overline{F|_U}$ contains no essential boundary strata.
	The ampleness of $F$ can also be preserved under the toric blowup by adding a small multiple of the exceptional divisors.
	Recall that the $R$-algebra $A \otimes_{R} R/\fm$ is finitely generated by \cref{prop:algebra_modulo_m}.
	Hence we conclude the proof by \cref{lem:finite_generation}.
\end{proof}

\section{Change of snc compactification} \label{sec:change_of_snc_compactification}

In this section, we show that the mirror algebra remains \emph{essentially} the same while we change the snc compactification $U\subset Y$.
We make it precise in \cref{prop:blowup_away_from_essential_boundary_strata} in the case when the change of snc compactification does not involve essential boundary strata, and then in \cref{prop:blowup} for the general case.
Consequently, we can define the mirror algebra for arbitrary (non-snc) compactification, and show that after setting all curve classes to 0, the mirror algebra is independent of the choice of compactification, see Remarks \ref{rem:any_compactification}, \ref{rem:A_U}.

\begin{proposition} \label{prop:blowup_away_from_essential_boundary_strata}
	Let $U\subset\tY$ be another snc compactification, together with a regular map $p\colon\tY\to Y$. 
	Then $p_*\colon N_1(\tY) \to N_1(Y)$ induces a surjection $p_*\colon \DC(\tY) \to \DC(Y)$. 
	If the exceptional locus of $p$ is disjoint from the essential boundary strata of $\tY$, then the pullback map $p^*\colon N_1(Y)\to N_1(\tY)$ restricts to an isomorphism of monoids $\DC(Y)\simeq\DC(\tY)$, with inverse $p_*$, and this induces an isomorphism $A_\DC\simeq\tA_\DC$.
\end{proposition}
\begin{proof}
	For any $[f\colon C \to Y^\an] \in \cM(U^\an,\bP,\beta)$ as in \cref{nota:moduli_spaces_analytic}, $f^{-1}(U^\an)\eqqcolon C^{\circ} \subset C$ is a Zariski dense open, so we have a rational map $C \supset C^\circ\to \tY^\an$.
	This is regular (true for any rational map from a curve to a proper variety), and so gives a canonical lift $\tf\colon C \to \tY^\an$.
	Hence the surjectivity of $p_*\colon \DC(\tY) \to \DC(Y)$ follows.
	
	Next let $F \subset \tY$ be the exceptional locus of $p$, and assume this contains no essential boundary strata. 
	By Lemmas \ref{lem:Msm_smooth}(\ref{lem:Msm_smooth:boundary}) and \ref{lem:restriction_to_skeleton}(\ref{lem:restriction_to_skeleton:sm}), the structure disks for $\tY$ are disjoint from $F^\an$, and the structure disks for $Y$ are disjoint from $p(F)^\an$.
	Hence the structure disks for $\tY$ are canonically identified with those of $Y$ via $p$.
	Now we conclude by \cref{lem:curve_class_blowup}.
\end{proof}

Now we will drop the assumption on the exceptional locus of $p$.
Here is the simple heuristic idea:
The increase in the number of parameters is the relative Picard number of $p\colon\tY\to Y$, which is the number of exceptional divisors.
But the mirror algebra for $(\tY,\tD)$ is equivariant for the $T_{\tD}$-action, so the increase in parameters is balanced by the increase in automorphisms.
Let us make this precise.

\begin{notation}
	When $\NE(Y)_\bbR$ is rational polyhedral, so is its dual $\Nef(Y)$.
	In this case, the associated toric variety $\TV(\Nef(Y))$ is given by $\Spec(\bbZ[\NE(Y)]$.
	In general, we just define $\TV(\Nef(Y))\coloneqq\Spec(\bbZ[\NE(Y)]$, which is an equivariant partial compactification of $T^{N_1(Y)}\coloneqq\Spec(\bbZ[N_1(Y)]$.
\end{notation}

Let $E$ be the set of exceptional divisors of $p\colon\tY\to Y$, and $K$ the kernel of $p_*\colon N_1(\tY)\to N_1(Y)$.
By \cref{lem:N1_blowup}, we have $K\simeq\bbZ^E$.
Let $T^E\coloneqq\Spec(\bbZ[\bbZ^E])$.
Let $\cU_{\Nef(Y)}\subset\TV(\Nef(\tY))$ denote the open subset associated to the embedding $\Nef(Y)\subset\Nef(\tY)$.
We have inclusions
\[\TV(\Nef(Y))\hooklongrightarrow\TV(\Nef(Y))\times T^E\xrightarrow{\ \sim\ }\cU_{\Nef(Y)}\hooklongrightarrow\TV(\Nef(\tY))\]
induced by the maps of monoids
\[\NE(Y)\longleftarrow\NE(Y)\oplus \bbZ^E\xleftarrow{\ \sim\ }\NE(\tY)+K\hooklongleftarrow\NE(\tY),\]
where the middle isomorphism follows from \cref{lem:N1_blowup}.

For the compactification $(\tY,\tD)$, we have $\tw\colon\Sk(U,\bbZ)\to\bbZ^{I_{\tD}}$ and $\tw\colon N_1(\tY)\to\bbZ^{I_{\tD}}$ as in the beginning of \cref{sec:torus_action}.
Let $w_E$ be the composition of $\tw$ with the projection
\[\bbZ^{I_{\tD}}=\Hom(I_{\tD},\bbZ)\longto\Hom(E,\bbZ)=\bbZ^E=\chi(T^E).\]

We denote $R\coloneqq\bbZ[\NE(Y)]$ and $\tR\coloneqq\bbZ[\NE(\tY)]$.
Let $\tA$ be the mirror $\tR$-algebra constructed from the pair $(\tY,\tD)$, and $\tcV\coloneqq\Spec(\tA)\to\Spec\tR$, the mirror family.

\begin{proposition} \label{prop:blowup}
	\begin{enumerate}[leftmargin=*]
		\item \label{prop:blowup:restriction} Let $\tR\twoheadrightarrow R$ be the natural surjection induced by $p_*\colon\allowbreak\NE(\tY)\allowbreak\twoheadrightarrow\NE(Y)$.
		This gives an $R$-algebra isomorphism $\tA\otimes_{\tR} R\xrightarrow{\sim }A$, sending $\theta_P$ to $\theta_P$ for each $P\in\Sk(U,\bbZ)$.
		In terms of spaces, this gives an isomorphism
		\[\tcV|_{\TV(\Nef(Y))}\simeq\cV.\]
		\item \label{prop:blowup:extension} The $T_{\tD}$-action on $\tcV$ gives an isomorphism
		\[\iota\colon\cV\times T^E\xrightarrow{\ \sim\ }\tcV|_{\cU_{\Nef(Y)}}.\]
		This induces an isomorphism of rings
		\[A\otimes_{\bbZ}\bbZ[\bbZ^E]\xleftarrow{\ \sim\ }\tA\otimes_{\tR}\bbZ[\NE(\tY)+K]\]
		which sends $z^\gamma\cdot\theta_P$ to $z^{p_*(\gamma)}\cdot\theta_P\otimes x^{w_E(P)+w_E(\gamma)}$.
		\item \label{prop:blowup:rho}
		Let $\pi_\cV\colon\cV\times T^E\to\cV$ denote the projection, and $\rho\coloneqq\pi_\cV\circ\iota\inv\colon\tcV|_{\cU_{\Nef(Y)}}\to\cV$.
		For each $n\ge 2$, let $\braket{,\dots,}_n$ be the $R$-multilinear map for $(Y,D)$ defined above \cref{thm:main}, and $\braket{,\dots,}^\sim_n$ the $\tR$-multilinear map for $(\tY,\tD)$.
		Then for all $a_1,\dots,a_n\in A$, we have
		\[\rho^*(\braket{a_1,\dots,a_n}_n)=\braket{\rho^* a_1,\dots,\rho^* a_n}^\sim_n,\]
		where the right hand side is extended over $\bbZ[\NE(\tY)+K]$ by multilinearity.
	\end{enumerate}
\end{proposition}
\begin{proof}
	For (\ref{prop:blowup:restriction}), it suffices to show that
	\begin{equation} \label{eq:blowup_restriction}
	\sum_{p_*\tgamma=\gamma}\tchi(P_1,\dots,P_n,Q,\tgamma)=\chi(P_1,\dots,P_n,Q,\gamma)
	\end{equation}
	for all $P_1,\dots,P_n,Q\in\Sk(U,\bbZ)$ and $\gamma\in\NE(Y)$.
	Let $\delta$ and $\tdelta$ be respectively the classes of the tails.
	We have $p_*\tdelta=\delta$.
	By \cref{lem:N1_blowup} and \cref{rem:compatible_curve_class}, there is a unique $\tgamma\in\NE(\tY)$ such that $\tgamma+\tdelta$ is compatible with $(P_1,\dots,P_n,-Q)$ and $p_*(\tgamma+\tdelta)=\gamma+\delta$.
	Put $\beta\coloneqq\gamma+\delta$ and $\tbeta\coloneqq\tgamma+\tdelta$.
	Let $H(\bP_Z,\beta)$ and $\tH(\bP_Z,\tbeta)$ be as in \cref{sec:intro:structure_constants} for the compactifications $U\subset Y$ and $U\subset\tY$ respectively.
	These two moduli spaces are isomorphic as they consist of exactly the same sets of punctured curves in $U$.
	Furthermore, the operations of taking fiber and imposing the toric tail condition in \cref{sec:intro:structure_constants} cut out the same 0-dimensional subspaces in $H(\bP_Z,\beta)$ and $\tH(\bP_Z,\tbeta)$, whose lengths give respectively the structure constants $\chi(P_1,\dots,P_n,Q,\gamma)$ and $\tchi(P_1,\dots,P_n,Q,\tgamma)$.
	Hence we obtain equation \eqref{eq:blowup_restriction}.
	Intuitively speaking, the analytic disks responsible for the structure constants only feel the interior $U$ rather than any compactification, except that the definition of curve classes relies on a chosen compactification.	
	
	Statement (\ref{prop:blowup:extension}) follows from the equivariant $T_{\tD}$-action of $\tcV$ (see \cref{thm:torus_action}).
	
	For (\ref{prop:blowup:rho}), we abbreviate $\tA\otimes_{\tR}\bbZ[\NE(\tY)+K]$ by $\tA[K]$; note this is a localization, inverting all monomials $z^\gamma$ for $\gamma\in K$.
	Since $\rho^*\colon A\to\tA[K]$ is a morphism of rings, for all $a_1,\dots,a_n\in A$, we have
	\begin{equation} \label{eq:blowup_product}
	\rho^*(a_1\cdots a_n)=\rho^*a_1\cdots\rho^*a_n.
	\end{equation}
	For each element $a\in\tA[K]$, write
	\[a=\sum_{P\in\Sk(U,\bbZ)}\Coeff_{\theta_P}(a)\cdot\theta_P.\]
	Equation \eqref{eq:blowup_product} implies
	\begin{equation} \label{prop:blowup:rho:eq:theta_0}
	\Coeff_{\theta_0}\rho^*(a_1\cdots a_n)=\Coeff_{\theta_0}(\rho^*a\cdots\rho^*a_n).
	\end{equation}
	Since $w_E(0)=0$, by the formula in statement (\ref{prop:blowup:extension}), we have $\rho^*\theta_0=\theta_0$; hence
	\[\Coeff_{\theta_0}\rho^*(a_1\cdots a_n)=\rho^*\Coeff_{\theta_0}(a_1\cdots a_n).\]
	By \cref{lem:structure_constants_trace}, we have
	\begin{align*}
	&\Coeff_{\theta_0}(a_1\cdots a_n)=\braket{a_1,\dots,a_n}_n,\\
	&\Coeff_{\theta_0}(\rho^* a_1\cdots\rho^* a_n)=\braket{\rho^*a_1,\dots,\rho^*a_n}^\sim_n.
	\end{align*}
	Therefore, statement (\ref{prop:blowup:rho}) follows from equation \eqref{prop:blowup:rho:eq:theta_0}.
\end{proof}

\begin{remark} \label{rem:any_compactification}
	Given any compactification $U\subset Z$ with $Z$ projective and normal,
		we define $A_Z\coloneqq A_{\tY}\otimes_{R_{\tY}} R_Z$, and $\DC(Z) \coloneqq \pi_*(\DC(\tY)) \subset \NE(Z)$, for any snc compactification $U\subset\tY$ (satisfying \cref{ass:strata}) with $\pi\colon\tY\to Z$ regular.
	In view of Propositions \ref{prop:blowup_away_from_essential_boundary_strata} and \ref{prop:blowup}(\ref{prop:blowup:restriction}), we see that they are independent of the choice of $U\subset \tY$.
\end{remark}

\begin{remark} \label{rem:A_U}
	Our mirror algebra $A_Y$ depends on the compactification $U \subset Y$. However, we can define $A_U \coloneqq A_Y \otimes_{R_Y} \bbZ$, where $R_Y \to \bbZ$ sends each $z^\gamma$ to $1$ (this corresponds to the fiber of the mirror family over the identity point of $T^{N_1(Y)}$).
	By \cref{prop:blowup}(\ref{prop:blowup:restriction}), $A_U$ is independent
	of the choice of $U \subset Y$.
\end{remark}

\section{Non-degeneracy of the trace map} \label{sec:non-degeneracy} 

In this section, we prove the non-degeneracy of the trace map (see \cref{thm:main}(\ref{thm:main:non-degeneracy})).
We will show that for any Zariski open $V\subset U$ which is also log Calabi-Yau, there is a natural degeneration of the mirror algebra for $U$ to that for $V$, see \cref{prop:flat_degeneration}.
This relies on the positivity properties of the structure disks in \cref{prop:structure_disk_interior_divisor}.
We will reduce the question of non-degeneracy for $U$ to that for $V$, and eventually to the toric case, where non-degeneracy is evident.

Let $V \subset U$ be an affine Zariski open subset which itself contains an algebraic torus (note this in particular implies $V$ is log Calabi-Yau).
Note any compactification of $U$ also compactifies $V$.
Let $\DC(V\subset Y)$ be the monoid of definition for $V$ (see \cref{def:monoid_of_definition}) and $A_\DC(V\subset Y)$ the mirror algebra for $V$ over $\bbZ[\DC(V\subset Y)]$.
By \cref{rem:any_compactification}, they make sense even though $V \subset Y$ need not be an snc compactification.

\begin{proposition} \label{prop:flat_degeneration}
	Assume
	$Z \coloneqq \overline{U \setminus V}$ contains no essential boundary strata.
	Then the following hold:
	\begin{enumerate}
	\item \label{prop:flat_degeneration:intersection} Each irreducible component of $Z$ has non-negative intersection with each element of $\DC(U\subset Y)$.
	The intersection is zero if and only if any associated structure disk $\bbD \to Y^{\an}$ factors through $(Z^c)^{\an} \subset Y^{\an}$, and thus is a structure disk for $V\subset Y$.
	\item \label{prop:flat_degeneration:monoid} We have $Z^\perp \cap \DC(U\subset Y) = \DC(V\subset Y)$.
	\item \label{prop:flat_degeneration:base} Let $I_Z \subset R$  be the monoidal ideal generated by disk classes with positive intersection with $Z$.
	We have $\bbZ[\DC(U\subset Y)]/I_Z \simeq \bbZ[\DC(V\subset Y)]$.
	\item \label{prop:flat_degeneration:algebra} The isomorphism above induces
	\[A_\DC(U\subset Y)\otimes_{\bbZ[\DC(U\subset Y)]} \bbZ[\DC(V\subset Y)]\simeq A_\DC(V\subset Y).\]
	\end{enumerate}
\end{proposition}
\begin{proof}
	(\ref{prop:flat_degeneration:intersection}) follows from \cref{prop:structure_disk_interior_divisor}.
	This implies (\ref{prop:flat_degeneration:monoid}), and (\ref{prop:flat_degeneration:monoid}) implies (\ref{prop:flat_degeneration:base}).
	As for (\ref{prop:flat_degeneration:algebra}), modulo $I_Z$, the only structure disks which contribute are those whose class has zero intersection with $Z$.
	By \cref{prop:structure_disk_interior_divisor}, these are exactly the structure disks for $V\subset Y$; this implies (\ref{prop:flat_degeneration:algebra}).
\end{proof}

For each $n\ge 2$, let $\braket{,\dots,}_n$ be the $R$-multilinear map for $(Y,D)$ defined above \cref{thm:main}.
Since $\braket{a_1,\dots,a_n,\theta_0}_{n+1}=\braket{a_1,\dots,a_n}_n$, for the property of non-degeneracy, it suffices to consider the case $n=2$.
So we now restrict to $n=2$ and drop the subscript $n$ from the notation.

\begin{lemma} \label{lem:non-degeneracy_change_compactification}
	Let $U \subset \tY$ be another projective snc compactification such that $p\colon \tY \to Y$ is regular.
	Then the non-degeneracy of the $\tR$-bilinear map $\braket{,}^\sim$ for $(\tY,\tD)$ implies the non-degeneracy of the $R$-bilinear map $\braket{,}$ for $(Y,D)$.
\end{lemma}
\begin{proof}
	Assume $\braket{,}^\sim$ is non-degenerate.
	Then for any $a\in A$, there exists $P\in\Sk(U,\bbZ)$ such that $\braket{\theta_P,\rho^*a}^\sim\neq 0$, ($\rho$ as in \cref{prop:blowup}(\ref{prop:blowup:rho})).
	By \cref{prop:blowup}(\ref{prop:blowup:extension}-\ref{prop:blowup:rho}), we obtain
	\[\rho^*\braket{\theta_P,a}=\braket{\rho^*\theta_P,\rho^*a}^\sim=z^{-w_E(P)}\braket{\theta_P,\rho^*a}^\sim.\]
	Hence $\braket{\theta_P,a}\neq 0$.
	Therefore, the pairing $\braket{,}$ is non-degenerate, completing the proof.
\end{proof}

For the proof of non-degeneracy, we will consider the mirror algebra over various monoidal coefficient rings.
For any map of monoids $\DC\coloneqq\DC(U\subset Y)\to P$, we write $A_P \coloneqq A_\DC(U\subset Y){\otimes}_{\bbZ[\DC]} \bbZ[P]$.
We say that $A_P$ is non-degenerate if the induced $\bbZ[P]$-bilinear pairing $\braket{,}$ on $A_P$ (as free $\bbZ[P]$-module) is non-degenerate.

\begin{lemma} \label{lem:base_extension}
	\begin{enumerate}[leftmargin=*]
		\item \label{lem:base_extension:injective} Given maps of monoids $\DC\to P_1$, $\DC\to P_2$ and $P_1\to P_2$, if $P_1 \to P_2$ is injective and $A_{P_2}$ is non-degenerate, then $A_{P_1}$ is non-degenerate.
		\item \label{lem:base_extension:gp} Given a map of monoids $\DC\to P$ such that $P\subset P^\gp$, then $A_P$ is non-degenerate if and only if $A_{P^\gp}$ is non-degenerate.
		\item \label{lem:base_extension:toric} If $U$ is a split algebraic torus, then $A_{P}$ is non-degenerate for any map of monoids $\DC\to P$.
	\end{enumerate}
\end{lemma}
\begin{proof}
	For (\ref{lem:base_extension:injective}), given $a\in A_{P_1}$, we have a commutative diagram
	\[\begin{tikzcd}[column sep=huge]
	A_{P_1}  \ar{r}{a \to \braket{a,\cdot}} \ar{d}{} & \Hom(A_{P_1},\bbZ[P_1]) \simeq \prod_{q \in\Sk(U,\bbZ)} \bbZ[P_1]
	\ar{d}{} \\
	A_{P_2}  \ar{r}{a \to \braket{a,\cdot}}  &\Hom(A_{P_2},\bbZ[P_2]) \simeq \prod_{q \in\Sk(U,\bbZ)} \bbZ[P_2].
	\end{tikzcd}\]
	If $P_1 \to P_2$ is injective, then so are both vertical maps.
	Hence (\ref{lem:base_extension:injective}) follows.
	For any $a,b\in A_{P^\gp}$ and $p\in P^\gp$, we have $\braket{z^p a, b} = z^p\braket{a,b} = \braket{a, z^p b}$, from which (\ref{lem:base_extension:gp}) follows.
	For (\ref{lem:base_extension:toric}), if $U\simeq T_M$, $M$ being the cocharacter lattice, then for any $a,b\in\Sk(U,\bbZ)\simeq M$, by \cref{lem:counts_in_toric_case}, we have $\theta_a \cdot \theta_b = z^{p(a,b)} \theta_{a + b}$ for some $p(a,b) \in P$, where the addition $a+b$ is computed in $M$.
	Thus $\braket{\theta_a,\theta_{b}} = z^{p(a,b)} \delta_{a,-b}$, so the non-degeneracy is clear.
\end{proof}

By \cref{lem:non-degeneracy_change_compactification}, for the purpose of proving the non-degeneracy, we are free to replace $(Y,D)$ by any toric blowup, so we can assume $Z\coloneqq\overline{U\setminus T_M}$ contains no essential boundary strata.
Let $Z_1,\dots,Z_n \subset Z$ be the divisorial irreducible components.
By \cref{prop:flat_degeneration}, we have
\[A_\DC\otimes_{\bbZ[\DC]} \bbZ[\braket{Z_1,\dots,Z_n}^\perp\cap\DC]\simeq A_\DC(T_M\subset Y),\]
where $^\perp$ denotes the orthogonal complement in $N_1(Y)$.
Tensoring with $\bbZ[\braket{Z_1,\dots,Z_n}^\perp]$ over $\bbZ[\braket{Z_1,\dots,Z_n}^\perp\cap\DC]\simeq\bbZ[\DC(T_M\subset Y)]$, we obtain
\[A_\DC\otimes_{\bbZ[\DC]} \bbZ[\braket{Z_1,\dots,Z_n}^\perp]\simeq A_\DC(T_M\subset Y)\times_{\bbZ[\DC(T_M\subset Y)]} \bbZ[\braket{Z_1,\dots,Z_n}^\perp].\]
By \cref{lem:base_extension}(\ref{lem:base_extension:toric}), the right hand side is non-degenerate, hence so is the left hand side.

Now we prove by decreasing induction on $i$ that $A_\DC\otimes_{\bbZ[\DC]} \bbZ[\braket{Z_1,\dots,Z_i}^\perp]$ is non-degenerate, where we interpret the $i=0$ case to be $A_\DC\otimes_{\bbZ[\DC]} \bbZ[N_1(Y)]$.
By Lemma \ref{lem:base_extension}(\ref{lem:base_extension:gp}), non-degeneracy for $N_1(Y)$ implies it for any $Q \subset N_1(Y)$ that generates $N_1(Y)$ as a group, e.g.\ $Q = \NE(Y)$.
So it remains to establish the inductive step.
Assume
$A_\DC\otimes_{\bbZ[\DC]}\bbZ[\braket{Z_1,\dots,Z_i}^\perp]$ is non-degenerate.
If $Z_i$ is trivial on $\braket{Z_1,\dots,Z_{i-1}}^\perp$, there is nothing to prove.
Otherwise $P\coloneqq\braket{Z_1,\dots,Z_{i-1}} ^\perp \cap\{Z_i \geq 0\}$ generates $\braket{Z_1,\dots,Z_{i-1}}^\perp$ as a group; so by \cref{lem:base_extension}(\ref{lem:base_extension:gp}), it is enough to prove non-degeneracy of $A_P$.
Let $Q\coloneqq\braket{Z_1,\dots,Z_i}^\perp$.
Note there is a natural surjection $\bbZ[P] \twoheadrightarrow \bbZ[Q]$, and the kernel, generated by
the monomials $z^c$ with $Z_i \cdot c > 0$, is principal, which is generated by such a monomial with $Z_i \cdot c$ minimal.

Take $0 \neq a = \sum r_P \theta_P \in A_P$, with $r_P \in \bbZ[P]$.
We argue that $\braket{a, \cdot } \neq 0$.
By the discussion above, the kernel of $A_P\twoheadrightarrow A_Q$ is $(z^c)\cdot A_P$.
Thus we can write $a = f b$ for some $0 \neq f \in (z^c) \subset \bbZ[P]$, and
\[0 \neq \ob  \in  A_Q\simeq A_P / ((z^c)\cdot A_P).\]
We have $\braket{a,\cdot}=f \braket{b,\cdot}$.
Since $A_P$ is $\bbZ[P]$-torsion free, in order to show that $\braket{a,\cdot}\neq 0$, its enough to show $\braket{b,\cdot } \neq 0$.
As $\braket{\ob,\cdot} \neq 0$ by induction, we deduce that $\braket{b,\cdot} \neq 0$.
This completes the proof of non-degeneracy.

\section{Geometry of the mirror family} \label{sec:geometry_of_the_mirror_family}

In this section, we study the geometry of the fibers of the mirror family $\cX\coloneqq\Spec A\to \Spec R$ over $\bbQ$, see \cref{prop:geometry_of_fiber}.
This implies \cref{thm:main}(\ref{thm:main:family}) in the introduction.

The idea is to use the equivariant torus action of \cref{sec:torus_action} on the mirror family.
It pushes every fiber to the fiber over the unique torus fixed point of the base.
This fiber has an explicit combinatorial description (\cref{prop:algebra_modulo_m}) so the geometric properties can be checked directly.
In order to deduce the geometric properties of its nearby fibers, and hence all fibers via the torus action, we must compactify all the fibers.
The fiberwise compactification is constructed via Proj of a graded ring, which we now describe.

By \cref{prop:blowup}, we are free to choose any different snc compactification $U\subset Y$.
So by \cref{lem:ample_boundary} we can assume $D$ supports an effective ample divisor $F$.
By \cref{lem:divisor_tropicalization}, $F^\trop\colon\Sk(U)\to\bbR$ is strictly positive away from 0.
Then we can construct a natural compactification of the mirror family as follows.

Consider the filtration on $A$ with $A_{\le n}$ having basis $\theta_P$, $F^\trop(P)\le n$.
By \cref{thm:Ftrop_structure_constants}(\ref{thm:Ftrop_structure_constants:nef}), we have $A_{\le m}\cdot A_{\le n}\subset A_{\le m+n}$.
Let $\tA \subset A[T]$ be the associated $\bbN$-graded $R$-algebra, having basis $T^n\cdot\theta_P$ with $F^\trop(P)\le n$.
Let $\ocX \coloneqq \Proj(\tA) \to \Spec(R)$, called the \emph{compactified mirror family}.
The function $T = T \cdot \theta_0$ gives a canonical section $T \in H^0(\ocX,\cO(1))$.
Let $\cZ \coloneqq \Proj(\tA/ (T\tA)) \subset \ocX$ be the associated zero scheme. 
Since $\tA/ (T\tA)$ is a free $R$-module with basis $T^{F^\trop(P)} \cdot \theta_P$, the family $\cZ \to \Spec(R)$ is also flat.
There is a canonical identification $A \simeq \tA_{(T)}$, where the right hand side denotes the subring of degree zero elements in the localization at $(T)$.
This gives $\ocX \setminus \cZ \simeq \cX$.

\begin{proposition} \label{prop:geometry_of_fiber}
	Let $(\ocX_\bbQ,\cZ_\bbQ) \to \Spec(R_\bbQ)$ denote the base change of the compactified mirror family to $\bbQ\supset\bbZ$.
	For any fiber $(\oX,Z)$ of this family, the following hold:
	\begin{enumerate}
		\item $\oX$ is Cohen-Macaulay, $Z \subset \oX$ is reduced, $(\oX,Z)$ is semi-log-canonical, and $K_{\oX} + Z$ is trivial.
		\item $X \coloneqq \oX \setminus Z$ is affine of the same dimension as $U$,  Gorenstein, semi-log-canonical and $K$-trivial.
			\end{enumerate}
	Moreover, the generic fiber $\oX$ is normal and $(\oX,Z)$ is log canonical.
	The generic fiber $X$ is log canonical and log Calabi-Yau.
\end{proposition}
\begin{proof}
	Let $s\in\Spec R_\bbQ$ be the point associated to the maximal monomial ideal.
		Observe that the equivariant $T_D$-action on $\cX=\Spec(A)\to\Spec R$ of \cref{thm:torus_action} extends also to the compactified mirror family $\ocX\to\Spec R$.
	By \cite[p.\ 38 Claims 1-2]{Fulton_Introduction_to_toric_varieties}, the ample Cartier divisor $F$ determines a one-parameter subgroup $\Gm \subset T_D$ that pushes any point of $\Spec(R_\bbQ)$ towards the 0-stratum, i.e.\ the point $s$.
	The closure of the orbit of this one-parameter subgroup gives a map $\bbA^1 \to \Spec(R_\bbQ)$ sending $0 \mapsto s$.
	We pullback the compactified mirror family to $\bbA^1$ and prove properties (1-2) for the pullback family.
	By \cite[Lemma 8.33]{Gross_Canonical_bases}\footnote{In that lemma the central fiber is assumed normal, but the result extends, with the same proof, to the present Cohen-Macaulay case.
	}, we see that the two properties are open.
	So by the $\Gm$-action, it is enough to prove these properties for the central fiber $\oX_s$.
	Using the explicit description of $\oX_s$ by \cref{prop:algebra_modulo_m}, we deduce these properties from \cite[Theorem 1.2.14]{Alexeev_Complete_moduli} and \cite[Definition-Lemma 5.10]{Kollar_Singularities_of_the_minimal_model_program}.
		
	Next we prove the statement concerning the generic fiber.
	Recall that log canonical is equivalent to normal plus semi-log-canonical, so it is enough to prove that the generic fiber is normal.
	As the locus of normal fibers is open, it suffices to exhibit a single normal fiber.
	Hence by \cref{prop:flat_degeneration}, we reduce to the case $U = T_M$.
	In this case, by \cref{lem:counts_in_toric_case}, every fiber over $T^{N_1(Y)} \subset \Spec(R)$ is a projective toric variety (associated to the rational polytope $\{F^\trop\le 1\}$), which is in particular normal, completing the proof.
\end{proof}

\begin{lemma} \label{lem:ample_boundary}
	Let $U \subset Y$ be a compactification of an affine variety over a field of characteristic zero.
	Then there is a proper birational map $\tY \to Y$ which is an isomorphism over $U$, such that $U \subset \tY$ is an snc compactification, and that the boundary $\tY \setminus U$ supports an ample effective divisor.
\end{lemma} 
\begin{proof}
	By assumption we have a closed embedding $U \subset \bbA^n$.
	Taking closure in the projective space $\bbP^n$ gives a compactification $U \subset Z$ with boundary the support of an ample Cartier divisor $A$.
	By the resolution of singularities (see \cite[Theorems 3.26-3.27]{Kollar_Lectures_on_resolution_of_singularities}), we can make $U \subset Z$ snc, and $Z \dasharrow Y$ regular, by an iterated sequence of blowups with centers in the complement of $U$.
	If $F \subset Z$ is an ample effective divisor with support $U^c$, then so is $\pi^{-1}(F) - \epsilon E$, for $E$ the exceptional divisor, and $\epsilon$ sufficiently small. 
\end{proof}

\section{Scattering diagram via infinitesimal analytic cylinders} \label{sec:scattering}

In this section, we give a direct geometric construction of a scattering diagram on the skeleton $\Sk(U)$ by counting infinitesimal analytic cylinders (see \cref{sec:intro:scattering} for a summary).
This differs and generalizes the combinatorial construction of scattering diagrams in \cite{Gross_Mirror_symmetry_for_log_Calabi-Yau_surfaces_I_v1,Gross_Canonical_bases}.
It will be the key to connecting our mirror algebra to cluster algebras in the next section.
The main results here are the wall-crossing homomorphism (\cref{thm:wall-crossing_homomorphism}), existence of wall-crossing function (\cref{prop:wall-crossing_function}), finite polyhedral approximation (\cref{prop:wall_decomposition}), theta function consistency (\cref{prop:theta_function_consistent}) and Kontsevich-Soibelman consistency (\cref{prop:KS_consistent}).

We define the counts of infinitesimal analytic cylinders in \cref{def:infinitesimal_spine}.
Using these counts, the wall-crossing transformations are constructed in Definitions \ref{def:wall-crossing_map}, \ref{def:wall-crossing_transformation},  and \ref{def:wall-crossing_transformation_extension}, via several extension steps, based on the results of adic convergence, wall-crossing homomorphism and existence of wall-crossing function.
The wall-crossing transformations give rise to a scattering diagram with finite polyhedral approximations.
In \cref{sec:consistency}, we introduce local theta functions, and prove the theta function consistency, i.e.\ the compatibility between the local theta functions and the scattering diagram.
We are obliged to separate the local theta functions into three parts in \cref{prop:theta_function_consistent}, because the wall-crossing functions are not invertible with curve classes.
In \cref{sec:setting_the_curve_classes_to_0}, we set all the curve classes to 0.
Then we show that the wall-crossing functions become invertible, and we obtain a consistent scattering diagram in the sense of Kontsevich-Soibelman.

\subsection{Construction} \label{sec:scattering:construction}

Let $T_M\subset U\subset Y$ be as in \cref{sec:log_CY}, and $N\coloneqq\Hom(M,\bbZ)$.

\begin{definition} \label{def:generic}
	Given $n\in N\setminus 0$, we say that a point $x\in n^\perp\subset M_\bbR$ is \emph{generic} if it does not lie in any other rational hyperplanes passing through 0.
\end{definition}

\begin{definition} \label{def:infinitesimal_spine}
	Given $n\in N\setminus 0$, $x\in n^\perp\subset M_\bbR$ generic, $v,w\in M\setminus n^\perp$, let $h\colon I^\epsilon\coloneqq[-\epsilon,\epsilon]\to M_\bbR$ be the continuous map, affine over $[-\epsilon,0]$ and $[0,\epsilon]$, with $h(0)=x$, $dh(-\epsilon)=-v$, $dh(\epsilon)=-w$.
	Given $\alpha\in\NE(Y)$, choose $A\in\bbN$ big with respect to $\alpha$.
	Shrink $\epsilon$ so that $h(I^\epsilon)$ does not meet $\Wall_A$ except at $x$.
	We obtain a spine $V^\epsilon_{x,v,w}\coloneqq[I^\epsilon,(-\epsilon,\epsilon),h]$ in $M_\bbR$ as in \cref{def:spine}.
	Define $N(V_{x,v,w},\alpha)\coloneqq N(V^\epsilon_{x,v,w},\alpha)$ as in \cref{def:naive_count_transverse}.
	It is independent of $\epsilon$ for $\epsilon$ sufficiently small.
\end{definition}

\begin{lemma} \label{lem:w-v}
	If $N(V_{x,v,w},\alpha)\neq 0$, then $w-v\in n^\perp$.
\end{lemma}
\begin{proof}
	If $w-v\neq 0$, then the tropical curve associated to any analytic curve contributing to $N(V^\epsilon_{x,v,w},\alpha)$ has twigs attached to $0\in I^\epsilon$.
	So $x=h(0)$ lies in $\Wall_A$.
	Since $x$ is generic, if $x\in\sigma$ for some open cell $\sigma\subset\Wall_A$, then $\sigma$ is of codimension 1 in $M_\bbR$, and $\sigma$ is contained in $n^\perp$; in particular the monomial of any twig starting from $x$ lies in $n^\perp$.
	Since $w-v$ is the sum of the monomials of all twigs attached to $0\in I^\epsilon$, we conclude that $w-v\in n^\perp$.
\end{proof}

\begin{definition} \label{def:wall-crossing_map}
	Given $n\in N\setminus 0$, $x\in n^\perp\subset M_\bbR$ generic and $v\in M$ with $\braket{n,v}>0$, we define the following \emph{wall-crossing transformation}, the formal sum:	\begin{equation} \label{eq:wall-crossing}
	\Psi_{x,n}(z^v) \coloneqq \sum_{\substack{w \in M,\; \braket{n,w}>0\\ \alpha \in\NE(Y)}} N(V_{x,v,w},\alpha) z^{\alpha} z^{w}.
	\end{equation}
	\end{definition}

We show first that this converges in a natural adic topology.
Fix a strictly convex toric monoid $Q\subset N_1(Y)$ containing $\NE(Y)$.
Let $Q_1\coloneqq Q \setminus 0$.
For $k\ge 1$, let $Q_k \subset Q_1$ be the set of elements which are sums of $k$ elements in $Q_1$.
Let $I \subset \bbZ[Q\oplus M]$ be the maximal monoid ideal, i.e.\ the ideal  generated by monomials $z^{(q,m)}$, $q \in Q_1$, $m\in M$.
Let $\hR$ be the $I$-adic completion of $\bbZ[Q\oplus M]$.

\begin{lemma} \label{lem:bend_bound}
	Given $k>0$, there are at most finitely many pairs $(m,n)\in M\times N$ satisfying the following:
	\begin{enumerate}
	\item $m\in n^\perp\setminus 0$;
	\item $n$ is primitive;
	\item $N(V_{x,v,v+m},\alpha)\neq 0$ for some $x\in n^\perp$ generic, $v\in M\setminus n^\perp$ and $\alpha\in Q\setminus Q_k$.
	\end{enumerate}
	Moreover, there exists $A\in\bbN$, such that for any $N(V_{x,v,v+m},\alpha)\neq 0$ as above, $x\in\Wall_A$ and $n^\perp$ is the span of a polyhedral cell in $\Wall_A$.
\end{lemma}
\begin{proof}
	We consider the twigs attached to the tropical curves associated to the analytic curves contributing to $N(V_{x,v,v+m},\alpha)$.
	Since $Q \setminus Q_k$ is finite, by \cref{lem:hgamma} and \cref{prop:tropicalization_of_stable_map}, the degrees of such twigs are bounded by some $A\in\bbN$.
	So by \cref{rem:twig_bound}, there are only finitely many combinatorial types for such twigs.
	This bounds $m$.
	Furthermore, by \cref{const:walls_by_induction}, $\Wall_A$ contains only finitely many cells, and the image of any twig of degree bounded by $A$ is contained in $\Wall_A$.
	In particular, $N(V_{x,v,v+m},\alpha)\neq 0$ implies that $x\in\Wall_A$.
	Since $x\in n^\perp$ is generic, we deduce that $n^\perp$ is the span of a polyhedral cell in $\Wall_A$.
	This also bounds $n$.
\end{proof}

\begin{lemma} \label{lem:wall-crossing_bound}
	Given $k>0$, there are at most finitely many $n\in N\setminus 0$ primitive such that $\Psi_{x,n}(z^v)\neq z^v$ modulo $I^k$ for some generic $x\in n^\perp$ and $v\in M\setminus n^\perp$.
	Given $k>0$, $n\in N\setminus 0$, $x\in n^\perp\subset M_\bbR$ generic, and $v\in M$ with $\braket{n,v}>0$, there are at most finitely many $w$ with $N(V_{x,v,w},\alpha)\neq 0$ for some $\alpha\in Q\setminus Q_k$.
	Hence the formal sum $\Psi_{x,n}(z^v)$ in \eqref{eq:wall-crossing} lies in $\hR$.
\end{lemma}
\begin{proof}
	It follows directly from \cref{lem:bend_bound}.
\end{proof}

\begin{definition} \label{def:wall-crossing_transformation}
	Let $n\in N\setminus 0$, $x \in n^\perp\subset M_\bbR$ generic, and $v\in M$ with $\braket{n,v}\ge 0$.
	If $\braket{n,v}>0$, let $\Psi_{x,n}(z^v)$ be as in \eqref{eq:wall-crossing}.
	If $\braket{n,v}=0$, let $\Psi_{x,n}(z^v)=z^v$.
	And we further extend to $\Psi_{x,n}\colon\bbZ[\NE(Y)\oplus M_{\braket{n,\cdot}\ge 0}]\to\hR$ by linearity over $\bbZ[\NE(Y)]$.
\end{definition}

\begin{theorem}\label{thm:wall-crossing_homomorphism}
	The map $\Psi_{x,n}$ is a ring homomorphism.
\end{theorem}
\begin{proof}
	We pick $m_1,m_2 \in M_{\braket{n,\cdot}\ge 0}$ and show $\Psi_{x,n}(z^{m_1 + m_2}) = \Psi_{x,n}(z^{m_1}) \cdot \Psi_{x,n}(z^{m_2})$.
	It is obvious when both are in $n^\perp$, so we may assume $\braket{n,m_1} > 0$ and $\braket{n,m_2} \geq 0$.
	It is enough to fix $k$ and prove the equality modulo $I^k$.
	Let $A$ be as in \cref{lem:bend_bound}.
	Now we fix $e \in M$ and prove the equality of the $z^e$ coefficients (as elements of $\bbZ[\NE(Y)]$).
	
	Let $\Gamma_0$ be a metric tree with three 1-valent vertices $v_1,v_2,v_3$, one 3-valent vertex $z$, and all edges having a same length $\epsilon$.
	Let $w$ be the midpoint of $[z,v_3]$, and glue $[0,v_4\coloneqq+\infty]$ to $\Gamma_0$, along $0$ and $w$.
	We denote the resulting metric tree by $\Gamma$.
	
	Let $\bP=(m_1,m_2,-e,0)$, $J=\{1,2,3,4\}$, $I=\{4\}$, $F = B = \{1,2,3\}$.
	Let $\gamma\colon[-\delta,\delta]\to M_\bbR$ be a linear segment such that $\operatorname{Im}\gamma\cap\Wall_A=\emptyset$, $\gamma(0)+\frac{\epsilon}{2}e=x$ and $\braket{n,\gamma(t)+\frac{\epsilon}{2}e}\cdot t < 0$ for $t\neq 0$.
	
	Let $\SP(M_\bbR,\bP^F)_\Gamma\subset\SP(M_\bbR,\bP^F)$ be the subset of spines with domain $\Gamma$.
	Let $\ev\colon\SP(M_\bbR,\bP^F)_\Gamma\allowbreak\to M_\bbR$ be evaluation at $v_4$.	
	By \cref{rem:internal_leg_contracted}, if $[\Gamma,(v_1,\dots,v_4),h]\in\SP(M_\bbR,\bP^F)_\Gamma$ with $h(v_4)\in\operatorname{Im}\gamma$, then $h$ is constant on the leg incident to $v_4$; in particular $h(v_4)=h(w)$.
	If $x\in\Wall_A$, let $\sigma$ be the open cell of $\Wall_A$ containing $x$; otherwise let $\sigma\coloneqq\emptyset$.
	Shrinking $\epsilon$ and $\delta$, we may assume the image of any spine in $\ev\inv(\operatorname{Im}\gamma)$ does not meet $\Wall_A\setminus\sigma$.
	For $t\in[-\delta,\delta]$, let $\SP_{\gamma(t)}\coloneqq\ev\inv(\gamma(t))$.
	Apply \cref{prop:deformation_invariance_skeletal} to a sufficiently small open neighborhood of $\{\Gamma\}\times\operatorname{Im}\gamma$ inside $\NT_J^F\times M_\bbR$, we obtain:
	\begin{claim} \label{claim:independent_of_t}
		\[\sum_{S\in \SP_{\gamma(t)}} N(S,\alpha)\]
		is independent of $t$.
	\end{claim}
	
	Note the following combinatorial observation:
	\begin{claim}[see \cref{fig:wc_homomorphism}] \label{claim:SPgammat}
		For $t\in[-\delta,0]$, $\SP_{\gamma(t)}$ is a singleton, which we denote by $S$.
		For $t\in(0,\delta]$, $\SP_{\gamma(t)}$ is in bijection with decompositions of $b\coloneqq e - (m_1 + m_2)$ into $b_1 + b_2$ with $b_1,b_2 \in n^\perp \subset M$; in this case, we denote the elements of $\SP_{\gamma(t)}$ by $S_{b_1,b_2}$ via the bijection.
		\begin{figure}[!ht]
			\centering
			\setlength{\unitlength}{0.7\textwidth}
			\begin{picture} (1,0.45)
			\put(0,0){\includegraphics[width=\unitlength]{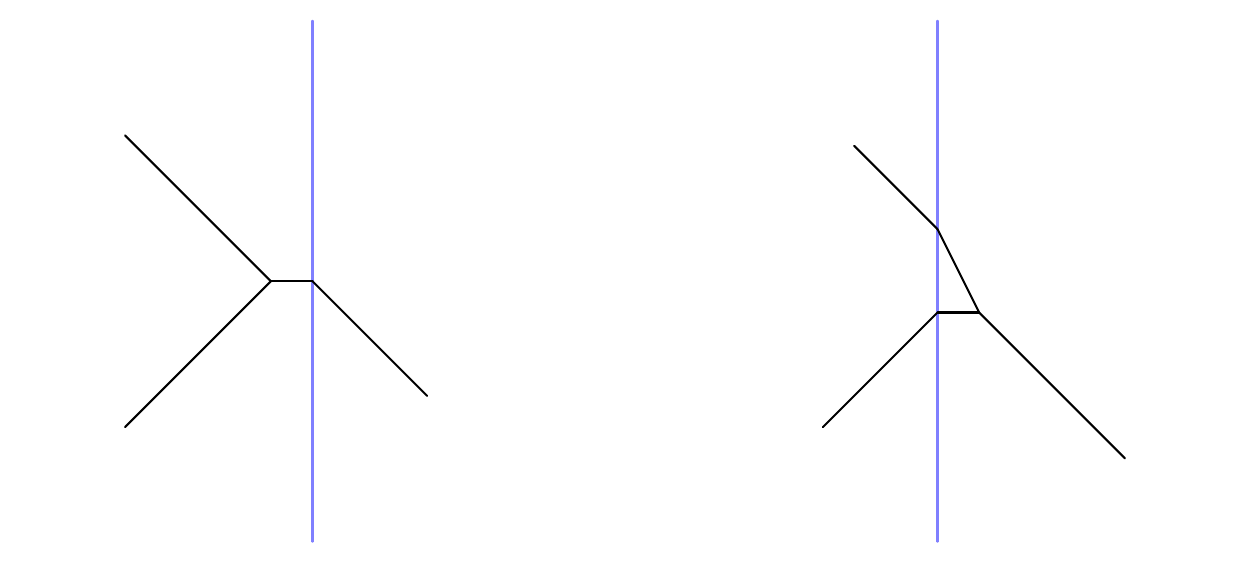}}
			\put(0.26,0.02){$\sigma$}
			\put(0.76,0.02){$\sigma$}
			\put(0.07,0.32){$v_1$}
			\put(0.07,0.13){$v_2$}
			\put(0.34,0.145){$v_3$}
			\put(0.655,0.31){$v_1$}
			\put(0.63,0.13){$v_2$}
			\put(0.895,0.10){$v_3$}
			\end{picture}
			\caption{Left is an example of $S$; right is an example of $S_{b_1,b_2}$.}
			\label{fig:wc_homomorphism}
		\end{figure}
	\end{claim}

	For any $t\in[-\delta,0)$, by cutting $S$ at a point in $[z,v_3]\subset\Gamma$ very close to $z$, using the gluing formula (\cref{thm:gluing_concatenate}), \cref{lem:counts_in_toric_case} and deformation invariance (\cref{thm:deformation_invariance_truncated}), we see that $N(S,\alpha)$ gives the coefficient of $z^\alpha z^e$ in $\Psi_{x,n}(z^{m_1+m_2})$.
	
	For any $t\in(0,\delta]$, we cut $S_{b_1,b_2}$ at two points: one point in $[v_1,z]\subset\Gamma$ very close to $z$, another point in $[v_2,z]\subset\Gamma$ very close to $z$.
	Applying the gluing formula, \cref{lem:counts_in_toric_case} and deformation invariance again, we see that
	\[\sum_{b_1+b_2 = b} N(S_{b_1,b_2},\alpha)\]
	is exactly the coefficient of $z^{\alpha} z^e$ in $\Psi_{x,n}(z^{m_1}) \Psi_{x,n}(z^{m_2})$.
	So we conclude the proof by \cref{claim:independent_of_t}.
\end{proof}

Now we will use \cref{thm:wall-crossing_homomorphism} to further extend the domain of the wall-crossing transformation $\Psi_{x,n}$.

\begin{lemma} \label{lem:wall-crossing_function}
	Given $n\in N$ primitive, $x\in n^\perp\subset M_\bbR$ generic, and $v\in M$ with $\braket{n,v}=1$.
	Write
	\[\Psi_{x,n}(z^v) = z^v f_{x,n,v}\]
	with $f_{x,n,v}\in\hR$.
	The function $f_{x,n,v}$ does not depend on $v$.
	Moreover, $f_{x,n,v}=f_{x,-n,-v}$.
	So we can denote $f_x\coloneqq f_{x,n,v}$ for any choice of $n\in N$ primitive with $x\in n^\perp$ and any $v\in M$ with $\braket{n,v}=1$, which we will refer to as \emph{wall-crossing function}.
	For any $k > 0$, there are at most finitely many primitive $n\in N$ such that $f_x\neq 1$ modulo $I^k$ for some generic $x \in n^\perp$.
\end{lemma}
\begin{proof}
	Say we have $v_1,v_2$ with $\braket{n,v_1}=\braket{n,v_2}=1$.
	Then $m\coloneqq v_1-v_2\in n^\perp$.
	By \cref{thm:wall-crossing_homomorphism}, we have $\Psi_{x,n}(z^{v_1})=\Psi_{x,n}(z^{v_2+m})=z^m \Psi_{x,n}(z^{v_2})$.
	Hence
	\[z^{v_1} f_{x,n,v_1} = \Psi_{x,n}(z^{v_1}) = z^m \Psi_{x,n}(z^{v_2}) = z^m z^{v_2} f_{x,n,v_2}=z^{v_1} f_{x,n,v_2},\]
	which gives $f_{x,n,v_1} = f_{x,n,v_2}$. 
	
	Next we compare $f_{x,n,v}$ with $f_{x,-n,-v}$.
	Take $\alpha\in\NE(Y)$ and $m \in n^\perp$, and consider the coefficient of $z^\alpha z^m$.
	The coefficient in question for $f_{x,n,v}$ is $N(V_{x,v,v+m},\alpha)$, which by $f_{x,n,v-m} = f_{x,n,v}$ is also $N(V_{x,v-m,v},\alpha)$.
	Note $N(V_{x,v-m,v},\alpha)=N(V_{x,-v,-v+m},\alpha)$, where the latter is defined as in \cref{def:infinitesimal_spine} by replacing $n$ with $-n$, which is exactly the $z^\alpha z^m$ coefficient of $f_{x,-n,-v}$.
	This shows the equality $f_{x,n,v}=f_{x,-n,-v}$.

	The last assertion follows from \cref{lem:wall-crossing_bound}.
\end{proof}

\begin{proposition} \label{prop:wall-crossing_function}
	For any $v \in M$ with $\braket{n,v} \geq 0$, we have
	\begin{equation} \label{eq:wall-crossing_function}
		\Psi_{x,n}(z^v) = z^v \cdot f_x^{\braket{n,v}}.
	\end{equation}
	Consequently, the counts $N(V_{x,v,w})$ in \cref{def:infinitesimal_spine} depends only on $\braket{n,v}$ and $w-v$.
\end{proposition}
\begin{proof}
	The case $\braket{n,v}=0$ follows directly from \cref{def:wall-crossing_transformation}.
	Now assume $\braket{n,v}\neq 0$.
	Write $v = \braket{n,v} v_0 + m$ with $\braket{n,v_0}=1$ and $m \in n^\perp$.
	We conclude by \cref{thm:wall-crossing_homomorphism} and \cref{lem:wall-crossing_function}.
\end{proof}

\begin{definition} \label{def:wall-crossing_transformation_extension}
	By \cref{prop:wall-crossing_function}, we can extend the map $\Psi_{x,n}$ of \cref{def:wall-crossing_transformation} to an automorphism of the fraction field $\Frac(\hR)$.
\end{definition}

\begin{remark}
	This may not give an automorphism of $\hR$ because $f_x$ need not be invertible in $\hR$.
	It will become invertible when we set all the curve classes to 0 in \cref{sec:setting_the_curve_classes_to_0}.
\end{remark}

\begin{proposition} \label{prop:wall_decomposition}
	For any $k>0$, there is a finite set of pairs $\fD_k=\{(\fd,f_\fd)\}$, where $\fd$ is a closed codimension-one rational convex cone in $M_\bbR$, and $f_\fd\in\bbZ[Q\oplus M]/I^k$, such that the union of all $\fd$ is the closure of the set of $x$ (generic in some $n^\perp$) with $f_x\neq 1$ modulo $I^k$, and that for any $x\in\fd$ generic, we have $f_\fd=f_x$ modulo $I^k$.
\end{proposition}
\begin{proof}
	Choose $A$ as in \cref{lem:bend_bound} and let $\{\fd\}$ be the union of $(d-1)$-dimensional closed polyhedral cells of $\Wall_A$.
	It follows from \cref{lem:bend_bound} that the union of all such $\fd$ contains the closure of the set of generic $x$ with $f_x\neq 1$ modulo $I^k$.
	Moreover, note that for any generic $x\in\fd$, the spine $V_{x,v,w}^\epsilon$ of \cref{def:infinitesimal_spine} is transverse with respect to $\Wall_A$ for $\epsilon$ sufficiently small.
	Then it follows from \cref{thm:deformation_invariance_truncated} that any two generic points in a given $\fd$ have the same $f_x$, so we conclude the proof.	
\end{proof}

\begin{definition}
	We denote
	\[\fD\coloneqq\Set{(x,f_x) | x\in n^\perp\subset M_\bbR\text{ generic for some } n\in N\setminus 0},\]
	and call it the \emph{scattering diagram} associated to $U\subset Y$ with respect to $T_M$.
	In view of \cref{prop:wall_decomposition}, we call $\fD_k$ a $k$-th order approximation of $\fD$.
\end{definition}

\begin{remark}
	Our scattering diagram here is slightly different from that of \cite[\S 3]{Gross_Mirror_symmetry_for_log_Calabi-Yau_surfaces_I_v1}, in the case when $U$ is two dimensional (the context of \cite{Gross_Mirror_symmetry_for_log_Calabi-Yau_surfaces_I_v1}), because of the presence of the piecewise-linear function $\varphi$ of \cite{Gross_Mirror_symmetry_for_log_Calabi-Yau_surfaces_I_v1} (see also \cite[\S 2.1]{Hacking_Secondary_fan} for a non-archimedean geometric interpretation of $\varphi$).
	\end{remark}

\subsection{Consistency} \label{sec:consistency}

Here we show that our scattering diagram $\fD$ is \emph{theta function consistent}, see \cref{prop:theta_function_consistent}.

\begin{definition} \label{def:local_theta_funcation}
	Given $x \in M_\bbR$ generic and $m \in M$, let $\SP_{x,m,e}$ be the set of spines in $M_\bbR$ with domain $[-\infty,0]$ such that $-\infty$ maps to $\partial\oM_\bbR$ with derivative $-m$ and $0$ maps to $x$ with derivative $-e$.
	we define the \emph{local theta function} $\theta_{x,m}$ to be the formal sum
	\[\theta_{x,m} \coloneqq \sum_{\substack{e\in M,\; S\in\SP_{x,m,e}\\ \alpha\in\NE(Y)}} N(S,\alpha) z^{\alpha} z^e.\]
\end{definition}

\begin{proposition}
	We have $\theta_{x,m} \in\hR$.
\end{proposition}
\begin{proof}
	It suffices to prove that for any $k>0$, there are only finitely many $e\in M$ such that $N(S,\alpha)\neq 0$ for some $S\in\SP_{x,m,e}$ and $\alpha\in Q\setminus Q_k$.
	Since $Q\setminus Q_k$ is finite, by \cref{lem:hgamma} and \cref{prop:tropicalization_of_stable_map}, there exists $A\in\bbN$ such that for any stable map $f$ contributing to some $N(S,\alpha)$ with $\alpha\in Q\setminus Q_k$, we have $\degtwig\Trop(f)\le A$.
	By \cref{rem:twig_bound}, there are only finitely many combinatorial types of twigs with degree bounded by $A$.
	Therefore, since $m$ is fixed, there are only finitely many combinatorial types of $\Trop(f)$, where $f$ is any stable map contributing to $N(S,\alpha)$ for some $e\in M$, $S\in\SP_{x,m,e}$ and $\alpha\in Q\setminus Q_k$.
	Consequently, there are only finitely many $e\in M$ and $S\in\SP_{x,m,e}$ such that $N(S,\alpha)\neq 0$ for some $\alpha\in Q\setminus Q_k$, completing the proof.
\end{proof}

\begin{proposition} \label{prop:theta_function_consistent}
	The scattering diagram $\fD$ is theta function consistent in the following sense:
	Given any $k>0$ and $(\fd,f_\fd)\in\fD_k$, choose $n\in N$ with $\fd\subset n^\perp$, and let $a,b\in M_\bbR$ be two general points near a general point $x\in\fd$ with $\braket{n,a}>0$, $\braket{n,b}<0$.
	Write $\theta_{a,m} = \theta_{a,m,+}+\theta_{a,m,0}+\theta_{a,m,-}$ and $\theta_{b,m} = \theta_{b,m,+}+\theta_{b,m,0}+\theta_{b,m,-}$, where we gather monomials $z^e$ according to the sign of the pairing $\braket{n,e}$, e.g.,
	\[\theta_{a,m,+} \coloneqq \sum_{\substack{e\in M,\; \braket{n,e}>0\\S\in\SP_{a,m,e},\; \alpha\in\NE(Y)}} N(S,\alpha) z^\alpha z^e.\]
	The following hold modulo $I^k$:
	\begin{align}
		\Psi_{x,n}(\theta_{a,m,+})=\theta_{b,m,+},\label{eq:theta_function_consistent_1}\\
		\Psi_{x,-n}(\theta_{b,m,-})=\theta_{a,m,-},\label{eq:theta_function_consistent_2}\\
		\Psi_{x,n}(\theta_{a,m,0})=\theta_{b,m,0},\label{eq:theta_function_consistent_3}\\
		\Psi_{x,-n}(\theta_{b,m,0})=\theta_{a,m,0}.\label{eq:theta_function_consistent_4}
	\end{align}
\end{proposition}
\begin{proof} 
	For \eqref{eq:theta_function_consistent_1}, we have
	\begin{align*}
	\Psi_{x,n}(\theta_{a,m,+})&=\Psi_{x,n}\Bigg(\sum_{\substack{e\in M,\; \braket{n,e}>0\\S\in\SP_{a,m,e},\; \alpha\in\NE(Y)}} N(S,\alpha) z^\alpha z^e\Bigg)\\
	&=\sum_{\substack{e\in M,\; \braket{n,e}>0\\S\in\SP_{a,m,e},\; \alpha\in\NE(Y)}} N(S,\alpha) z^\alpha \sum_{\substack{w \in M,\; \braket{n,w}>0\\ \beta \in\NE(Y)}} N(V_{x,e,w},\beta) z^\beta z^{w}\\
	&=\sum_{\substack{w\in M,\; \braket{n,w}>0\\S'\in\SP_{b,m,w},\; \gamma\in\NE(Y)}} N(S',\gamma) z^\gamma z^w\\
	&=\theta_{b,m,+},
	\end{align*}
	where the third equality above follows from \cref{thm:gluing_concatenate} and \cref{thm:deformation_invariance_truncated} by gluing $S$ and $V_{x,e,w}^\epsilon$ at $a$ after a small deformation of the latter for alignment.
	
	Equation \eqref{eq:theta_function_consistent_2} follows from \eqref{eq:theta_function_consistent_1} by replacing $n$ with $-n$.
	
	Equations \eqref{eq:theta_function_consistent_3} and \eqref{eq:theta_function_consistent_4} follow from \cref{prop:deformation_invariance_skeletal}.
\end{proof}

\subsection{Setting the curve classes to $0$}  \label{sec:setting_the_curve_classes_to_0}

Our construction of the wall-crossing transformations and the scattering diagram depends on the compactification $U\subset Y$ via the usage of curve classes in $\NE(Y)$.
We can remove this dependence by setting all the curve classes to $0$.
But for the adic convergence of wall-crossing transformations, we need to impose a condition on the bend of infinitesimal analytic cylinders.

Let $P \subset M$ be a strictly convex toric monoid, $P_1\coloneqq P \setminus 0$, $J\subset\bbZ[P]$ the monomial ideal associated to $P_1$, and $P_k \subset P_1$ the subset of elements which are sums of $k$ elements of $P_1$.
Let $\hL^0$ be the $J$-adic completion of $\bbZ[P]$ and $\hL \coloneqq \hL^0 \otimes_{\bbZ[P]} \bbZ[M]$.
Let $\hR$ be as in \cref{sec:scattering:construction}, and $\hR_P \subset \hR$ the closure of $\bbZ[Q \oplus P]$ inside $\hR$.

From now on until the end of this section, we assume the following:

\begin{assumption} \label{ass:bend_in_P}
	For any $N(V_{x,v,w},\alpha) \neq 0$, the bend $w - v \in P$.
\end{assumption}

\begin{remark} \label{rem:bend_in_P}
	\cref{ass:bend_in_P} implies that the wall-crossing function $f_x$ (as in \cref{lem:wall-crossing_function}) lies in the subring $\hR_P \subset \hR$.
\end{remark}

\begin{lemma} \label{lem:mod_J_bound}
	Given $v \in M$ and $k > 0$, there are at most finitely many $\alpha\in\NE(Y)$ such that $N(V_{x,v,w},\alpha) \neq 0$ for some $x,w$ with $w - v \in P\setminus P_k$.
\end{lemma}
\begin{proof}
	Let $A$ be the sum of the norm (see \cref{def:norm_of_weight_vector}) of every vector in $P\setminus P_k$.
	Let $f$ be a stable map contributing to $N(V_{x,v,w},\alpha)\neq 0$.
	Then $w - v$ is the sum of monomials of all twigs of $\Trop(f)$.
	So by the balancing condition, if $w - v\in P\setminus P_k$, we have $\degtwig\Trop(f)\le A$;
	moreover, as there are only finitely many possibilities for $w-v$ and $v$ is fixed, there are only finitely many possibilities for $w$.
	Hence we conclude from \cref{lem:bound_on_class}(\ref{lem:bound_on_class:degtwig}).
\end{proof}

\begin{lemma} \label{lem:fx_mod_J}
	For any $k>0$, the image of the wall-crossing function $f_x$ in $\hR_P/(J^k)$ is a polynomial, i.e.\ it lies in the image of the inclusion $\bbZ[Q \oplus P] \hookrightarrow \hR_P/(J^k)$. 
\end{lemma}
\begin{proof}
	This follows from \cref{lem:mod_J_bound}.
\end{proof}

\begin{lemma} \label{lem:local_theta_function_mod_J}
	Fix $x \in M_\bbR$ generic and $m \in M$, let
	\[\theta_{x,m} = \sum_{\substack{e\in M,\; S\in\SP_{x,m,e}\\ \alpha\in\NE(Y)}} N(S,\alpha) z^{\alpha} z^e.\]
	be the local theta function as in \cref{def:local_theta_funcation}.
	For any $k>0$, there are at most finitely many $\alpha\in\NE(Y)$ such that $N(S,\alpha)\neq 0$ for some $e\in M$ with $e-m\in P\setminus P_k$ and $S\in\SP_{x,m,e}$.
	Consequently, the local theta function $\theta_{x,m}$, viewed as element of $z^m(\hR_P/(J^k))$, lies in the image of the inclusion $z^m(\bbZ[Q \oplus P])\hookrightarrow z^m(\hR_P/(J^k))$.
\end{lemma}
\begin{proof}
	The proof parallels that of \cref{lem:mod_J_bound}.
	\end{proof}

\begin{definition} \label{def:scattering_diagram_without_curve_classes}
	Fix $k>0$.
	For any $x$ generic in some hyperplane $n^\perp$, let $\of_{x,k}\in\bbZ[P]/J^k$ be the image of $f_x$ under the map induced by the projection $Q\oplus P\to P$, i.e.\ setting $z^\alpha = 1$ for all $\alpha \in Q$.
	This is well-defined by \cref{lem:fx_mod_J}.
	Now there is a unique $\of_x \in \hL^0\subset\hL$ such that $\of_x=\of_{x,k}$ modulo $J^k$ for all $k$.
	
	Similarly, for the local theta function $\theta_{x,m}$, by \cref{lem:local_theta_function_mod_J}, we define $\otheta_{x,m,k}\in z^m(\bbZ[P]/J^k)$ and $\otheta_{x,m}\in z^m(\hL^0)\subset\hL$.
	
	We denote
	\[\fD_U\coloneqq\Set{(x,\of_x) | x\in n^\perp\subset M_\bbR\text{ generic for some } n\in N\setminus 0},\]
	and call it the \emph{scattering diagram} associated to $U$ with respect to $T_M$.
	
	For every $(\fd,f_\fd)\in\fD_k$ in \cref{prop:wall_decomposition}, we have $f_{\fd} \in \bbZ[Q \oplus P]/(J^k)$ by \cref{lem:fx_mod_J}, and we define $\of_\fd\in\bbZ[P]/J^k$ by projection.
	We denote $\fD_{U,k}\coloneqq\set{(\fd,\of_\fd) | (\fd,f_\fd)\in\fD_k}$, and call it a $k$-th order approximation of $\fD_U$.
\end{definition}

\begin{proposition} \label{prop:invertible}
	The function $\of_x\in\hL^0$ is invertible in $\hL$, in other words, $\of_x=1$ modulo $J$.
	Consequently, replacing $f_x$ with $\of_x$ in the formula \eqref{eq:wall-crossing_function}, we obtain an automorphism $\oPsi_{x,n}$ of $\hL$.
\end{proposition}
\begin{proof} 
	Let $g$ be any stable map contributing to $\of_x$ modulo $J$, in other words, $g$ contributes to some $N(V_{x,v,v},\alpha)\neq 0$.
	By \cref{ass:bend_in_P}, $\Trop(g)$ does not have any twigs.
	It follows that $g$ has image in $W^\an\subset Y^\an$, $W$ as in \cref{lem:toric_model}.
	Hence by \cref{lem:counts_in_toric_case}, there is a unique $\alpha$ such that $N(V_{x,v,v},\alpha)\neq 0$, and for this $\alpha$ we have $N(V_{x,v,v},\alpha)=1$.
	Therefore, we have $\of_x=1$ modulo $J$, completing the proof.
\end{proof}

\begin{lemma} \label{lem:otheta}
		For any $x\in M_\bbR$ generic and $m\in M$, we have $\otheta_{x,m}=z^m(1+\eta)$ for some $\eta\in J$.
\end{lemma}
\begin{proof} This follows from \cref{ass:bend_in_P} just as in the proof of \cref{prop:invertible}.
\end{proof}

\begin{proposition} \label{prop:fD_U_consistent}
	The scattering diagram $\fD_U$ is theta function consistent in the following sense:
	Given any $k>0$ and $(\fd,f_\fd)\in\fD_k$, choose $n\in N$ with $\fd\subset n^\perp$, and let $a,b\in M_\bbR$ be two general points near a general point $x\in\fd$ with $\braket{n,a}>0$, $\braket{n,b}<0$.
	We have the following equalities in $z^m(\bbZ[P]/J^k)$
	\begin{align*}
	\oPsi_{x,n}(\otheta_{a,m})=\otheta_{b,m},\\
	\oPsi_{x,-n}(\otheta_{b,m})=\otheta_{a,m}.
	\end{align*}
\end{proposition}
\begin{proof}
	This follows from \cref{prop:theta_function_consistent} and \cref{lem:fx_mod_J,lem:local_theta_function_mod_J}, via the quotient induced by the projection $Q\oplus P\to P$.
\end{proof}

\begin{proposition} \label{prop:KS_consistent}
	The scattering diagram $\fD_U$ is consistent in sense of Kontsevich-Soibelman, i.e.\ for any general loop $l\colon [0,1] \to M_\bbR$ with $l(0) = l(1)$, the composition $\oPsi$ of wall-crossing automorphisms $\oPsi_{x,n}$ along $l$ is the identity on $\hL$.
\end{proposition}
\begin{proof}
	Denote $y\coloneqq l(0) = l(1)$.
	By \cref{prop:fD_U_consistent}, we have $\oPsi(\otheta_{y,m})=\otheta_{y,m}$ for all $m\in M$.
	By \cref{lem:otheta}, the subring of $\hL$ generated by all $\otheta_{y,m}$ is $J$-adically dense.
	So $\oPsi$ is identity on $\hL$.
\end{proof}

\section{Cluster case: comparison with Gross-Hacking-Keel-Kontsevich} \label{sec:cluster_case} 

In this section, we compare our mirror algebra with the canonical algebra of Gross-Hacking-Keel-Kontsevich \cite{Gross_Canonical_bases} in the case of cluster varieties, for both the A-cluster case (\cref{thm:comparison_with_GHKK_A-type}) and the X-cluster case (\cref{cor:comparison_with_GHKK_X-type}).
We refer to the final remarks in \cref{sec:intro:scattering} for implications of the comparison.

The idea is the following:
Using the geometric construction of scattering diagrams in the previous section, it suffices to prove the comparison at the level of scattering diagrams (see \cref{thm:scattering_diagram_comparison}).
By the consistency, it suffices to identify all the incoming walls of our scattering diagram and compute the associated scattering functions.
The notions of twigs and walls introduced in \cref{sec:tropical} are too rough for this purpose.
Thus in \cref{sec:C-walls}, we refine the notions of walls and twigs for the cluster case.
They will helps us identify the incoming walls in \cref{lem:incoming_walls}, and control the monomials in the scattering functions in \cref{lem:scattering_function_cluster}.

\begin{notation} \label{nota:cluster_blowup}
Let $M$ be a lattice and $\braket{,}$ an integer valued non-degenerate skew-symmetric form on $M$.
Fix $S \subset M \setminus 0$ a finite set of elements such that for every $e\in S$, $\braket{e ,\cdot} \in M^*$ is primitive (in particular, nonzero);
moreover, we assume the submonoid $P \subset M$ generated by $S$ is strictly convex.

Let $\Sigma$ be the (incomplete) fan in $M_\bbR$ consisting of rays $\bbR_{\geq 0} e$ for all $e \in S$.
Let $\TV(\Sigma)$ be the associated toric variety.
For every $e\in S$, let $D_e\subset\TV(\Sigma)$ be the corresponding toric boundary divisor.
Fix $\lambda_e\in k^*$ for each $e \in S$. 
Note that $\braket{e,\cdot}$ vanishes on $e$, so gives an element in the dual of $M/(\bbZ e)$, hence the equation $z^{\braket{e,\cdot}}+\lambda_e=0$ gives a subvariety $Z_e\subset D_e$
(note $D_e$ is isomorphic to the algebraic torus $T_{M/(\bbZ e)}$ and $Z_e$ is a coset for the codimension-one subtorus $\ker z^{\braket{e,\cdot}}\colon D_e \to \bbG_m$).
Let $\oU' \to \TV(\Sigma)$ be the blowup along the (disjoint) union of all $Z_e$, $e \in S$.
Let $\partial\oU'$ be the strict transform of the toric boundary, $U'\subset \oU'$ the complement of $\partial\oU'$, and $U \coloneqq \Spec(H^0(U',\cO_{U'}))$.
\end{notation}

\begin{assumption} \label{ass:cluster_blowup}
We assume the following:
\begin{enumerate}
	\item \label{ass:cluster_blowup:fg} $H^0(U',\cO_{U'})$ is finitely generated;
	\item \label{ass:cluster_blowup:iso} The natural map $U' \to U$ is an isomorphism outside of closed subsets (of domain and range) of codimension at least two;
	\item \label{ass:cluster_blowup:torus} The induced map $T_M \to U$ is an open immersion.
	\item \label{ass:cluster_blowup:smooth} $U$ is smooth.
\end{enumerate}
Note that (\ref{ass:cluster_blowup:fg}), (\ref{ass:cluster_blowup:iso}) and (\ref{ass:cluster_blowup:smooth}) imply that $U$ is log Calabi-Yau.
\end{assumption}

\begin{remark}
In our application later in \cref{sec:comparison}, $U$ will be the spectrum of an upper A-cluster algebra, $\braket{,}$ will be assumed unimodular, $S'$ will be a seed (i.e.\ a basis of $M$), $S \subset S'$ will be the set of unfrozen basis elements, and $\oU'$ will be a \emph{toric model} determined by the seed, as in \cite[3.4]{Gross_Birational_geometry_of_cluster_algebras}.
\end{remark}

\subsection{C-twigs and C-walls} \label{sec:C-walls}

Here we introduce a more restrictive notion of twigs and walls than in \cref{sec:tropical}, that is better adapted to the case of cluster varieties.
We will call them \emph{C-twigs} and \emph{C-walls}, where ``C'' is short for ``cluster''.

\begin{construction} \label{const:LRS}
	Following \cref{nota:cluster_blowup}, we construct a partial compactification $M_\bbR^S$ of $M_\bbR$ as follows:
	We add more cells to the fan $\Sigma$ to make it into a complete fan $\Sigma'$.
	Let $\oSigma'$ denote the compactified fan associated to $\Sigma'$ (as in \cref{nota:toric_model}).
	For every $e\in S$, let $\sigma_e\subset\oSigma'$ denote the open cell of $\partial\oSigma'$ perpendicular to $\bbR_{\ge 0}e$.
	Note $\sigma_e$ is isomorphic to $(M/(\bbZ e))_\bbR$.
	Define
	\[M_\bbR^S\coloneqq M_\bbR \cup \bigcup_{e\in S} \sigma_e, \quad\partial M_\bbR^S\coloneqq M_\bbR^S\setminus M_\bbR.\]
	Note they are independent of the choice of $\Sigma'$.
	
	For every $e\in S$, let $\eta_e\subset\sigma_e$ be the limit of the hyperplane $e^\perp\subset M_\bbR$ in $\sigma_e$.
%	We define
%	\[M_\bbR^E\coloneqq M_\bbR \cup \bigcup_{e\in S} \eta_e, \quad\partial M_\bbR^E\coloneqq M_\bbR^E\setminus M_\bbR.\]
	We have a canonical retraction $\rho_\rt\colon\TV(\Sigma)^\an\to M_\bbR^S$, which induces a retraction $\rho\colon\oU'^\an\to M_\bbR^S$.
	For every $e\in S$, we have $\rho_\rt(D_e)=\sigma_e$ and $\rho_\rt(Z_e)=\eta_e$.
\end{construction}

\begin{definition} \label{def:C-twig}
	A \emph{C-twig} consists of a nodal metric tree $\Gamma$, a set of 1-valent vertices $(r,u_1,\dots,\allowbreak u_m)$, and a \Zaffine map $h\colon\Gamma\to M_\bbR^S$ (see \cref{def:pointedtree}(\ref{def:pointedtree:Zaffine})) satisfying the following conditions:
	\begin{enumerate}
		\item The vertex $r$ is a 1-valent finite vertex called \emph{root}; the vertices $u_1,\dots,u_m$ are different 1-valent infinite vertices.
		These are the only 1-valent vertices of $\Gamma$.
		\item We have $h\inv(\partial M_\bbR^S)=\{u_1,\dots,u_m\}$.
		\item \label{def:C-twig:leg} For $j=1,\dots,m$, let $l_j$ denote the leg incident to $u_j$; then $l_j\setminus u_j$ maps into a hyperplane $e^\perp\subset M_\bbR$ for some $e\in S$, with outward weight vector $k e$ for some positive integer $k$.
		\item \label{def:C-twig:balancing} (Balancing condition) The \Zaffine map $h$ is balanced at every vertex of $\Gamma$ of valency greater than 1.
	\end{enumerate}
	For every edge $e$ of $\Gamma$, let $\widetilde e$ denote $e$ minus its possible infinite endpoint.
	Let $e_r$ be the edge of $\Gamma$ incident to the root $r$.
	We call the weight vector $w_{(r,e_r)}$ the \emph{monomial} of the C-twig, and $-w_{(r,e_r)}$ the \emph{direction} of the C-twig.
\end{definition}

\begin{lemma} \label{lem:C-twig_weight_vector}
	Let $[\Gamma,(r,u_1,\dots,u_m),h]$ be a C-twig in $M_\bbR$.
	For every edge $e$ of $\Gamma$, let $w_e$ denote the weight vector of $e$ at the endpoint closer to the root.
	We have $w_e\in P\setminus 0$ and $h(\widetilde e)\subset w_e^\perp$.
	In particular, $h$ is an immersion and $\Gamma$ is irreducible.
\end{lemma}
\begin{proof}
	It follows from \cref{def:C-twig}(\ref{def:C-twig:leg}) that the claim holds for all the legs of $\Gamma$.
	Hence by induction (from the legs towards the root), it suffices to prove the following:
	for any finite vertex $v$ of $\Gamma$, let $e_1,\dots,e_l,f$ denote the edges incident to $v$ where $f$ is the edge closer to $r$; if the statement holds for all $e_1,\dots,e_l$, then it also holds for $f$.
	So let us suppose $w_{e_i}\in P\setminus 0$ and $h(\widetilde e_i)\subset w_{e_i}^\perp$ for all $i$.
	Since $v$ is the intersection of all $\widetilde e_i$, we get $h(v)\in w_{e_i}^\perp$ for all $i$.
	By the balancing condition, we have $w_f=\sum w_{e_i}$.
	We obtain $w_f\in P\setminus 0$ and $h(v)\in w_f^\perp$.
	We deduce that
	\[h(f)\ \subset\ h(v)-\bbR_{\ge 0}\cdot w_f\ \subset\ w_f^\perp,\]
	completing the proof.
\end{proof}

\begin{lemma} \label{lem:C-twig_generic}
	Let $[\Gamma,(r,u_1,\dots,u_m),h]$ be a C-twig in $M_\bbR$.
	For every edge $e$ of $\Gamma$, we say that $f(\widetilde e)\subset M_\bbR$ is generic if it is contained in at most one rational hyperplane passing through 0.
	Let $e_r$ be the edge incident to $r$.
	If $h(e_r)$ is generic, then $h(e)$ is generic for every edge $e$ of $\Gamma$.
	In particular, this holds when $h(r)$ is generic (in the same sense as above).
\end{lemma}
\begin{proof}
	This follows from the following more general lemma.
\end{proof}

\begin{lemma} \label{lem:edge_generic}
	Let $[\Gamma,(r,t),h]$ be a pointed tree in $M_\bbR$ where $\Gamma$ consists of only three vertices $r,s,t$, an edge $e$ connecting $r,s$, and another edge $f$ connecting $s,t$.
	Let $w_e$ be the weight vector of $e$ at $r$, and $w_f$ the weight vector of $f$ at $s$.
	Assume $h(e)\subset w_e^\perp$, $h(f)\subset w_f^\perp$, and $h(e)\subset M_\bbR$ is generic (in the same sense as in \cref{lem:C-twig_generic}).
	Then $h(f)\subset M_\bbR$ is also generic.
\end{lemma}
\begin{proof}
	Suppose to the contrary that $h(f)$ is not generic.
	Then there is $a\in M$ such that $h(f)\subset (w_f,a)^\perp$ with $\rank(w_f,a)=2$.
	Since $w_f$ is the direction of $h(f)$, we note that $(w_f,a)$ is isotropic with respect to the skew-symmetric form.
	
	Observe that
	\begin{equation} \label{eq:edge_generic}
	h(e) \subset h(s)+\bbR w_e \subset h(e)\cap h(f)+\bbR w_e \subset (w_e,w_f,a)^\perp + \bbR w_e.
	\end{equation}
	Since $h(e)$ is generic, $(w_e,w_f,a)$ has rank at most 2.
	As $(w_f,a)$ has rank 2, we deduce that $(w_e,w_f,a)$ has rank 2, and $w_e$ is a linear combination of $w_f$ and $a$.
	As $(w_f,a)$ is isotropic, we obtain $w_e\in (w_e,w_f,a)^\perp$.
	Then \eqref{eq:edge_generic} implies that $h(e)\subset(w_e,w_f,a)^\perp$, contradicting the genericity assumption on $h(e)$.
\end{proof}

\begin{definition} \label{def:C-wall}
	A \emph{C-wall} in $M_\bbR$ is a pair $(\fd,n)$ where $n \in P\setminus 0$ and $\fd \subset n^\perp$ is a closed convex rational polyhedral cone.
	%We call a C-wall $(\fd,n)$ \emph{full} if the linear span $(\fd)$ has codimension one, and thus $(\fd)=n^\perp$.
	We call a C-wall $(\fd,n)$ \emph{incoming} if $n \in \fd$; otherwise we call it \emph{outgoing}.
	We call $n$ the \emph{monomial} of the C-wall, and $-n$ the \emph{direction} of the C-wall.
\end{definition}

\begin{remark} \label{rem:C-wall}
	The notion of C-wall is more restrictive than the walls we considered in \cref{const:walls_by_induction}; in particular, the support of a C-wall determines its monomial up to an integer multiple.
	This more restrictive notation is well-adapted to the cluster case by Lemmas \ref{lem:C-twig_in_C-wall} and \ref{lem:C-twig_from_analytic_curve}.
	It helps us identify the incoming walls in \cref{lem:incoming_walls}, and control the monomials in the scattering functions in \cref{lem:scattering_function_cluster}.
\end{remark}

\begin{definition}
	For any $n\in P$, we define its \emph{degree} $d(n)$ to be the largest integer $k$ such that $n=n_1+\dots+n_k$ with every $n_i\in P\setminus 0$.
\end{definition}

\begin{construction} \label{const:C-wall}
	For any positive integer $d$, we define in the following a collection $W_d$ of C-walls whose monomial has degree at most $d$.
	For each $e\in S$, let $P_e \subset P$ denote the submonoid generated by $e$.
	Let $W_d^0$ be the collection of C-walls of form $(e^\perp, n)$ with $n \in P_e\setminus 0$ and $d(n) \leq d$.
	We call these the \emph{initial C-walls} of $W_d$.
	Having defined $W_d^0 \subset W_d^1 \dots \subset W_d^t$, we define $W_d^{t+1}$ as follows:
	
	Let $(\fd_1,n_1), (\fd_2,n_2) \in W_d^t$ and assume either:
	\begin{enumerate}
		\item $\braket{n_1,n_2} \neq 0$, or
		\item $n_1$ and $n_2$ are parallel (or equivalently, $n_1^\perp = n_2^\perp$).
	\end{enumerate}
	In both cases we define
	\begin{gather*}
	\fd_1 + \fd_2 \coloneqq \fd_1 \cap \fd_2 - \bbR_{\geq 0} (n_1 + n_2), \\
	(\fd_1,n_1) + (\fd_2,n_2) \coloneqq (\fd_1 + \fd_2,n_1 + n_2).
	\end{gather*}
	It is easy to check that $(\fd_1,n_1) + (\fd_2,n_2)$ is a C-wall.
	Let $W_d^{t+1}$ be obtained by adding to $W_d^t$ all such sums whose monomials have degree at most $d$.
	It follows from the definition of degree that $W_d^t$ becomes constant for sufficiently large $t$.
	We let $W_d$ be this constant set (which is thus the union of $W_d^t$ over all $t\in\bbN$).
\end{construction}

\begin{lemma} \label{lem:incoming_C-wall}
	The incoming C-walls of $W_d$ are exactly the initial C-walls, i.e.\ those in $W_d^0$.
\end{lemma}
\begin{proof}
	Let $(\fd_1,n_1)$ and $(\fd_2,n_2)$ be any two C-walls in $W_d$.
	Assume $(\fd_1,n_1)+(\fd_2,n_2)$ is incoming.
	Suppose first we are in case (2).
	Then $\fd_1 + \fd_2 = \fd_1 \cap \fd_2$.
	So $n_1 + n_2 \in \fd_1 + \fd_2$ if and only if $n_i \in \fd_i$, $i = 1,2$.
	Hence both $(\fd_1,n_1)$ and $(\fd_2,n_2)$ are incoming.
	If both are initial C-walls, their sum is also an initial C-wall.
	
	Next suppose we are in case (1).
	Since $(\fd_1,n_1)+(\fd_2,n_2)$ is incoming, we have $n_1+n_2 = x - \lambda(n_1 + n_2)$ for some $x\in\fd_1\cap\fd_2$ and $\lambda\in\bbR_{\ge 0}$.
	Then $(1+\lambda)(n_1 + n_2) = x \in (n_1,n_2)^\perp$.
	This implies $\braket{n_1,n_2} = 0$, a contradiction.
	Therefore, the sums of walls in case (1) never produce incoming walls.
	This completes the proof.
\end{proof}

\begin{lemma} \label{lem:C-twig_in_C-wall}
	Let $[\Gamma,(r,u_1,\dots,u_m),h]$ be a C-twig in $M_\bbR$ whose monomial has degree bounded by $d$.
	Let $e_r$ be the edge incident to the root $r$.
	If $h(e_r)\subset M_\bbR$ is generic, then for every edge $e$ of $\Gamma$, there exists a C-wall $(\fd,n)\in W_d$ such that $h(\widetilde e)\subset\fd$, with derivative (pointing away from $r$) equal to $n$.
\end{lemma}
\begin{proof}
	It follows from \cref{def:C-twig}(\ref{def:C-twig:leg}) that the claim holds for all the infinite legs of $\Gamma$.
	Hence by induction (from the legs towards the root), it suffices to prove the following:
	for any finite vertex $v$ of $\Gamma$, let $e_1,\dots,e_l,f$ denote the edges incident to $v$ where $f$ is the edge closer to $r$; if the statement holds for all $e_1,\dots,e_l$, then it also holds for $f$.
	So let us assume that for $j=1,\dots,l$, there exists a wall $(\fd_j,n_j)\in W_d$ such that $h(\widetilde e_j)\subset\fd_j$, with derivative (pointing away from $r$) equal to $n_j$.
	Now it suffices to show that the sum
	\[(\fd,n)\coloneqq(\fd_1,n_1)+\dots+(\fd_l,n_l)\]
	makes sense; in other words, if we add successively from left to right, at each step, we are in the two allowed cases of \cref{const:C-wall}.
	Then it will follow from the balancing condition that $h(f)\subset\fd$, with derivative (point away from $r$) equal to $n$.
	
	For $j=1,\dots,l-1$, let $s_j\coloneqq n_1+\dots+n_j$.
	It remains to check that either $\braket{s_j,n_{j+1}}\neq 0$, or $s_j^\perp=n_{j+1}^\perp$.
	By \cref{lem:C-twig_weight_vector}, we have $h(\widetilde e_i)\subset n_i^\perp$ for all $i$, hence $h(v)\in n_i^\perp$ for all $i$.
	Suppose $\braket{s_j,n_{j+1}}=0$, then we obtain $h(\widetilde e_{j+1})\subset (s_j,n_{j+1})^\perp$.
	As $h(f)\subset M_\bbR$ is generic, by \cref{lem:edge_generic}, $h(\widetilde e_{j+1})$ is also generic, so $s_j^\perp=n_{j+1}^\perp$, completing the proof.
\end{proof}

\begin{lemma} \label{lem:C-twig_from_analytic_curve}
	Choose any snc compactification $U\subset Y$ satisfying \cref{ass:strata} with respect to $T_M\subset U$.
	Let $[C,(p_j)_{j\in J},f\colon C\to Y^\an]$ be a skeletal curve in $\cM(U^\an,\bP,\beta)$.
	Then the twigs of $\Trop(f)$ are C-twigs.
	(By \cref{lem:restriction_to_skeleton}(\ref{lem:restriction_to_skeleton:sm}), $f$ belongs to $\cM^\sm(U^\an,\bP,\beta)$, so $\Trop(f)$ is well-defined by \cref{prop:tropicalization_of_stable_map}.)
\end{lemma}
\begin{proof}
	We follow \cref{nota:cluster_blowup}.
	Let $Z\subset Y$ be the locus where $Y\dasharrow\oU'$ fails to be an isomorphism.
	Since $U \dasharrow U'$ is an isomorphism outside codimension two, after replacing $Y$ by a toric blowup and adding to $\Sigma$ extra rays (without changing $S$), we can assume that $Z\cap(U\cup D^\ess)$ has codimension at least two.
	We have a commutative diagram
	\[\begin{tikzcd}
	(Y\setminus Z)^\an \rar{\tau} \dar & \oM_\bbR \\
	\oU' \rar{\rho} & M_\bbR^S \uar,
	\end{tikzcd}\]
	where the upper arrow is defined as in \cref{nota:Et}, and the lower arrow as in \cref{const:LRS}.
	
	By \cref{lem:restriction_to_skeleton}(\ref{lem:restriction_to_skeleton:sm}), the image $f(C)\subset Y^\an$ is disjoint from $Z^\an$, thus we obtain a map $g\colon C\to\oU'$.
	The commutative diagram above identifies $\Trop(f)$ with $\Trop(g)$.
	
	Let $[\Gamma,(r,u_1,\dots,u_m),h]$ be a twig of $\Trop(g)$.
	The fact that it is also a twig of $\Trop(f)$ implies all the conditions in \cref{def:C-twig} except Condition (\ref{def:C-twig:leg}).
	So it remains to verity Condition (\ref{def:C-twig:leg}).
	Let $b\colon\oU' \to \TV(\Sigma)$ denote the blowup map.
	For each $j=1,\dots,m$, we have $g(u_j)\notin\partial\oU'$, and $(b\circ g)(u_j)\in\TV(\Sigma)$ lies in the toric boundary.
	Therefore, there exists $e\in S$ such that $(b\circ g)(u_j)\in Z_e\subset D_e$.
	Recall from \cref{const:LRS} that $\rho_\rt(Z_e)=\eta_e$, where $\eta_e$ is the limit of the hyperplane $e^\perp\subset M_\bbR$ in $\sigma_e$.
	So we have $h(u_j)=(\rho\circ g)(u_j)=(\rho_\rt\circ b\circ g)(u_j)\in\eta_e\subset M_\bbR^S$.
	Then the leg $l_j\setminus u_j$ must map into the hyperplane $e^\perp\subset M_\bbR$ with nonzero derivative parallel to $e$, completing the proof.
\end{proof}

\begin{lemma} \label{lem:w-v_cluster}
	For any $n\in N\setminus 0$, $x\in n^\perp$ generic, $v,w\in M$, if $N(V_{x,v,w},\alpha)\neq 0$, then we have
	\begin{enumerate}
		\item \label{lem:w-v_cluster:P} $w-v\in P\setminus 0$,
		\item \label{lem:w-v_cluster:parallel} $w-v$ is parallel (up to sign) to $n$.
	\end{enumerate}
\end{lemma}
\begin{proof}
	Let $f$ be a stable map contributing to $N(V_{x,v,w},\alpha)$.
	Then $w-v$ is equal to the sum of monomials of all twigs of $\Trop(f)$.
	By \cref{lem:C-twig_from_analytic_curve}, every twig of $\Trop(f)$ is a C-twig.
	So it follows from \cref{lem:C-twig_weight_vector} that $w-v\in P\setminus 0$.
	Moreover, it follows from \cref{lem:C-twig_in_C-wall} that the monomial of every twig of $\Trop(f)$ is parallel (up to sign) to $n$.
	Hence $w-v$ is parallel (up to sign) to $n$.
\end{proof}

\begin{construction} \label{const:DU}
	By \cref{lem:w-v_cluster}(\ref{lem:w-v_cluster:P}), \cref{ass:bend_in_P} is satisfied.
	So \cref{sec:setting_the_curve_classes_to_0} applies here, in other words, we can set curve classes to 0 and obtain a consistent scattering diagram $\fD_U$ as well as its finite-order approximations $\fD_{U,k}$ as in \cref{def:scattering_diagram_without_curve_classes}.
\end{construction}

\begin{lemma} \label{lem:scattering_function_cluster}
	For any $(x,\of_x)\in\fD_U$ with $x\in n^\perp$ and $\of_x\neq 1$, there exists $n_0\in P\setminus 0$, such that $x\in n_0^\perp$, and $\of_x$ has the form $1+\sum_{k=1}^\infty c_k z^{k n_0}\in\hL^0\subset\hL$.
\end{lemma}
\begin{proof}
	This follows from \cref{prop:invertible} and \cref{lem:w-v_cluster}.
\end{proof}

\subsection{Comparison theorem} \label{sec:comparison}

In this subsection, we compare the scattering diagram $\fD_U$ of \cref{const:DU} with the one in \cite{Gross_Canonical_bases}, and then we deduce the comparison theorem of the mirror algebras.

\begin{notation} \label{nota:cluster_data}
	Let $M$ be a lattice together with an integer valued skew-symmetric form $\braket{,}$.
	Let $S'$ be a basis of $M$ and $S\subset S'$ a subset.
	This constitutes a \emph{seed} for a skew-symmetric cluster algebra of geometric type.
	The subset $S$ corresponds to the \emph{unfrozen} variables, while $S'\setminus S$ corresponds to the \emph{frozen} variables.
	Let $\cA$ be the associated Fock-Goncharov A-cluster variety (see \cite{Fock_Cluster_ensembles}).
	Let $A^\mathrm{up}\coloneqq\Gamma(\cA,\cO_\cA)$, the \emph{upper cluster algebra}.
	Let $P\subset M$ denote the strictly convex submonoid generated by $S$.
\end{notation}

\begin{assumption} \label{ass:cluster_algebra}
	We assume the following
	\begin{enumerate}
		\item \label{ass:cluster_algebra:unimodular} The skew-symmetric form $\braket{,}$ is unimodular.
		\item The upper cluster algebra $A^\mathrm{up}$ is finitely generated.
		\item $\Spec(A^\mathrm{up})$ is smooth.
	\end{enumerate}
	Although (\ref{ass:cluster_algebra:unimodular}) does not hold in general, it does hold in the principle coefficient case.
	We will deduce our results for more general cluster algebras from the principle coefficient case.
	We assume (\ref{ass:cluster_algebra:unimodular}) for the following reasons:
	\begin{enumerate}[label=(\roman*), ref=\roman*]
		\item Unimodular implies non-degeneracy, which guarantees that the all mutation equivalent seeds are coprime (see \cite[Definition 1.4, Proposition 1.8]{Berenstein_Cluster_algebras_III}).
		\item We need non-degeneracy in order to apply \cref{sec:C-walls}.
		\item Unimodular allows us to identify $M$ with its dual $N\coloneqq M^*$, so as to have compatible notations with \cite[\S 1]{Gross_Canonical_bases}.
	\end{enumerate}
\end{assumption}

We construct $U'\subset\oU'$ and $U$ as in \cref{nota:cluster_blowup} where we set $\lambda_e=1$ for all $e\in S$.

\begin{lemma} \label{lem:isomorphic_outside_codim_2}
	The canonical map $\cA \to \Spec(A^\mathrm{up})$ is an open immersion.
	The varieties $\cA$, $\Spec(A^\mathrm{up})$ and $U'$ are isomorphic outside codimension two.
	So they have the same algebras of global functions; in particular, we have $U\simeq\Spec(A^\mathrm{up})$.
	Thus \cref{ass:cluster_algebra} implies \cref{ass:cluster_blowup}.
\end{lemma}
\begin{proof}
	The canonical map $\cA \to \Spec(A^\mathrm{up})$ is an open immersion by \cite[Theorem 3.14]{Gross_Birational_geometry_of_cluster_algebras}.
	As the two varieties have the same algebras of global functions by construction, the
	complement of $\cA \subset \Spec(\Gamma(\cA,\cO_\cA))$ is of codimension at least two.
	The varieties $\cA$ and $U'$ are isomorphic outside codimension two by \cite[Theorem 3.9(2)]{Gross_Birational_geometry_of_cluster_algebras}.
	\end{proof}

\begin{remark}
	Thanks to \cref{lem:isomorphic_outside_codim_2}, readers unfamiliar with cluster algebras can take $\Gamma(U',\cO_{U'})$ as definition of the upper cluster algebra.
	Note that $U'$ has the simple blowup description, while $\cA$ requires gluing of infinitely many tori via cluster mutations.
\end{remark}

\begin{remark} \label{rem:terminology_scattering_diagram}
	As we will be comparing with \cite{Gross_Canonical_bases}, let us make a few remarks concerning the terminology of scattering diagram.
	\begin{enumerate}[leftmargin=*]
		\item For every $k>0$, we have a canonical set-theoretic inclusion $\iota_k\colon\bbZ[P]/J^k\hookrightarrow\bbZ[P]$ via the basis of monomials.
		Under this inclusion, our finite-order scattering diagram $\fD_{U,k}$ is a scattering diagram in the sense of \cite[Definition 1.6, Remark 1.5]{Gross_Canonical_bases}.
		On the other hand, our scattering diagram $\fD_U$ does not fit \cite[Definition 1.6]{Gross_Canonical_bases}.
		This is not problematic because it suffices to work at finite orders.
		Although it is possible to introduce an additional infinite polyhedral structure to make our $\fD_U$ fit \cite[Definition 1.6]{Gross_Canonical_bases}, it is artificial to do so.
		\item The notion of wall in a scattering diagram of \cite{Gross_Canonical_bases} includes the attached scattering function, which differs from our notions of wall (\cref{def:wall}) and C-wall (\cref{def:C-wall}).
		Given a wall $(\fd,g_\fd)$ in a scattering diagram,  assume $\fd\subset n^\perp$ with $n\in P\setminus 0$, and $f_\fd=1+\sum_{k=1}^\infty c_k z^{k n}\in\hL^0\subset\hL$;
		we obtain a C-wall $(\fd,n)$.
		Note that $(\fd,g_\fd)$ is incoming in the sense of \cite[Definition 1.11]{Gross_Canonical_bases} if and only if $(\fd,n)$ is incoming in the sense of \cref{def:C-wall}.
	\end{enumerate}
\end{remark}

Let $\fD_\init$ be the initial scattering diagram that consists of walls $(e^\perp,1 + z^e)$ for all $e\in S$.
Let $\fD_\GHKK\supset\fD_\init$ be the consistent scattering diagram produced by the Kontsevich-Soibelman algorithm from $\fD_\init$ as in \cite[Thereom 1.12]{Gross_Canonical_bases}.

\begin{lemma} \label{lem:incoming_walls}
	For every $k>0$, the incoming walls in $\fD_{U,k}$ are exactly the walls in $\fD_\init$.
\end{lemma}
\begin{proof}
	Let $(\fd,f_\fd)$ be an incoming wall in $\fD_{U,k}$.
	By \cref{lem:incoming_C-wall}, there exists a unique $e\in S$ such that $\fd\subset e^\perp$.
	Now it suffices to prove that for any $x\in e^\perp$ generic, we have $\of_x=1+z^e$, where $(x,\of_x)\in\fD_U$ as in \cref{const:DU}.
	So we compute $N(V_{x,v,w},\alpha)$ as in \cref{def:infinitesimal_spine}.
	Let $[C,(p_1,p_2,p_3),f\colon C\to Y^\an]$ be a stable map contributing to $N(V_{x,v,w},\alpha)$, with $f(p_1),f(p_2)\in D^\an$ and $f(p_3)\in U^\an$.
	
	\begin{claim} \label{claim:twig_straight}
		The image of any twig $T$ of $\Trop(f)$ in $M_\bbR$ is contained in $x+\bbR_{\ge 0} e$.
	\end{claim}
	
	By \cref{lem:C-twig_from_analytic_curve}, $T$ is a C-twig.
	Hence by \cref{lem:C-twig_in_C-wall}, the monomial of $T$ is a multiple of $e$.
	Recall that $S$ is a partial basis of $M$.
	Then by \cref{def:C-twig} Conditions (\ref{def:C-twig:leg}) and (\ref{def:C-twig:balancing}), the derivative of every edge of $T$ must be a multiple of $e$.
	So the claim holds.
	
	Let $\Sigma_e$ be the (incomplete) fan in $M_\bbR$ consisting of the single ray $\bbR_{\ge 0}e\subset M_\bbR$, and let $\TV(\Sigma_e)$ be the associated toric variety.
	Let $D_e\subset\TV(\Sigma_e)$ be the toric boundary, and $Z_e\subset D_e$ the subvariety given by the $z^{\braket{e,\cdot}}+1=0$.
	Let $\oU_e$ be the blowup of $\TV(\Sigma_e)$ along $Z_e$, $\partial\oU_e$ the strict transform of the toric boundary, and $U_e\subset\oU_e$ the complement of $\partial\oU_e$.
	Let $U'$ be as in \cref{nota:cluster_blowup}.
	Then $U_e\subset U'$ is exactly the complement of all exceptional divisors except the one corresponding to $e$.
	
	Let $d$ be the rank of the lattice $M$.
	Write $M\simeq\bbZ\times\bbZ\times\bbZ^{d-2}$ such that $e=(1,0,0)$ and that the projection of $v$ to the factor $\bbZ^{d-2}$ is zero.
	Since $w-v$ is equal to the sum of monomials of all twigs of $\Trop(f)$, $w-v$ is necessarily a multiple of $e$.
	Hence the projection of $w$ to the factor $\bbZ^{d-2}$ is also zero.
	
	The decomposition of $M$ induces a decomposition
	\[\TV(\Sigma_e)\simeq\bbA^1\times \Gm \times \Gm^{d-2}.\]
	Choose coordinates so that $Z_e\subset\TV(\Sigma_e)$ is given by $(0,-1)\times\Gm^{d-2}$.
	Then $\oU_e$ is isomorphic to the blowup $\Bl_{(0,-1)}(\bbA^1\times\Gm)\times\Gm^{d-2}$.
	Let $x$ be a general rigid point of $\Bl_{(0,-1)}(\bbA^1\times\Gm)^\an$, and $y$ a general rigid point of $(\Gm^{d-2})^\an$.
	
	Since $U$ and $U'$ are isomorphic outside codimension two, by \cref{lem:restriction_to_skeleton}(\ref{lem:restriction_to_skeleton:sm}), we can assume $f(C\setminus\{p_1,p_2\})\subset (U')^\an$.
	Now \cref{claim:twig_straight} implies that $f(C\setminus\{p_1,p_2\})\subset U_e^\an$.
	If we ask furthermore that $f(p_3)=(x,y)\in U_e^\an$, then the image $f(C\setminus\{p_1,p_2\})$ will lie completely in the slice $\Bl_{(0,-1)}(\bbA^1\times\Gm)^\an\times\{y\}\subset\oU_e^\an$.
	Therefore, we can reduce the computation of $N(V_{x,v,w},\alpha)$ to the two-dimensional case considered in \cite[\S 7]{Yu_Enumeration_of_holomorphic_cylinders_I}.
	We conclude from Theorem 7.1 and Remark 7.2 in loc.\ cit.\ that $\of_x=1+z^e$, completing the proof.
\end{proof}

\begin{theorem} \label{thm:scattering_diagram_comparison}
	The scattering diagrams $\fD_U$ and $\fD_\mathrm{GHKK}$ are equivalent in the sense of \cite[Definition 1.8]{Gross_Canonical_bases}.
\end{theorem}
\begin{proof}
	For every $k>0$, let $\fD_{\mathrm{GHKK},k}$ be the $k$-th order approximation of $\fD_\mathrm{GHKK}$, which is a scattering diagram for the Lie algebra $\mathfrak g^{\le k}$ as in \cite[Construction C.1]{Gross_Canonical_bases}.
	By \cref{lem:incoming_walls}, the scattering diagrams $\fD_{U,k}$ and $\fD_{\mathrm{GHKK},k}$ have the same incoming walls.
	Applying \cite[Theorem 1.12]{Gross_Canonical_bases} for the Lie algebra $\mathfrak g^{\le k}$, we see that $\fD_{U,k}$ and $\fD_{\mathrm{GHKK},k}$ are equivalent.
	Hence $\fD_U$ and $\fD_\mathrm{GHKK}$ are equivalent.
\end{proof}

\begin{theorem} \label{thm:comparison_with_GHKK_A-type}
	Let $\cA$ and $U=\Spec(A^\mathrm{up})$ be as in \cref{nota:cluster_data}.
	Let $\cA^\vee$ be the Fock-Goncharov dual, and let $\can(\cA^\vee)$ be as in \cite[Theorem 0.3]{Gross_Canonical_bases}.
	Let $A_U$ be our mirror algebra as in \cref{rem:A_U}.
	The following hold:
	\begin{enumerate}
		\item The (combinatorially defined) structure constants of \cite[Theorem 0.3(1)]{Gross_Canonical_bases} are equal to our (geometrically defined) structure constants.
		Hence they give $\can(\cA^{\vee})$ an algebra structure, equal to our mirror algebra $A_U$.
		\item The mirror algebra $\can(\cA^{\vee}) \simeq A_U$, together with its theta function basis, is independent of the cluster structure; it is canonically determined by the variety $U$.
	\end{enumerate}
\end{theorem}
\begin{proof}
	\cref{thm:scattering_diagram_comparison} gives the equivalence of scattering diagrams.
	Thus by theta function consistency (\cref{prop:fD_U_consistent}), the coefficients of monomials attached to broken lines in \cite[\S 3]{Gross_Canonical_bases} are exactly the counts in \cref{def:local_theta_funcation} (after setting all curve classes to 0).
	Since the structure constants can be equivalently defined using broken lines (see \cite[6.2]{Gross_Canonical_bases}), we conclude that the canonical map between $\can(\cA^\vee)$ and $A_U$ (by identifying their bases as abelian groups) is an isomorphism of algebras.
	Statement (2) follows from (1) and \cref{prop:structure_constants_varying_tori}.
\end{proof}

\begin{corollary} \label{cor:comparison_with_GHKK_X-type}
	Let $\cX$ be a Fock-Goncharov skew-symmetric X-cluster variety (possibly with frozen variables).
	Assume $U \coloneqq \Spec(H^0(X,\cO_\cX))$ is smooth and that the canonical map $\cX \to U$ is an open immersion.
	Let $\cX^\vee$ be the Fock-Goncharov dual, and let $\can(\cX^\vee)$ be as in \cite[Theorem 0.3]{Gross_Canonical_bases}.
	Let $A_U$ be our mirror algebra as in \cref{rem:A_U}.
	The following hold:
	\begin{enumerate}
		\item The (combinatorially defined) structure constants of \cite[Theorem 0.3(1)]{Gross_Canonical_bases} are equal to our (geometrically defined) structure constants.
		Hence they give $\can(\cX^{\vee})$ an algebra structure, equal to our mirror algebra $A_U$.
		\item The mirror algebra $\can(\cX^{\vee}) \simeq A_U$, together with its theta function basis, is independent of the cluster structure; it is canonically determined by the variety $U$.
	\end{enumerate}
\end{corollary}
\begin{proof}
	We follow the notation of \cite[\S 2]{Gross_Birational_geometry_of_cluster_algebras}.
	We have the $T_N$-principal bundle $\cA_{\prin} \to \cX$.
	Note $U \setminus \cX$ has codimension at least two, so this extends canonically to a bundle $V \to U$.
	Note $\cA_{\prin} \subset V$ is an open immersion with complement of codimension at least two, and $V\simeq\Spec(H^0(\cA_{\prin},\cO))$ is affine.
	So \cref{thm:comparison_with_GHKK_A-type} applies to $\cA_{\prin}$.
	The results for $\cX$ follow by taking $T_N$ invariants.
\end{proof}

\begin{example} \label{ex:mirror_algebra}
	We describe the mirror algebra $A$ for \cref{ex:surface}.
	It is generated by $\theta_{v_i}$, where $v_i \in \rho_i$ is the first lattice point (or equivalently, $v_i \in \Sk(U,\bbZ)$ is the divisorial valuation given by $D_i$).
	The multiplication is determined by
	\begin{align*}
		\theta_{v_i}^a \cdot \theta^b_{v_{i+1}}  &= \theta_{av_i + bv_{i+1}} \\
		\theta_{v_{i-1}} \cdot \theta_{v_{i+1}} &= z^{[D_i]}(\theta_{v_i} + z^{[E_i]})
	\end{align*}
	for $a,b \in \bbN$, and indices mod $5$, where the addition $a v_i+b v_{i+1}$ is carried out in the maximal cone of the cone complex $\Sigma_{(Y,D)}$ containing $v_i$ and $v_{i+1}$ (cf.\ \cite[Example 3.7]{Gross_Mirror_symmetry_for_log_Calabi-Yau_surfaces_I_v1}).
	
	These formulas can be deduced as follows:
	First consider the analogous formulas for the mirror algebra $A_U$ with all curve classes set to 0, as in \cref{rem:A_U}.
	Since $U$ is isomorphic to the $A_2$-cluster variety, the description of $A_U$ follows from \cref{thm:comparison_with_GHKK_A-type} and \cite[Example 1.15]{Gross_Canonical_bases}.	
	Since the nonzero structure constants in the analogous formulas for $A_U$ are either zero or one, each nonzero structure constants in the formulas for $A$ must be of the form $z^\alpha$ for some $\alpha\in\NE(Y,\bbZ)$.
	Such class $\alpha$ can be deduced by degeneration to the toric case:
	Following the notations in \cref{sec:non-degeneracy}, the monoid of definition $\DC(U\subset Y)\subset\NE(Y)$ is generated by the $D_i$, as these are exactly the classes with non-negative intersections with all the $E_i$.
	So the mirror algebra $A_\DC(U\subset Y)$ is defined over the polynomial ring $\bbZ[\DC(U \subset Y)] \simeq \bbZ[z^{[D_1]},\dots,z^{[D_5]}]$.
		Applying \cref{prop:flat_degeneration}(\ref{prop:flat_degeneration:algebra}) to $V\coloneqq T_M$, the degeneration $A_\DC(V\subset Y)$ is given by killing $[D_1]$ and $[D_2]$ (the classes with positive intersection with $E= E_1 + E_2$).
	Now the above formulas can be deduced from the toric case and the $\mu_5$-symmetry.
\end{example}

\bibliographystyle{plain}
\bibliography{dahema}

\end{document}